\documentclass[a4paper,11pt,reqno]{amsart}
\usepackage{amsthm}
\usepackage{mathtools}
\usepackage{tensor}
\usepackage{enumerate}
\usepackage{amsmath}
\usepackage{amssymb}
\usepackage{enumitem}
\usepackage{bbm}
\usepackage{tikz}
\usepackage{hyperref}
\usepackage{thmtools}
\usetikzlibrary{positioning}
\usepackage[english]{babel}
\usetikzlibrary{patterns}
\usepackage[enableskew]{youngtab}
\usepackage{cleveref}
\usepackage{tikzit}
\usepackage[textfont=normalsize, margin=25pt, skip=5pt]{caption}
\usepackage{orcidlink}

\allowdisplaybreaks

\numberwithin{equation}{section}

\theoremstyle{plain}
\newtheorem{theorem}{Theorem}[section]
\newtheorem{lemma}[theorem]{Lemma}
\newtheorem{proposition}[theorem]{Proposition}
\newtheorem{corollary}[theorem]{Corollary}

\theoremstyle{definition}
\newtheorem{definition}[theorem]{Definition}
\newtheorem{example}[theorem]{Example}

\theoremstyle{remark}
\newtheorem{remark}[theorem]{Remark}

\makeatletter
\@namedef{subjclassname@2020}{%
	\textup{2020} Mathematics Subject Classification}
\makeatother

\newcommand{\Z}{\mathbb{Z}}

\newcommand{\N}{\mathbb{N}}

\DeclareMathOperator{\sgn}{sgn}

\DeclareMathOperator{\emp}{emp}
\DeclareMathOperator{\bd}{bd}
\DeclareMathOperator{\tp}{tp}
\DeclareMathOperator{\bt}{bt}
\DeclareMathOperator{\Ms}{Ms}
\DeclareMathOperator{\ms}{ms}

\newcommand{\List}[2]{#1_1,#1_2,\dots ,#1_{#2}}

\newcommand{\dbal}{$d$-balanced }
\newcommand{\dunb}{$d$-skewed }
\newcommand{\dsbal}{$d$-shift balanced }
\newcommand{\dsunb}{$d$-shift skewed }

\tikzstyle{Excluded}=[fill={rgb,255: red,191; green,191; blue,191}, draw=none, shape=rectangle, minimum width=0.875cm, minimum height=0.875cm]
\tikzstyle{Small excluded}=[fill={rgb,255: red,191; green,191; blue,191}, draw=none, shape=rectangle, minimum width=0.375cm, minimum height=0.375cm]
\tikzstyle{Big excluded}=[fill={rgb,255: red,191; green,191; blue,191}, draw=none, shape=rectangle, minimum width=1.75cm, minimum height=1.75cm]
\tikzstyle{Without excluded}=[fill={rgb,255: red,191; green,191; blue,191}, draw=none, shape=rectangle, minimum width=0.5cm, minimum height=0.5cm]
\tikzstyle{hdomino}=[fill=none, draw=black, shape=rectangle, minimum width=1cm, minimum height=0.5cm]
\tikzstyle{vdomino}=[fill=none, draw=black, shape=rectangle, minimum width=1cm, minimum height=0.5cm, rotate=90]
\tikzstyle{hnew}=[fill={rgb,255: red,197; green,197; blue,197}, draw=black, shape=rectangle, minimum width=1cm, minimum height=0.5cm]
\tikzstyle{vnew}=[fill={rgb,255: red,197; green,197; blue,197}, draw=black, shape=rectangle, minimum width=1cm, minimum height=0.5cm, rotate=90]
\tikzstyle{Rotated none}=[fill=none, draw=none, shape=circle, rotate=90]
\tikzstyle{Numbers}=[fill=white, draw=black, shape=circle]
\tikzstyle{Empty node}=[fill=black, draw=black, shape=circle, inner sep=0cm, minimum height=0.2cm]
\tikzstyle{Bead}=[fill=white, draw=black, shape=circle, minimum height=0.4cm, inner sep=0cm]
\tikzstyle{rot45}=[shape=circle, rotate=51]
\tikzstyle{rot-45}=[shape=circle, rotate=-51]
\tikzstyle{unlabelled beads}=[fill=black, draw=black, shape=circle]
\tikzstyle{Selected}=[fill={rgb,255: red,197; green,197; blue,197}, draw=black, shape=circle]
\tikzstyle{Small Empty}=[fill=white, draw=black, shape=circle, minimum height=0.1cm, inner sep=0cm]
\tikzstyle{Lead}=[fill={rgb,255: red,197; green,197; blue,197}, draw=black, shape=circle, inner sep=0cm, minimum height=0.2cm]
\tikzstyle{mse}=[fill={rgb,255: red,197; green,197; blue,197}, draw=black, shape=rectangle]
\tikzstyle{vertical}=[fill=none, draw=none, shape=circle, rotate=90]
\tikzstyle{square}=[fill=black, draw=black, shape=rectangle]
\tikzstyle{empty square}=[fill=white, draw=black, shape=rectangle]
\tikzstyle{midsize}=[fill=none, draw=black, shape=circle, minimum height=0.2cm, inner sep=0cm]

\tikzstyle{Border edge}=[-, thick, fill=none]
\tikzstyle{measuredots}=[<->]
\tikzstyle{Grey diagram}=[-, fill={rgb,255: red,122; green,122; blue,122}, thick]
\tikzstyle{Light grey column}=[-, fill={rgb,255: red,197; green,197; blue,197}]
\tikzstyle{Extra box}=[-, fill=none, dashed]
\tikzstyle{normal green}=[-, fill={rgb,255: red,132; green,255; blue,65}]
\tikzstyle{Move it}=[->]
\tikzstyle{special column}=[-, fill={rgb,255: red,8; green,243; blue,255}]
\tikzstyle{Temp Gray}=[-, fill={rgb,255: red,191; green,191; blue,191}]
\tikzstyle{bodrer gray}=[-, fill={rgb,255: red,197; green,197; blue,197}, thick]
\tikzstyle{Finish diagram}=[->, dashed]
\tikzstyle{Background}=[-, draw=none, fill={rgb,255: red,238; green,238; blue,238}]

\title{Mullineux map: $d$-balanced partitions and $d$-runner matrices}
\author[Pavel Turek]{Pavel Turek \orcidlink{0000-0002-6190-0745}}
\address{Okinawa Institute of Science and Technology, Onna, Okinawa, 904-0495, Japan}
\email{pavel.turek@oist.jp}
\date{\today}
\begin{document}	 
\begin{abstract}
	Let $d,e>1$ be two integers. For $e$ prime, the Mullineux map $m_e$ describes tensor products of the irreducible modules of symmetric groups with the sign in characteristic $e$ as well as certain entries of decomposition matrices. Motivated by understanding new columns of decomposition matrices, we prove that if $\lambda$ is an $e$-regular partition such that $d$ divides the arm length of any rim hook of $\lambda$ of size divisible by $e$, then $m_e(\lambda)'$ is a partition such that the arm length of any of its rim hooks of size divisible by $e$ is congruent to $-1$ modulo $d$. We introduce a new parameter for partitions called the $d$-runner matrix and show that if $\lambda$ is as above, then the $d$-runner matrices of $\lambda$ and $m_e(\lambda)'$ agree. This determines $m_e(\lambda)'$ uniquely. We approach the whole problem combinatorially and take advantage of a new Abacus Mullineux Algorithm introduced in this paper. We also establish equivalent descriptions of the above partitions which provide an alternative version of the main result about the Mullineux map that becomes particularly strong when $d=2$. 
\end{abstract}
\maketitle
	
\thispagestyle{empty}

\section{Introduction}\label{se:intro}

Let $e$ be a positive integer greater than one. If $e$ is prime, then the irreducible modules of symmetric groups over a field of characteristic $e$ are labelled by $e$-regular partitions. The \textit{Mullineux map} $m_e$ describes the effect of the tensor product with the sign on these irreducible modules. More precisely, if $\lambda$ is an $e$-regular partition, then $m_e(\lambda)$ is defined by
\[
D^{\lambda}\otimes \sgn \cong D^{m_e(\lambda)},
\]
where $D^{\mu}$ is the irreducible module labelled by an $e$-regular partition $\mu$. Using an involution on Hecke algebras in quantum characteristic $e$, this algebraic definition can be generalised to all integers $e>1$; see \cite{BrundanMullineuxHecke98}.

In this paper, we establish a new, surprising behaviour of the Mullineux map applied to particular families of partitions, revealing more of its combinatorial properties. Our results are motivated by the connection between the Mullineux map and the decomposition matrices and the results of Giannelli and Wildon \cite{GiannelliWildonFoulkesandDecomposition15}; see the background section for more details. Throughout the paper, we use the following terminology.

\begin{definition}\label{de:partitions}
	We say that a rim hook of a partition is an \textit{$e$-divisible hook} if its size is divisible by $e$. For a positive integer $d$, we say that a partition $\lambda$ is
	\begin{enumerate}[label=\textnormal{(\roman*)}]
		\item \textit{\dbal}if the arm length of any $e$-divisible hook of $\lambda$ is congruent to $0$ modulo $d$;
		\item \textit{\dsbal}if the arm length of any $e$-divisible hook of $\lambda$ is congruent to $-1$ modulo $d$;
		\item \textit{\dunb}if the arm length of any $e$-divisible hook of $\lambda$ is \emph{not} congruent to $0$ modulo $d$;
		\item \textit{\dsunb}if the arm length of any $e$-divisible hook of $\lambda$ is \emph{not} congruent to $-1$ modulo $d$.  
	\end{enumerate}
\end{definition}

It is clear that any \dbal partition is \dsunb and any \dsbal partition is $d$-skewed. A surprising equivalent description of \dsunb partitions appears in \Cref{pr:equivalence max}. \Cref{pr:equivalence min} provides a similar result (with a small caveat) for \dunb partitions; see also \Cref{fig:BigPicture} for a diagrammatic summary of all the results. These two propositions are the motivation behind introducing the partitions in \Cref{de:partitions}; they relate these partitions to the sets $\mathcal{E}_{\mathcal{R}}(\gamma)$ defined below, which are `finer' versions of the sets $\mathcal{E}_k(\gamma)$ from \cite{GiannelliWildonFoulkesandDecomposition15} (see \Cref{ex:finerSets}). To define the sets $\mathcal{E}_{\mathcal{R}}(\gamma)$ and state our first main result we need the following new parameter for partitions called the $d$-runner matrix.

\begin{definition}\label{de:runner matrix}
	Let $d>1$ be an integer. The \textit{$d$-runner matrix} $\mathcal{R}_d(\lambda)$ of a partition $\lambda$ is a $(d-1)\times e$ matrix indexed by $1\leq x\leq d-1$ and $0\leq y\leq e-1$ where $\mathcal{R}_d(\lambda)_{x,y}$ counts the number of indices $i$ such that $\lambda_i \equiv x\, (\textnormal{mod } d)$ and $\lambda_i\equiv i+ y\, (\textnormal{mod } e)$. 
\end{definition}

Examples and an equivalent definition of $d$-runner matrices are presented later in \Cref{fig:Overview}, \Cref{fig:runner} and \Cref{de:runner matrix B}, respectively. Our first main result shows that the above-introduced concepts arise naturally in the context of the Mullineux map.

\begin{theorem}\label{th:Mullineuxbalanced}
	Let $d,e>1$ be two integers and $\lambda$ be an $e$-regular \dbal partition. Then $m_e(\lambda)'$ is an $e$-restricted \dsbal partition with the same $d$-runner matrix and $e$-core as $\lambda$. Moreover, this determines $m_e(\lambda)'$ uniquely.
\end{theorem}

\begin{remark}\label{re:e-regular}
	The statements that $m_e(\lambda)'$ is $e$-restricted and that $\lambda$ and $m_e(\lambda)'$ have the same $e$-core are well-known. They are included in the statement, so the `moreover' part holds. One can replace `$e$-restricted' by `$l$-restricted' where $l$ is the least common multiple of $d$ and $e$ and omit it entirely if $d$ and $e$ are coprime; see \Cref{le:e-restricted} and its preceding paragraph.
\end{remark}

The family of partitions $\lambda$ to which \Cref{th:Mullineuxbalanced} applies is rich even when we restrict to $d=2$ and odd $e$ --- such partitions are in bijections with pairs of an $e$-core partition and a $1\times e$ matrix of non-negative integers with at least one $0$; see \Cref{pr:equivalence max} (and use that $2$-balanced is the same as $2$-shift skewed), or \Cref{co:d=2} (which is a reformulation of \Cref{th:Mullineuxbalanced} for $d=2$) and its following paragraph. Computations done in GAP \cite{GAP4} suggest that the family of $e$-regular $d$-balanced partitions is rich whenever $d<e$.

The proof of \Cref{th:Mullineuxbalanced} splits into two pieces. The `moreover' part will be justified shortly using other main results of this paper. The main part of the statement is established in \Cref{se:pairs} using a new algorithm which computes $m_e(\lambda)'$ for an $e$-regular partition $\lambda$ --- the Abacus Mullineux Algorithm, which is introduced in \Cref{se:mullineux}. Unlike the existing algorithms for computing the Mullineux map, our Abacus Mullineux Algorithm is performed on the James abacus; this is advantageous for studying how $e$-divisible hooks change when the Mullineux map is applied. As it is the case for \Cref{th:Mullineuxbalanced} and the Abacus Mullineux Algorithm, rather than considering the Mullineux map $m_e$ itself, we focus on the composition of $m_e$ followed by the conjugation. This composition arises naturally in the decomposition matrices; see the second paragraph of the background section. A short example of the Abacus Mullineux Algorithm is in \Cref{fig:SmallAbacusAlgorithm}. A further example appears in \Cref{fig:MA}.

\begin{figure}[h]
	\centering

		\caption{The Abacus Mullineux Algorithm applied to $\lambda = (3,1^2)$ with $e=5$. It starts with two abaci, the first one (called $S$) displays a $\beta$-set of our partition $\lambda$, while the second (called $T$) is empty. Following the algorithm, explained in detail after \Cref{le:Ms is J}, in each step, we move all the leading beads of $S$ (coloured grey and defined in \Cref{de:leading}) up by one place and add a new bead $\ms_e(S)$ (denoted by a square and defined in \Cref{de:shift}) to $T$. Note that in all but the first step, this new bead is the bead `removed' from $S$. On the second to last line $S$ becomes a $\beta$-set of the empty partition, at which point we move all beads from $S$ to $T$ and terminate the algorithm. The resulting set $T$ is a $\beta$-set of $m_e(\lambda)'$. In our example, we see that $m_e(\lambda)' = (2,1^3)$, which is an instance of the formula $m_e(\lambda)' = (a-1,1^{e+1-a})$ for $e$-hook partitions $\lambda = (a,1^{e-a})$ with $2\leq a\leq e$.}
		\label{fig:SmallAbacusAlgorithm}
	\end{figure}

Suppose that $d,e>1$ are two integers. Given an $e$-core partition $\gamma$ and a $(d-1)\times e$ matrix of non-negative integers $\mathcal{R}$, we say that $\mathcal{R}$ is \textit{$\gamma$-realisable} if there is a partition with $e$-core $\gamma$ and $d$-runner matrix $\mathcal{R}$; see \Cref{cor:existence} and \eqref{eq:realisable} for the (slightly technical) characterisation of $\gamma$-realisable matrices which, in particular, shows that \emph{all} $\mathcal{R}$ are $\gamma$-realisable if $d$ and $e$ are coprime (which is the case in our examples and figures). If $\mathcal{R}$ is $\gamma$-realisable, we let $w_{\mathcal{R}}(\gamma)$ be the minimal $e$-weight of a partition with $e$-core $\gamma$ and $d$-runner matrix $\mathcal{R}$. We then let $\mathcal{E}_{\mathcal{R}}(\gamma)$ be the set of partitions with $e$-core $\gamma$, $d$-runner matrix $\mathcal{R}$ and $e$-weight $w_{\mathcal{R}}(\gamma)$. An example is presented later in \Cref{ex:summary}. As mentioned earlier, the motivation of this definition comes from \cite{GiannelliWildonFoulkesandDecomposition15} by Giannelli and Wildon, as explained in detail in the background section. Our second main result is as follows.

\begin{theorem}\label{th:max to min}
	Let $d,e>1$ be two integers, $\gamma$ be an $e$-core partition and $\mathcal{R}$ be a $\gamma$-realisable $(d-1)\times e$ matrix of non-negative integers with at least one $0$ in each row. The set $\mathcal{E}_{\mathcal{R}}(\gamma)$ has a unique maximal element $\lambda$ in the dominance order. Moreover, if $\lambda$ is $d$-balanced, then $m_e(\lambda)'$ is a unique minimal partition of $\mathcal{E}_{\mathcal{R}}(\gamma)$ in the dominance order. 
\end{theorem}

Again, the proof of \Cref{th:max to min} is divided into two parts. In \Cref{le:dominance} we see that the set $\mathcal{E}_{\mathcal{R}}(\gamma)$ has, in fact, a unique maximal element in the dominance order even when $\mathcal{R}$ has a row with no $0$, and, under the minor assumption that $\mathcal{R}$ contains $0$ in each of its rows, it also has a unique minimal element in the dominance order. The `moreover' part of \Cref{th:max to min} (which is, despite the label `moreover', the more substantial part of \Cref{th:max to min}) is proved below. It follows easily from \Cref{th:Mullineuxbalanced} after applying the next two surprising combinatorial results which describe the maximal and the minimal element of $\mathcal{E}_{\mathcal{R}}(\gamma)$ using partitions from \Cref{de:partitions}. The first result also justifies that the assumption in \Cref{th:max to min} that $\lambda$ is $d$-balanced is sensible (and shows that it is \emph{not} needed if $d=2$), and both results provide algorithms to find these extremal elements of $\mathcal{E}_{\mathcal{R}}(\gamma)$; see the paragraph after \Cref{le:dominance}. 

\begin{proposition}\label{pr:equivalence max}
	Let $d,e>1$ be two integers and $\lambda$ be a partition with $e$-core $\gamma$ and $d$-runner matrix $\mathcal{R}$. Then $\lambda$ is \dsunb if and only if it is the maximal element of $\mathcal{E}_{\mathcal{R}}(\gamma)$. In that case, it is $e$-regular if and only if each row of $\mathcal{R}$ contains $0$.
\end{proposition}

\begin{proposition}\label{pr:equivalence min}
	Let $d,e>1$ be two integers with least common multiple $l$ and $\lambda$ be a partition with $e$-core $\gamma$ and $d$-runner matrix $\mathcal{R}$ such that each row of $\mathcal{R}$ contains $0$. Then $\lambda$ is $l$-restricted and \dunb if and only if it is the minimal element of $\mathcal{E}_{\mathcal{R}}(\gamma)$.   
\end{proposition}

\begin{remark}\label{re:min}
	At the end of \Cref{se:unbalances}, we discuss the assumption in \Cref{pr:equivalence min}. For example, in \Cref{le:e-restricted} we show that if $d$ and $e$ are coprime, then one can omit `$l$-restricted' from the statement. We also state the complete description of all $d$-skewed partitions. 
\end{remark}

We are now ready to prove the `moreover' parts of \Cref{th:Mullineuxbalanced} and \Cref{th:max to min}. We write $l$ for the least common multiple of $d$ and $e$.

\begin{proof}[Proof of the `moreover' part of \Cref{th:max to min}]
	The `moreover' part of \Cref{pr:equivalence max} shows that $\lambda$ is $e$-regular. Using the main part of \Cref{th:Mullineuxbalanced}, $m_e(\lambda)'$ is an $e$-restricted \dsbal partition with $e$-core $\gamma$ and $d$-runner matrix $\mathcal{R}$. In particular, it is $l$-restricted and \dunb and thus it is the minimal element of $\mathcal{E}_{\mathcal{R}}(\gamma)$ by \Cref{pr:equivalence min}, as required
\end{proof}

This argument also shows that $m_e(\lambda)'$ in \Cref{th:Mullineuxbalanced} is the minimal element of $\mathcal{E}_{\mathcal{R}}(\gamma)$ (where $\gamma$ and $\mathcal{R}$ are the $e$-core and the $d$-runner matrix of $\lambda$, respectively), proving the `moreover' part of \Cref{th:Mullineuxbalanced}. We can further conclude that \Cref{th:Mullineuxbalanced} and \Cref{th:max to min} concern the same partitions: any \dbal partition is also \dsunb and hence, by \Cref{pr:equivalence max}, the maximal element in the corresponding set $\mathcal{E}_{\mathcal{R}}(\gamma)$, and it is $e$-regular if and only if $\mathcal{R}$ contains $0$ in each row.

In general there does not have to be a \dbal partition with a given $e$-core $\gamma$ and $\gamma$-realisable $d$-runner matrix $\mathcal{R}$. However, if $d=2$, then \dbal partitions and \dsunb partitions coincide (and similarly, \dsbal partitions and \dunb partitions coincide). Therefore, by \Cref{pr:equivalence max}, the maximal partition of a set $\mathcal{E}_{\mathcal{R}}(\gamma)$ is \dbal and hence \Cref{th:max to min} gives the following.

\begin{corollary}\label{co:d=2}
	Let $d=2$ and $e>1$. If $\gamma$ is an $e$-core partition, $\mathcal{R}$ is a $\gamma$-realisable $1\times e$ matrix of non-negative integers with at least one $0$ and $\lambda$ denotes the maximal element of $\mathcal{E}_{\mathcal{R}}(\gamma)$, then $m_e(\lambda)'$ is the minimal element of $\mathcal{E}_{\mathcal{R}}(\gamma)$.
\end{corollary}

We recall that if $e$ is odd, then $\gamma$-realisability condition on $\mathcal{R}$ can be removed from the statement. That is, \Cref{co:d=2} applies to any $e$-core partition $\gamma$ and $1\times e$ matrix $\mathcal{R}$ of non-negative integers with at least one $0$. 

The hierarchy of the introduced partitions is captured in \Cref{fig:BigPicture}.

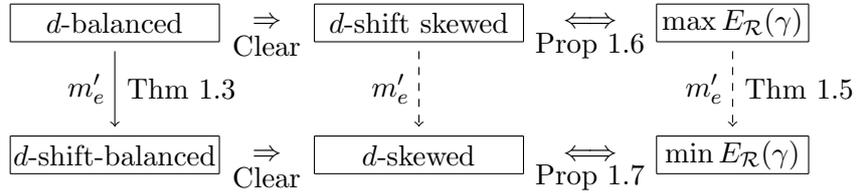
\begin{figure}[h]
	\centering
	\begin{tikzpicture}[x=0.5cm, y=0.5cm]
		\begin{pgfonlayer}{nodelayer}
			\node [style=none] (0) at (-5.75, 2.5) {$d$-balanced};
			\node [style=none] (1) at (-1.75, 2.5) {$\Rightarrow$};
			\node [style=none] (2) at (2.25, 2.5) {$d$-shift skewed};
			\node [style=none] (3) at (6.75, 2.5) {$\iff$};
			\node [style=none] (4) at (10.5, 2.5) {$\max \mathcal{E}_{\mathcal{R}}(\gamma)$};
			\node [style=none] (5) at (-5.75, -1) {$d$-shift balanced};
			\node [style=none] (6) at (-1.75, -1) {$\Rightarrow$};
			\node [style=none] (7) at (2.25, -1) {$d$-skewed};
			\node [style=none] (8) at (6.75, -1) {$\iff$};
			\node [style=none] (9) at (10.5, -1) {$\min \mathcal{E}_{\mathcal{R}}(\gamma)$};
			\node [style=none] (10) at (-5.75, 1.75) {};
			\node [style=none] (11) at (-5.75, -0.25) {};
			\node [style=none] (12) at (-6.5, 0.75) {$m_e'$};
			\node [style=none] (13) at (-1.75, 1.9) {Clear};
			\node [style=none] (14) at (-1.75, -1.6) {Clear};
			\node [style=none] (15) at (6.75, -1.6) {Prop \ref{pr:equivalence min}};
			\node [style=none] (16) at (6.75, 1.9) {Prop \ref{pr:equivalence max}};
			\node [style=none] (17) at (-4, 0.75) {Thm \ref{th:Mullineuxbalanced}};
			\node [style=none] (18) at (10.5, 1.75) {};
			\node [style=none] (19) at (10.5, -0.25) {};
			\node [style=none] (20) at (9.75, 0.75) {$m_e'$};
			\node [style=none] (21) at (12.25, 0.75) {Thm \ref{th:max to min}};
			\node [style=none] (22) at (-8.5, 2) {};
			\node [style=none] (23) at (-3, 2) {};
			\node [style=none] (24) at (-3, 3) {};
			\node [style=none] (25) at (-8.5, 3) {};
			\node [style=none] (26) at (-8.5, -0.5) {};
			\node [style=none] (27) at (-8.5, -1.5) {};
			\node [style=none] (28) at (-3, -0.5) {};
			\node [style=none] (29) at (-3, -1.5) {};
			\node [style=none] (30) at (-0.5, 3) {};
			\node [style=none] (31) at (5, 3) {};
			\node [style=none] (32) at (5, 2) {};
			\node [style=none] (33) at (-0.5, 2) {};
			\node [style=none] (34) at (-0.5, -0.5) {};
			\node [style=none] (35) at (-0.5, -1.5) {};
			\node [style=none] (36) at (5, -0.5) {};
			\node [style=none] (37) at (5, -1.5) {};
			\node [style=none] (38) at (8.5, 3) {};
			\node [style=none] (39) at (12.5, 3) {};
			\node [style=none] (40) at (12.5, 2) {};
			\node [style=none] (41) at (8.5, 2) {};
			\node [style=none] (42) at (8.5, -0.5) {};
			\node [style=none] (43) at (8.5, -1.5) {};
			\node [style=none] (44) at (12.5, -1.5) {};
			\node [style=none] (45) at (12.5, -0.5) {};
			\node [style=none] (46) at (2.25, 1.75) {};
			\node [style=none] (47) at (2.25, -0.25) {};
			\node [style=none] (48) at (1.5, 0.75) {$m_e'$};
		\end{pgfonlayer}
		\begin{pgfonlayer}{edgelayer}
			\draw [style=Move it] (10.center) to (11.center);
			\draw [style=Finish diagram] (18.center) to (19.center);
			\draw (25.center) to (24.center);
			\draw (24.center) to (23.center);
			\draw (23.center) to (22.center);
			\draw (22.center) to (25.center);
			\draw (26.center) to (28.center);
			\draw (28.center) to (29.center);
			\draw (29.center) to (27.center);
			\draw (27.center) to (26.center);
			\draw (30.center) to (31.center);
			\draw (31.center) to (32.center);
			\draw (32.center) to (33.center);
			\draw (33.center) to (30.center);
			\draw (34.center) to (35.center);
			\draw (34.center) to (36.center);
			\draw (36.center) to (37.center);
			\draw (37.center) to (35.center);
			\draw (38.center) to (39.center);
			\draw (39.center) to (40.center);
			\draw (40.center) to (41.center);
			\draw (41.center) to (38.center);
			\draw (42.center) to (45.center);
			\draw (45.center) to (44.center);
			\draw (44.center) to (43.center);
			\draw (43.center) to (42.center);
			\draw [style=Finish diagram] (46.center) to (47.center);
		\end{pgfonlayer}
	\end{tikzpicture}
	\caption{The hierarchy of the introduced partitions, where $m_e'$ is the composition of the Mullineux map $m_e$ followed by the conjugation, that is $m_e'(\lambda) = m_e(\lambda)'$. The dashed arrows show the effect of $m_e'$ if we apply it to a \dbal partition. If $d=2$, then both implications become equivalences, the dashed arrows become solid and in turn we obtain \Cref{co:d=2}. Full details are presented in our main results (and their accompanying remarks).}
	\label{fig:BigPicture}
\end{figure}

Our results are demonstrated in the following running example.

\begin{example}\label{ex:summary}
	Let $d=3$ and $e=5$.
	\begin{enumerate}[label=\textnormal{(\roman*)}]
		\item Looking at the top partition of \Cref{fig:Overview} we observe that $\lambda=(6,4,2)$ is $e$-regular and $d$-balanced. We also see that its $e$-core is $\gamma=(1^2)$ and its $d$-runner matrix is $\mathcal{R}=\big(\begin{smallmatrix}
			0 & 0 & 1 & 0 & 0\\ 
			0 & 0 & 0 & 0 & 1
		\end{smallmatrix}\big)$; see also \Cref{fig:core} and \Cref{fig:runner}. Looking at the bottom partition of \Cref{fig:Overview}, one similarly finds out that $\mu=(5,3^2,1)$ is $e$-restricted, $d$-shift balanced and has the same $e$-core and $d$-runner matrix as $\lambda$. By \Cref{th:Mullineuxbalanced}, we have $m_e(\lambda)' = \mu$.
		\item We can reach the same conclusion by applying \Cref{th:max to min}, since, by \Cref{pr:equivalence max}, we know that $\lambda$ is the maximal element of $\mathcal{E}_{\mathcal{R}}(\gamma)$. Indeed, any partition with $d$-runner matrix $\mathcal{R}$ is of size divisible by $3$. In turn any such partition with $e$-core $\gamma$ has $e$-weight congruent to $2$ modulo $3$, so $w_{\mathcal{R}}(\gamma)\geq 2$. The $e$-weight $2$ is attainable and one computationally verifies that
		\[
		\mathcal{E}_{\mathcal{R}}(\gamma) = \left\lbrace (6,4,2), (6,3,2,1), (5,4,3), (5,3^2,1) \right\rbrace; 
		\]
		see \Cref{fig:Overview} for the display of all partitions in $\mathcal{E}_{\mathcal{R}}(\gamma)$. In the dominance order, the maximal element is $\lambda=(6,4,2)$ and the minimal element is $\mu=(5,3^2,1)$. As $\lambda$ is $d$-balanced, \Cref{th:max to min} gives $m_e(\lambda)' = \mu$ (without verifying that $\mu$ is $d$-shift balanced). 
	\end{enumerate}
\end{example}

\begin{figure}[h]
	\centering
	\begin{tikzpicture}[x=0.5cm, y=0.5cm]
		\begin{pgfonlayer}{nodelayer}
			\node [style=none] (418) at (-1, 10) {};
			\node [style=none] (419) at (5, 10) {};
			\node [style=none] (420) at (5, 9) {};
			\node [style=none] (421) at (3, 9) {};
			\node [style=none] (422) at (3, 8) {};
			\node [style=none] (423) at (1, 8) {};
			\node [style=none] (424) at (1, 7) {};
			\node [style=none] (425) at (-1, 7) {};
			\node [style=none] (426) at (0, 7) {};
			\node [style=none] (427) at (-1, 8) {};
			\node [style=none] (428) at (-1, 9) {};
			\node [style=none] (429) at (0, 10) {};
			\node [style=none] (430) at (1, 10) {};
			\node [style=none] (431) at (2, 10) {};
			\node [style=none] (432) at (2, 8) {};
			\node [style=none] (433) at (4, 10) {};
			\node [style=none] (434) at (4, 9) {};
			\node [style=none] (435) at (3, 10) {};
			\node [style=none] (436) at (-1, 0) {};
			\node [style=none] (439) at (3, -1) {};
			\node [style=none] (440) at (-1, -4) {};
			\node [style=none] (441) at (2, -3) {};
			\node [style=none] (442) at (1, -3) {};
			\node [style=none] (443) at (-1, -3) {};
			\node [style=none] (444) at (0, -4) {};
			\node [style=none] (445) at (-1, -2) {};
			\node [style=none] (446) at (-1, -1) {};
			\node [style=none] (447) at (0, 0) {};
			\node [style=none] (448) at (1, 0) {};
			\node [style=none] (449) at (2, 0) {};
			\node [style=none] (450) at (2, -1) {};
			\node [style=none] (451) at (4, 0) {};
			\node [style=none] (452) at (4, -1) {};
			\node [style=none] (453) at (3, 0) {};
			\node [style=none] (472) at (2, -2) {};
			\node [style=none] (474) at (0, -3) {};
			\node [style=none] (484) at (-0.5, 7.5) {};
			\node [style=none] (485) at (0.5, 7.5) {};
			\node [style=none] (486) at (0.5, 8.25) {};
			\node [style=none] (487) at (2.5, 8.25) {};
			\node [style=none] (488) at (1.5, 8.75) {};
			\node [style=none] (489) at (2.5, 8.75) {};
			\node [style=none] (490) at (2.5, 9.5) {};
			\node [style=none] (491) at (4.5, 9.5) {};
			\node [style=none] (492) at (-0.5, -3.5) {};
			\node [style=none] (493) at (-0.5, -2.75) {};
			\node [style=none] (494) at (1.75, -2.75) {};
			\node [style=none] (495) at (1.75, -1.5) {};
			\node [style=none] (496) at (1.25, -2.25) {};
			\node [style=none] (497) at (1.25, -0.5) {};
			\node [style=none] (498) at (3.5, -0.5) {};
			\node [style=none] (499) at (-9, 4.75) {};
			\node [style=none] (500) at (-3, 4.75) {};
			\node [style=none] (501) at (-3, 3.75) {};
			\node [style=none] (502) at (-5, 3.75) {};
			\node [style=none] (504) at (-7, 2.75) {};
			\node [style=none] (505) at (-7, 1.75) {};
			\node [style=none] (506) at (-9, 1.75) {};
			\node [style=none] (507) at (-8, 1.75) {};
			\node [style=none] (508) at (-9, 2.75) {};
			\node [style=none] (509) at (-9, 3.75) {};
			\node [style=none] (510) at (-8, 4.75) {};
			\node [style=none] (511) at (-7, 4.75) {};
			\node [style=none] (512) at (-6, 4.75) {};
			\node [style=none] (513) at (-6, 2.75) {};
			\node [style=none] (514) at (-4, 4.75) {};
			\node [style=none] (515) at (-4, 3.75) {};
			\node [style=none] (516) at (-5, 4.75) {};
			\node [style=none] (517) at (-8.5, 2.25) {};
			\node [style=none] (518) at (-7.5, 2.25) {};
			\node [style=none] (519) at (-7.5, 3) {};
			\node [style=none] (521) at (-6.5, 3.5) {};
			\node [style=none] (523) at (-6.5, 4.25) {};
			\node [style=none] (524) at (-3.5, 4.25) {};
			\node [style=none] (525) at (-6, 3.75) {};
			\node [style=none] (526) at (-9, 0.75) {};
			\node [style=none] (527) at (-8, 0.75) {};
			\node [style=none] (528) at (-6.5, 3) {};
			\node [style=none] (529) at (-8.5, 1.25) {};
			\node [style=none] (530) at (7.25, 4.75) {};
			\node [style=none] (533) at (11.25, 3.75) {};
			\node [style=none] (534) at (11.25, 2.75) {};
			\node [style=none] (535) at (9.25, 2.75) {};
			\node [style=none] (536) at (9.25, 1.75) {};
			\node [style=none] (537) at (7.25, 1.75) {};
			\node [style=none] (538) at (8.25, 1.75) {};
			\node [style=none] (539) at (7.25, 2.75) {};
			\node [style=none] (540) at (7.25, 3.75) {};
			\node [style=none] (541) at (8.25, 4.75) {};
			\node [style=none] (542) at (9.25, 4.75) {};
			\node [style=none] (543) at (10.25, 4.75) {};
			\node [style=none] (544) at (10.25, 2.75) {};
			\node [style=none] (545) at (12.25, 4.75) {};
			\node [style=none] (546) at (12.25, 3.75) {};
			\node [style=none] (547) at (11.25, 4.75) {};
			\node [style=none] (548) at (7.75, 2) {};
			\node [style=none] (549) at (10, 2) {};
			\node [style=none] (550) at (10, 3) {};
			\node [style=none] (551) at (10.75, 3) {};
			\node [style=none] (552) at (9.5, 3.5) {};
			\node [style=none] (553) at (10.75, 3.5) {};
			\node [style=none] (554) at (10.75, 4.25) {};
			\node [style=none] (555) at (11.75, 4.25) {};
			\node [style=none] (556) at (10.25, 1.75) {};
			\node [style=none] (557) at (9.5, 2.5) {};
			\node [style=none] (558) at (-2.75, 8) {};
			\node [style=none] (559) at (-5.5, 5.75) {};
			\node [style=none] (560) at (-5.5, 0.5) {};
			\node [style=none] (561) at (-2.75, -1.5) {};
			\node [style=none] (562) at (5.75, -1.5) {};
			\node [style=none] (563) at (9.25, 0.75) {};
			\node [style=none] (564) at (9.25, 5.75) {};
			\node [style=none] (565) at (5.75, 7.75) {};
			\node [style=none] (566) at (-1.5, 8.5) {$\scriptstyle 1,2$};
			\node [style=none] (567) at (-1.5, 7.5) {$\scriptstyle 2,4$};
			\node [style=none] (568) at (-9.5, 2.25) {$\scriptstyle 2,4$};
			\node [style=none] (569) at (-1.5, -0.5) {$\scriptstyle 2,4$};
			\node [style=none] (570) at (6.75, 4.25) {$\scriptstyle 2,4$};
			\node [style=none] (571) at (-9.5, 1.25) {$\scriptstyle 1,2$};
			\node [style=none] (572) at (-1.5, -3.5) {$\scriptstyle 1,2$};
			\node [style=none] (573) at (6.75, 3.25) {$\scriptstyle 1,2$};
		\end{pgfonlayer}
		\begin{pgfonlayer}{edgelayer}
			\draw (418.center) to (419.center);
			\draw (428.center) to (420.center);
			\draw (427.center) to (422.center);
			\draw (425.center) to (424.center);
			\draw (418.center) to (425.center);
			\draw (429.center) to (426.center);
			\draw (430.center) to (424.center);
			\draw (431.center) to (432.center);
			\draw (435.center) to (422.center);
			\draw (433.center) to (434.center);
			\draw (419.center) to (420.center);
			\draw (436.center) to (451.center);
			\draw (446.center) to (452.center);
			\draw (445.center) to (472.center);
			\draw (443.center) to (441.center);
			\draw (440.center) to (444.center);
			\draw (436.center) to (440.center);
			\draw (447.center) to (444.center);
			\draw (448.center) to (442.center);
			\draw (449.center) to (441.center);
			\draw (453.center) to (439.center);
			\draw (451.center) to (452.center);
			\draw [style=Extra box] (484.center) to (485.center);
			\draw [style=Extra box] (485.center) to (486.center);
			\draw [style=Extra box] (486.center) to (487.center);
			\draw [style=Extra box] (488.center) to (489.center);
			\draw [style=Extra box] (489.center) to (490.center);
			\draw [style=Extra box] (490.center) to (491.center);
			\draw [style=Extra box] (492.center) to (493.center);
			\draw [style=Extra box] (493.center) to (494.center);
			\draw [style=Extra box] (494.center) to (495.center);
			\draw [style=Extra box] (496.center) to (497.center);
			\draw [style=Extra box] (497.center) to (498.center);
			\draw (499.center) to (500.center);
			\draw (509.center) to (501.center);
			\draw (508.center) to (513.center);
			\draw (506.center) to (505.center);
			\draw (526.center) to (527.center);
			\draw (499.center) to (526.center);
			\draw (510.center) to (527.center);
			\draw (511.center) to (505.center);
			\draw (512.center) to (513.center);
			\draw (516.center) to (502.center);
			\draw (514.center) to (515.center);
			\draw (500.center) to (501.center);
			\draw [style=Extra box] (524.center) to (523.center);
			\draw [style=Extra box] (523.center) to (521.center);
			\draw [style=Extra box] (528.center) to (519.center);
			\draw [style=Extra box] (519.center) to (518.center);
			\draw [style=Extra box] (518.center) to (517.center);
			\draw [style=Extra box] (517.center) to (529.center);
			\draw (539.center) to (534.center);
			\draw (530.center) to (537.center);
			\draw (541.center) to (538.center);
			\draw (542.center) to (536.center);
			\draw (547.center) to (534.center);
			\draw (545.center) to (546.center);
			\draw [style=Extra box] (548.center) to (549.center);
			\draw [style=Extra box] (549.center) to (550.center);
			\draw [style=Extra box] (550.center) to (551.center);
			\draw [style=Extra box] (552.center) to (553.center);
			\draw [style=Extra box] (553.center) to (554.center);
			\draw [style=Extra box] (554.center) to (555.center);
			\draw (530.center) to (545.center);
			\draw (540.center) to (546.center);
			\draw (543.center) to (556.center);
			\draw (537.center) to (556.center);
			\draw [style=Extra box] (552.center) to (557.center);
			\draw [style=Border edge] (559.center) to (558.center);
			\draw [style=Border edge] (560.center) to (561.center);
			\draw [style=Border edge] (562.center) to (563.center);
			\draw [style=Border edge] (565.center) to (564.center);
		\end{pgfonlayer}
	\end{tikzpicture}
	\caption{The Hasse diagram of partitions in $\mathcal{E}_{\mathcal{R}}(\gamma)$ with $\gamma=(1^2)$ and $\mathcal{R}=\big(\begin{smallmatrix}
			0 & 0 & 1 & 0 & 0\\ 
			0 & 0 & 0 & 0 & 1
	\end{smallmatrix}\big)$ from \Cref{ex:summary}. Their $e$-divisible hooks are denoted with dashed lines. For each row, we wrote numbers $x,y$ in front of it if the row contributes by $1$ to the $(x,y)$-entry of the $d$-runner matrix of the corresponding partition. Thus we see that all displayed partitions have $d$-runner matrix equal to $\mathcal{R}$. One easily checks that all displayed partitions have $e$-core equal to $\gamma$.}
	\label{fig:Overview}
\end{figure}
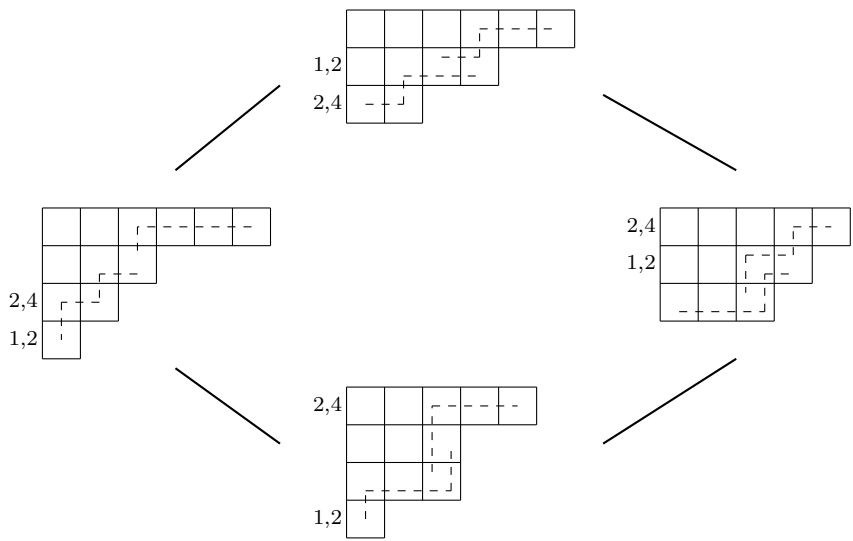

\subsection*{Background}\label{su:background}

In 1979 Mullineux \cite{MullineuxMap79} introduced a combinatorial map $m_e$ from the set of $e$-regular partitions of size $n$ to itself. He conjectured that if $e$ is a prime and $\lambda$ is an $e$-regular partition of $n$, then over a field of characteristic $e$, $m_e(\lambda)$ labels the irreducible module of the symmetric group $S_n$ isomorphic to $D^{\lambda}\otimes\sgn$. This long-standing conjecture was finally proved by Ford and Kleshchev \cite{FordKleschevMullineux97} in 1997, using Kleshchev's modular branching rules \cite{KleshchevBranchingI95, KleshchevBranchingII95, KleshchevBranchingIII96}. A slight reformulation of the Mullineux map $m_e$ was then used to find irreducible modules of the alternating group in odd characteristic \cite{FordAlternating97}. Since then, different combinatorial algorithms for computing $m_e(\lambda)$ have been found. Among others, there are algorithms by Xu \cite{XuMullineux97}, Brundan and Kujawa \cite{BrundanKujawaMullineux03}, Fayers \cite{FayersMullineux22} and Jacon \cite{JaconMullineux23}. The algorithm we use in this paper is based on the map $J$ defined by Xu \cite{XuMullineux97} (a shorter proof that the algorithm of Xu coincides with the composition of the Mullineux map followed by the conjugation is in \cite{XupSeries99}). In 1998 Brundan \cite{BrundanMullineuxHecke98} found an algebraic interpretation of the Mullineux map for any integer $e>1$ using an involution of the Hecke algebra in quantum characteristic $e$.

The Mullineux map is closely tied to the decomposition numbers of symmetric groups and Hecke algebras --- if the rows on the decomposition matrix are ordered by any linear extension of the dominance order, then the partition $m_e(\lambda)'$ labels the bottommost row with a non-zero entry in the column labelled by $\lambda$, and this entry is $1$ --- and has been used as a tool to finding decomposition numbers; for instance, the combinatorial formula for the decomposition numbers in blocks of weight $2$ due to Richards \cite{RichardsDecomposition96} is proved with an aid of the Mullineux map.

Our results are motivated by this connection to decomposition numbers. Namely, the definition of $\mathcal{E}_{\mathcal{R}}(\gamma)$ is inspired by the sets $\mathcal{E}_{k}(\gamma)$ introduced by Giannelli and Wildon \cite{GiannelliWildonFoulkesandDecomposition15} (defined for $d=2$), which can be partitioned by our `finer' sets $\mathcal{E}_{\mathcal{R}}(\gamma)$. In the paper, Giannelli and Wildon showed that under mild conditions there are columns of the decomposition matrix of a symmetric group whose sum is a vector whose non-zero entries are one and lie in the rows labelled by the partitions in $\mathcal{E}_{k}(\gamma)$. They further analyse the set $\mathcal{E}_{0}(\gamma)$ in \cite{GiannelliWildonIndecomposable16} with particular focus on the cases when the partitions in $\mathcal{E}_{0}(\gamma)$ have $e$-weight at most $2$.

The results of Giannelli and Wildon are already impressive, however, the columns of the decomposition matrix that their results concern and their explicit descriptions are not found. The sets $\mathcal{E}_{\mathcal{R}}(\gamma)$ introduced in this paper aim to both complete and vastly generalise the results by Giannelli and Wildon by describing these (and other) columns --- this is a current work in progress of the author and David Hemmer (for $d=2$), and Bim Gustavsson, Stacey Law, Lorenzo Putignano and Liron Speyer (for $d>2$), that builds on the results of this paper. We demonstrate the potential of sets $\mathcal{E}_{\mathcal{R}}(\gamma)$ in the following example.

\begin{example}\label{ex:finerSets}
	Let $e>1$ be odd. Analogously to our definition of $w_{\mathcal{R}}(\gamma)$ and $\mathcal{E}_{\mathcal{R}}(\gamma)$ (for $d=2$ and odd $e>1$), in \cite{GiannelliWildonFoulkesandDecomposition15} the number $w_k(\gamma)$ is defined as the minimal $e$-weight of a partition with $e$-core $\gamma$ and $k$ odd parts and the set $\mathcal{E}_k(\gamma)$ is defined as the set of partitions with $e$-core $\gamma$, $k$ odd parts and $e$-weight $w_k(\gamma)$. When $\mathcal{R}$ is the one-row $2$-runner matrix consisting of only zeros, a partition has $2$-runner matrix $\mathcal{R}$ if and only if all its parts are even; thus, $w_{\mathcal{R}}(\gamma) = w_0(\gamma)$ and $\mathcal{E}_{\mathcal{R}}(\gamma)=\mathcal{E}_0(\gamma)$. By \Cref{co:d=2} and the relation between the Mullineux map and decomposition matrices, the column of the decomposition matrix labelled by the maximal element of $\mathcal{E}_{\mathcal{R}}(\gamma)$ has the bottommost non-zero entry in the row labelled by the minimal element of $\mathcal{E}_{\mathcal{R}}(\gamma)$, and this entry is one.
	
	Looking at the first part of \cite[Example~6.2]{GiannelliWildonFoulkesandDecomposition15}, for $e=3$ and $\gamma=(3,1^2)$ we have
	\[
	\mathcal{E}_{\mathcal{R}}(\gamma)=\mathcal{E}_0(\gamma) = \left\lbrace (8, 4, 2), (6^2, 2), (6, 4^2),(6, 4, 2^2)\right\rbrace. 
	\]
	Furthermore, in the same example it is concluded that the column of the decomposition matrix labelled by $(8,4,2)$, the maximal partition of the set $\mathcal{E}_{\mathcal{R}}(\gamma)$, contains ones in rows labelled by the partitions in $\mathcal{E}_{\mathcal{R}}(\gamma)$ and zeros elsewhere; this agrees with \Cref{co:d=2}.
	
	More generally, a partition has $k$ odd parts if and only if the entries of its $2$-runner matrix add up to $k$. Untangling the definitions, we conclude that for any $2$-runner matrix $\mathcal{R}$ with the sum of entries equal to $k$, we have $w_{\mathcal{R}}(\gamma) \geq w_k(\gamma)$ and
	\[
	\mathcal{E}_k(\gamma) = \bigsqcup_{\mathcal{R}} \mathcal{E}_{\mathcal{R}}(\gamma)
	\]
	where the disjoint union is taken over all such $\mathcal{R}$ with $w_{\mathcal{R}}(\gamma) = w_k(\gamma)$.
	
	Moving to the second part of \cite[Example~6.2]{GiannelliWildonFoulkesandDecomposition15}, the set $\mathcal{E}_6(\gamma)$ with $e=7$ and $\gamma = (4^3)$ is the union of sets
	\[
	\mathcal{X} = \left\lbrace (11, 4^2, 3, 1^4), (11, 4^2, 2, 1^5), (10, 5, 4, 3, 1^4), (10, 5, 4, 2, 1^5) \right\rbrace, 
	\]
	and
	\[
	\mathcal{X}' = \left\lbrace (9, 5^3, 1^2),(9, 5^2, 4, 1^3),(8, 5^3, 1^3) \right\rbrace.
	\]
	One can verify that these sets coincide with $\mathcal{E}_{\mathcal{R}}(\gamma)$ for suitable $\mathcal{R}$ (with $d=2$). Namely, $\mathcal{R} = (1\; 1\; 1\; 2\; 0\; 0\; 1)$ for $\mathcal{X}$ and $\mathcal{R} = (0\; 2\; 2\; 2\; 0\; 0\; 0)$ for $\mathcal{X}'$. Using the conclusion of the second part of \cite[Example~6.2]{GiannelliWildonFoulkesandDecomposition15}, the columns of the decomposition matrix labelled by the maximal partitions of these two sets $\mathcal{E}_{\mathcal{R}}(\gamma)$ contain ones in rows labelled by the partitions in the corresponding set $\mathcal{E}_{\mathcal{R}}(\gamma)$ and zeros elsewhere. Notice that this is again in accordance with \Cref{co:d=2}, which tells us that the bottommost ones in these columns lie in the rows labelled by the minimal element of the corresponding set $\mathcal{E}_{\mathcal{R}}(\gamma)$. 
\end{example}

Results like those presented in this paper are rare --- most known results about the Mullineux map concern its general properties rather than its behaviour when applied to particular partitions. An exception to this is the result by Paget \cite{PagetMullineux06} who proves a simple behaviour of the Mullineux map when applied to partitions in RoCK blocks. Paget's algebraic proof was later accompanied by a combinatorial proof by Fayers \cite{FayersMullineux22} who established a more general runner removal result for the Mullineux map applicable to partitions in the RoCK blocks.

\subsection*{Outline}

We recall the notation, terminology and results about partitions and the James abacus in \Cref{se:pre}. We end the section by describing $\gamma$-realisable matrices. In \Cref{se:mullineux} we describe the Abacus Mullineux Algorithm and prove that it has the same effect as the algorithm from \cite{XuMullineux97}. Thus for an $e$-regular partition $\lambda$ as an input it returns $m_e(\lambda)'$. Several results about \dbal partitions are proved in \Cref{se:balanced}. In \Cref{se:pairs} we introduce $d$-combined pairs and use them to prove \Cref{th:Mullineuxbalanced}. Finally, in \Cref{se:unbalances} we prove \Cref{pr:equivalence max} and \Cref{pr:equivalence min}. This is done by using two new algorithms A1 and A2 of independent interest performed on the James abacus. Note that \Cref{se:unbalances} can be read immediately after \Cref{se:pre}. 

\section{Preliminaries}\label{se:pre}

Throughout, $e>1$ is a fixed positive integer. Let $\lambda=(\List{\lambda}{t})$ be a partition. We denote by $|\lambda|$ its \textit{size} given by $\sum_{i=1}^t \lambda_i$. We say that $\lambda$ is \textit{$e$-regular} if there are no $e$ equal parts of $\lambda$. The \textit{Young diagram} of $\lambda$ is the set $Y(\lambda) = \left\lbrace (i,j)\in \N^2 : i\leq t,\: j\leq \lambda_i \right\rbrace $. The \textit{conjugate partition} $\lambda'$ is the partition such that $Y(\lambda') = \left\lbrace (i,j) : (j,i)\in Y(\lambda) \right\rbrace$. We say that $\lambda$ is \textit{$e$-restricted} if $\lambda'$ is $e$-regular. Equivalently, $\lambda$ is $e$-restricted if for all $i\leq t$ we have $\lambda_i - \lambda_{i+1}\leq e-1$ (where we let $\lambda_{t+1}=0$).

Given any $(i,j)\in Y(\lambda)$ the \textit{rim hook $R_{i,j}$} of $\lambda$ is defined as the set $\left\lbrace (i', j')\in Y(\lambda) : i'\geq i,\: j'\geq j,\: (i'+1,j'+1)\notin Y(\lambda) \right\rbrace$. We refer to the rim hook $R_{1,1}$ as the \textit{rim} of $\lambda$. A rim hook is called \textit{$e$-hook} if its size is $e$ and, as defined in \Cref{de:partitions}, is called $e$-divisible hook if its size is divisible by $e$. We define the \textit{arm length} and \textit{leg length} of a rim hook $R_{i,j}$ by $\lambda_i-j$ and $\lambda'_j - i$, respectively. It follows that the sum of the arm length and the leg length of a rim hook equals its size decreased by $1$. We also define the \emph{top} and \emph{bottom} of a rim hook $R$ as the least and the greatest $i$ such that $(i,j)\in R$ for some $j$, respectively. We denote them by $\tp(R)$ and $\bt(R)$, respectively. See \Cref{fig:Young} for an example.

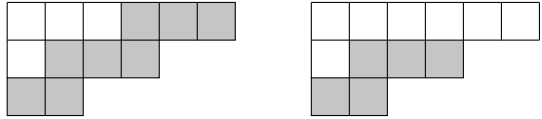
\begin{figure}[h]
	\centering
	\begin{tikzpicture}[x=0.5cm, y=0.5cm]
		\begin{pgfonlayer}{nodelayer}
			\node [style=none] (0) at (-6, 3) {};
			\node [style=none] (1) at (0, 3) {};
			\node [style=none] (2) at (0, 2) {};
			\node [style=none] (3) at (-2, 2) {};
			\node [style=none] (4) at (-2, 1) {};
			\node [style=none] (5) at (-4, 1) {};
			\node [style=none] (6) at (-4, 0) {};
			\node [style=none] (7) at (-6, 0) {};
			\node [style=none] (8) at (-6, 1) {};
			\node [style=none] (9) at (-6, 2) {};
			\node [style=none] (10) at (-5, 0) {};
			\node [style=none] (11) at (-5, 3) {};
			\node [style=none] (12) at (-4, 3) {};
			\node [style=none] (13) at (-3, 3) {};
			\node [style=none] (14) at (-2, 3) {};
			\node [style=none] (15) at (-1, 3) {};
			\node [style=none] (16) at (-1, 2) {};
			\node [style=none] (17) at (-3, 1) {};
			\node [style=none] (18) at (-3, 2) {};
			\node [style=none] (19) at (-5, 2) {};
			\node [style=none] (20) at (-5, 1) {};
			\node [style=none] (21) at (2, 3) {};
			\node [style=none] (22) at (8, 3) {};
			\node [style=none] (23) at (8, 2) {};
			\node [style=none] (24) at (6, 2) {};
			\node [style=none] (25) at (6, 1) {};
			\node [style=none] (26) at (4, 1) {};
			\node [style=none] (27) at (4, 0) {};
			\node [style=none] (28) at (2, 0) {};
			\node [style=none] (29) at (2, 1) {};
			\node [style=none] (30) at (2, 2) {};
			\node [style=none] (31) at (3, 0) {};
			\node [style=none] (32) at (3, 3) {};
			\node [style=none] (33) at (4, 3) {};
			\node [style=none] (34) at (5, 3) {};
			\node [style=none] (35) at (6, 3) {};
			\node [style=none] (36) at (7, 3) {};
			\node [style=none] (37) at (7, 2) {};
			\node [style=none] (38) at (5, 1) {};
			\node [style=none] (39) at (5, 2) {};
			\node [style=none] (40) at (3, 2) {};
			\node [style=none] (41) at (3, 1) {};
			\node [style=none] (42) at (-1, 2) {};
		\end{pgfonlayer}
		\begin{pgfonlayer}{edgelayer}
			\draw [style=Light grey column] (2.center)
			to (3.center)
			to (4.center)
			to (5.center)
			to (6.center)
			to (7.center)
			to (8.center)
			to (20.center)
			to (19.center)
			to (18.center)
			to (13.center)
			to (1.center)
			to cycle;
			\draw (0.center) to (13.center);
			\draw (0.center) to (8.center);
			\draw (11.center) to (19.center);
			\draw (20.center) to (10.center);
			\draw (20.center) to (5.center);
			\draw (5.center) to (12.center);
			\draw (18.center) to (3.center);
			\draw (9.center) to (19.center);
			\draw (14.center) to (3.center);
			\draw (18.center) to (17.center);
			\draw (15.center) to (42.center);
			\draw [style=Light grey column] (27.center)
			to (28.center)
			to (29.center)
			to (41.center)
			to (40.center)
			to (24.center)
			to (25.center)
			to (26.center)
			to cycle;
			\draw (29.center) to (21.center);
			\draw (21.center) to (22.center);
			\draw (22.center) to (23.center);
			\draw (23.center) to (24.center);
			\draw (36.center) to (37.center);
			\draw (35.center) to (24.center);
			\draw (34.center) to (38.center);
			\draw (33.center) to (26.center);
			\draw (32.center) to (40.center);
			\draw (41.center) to (31.center);
			\draw (30.center) to (40.center);
			\draw (41.center) to (26.center);
		\end{pgfonlayer}
	\end{tikzpicture}
	\caption{These are two copies of the Young diagram of partition $\lambda = (6,4,2)$. The highlighted boxes in the left diagram form the rim of $\lambda$, while in the right diagram they form the $5$-hook $R_{2,1}$ with top $2$, bottom $3$, arm length $3$ and leg length $1$. Note that $\lambda$ has one more $5$-hook, namely, $R_{1,3}$. It has top $1$, bottom $2$, arm length $3$ and leg length $1$.}
	\label{fig:Young}
\end{figure}

We say a partition is an \textit{$e$-core} if it does not have any $e$-hooks. Given a partition $\lambda$ and its $e$-hook $R_{i,j}$, one can \textit{remove} $R_{i,j}$ to obtain a new partition with Young diagram $Y(\lambda)\setminus R_{i,j}$. Repeating this procedure yields a partition $\gamma$ with no $e$-hooks, the \textit{$e$-core of $\lambda$} and we define the \textit{$e$-weight of $\lambda$} to be the number of $e$-hooks removed from $\lambda$ during this procedure. As we remark after \Cref{ex:empty}, the $e$-core of $\lambda$ and, consequently, the $e$-weight of $\lambda$ are well-defined. An example of the $e$-core and the $e$-weight of a partition is in \Cref{fig:core}.

The length of a partition $\lambda$, denoted by $\ell(\lambda)$ equals the number of parts of $\lambda$. The \textit{dominance order} on the set of partitions of size $n$ is defined as $\mu \unlhd \lambda$ if $\ell(\lambda)\leq \ell(\mu)$ and for all $i\leq \ell(\lambda)$ we have $\sum_{j=1}^i \mu_j \leq \sum_{j=1}^i \lambda_j$. We will later need a simple fact that if $\lambda$ and $\mu$ have equal sizes and $R$ and $R'$ are rim hooks of $\lambda$ and $\mu$, respectively, such that the removal of $R$ from $\lambda$ results in the same partition as does the removal of $R'$ from $\mu$, then $\mu\unlhd \lambda$ if $\bt(R) < \bt(R')$. The dominance order can be refined to the usual lexicographical order, which is a total order on the set of partitions.

\begin{example}\label{ex:dominance}
	In the dominance order, the partition $(5,4,3)$ is less than $(6,4,2)$ and is incomparable with $(6,3,2,1)$. However, in the lexicographic order $(5,4,3)$ is less than $(6,3,2,1)$.
\end{example}

Following \cite[Section~2.7]{JamesKerberSymmetric81}, one can think about partitions using the James abacus: a \textit{set of $\beta$-numbers} (or a \textit{$\beta$-set}, for short) of a partition $\lambda = (\List{\lambda}{t})$ is a set $B\subset \Z$ of the form $B_s(\lambda):=\left\lbrace \lambda_i-i+s : i\in \N \right\rbrace $ for some $s\in\Z$ (where $\lambda_i$ is taken to be $0$ for $i>t$). We call the set $B_0(\lambda)$ the \textit{canonical $\beta$-set} of $\lambda$. We note that all our work can be carried out only with the canonical $\beta$-sets but it is often convenient to pick a suitable $s$ such that $B_s(\lambda)$, for instance, contains all non-positive integers.

We refer to the elements of a $\beta$-set $B$ of a partition (and more generally, of any subset $B\subseteq \Z$) as \textit{beads} and to the integers outside $B$ as \textit{empty spaces}. If a bead $b$ equals $ke+r$ with $k\in \Z$ and $r\in \left\lbrace 0,1,\dots, e-1 \right\rbrace $ we say that $b$ lies \textit{in row $k$ and on runner $r$} (\textit{of the James abacus with $e$ runners}). See \Cref{fig:abacus} for an intuitive description of these terms. This definition trivially adapts to James abaci with any given number of runners. However, in this paper we will only need $e$ runners and occasionally $ae$ runners with $a$ a positive integer. Thus, if not said otherwise, we work on the James abacus with $e$ runners.

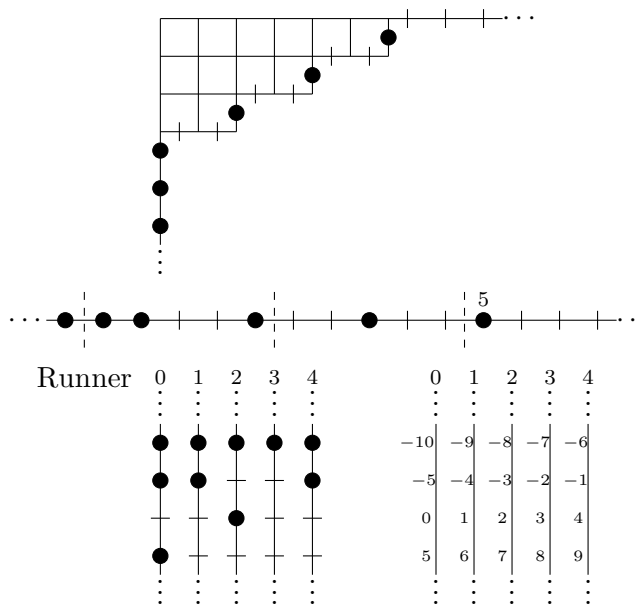
\begin{figure}[h]
	\centering
	\begin{tikzpicture}[x=0.5cm, y=0.5cm]
		\begin{pgfonlayer}{nodelayer}
			\node [style=none] (0) at (0.5, 3) {};
			\node [style=none] (9) at (0.5, 2) {};
			\node [style=none] (10) at (1.5, 0) {};
			\node [style=none] (11) at (1.5, 3) {};
			\node [style=none] (12) at (2.5, 3) {};
			\node [style=none] (14) at (4.5, 3) {};
			\node [style=none] (15) at (5.5, 3) {};
			\node [style=none] (16) at (5.5, 2) {};
			\node [style=none] (17) at (3.5, 1) {};
			\node [style=none] (21) at (5.5, 2) {};
			\node [style=none] (22) at (2.5, 0) {};
			\node [style=none] (23) at (0.5, 0) {};
			\node [style=none] (24) at (2.5, 1) {};
			\node [style=none] (25) at (4.5, 1) {};
			\node [style=none] (26) at (4.5, 2) {};
			\node [style=none] (27) at (6.5, 2) {};
			\node [style=none] (28) at (6.5, 3) {};
			\node [style=none] (29) at (3.5, 3) {};
			\node [style=none] (30) at (0.5, 1) {};
			\node [style=none] (31) at (9.5, 3) {};
			\node [style=none] (32) at (10, 3) {$\cdots$};
			\node [style=none] (33) at (0.5, -3) {};
			\node [style=none] (34) at (0.5, -3.25) {$\vdots$};
			\node [style=Empty node] (35) at (6.5, 2.5) {};
			\node [style=Empty node] (36) at (4.5, 1.5) {};
			\node [style=Empty node] (37) at (2.5, 0.5) {};
			\node [style=Empty node] (38) at (0.5, -0.5) {};
			\node [style=Empty node] (39) at (0.5, -1.5) {};
			\node [style=Empty node] (40) at (0.5, -2.5) {};
			\node [style=none] (41) at (-2.5, -5) {};
			\node [style=none] (42) at (12.5, -5) {};
			\node [style=Empty node] (43) at (-2, -5) {};
			\node [style=Empty node] (44) at (-1, -5) {};
			\node [style=Empty node] (45) at (0, -5) {};
			\node [style=Empty node] (46) at (3, -5) {};
			\node [style=Empty node] (47) at (6, -5) {};
			\node [style=Empty node] (48) at (9, -5) {};
			\node [style=none] (49) at (1, 0.25) {};
			\node [style=none] (50) at (1, -0.25) {};
			\node [style=none] (51) at (2, 0.25) {};
			\node [style=none] (52) at (2, -0.25) {};
			\node [style=none] (53) at (3, 1.25) {};
			\node [style=none] (54) at (3, 0.75) {};
			\node [style=none] (55) at (4, 1.25) {};
			\node [style=none] (56) at (4, 0.75) {};
			\node [style=none] (57) at (6, 2.25) {};
			\node [style=none] (58) at (6, 1.75) {};
			\node [style=none] (59) at (5, 2.25) {};
			\node [style=none] (60) at (5, 1.75) {};
			\node [style=none] (61) at (7, 3.25) {};
			\node [style=none] (62) at (7, 2.75) {};
			\node [style=none] (63) at (8, 3.25) {};
			\node [style=none] (64) at (8, 2.75) {};
			\node [style=none] (65) at (9, 3.25) {};
			\node [style=none] (66) at (9, 2.75) {};
			\node [style=none] (67) at (1, -4.75) {};
			\node [style=none] (68) at (1, -5.25) {};
			\node [style=none] (69) at (2, -4.75) {};
			\node [style=none] (70) at (2, -5.25) {};
			\node [style=none] (73) at (4, -4.75) {};
			\node [style=none] (74) at (4, -5.25) {};
			\node [style=none] (75) at (5, -4.75) {};
			\node [style=none] (76) at (5, -5.25) {};
			\node [style=none] (77) at (7, -4.75) {};
			\node [style=none] (78) at (7, -5.25) {};
			\node [style=none] (79) at (8, -4.75) {};
			\node [style=none] (80) at (8, -5.25) {};
			\node [style=none] (81) at (10, -4.75) {};
			\node [style=none] (82) at (10, -5.25) {};
			\node [style=none] (83) at (11, -4.75) {};
			\node [style=none] (84) at (11, -5.25) {};
			\node [style=none] (85) at (12, -4.75) {};
			\node [style=none] (86) at (12, -5.25) {};
			\node [style=none] (87) at (-3, -5) {$\cdots$};
			\node [style=none] (88) at (13, -5) {$\cdots$};
			\node [style=none] (89) at (0.5, -7.75) {};
			\node [style=none] (90) at (0.5, -7) {$\vdots$};
			\node [style=none] (91) at (1.5, -7.75) {};
			\node [style=none] (92) at (1.5, -7) {$\vdots$};
			\node [style=none] (93) at (2.5, -7.75) {};
			\node [style=none] (94) at (2.5, -7) {$\vdots$};
			\node [style=none] (95) at (3.5, -7.75) {};
			\node [style=none] (96) at (3.5, -7) {$\vdots$};
			\node [style=none] (97) at (4.5, -7.75) {};
			\node [style=none] (98) at (4.5, -7) {$\vdots$};
			\node [style=none] (99) at (0.5, -12) {$\vdots$};
			\node [style=none] (100) at (1.5, -12) {$\vdots$};
			\node [style=none] (101) at (0.5, -11.75) {};
			\node [style=none] (102) at (1.5, -11.75) {};
			\node [style=none] (103) at (2.5, -11.75) {};
			\node [style=none] (104) at (2.5, -12) {$\vdots$};
			\node [style=none] (105) at (3.5, -12) {$\vdots$};
			\node [style=none] (106) at (3.5, -11.75) {};
			\node [style=none] (107) at (4.5, -11.75) {};
			\node [style=none] (108) at (4.5, -12) {$\vdots$};
			\node [style=Empty node] (109) at (0.5, -8.25) {};
			\node [style=Empty node] (110) at (1.5, -8.25) {};
			\node [style=Empty node] (111) at (2.5, -8.25) {};
			\node [style=Empty node] (112) at (3.5, -8.25) {};
			\node [style=Empty node] (113) at (4.5, -8.25) {};
			\node [style=Empty node] (114) at (0.5, -9.25) {};
			\node [style=Empty node] (115) at (1.5, -9.25) {};
			\node [style=Empty node] (116) at (4.5, -9.25) {};
			\node [style=Empty node] (117) at (2.5, -10.25) {};
			\node [style=Empty node] (118) at (0.5, -11.25) {};
			\node [style=none] (119) at (2.25, -9.25) {};
			\node [style=none] (120) at (2.75, -9.25) {};
			\node [style=none] (121) at (3.25, -9.25) {};
			\node [style=none] (122) at (3.75, -9.25) {};
			\node [style=none] (123) at (3.25, -10.25) {};
			\node [style=none] (124) at (3.75, -10.25) {};
			\node [style=none] (125) at (4.25, -10.25) {};
			\node [style=none] (126) at (4.75, -10.25) {};
			\node [style=none] (127) at (4.75, -11.25) {};
			\node [style=none] (128) at (4.25, -11.25) {};
			\node [style=none] (129) at (3.75, -11.25) {};
			\node [style=none] (130) at (3.25, -11.25) {};
			\node [style=none] (131) at (2.75, -11.25) {};
			\node [style=none] (132) at (2.25, -11.25) {};
			\node [style=none] (133) at (1.75, -11.25) {};
			\node [style=none] (134) at (1.25, -11.25) {};
			\node [style=none] (135) at (1.75, -10.25) {};
			\node [style=none] (136) at (1.25, -10.25) {};
			\node [style=none] (137) at (0.75, -10.25) {};
			\node [style=none] (138) at (0.25, -10.25) {};
			\node [style=none] (139) at (0.5, -6.5) {$\scriptstyle 0$};
			\node [style=none] (140) at (1.5, -6.5) {$\scriptstyle 1$};
			\node [style=none] (141) at (2.5, -6.5) {$\scriptstyle 2$};
			\node [style=none] (142) at (3.5, -6.5) {$\scriptstyle 3$};
			\node [style=none] (143) at (4.5, -6.5) {$\scriptstyle 4$};
			\node [style=none] (144) at (-1.5, -4.25) {};
			\node [style=none] (145) at (-1.5, -5.75) {};
			\node [style=none] (146) at (3.5, -4.25) {};
			\node [style=none] (147) at (3.5, -5.75) {};
			\node [style=none] (148) at (8.5, -4.25) {};
			\node [style=none] (149) at (8.5, -5.75) {};
			\node [style=none] (150) at (9, -4.45) {$\scriptstyle 5$};
			\node [style=none] (151) at (-1.5, -6.5) {Runner};
			\node [style=none] (152) at (7.75, -7.75) {};
			\node [style=none] (153) at (7.75, -7) {$\vdots$};
			\node [style=none] (154) at (8.75, -7.75) {};
			\node [style=none] (155) at (8.75, -7) {$\vdots$};
			\node [style=none] (156) at (9.75, -7.75) {};
			\node [style=none] (157) at (9.75, -7) {$\vdots$};
			\node [style=none] (158) at (10.75, -7.75) {};
			\node [style=none] (159) at (10.75, -7) {$\vdots$};
			\node [style=none] (160) at (11.75, -7.75) {};
			\node [style=none] (161) at (11.75, -7) {$\vdots$};
			\node [style=none] (162) at (7.75, -12) {$\vdots$};
			\node [style=none] (163) at (8.75, -12) {$\vdots$};
			\node [style=none] (164) at (7.75, -11.75) {};
			\node [style=none] (165) at (8.75, -11.75) {};
			\node [style=none] (166) at (9.75, -11.75) {};
			\node [style=none] (167) at (9.75, -12) {$\vdots$};
			\node [style=none] (168) at (10.75, -12) {$\vdots$};
			\node [style=none] (169) at (10.75, -11.75) {};
			\node [style=none] (170) at (11.75, -11.75) {};
			\node [style=none] (171) at (11.75, -12) {$\vdots$};
			\node [style=none] (188) at (11.5, -10.25) {$\scriptscriptstyle 4$};
			\node [style=none] (191) at (11.5, -11.25) {$\scriptscriptstyle 9$};
			\node [style=none] (193) at (10.5, -11.25) {$\scriptscriptstyle 8$};
			\node [style=none] (195) at (9.5, -11.25) {$\scriptscriptstyle 7$};
			\node [style=none] (197) at (8.5, -11.25) {$\scriptscriptstyle 6$};
			\node [style=none] (202) at (7.75, -6.5) {$\scriptstyle 0$};
			\node [style=none] (203) at (8.75, -6.5) {$\scriptstyle 1$};
			\node [style=none] (204) at (9.75, -6.5) {$\scriptstyle 2$};
			\node [style=none] (205) at (10.75, -6.5) {$\scriptstyle 3$};
			\node [style=none] (206) at (11.75, -6.5) {$\scriptstyle 4$};
			\node [style=none] (207) at (10.5, -10.25) {$\scriptscriptstyle 3$};
			\node [style=none] (208) at (11.325, -9.25) {$\scriptscriptstyle -1$};
			\node [style=none] (209) at (11.3, -8.25) {$\scriptscriptstyle -6$};
			\node [style=none] (210) at (10.3, -8.25) {$\scriptscriptstyle -7$};
			\node [style=none] (211) at (10.3, -9.25) {$\scriptscriptstyle -2$};
			\node [style=none] (212) at (9.3, -8.25) {$\scriptscriptstyle -8$};
			\node [style=none] (213) at (9.3, -9.25) {$\scriptscriptstyle -3$};
			\node [style=none] (214) at (9.5, -10.25) {$\scriptscriptstyle 2$};
			\node [style=none] (215) at (8.5, -10.25) {$\scriptscriptstyle 1$};
			\node [style=none] (216) at (8.3, -9.25) {$\scriptscriptstyle -4$};
			\node [style=none] (218) at (7.25, -8.25) {$\scriptscriptstyle -10$};
			\node [style=none] (219) at (7.3, -9.25) {$\scriptscriptstyle -5$};
			\node [style=none] (220) at (7.5, -10.25) {$\scriptscriptstyle 0$};
			\node [style=none] (221) at (7.5, -11.25) {$\scriptscriptstyle 5$};
			\node [style=none] (222) at (8.3, -8.25) {$\scriptscriptstyle -9$};
		\end{pgfonlayer}
		\begin{pgfonlayer}{edgelayer}
			\draw (9.center) to (27.center);
			\draw (0.center) to (31.center);
			\draw (30.center) to (25.center);
			\draw (23.center) to (22.center);
			\draw (0.center) to (33.center);
			\draw (11.center) to (10.center);
			\draw (12.center) to (22.center);
			\draw (29.center) to (17.center);
			\draw (14.center) to (25.center);
			\draw (15.center) to (21.center);
			\draw (28.center) to (27.center);
			\draw (41.center) to (42.center);
			\draw (49.center) to (50.center);
			\draw (67.center) to (68.center);
			\draw (69.center) to (70.center);
			\draw (73.center) to (74.center);
			\draw (75.center) to (76.center);
			\draw (77.center) to (78.center);
			\draw (79.center) to (80.center);
			\draw (81.center) to (82.center);
			\draw (83.center) to (84.center);
			\draw (85.center) to (86.center);
			\draw (51.center) to (52.center);
			\draw (53.center) to (54.center);
			\draw (55.center) to (56.center);
			\draw (59.center) to (60.center);
			\draw (57.center) to (58.center);
			\draw (61.center) to (62.center);
			\draw (63.center) to (64.center);
			\draw (65.center) to (66.center);
			\draw (89.center) to (101.center);
			\draw (91.center) to (102.center);
			\draw (95.center) to (106.center);
			\draw (97.center) to (107.center);
			\draw (119.center) to (120.center);
			\draw (121.center) to (122.center);
			\draw (123.center) to (124.center);
			\draw (125.center) to (126.center);
			\draw (128.center) to (127.center);
			\draw (130.center) to (129.center);
			\draw (132.center) to (131.center);
			\draw (134.center) to (133.center);
			\draw (136.center) to (135.center);
			\draw (138.center) to (137.center);
			\draw (93.center) to (103.center);
			\draw [style=Extra box] (144.center) to (145.center);
			\draw [style=Extra box] (146.center) to (147.center);
			\draw [style=Extra box] (148.center) to (149.center);
			\draw (152.center) to (164.center);
			\draw (154.center) to (165.center);
			\draw (158.center) to (169.center);
			\draw (160.center) to (170.center);
			\draw (156.center) to (166.center);
		\end{pgfonlayer}
	\end{tikzpicture}
	\caption{Ignoring the dashed lines, the diagram in the middle is $\lambda=(6,4,2)$ displayed on the James abacus. It is created from $\lambda$ by extending the left and upper borderlines of $Y(\lambda)$, placing beads on every vertical unit interval of its inner border and unbending the inner border into a straight line. We obtain a $\beta$-set of $\lambda$ by identifying the James abacus with the real line such that its positions lie at integers. If we identify the rightmost bead with $5$ as in the figure, we get the canonical $\beta$-set $\left\lbrace 5,2,-1,-4,-5,-6, \dots \right\rbrace $. The bottom left picture is a James abacus display of $\lambda$ with $5$ runners, obtained by dividing the middle picture into intervals of length $5$ (separated by the dashed lines) and placing them on top of each other. We obtain the canonical $\beta$-set by labelling the last displayed row as row $1$, so the bead on this row, which lies on runner $0$, is identified with $1*e + 0 =5$. The identifications of all shown positions in the canonical $\beta$-set with integers are displayed in the bottom right diagram. In the rest of the figures we will omit the dots indicating that all runners are infinite.}
	\label{fig:abacus}
\end{figure}

Note that any $\beta$-set of a partition is bounded above and it contains all integers less or equal to some integer $L$. In fact, these are the necessary and sufficient conditions for a set $B\subseteq \Z$ to arise as a $\beta$-set of a partition. To recover the unique partition from such $B$ we need some more notation.

\begin{definition}\label{de:emptyness}
	Let $B$ be a subset of integers and $u$ and $v$ be two integers such that $u\leq v+1$. We define
	\begin{enumerate}[label=\textnormal{(\roman*)}]
		\item $\bd_B(u,v)$ to be the number of beads $b$ of $B$ such that $u\leq b\leq v$, that is $\bd_B(u,v) = |B\cap \left\lbrace u, u+1, \dots, v \right\rbrace |$,
		\item $\emp_B(u,v)$ to be the number of empty spaces $f$ of $B$ such that $u\leq f \leq v$, that is $\emp_B(u,v)=1+v-u-\bd_B(u,v)$,
		\item provided that $B$ contains all integers less or equal to some $L$, the \textit{emptiness} of $v$, denoted by $\emp_B(v)$ to be the number of empty spaces $f$ of $B$ such that $f\leq v$, that is $\emp_B(v)=\emp_B(l,v) = \emp_B(l-1,v)=\dots$. 
	\end{enumerate}
\end{definition}

Note that for $u=v+1$ we have $\bd_B(u,v)=\emp_B(u,v)=0$. With this convention, the expected equality $\bd_B(u,v) + \bd_B(v+1,w)=\bd_B(u,w)$ holds whenever all terms are defined (that is, $u\leq v+1$ and $v\leq w$). The same remains true if $\bd_B$ is replaced by $\emp_B$. In turn we obtain $\emp_B(v+1,w) + \emp_B(v) = \emp_B(w)$ if $B$ contains all integers less or equal to some $L$ and $v\leq w$. 

If $B$ is bounded above and contains all integers less or equal to some integer $L$, then $B$ is a $\beta$-set of the unique partition $\lambda$ obtained from
\begin{equation}\label{eq:recover partition}
(\emp_B(b^{(1)}), \emp_B(b^{(2)}),\dots)
\end{equation}
by omitting the trailing zeros, where $b^{(i)}$ is the $i$'th greatest element of $B$. In turn if $B$ is a $\beta$-set of a partition $\lambda$, which contains all non-positive integers, and $0$ is the $N$'th largest bead of $B$, then
\begin{equation}\label{eq:size}
|\lambda| = \sum_{i=1}^N \emp_B(b^{(i)}) =\sum_{i=1}^N b^{(i)} - (N-i) = \sum_{0\leq b \in B}b - \binom{N}{2}.
\end{equation}

\begin{example}\label{ex:empty}
	If we label the rightmost displayed position of the middle James abacus in \Cref{fig:abacus} by $8$ (so the leftmost position is $-6$), then, for the corresponding $\beta$-set $B$, we have $\bd_B(-6,8) = 6$ and $\emp_B(-6,8)=9$. One also computes $\emp_B(5)=6, \emp_B(2) = 4, \emp_B(-1)=2$ and $\emp_B(b) = 0$ for the remaining beads of $B$, so by \eqref{eq:recover partition}, $B$ is a $\beta$-set of $(6,4,2)$.
\end{example}

One can use a $\beta$-set $B$ of a partition $\lambda$ to find its $e$-core using the well-known correspondence between $e$-hooks of $\lambda$ and beads $b\in B$ such that $b-e$ is an empty space of $B$; see \cite[Lemma~2.7.13]{JamesKerberSymmetric81}. This correspondence assigns to an $e$-hook $R$ of $\lambda$ the $\tp(R)$'th largest bead $b$ of $B$ and it is easy to see that the effect of removing $R$ from $\lambda$ corresponds to \textit{sliding bead $b$ up by one place on its runner}, that is by replacing $b$ in $B$ by $b-e$. Thus the $e$-core of $\lambda$ is unique and one obtains its $\beta$-set by sliding all beads up on their runners on a $\beta$-set of $\lambda$; see \Cref{fig:core} for details.

In fact if $B=B_s(\lambda)$, then the procedure of sliding all beads up on their runners yields $B_s(\gamma)$, where $\gamma$ is the $e$-core of $\lambda$. In particular for any partitions $\lambda$ and $\mu$ if $B_s(\lambda)$ and $B_s(\mu)$ contain all non-positive integers, then $\lambda$ and $\mu$ have the same $e$-core if and only if $B_s(\lambda)$ and $B_s(\mu)$ have the same number of positive beads on each runner.

\begin{figure}[h]
	\centering
	\begin{tikzpicture}[x=0.5cm, y=0.5cm]
		\begin{pgfonlayer}{nodelayer}
			\node [style=none] (0) at (2, 3) {};
			\node [style=none] (1) at (8, 3) {};
			\node [style=none] (2) at (8, 2) {};
			\node [style=none] (3) at (6, 2) {};
			\node [style=none] (4) at (6, 1) {};
			\node [style=none] (5) at (4, 1) {};
			\node [style=none] (6) at (4, 0) {};
			\node [style=none] (7) at (2, 0) {};
			\node [style=none] (8) at (2, 1) {};
			\node [style=none] (9) at (2, 2) {};
			\node [style=none] (10) at (3, 0) {};
			\node [style=none] (11) at (3, 3) {};
			\node [style=none] (12) at (4, 3) {};
			\node [style=none] (13) at (5, 3) {};
			\node [style=none] (14) at (6, 3) {};
			\node [style=none] (15) at (7, 3) {};
			\node [style=none] (16) at (7, 2) {};
			\node [style=none] (17) at (5, 1) {};
			\node [style=none] (18) at (5, 2) {};
			\node [style=none] (19) at (3, 2) {};
			\node [style=none] (20) at (3, 1) {};
			\node [style=none] (21) at (9.75, 3.5) {};
			\node [style=none] (22) at (10.75, 3.5) {};
			\node [style=none] (23) at (11.75, 3.5) {};
			\node [style=none] (24) at (12.75, 3.5) {};
			\node [style=none] (25) at (13.75, 3.5) {};
			\node [style=none] (26) at (9.75, -0.5) {};
			\node [style=none] (27) at (10.75, -0.5) {};
			\node [style=none] (28) at (11.75, -0.5) {};
			\node [style=none] (29) at (12.75, -0.5) {};
			\node [style=none] (30) at (13.75, -0.5) {};
			\node [style=Empty node] (31) at (9.75, 3) {};
			\node [style=Empty node] (32) at (10.75, 3) {};
			\node [style=Empty node] (33) at (11.75, 3) {};
			\node [style=Empty node] (34) at (12.75, 3) {};
			\node [style=Empty node] (35) at (13.75, 3) {};
			\node [style=Empty node] (36) at (9.75, 2) {};
			\node [style=Empty node] (37) at (10.75, 2) {};
			\node [style=Empty node] (38) at (13.75, 2) {};
			\node [style=Empty node] (39) at (9.75, 0) {};
			\node [style=none] (41) at (11.5, 2) {};
			\node [style=none] (42) at (12, 2) {};
			\node [style=none] (43) at (12.5, 2) {};
			\node [style=none] (44) at (13, 2) {};
			\node [style=none] (45) at (12.5, 1) {};
			\node [style=none] (46) at (13, 1) {};
			\node [style=none] (47) at (13.5, 1) {};
			\node [style=none] (48) at (14, 1) {};
			\node [style=none] (49) at (14, 0) {};
			\node [style=none] (50) at (13.5, 0) {};
			\node [style=none] (51) at (13, 0) {};
			\node [style=none] (52) at (12.5, 0) {};
			\node [style=none] (53) at (12, 0) {};
			\node [style=none] (54) at (11.5, 0) {};
			\node [style=none] (55) at (11, 0) {};
			\node [style=none] (56) at (10.5, 0) {};
			\node [style=none] (57) at (11, 1) {};
			\node [style=none] (58) at (10.5, 1) {};
			\node [style=none] (59) at (10, 1) {};
			\node [style=none] (60) at (9.5, 1) {};
			\node [style=none] (61) at (2, -2.75) {};
			\node [style=none] (62) at (8, -2.75) {};
			\node [style=none] (63) at (8, -3.75) {};
			\node [style=none] (64) at (6, -3.75) {};
			\node [style=none] (69) at (2, -4.75) {};
			\node [style=none] (70) at (2, -3.75) {};
			\node [style=none] (72) at (3, -2.75) {};
			\node [style=none] (73) at (4, -2.75) {};
			\node [style=none] (74) at (5, -2.75) {};
			\node [style=none] (75) at (6, -2.75) {};
			\node [style=none] (76) at (7, -2.75) {};
			\node [style=none] (77) at (7, -3.75) {};
			\node [style=none] (79) at (5, -3.75) {};
			\node [style=none] (80) at (3, -3.75) {};
			\node [style=none] (82) at (9.75, -2.25) {};
			\node [style=none] (83) at (10.75, -2.25) {};
			\node [style=none] (84) at (11.75, -2.25) {};
			\node [style=none] (85) at (12.75, -2.25) {};
			\node [style=none] (86) at (13.75, -2.25) {};
			\node [style=none] (87) at (9.75, -6.25) {};
			\node [style=none] (88) at (10.75, -6.25) {};
			\node [style=none] (89) at (11.75, -6.25) {};
			\node [style=none] (90) at (12.75, -6.25) {};
			\node [style=none] (91) at (13.75, -6.25) {};
			\node [style=Empty node] (92) at (9.75, -2.75) {};
			\node [style=Empty node] (93) at (10.75, -2.75) {};
			\node [style=Empty node] (94) at (11.75, -2.75) {};
			\node [style=Empty node] (95) at (12.75, -2.75) {};
			\node [style=Empty node] (96) at (13.75, -2.75) {};
			\node [style=Empty node] (97) at (9.75, -3.75) {};
			\node [style=Empty node] (98) at (10.75, -3.75) {};
			\node [style=Empty node] (99) at (13.75, -3.75) {};
			\node [style=Empty node] (101) at (11.75, -3.75) {};
			\node [style=none] (102) at (11.5, -4.75) {};
			\node [style=none] (103) at (12, -4.75) {};
			\node [style=none] (104) at (12.5, -3.75) {};
			\node [style=none] (105) at (13, -3.75) {};
			\node [style=none] (106) at (12.5, -4.75) {};
			\node [style=none] (107) at (13, -4.75) {};
			\node [style=none] (108) at (13.5, -4.75) {};
			\node [style=none] (109) at (14, -4.75) {};
			\node [style=none] (110) at (14, -5.75) {};
			\node [style=none] (111) at (13.5, -5.75) {};
			\node [style=none] (112) at (13, -5.75) {};
			\node [style=none] (113) at (12.5, -5.75) {};
			\node [style=none] (114) at (12, -5.75) {};
			\node [style=none] (115) at (11.5, -5.75) {};
			\node [style=none] (116) at (11, -5.75) {};
			\node [style=none] (117) at (10.5, -5.75) {};
			\node [style=none] (118) at (11, -4.75) {};
			\node [style=none] (119) at (10.5, -4.75) {};
			\node [style=none] (120) at (10, -4.75) {};
			\node [style=none] (121) at (9.5, -4.75) {};
			\node [style=none] (122) at (3, -4.75) {};
			\node [style=Lead] (123) at (11.75, 1) {};
			\node [style=Lead] (124) at (9.75, -5.75) {};
			\node [style=none] (125) at (2, -8.5) {};
			\node [style=none] (129) at (2, -10.5) {};
			\node [style=none] (130) at (2, -9.5) {};
			\node [style=none] (131) at (3, -8.5) {};
			\node [style=none] (138) at (3, -9.5) {};
			\node [style=none] (139) at (9.75, -8) {};
			\node [style=none] (140) at (10.75, -8) {};
			\node [style=none] (141) at (11.75, -8) {};
			\node [style=none] (142) at (12.75, -8) {};
			\node [style=none] (143) at (13.75, -8) {};
			\node [style=none] (144) at (9.75, -12) {};
			\node [style=none] (145) at (10.75, -12) {};
			\node [style=none] (146) at (11.75, -12) {};
			\node [style=none] (147) at (12.75, -12) {};
			\node [style=none] (148) at (13.75, -12) {};
			\node [style=Empty node] (149) at (9.75, -8.5) {};
			\node [style=Empty node] (150) at (10.75, -8.5) {};
			\node [style=Empty node] (151) at (11.75, -8.5) {};
			\node [style=Empty node] (152) at (12.75, -8.5) {};
			\node [style=Empty node] (153) at (13.75, -8.5) {};
			\node [style=Empty node] (154) at (9.75, -9.5) {};
			\node [style=Empty node] (155) at (10.75, -9.5) {};
			\node [style=Empty node] (156) at (13.75, -9.5) {};
			\node [style=Empty node] (157) at (9.75, -10.5) {};
			\node [style=none] (160) at (12.5, -9.5) {};
			\node [style=none] (161) at (13, -9.5) {};
			\node [style=none] (162) at (12.5, -10.5) {};
			\node [style=none] (163) at (13, -10.5) {};
			\node [style=none] (164) at (13.5, -10.5) {};
			\node [style=none] (165) at (14, -10.5) {};
			\node [style=none] (166) at (14, -11.5) {};
			\node [style=none] (167) at (13.5, -11.5) {};
			\node [style=none] (168) at (13, -11.5) {};
			\node [style=none] (169) at (12.5, -11.5) {};
			\node [style=none] (170) at (12, -11.5) {};
			\node [style=none] (171) at (11.5, -11.5) {};
			\node [style=none] (172) at (11, -11.5) {};
			\node [style=none] (173) at (10.5, -11.5) {};
			\node [style=none] (174) at (11, -10.5) {};
			\node [style=none] (175) at (10.5, -10.5) {};
			\node [style=none] (176) at (10, -11.5) {};
			\node [style=none] (177) at (9.5, -11.5) {};
			\node [style=none] (178) at (3, -10.5) {};
			\node [style=none] (179) at (11.5, -10.5) {};
			\node [style=none] (180) at (12, -10.5) {};
			\node [style=Empty node] (181) at (11.75, -9.5) {};
			\node [style=none] (182) at (4, -3.75) {};
		\end{pgfonlayer}
		\begin{pgfonlayer}{edgelayer}
			\draw [style=Light grey column] (6.center)
			to (7.center)
			to (8.center)
			to (20.center)
			to (19.center)
			to (3.center)
			to (4.center)
			to (5.center)
			to cycle;
			\draw (8.center) to (0.center);
			\draw (0.center) to (1.center);
			\draw (1.center) to (2.center);
			\draw (2.center) to (3.center);
			\draw (15.center) to (16.center);
			\draw (14.center) to (3.center);
			\draw (13.center) to (17.center);
			\draw (12.center) to (5.center);
			\draw (11.center) to (19.center);
			\draw (20.center) to (10.center);
			\draw (9.center) to (19.center);
			\draw (20.center) to (5.center);
			\draw (21.center) to (26.center);
			\draw (22.center) to (27.center);
			\draw (24.center) to (29.center);
			\draw (25.center) to (30.center);
			\draw (41.center) to (42.center);
			\draw (43.center) to (44.center);
			\draw (45.center) to (46.center);
			\draw (47.center) to (48.center);
			\draw (50.center) to (49.center);
			\draw (52.center) to (51.center);
			\draw (54.center) to (53.center);
			\draw (56.center) to (55.center);
			\draw (58.center) to (57.center);
			\draw (60.center) to (59.center);
			\draw (23.center) to (28.center);
			\draw (69.center) to (61.center);
			\draw (76.center) to (77.center);
			\draw (75.center) to (64.center);
			\draw (70.center) to (80.center);
			\draw (82.center) to (87.center);
			\draw (83.center) to (88.center);
			\draw (85.center) to (90.center);
			\draw (86.center) to (91.center);
			\draw (102.center) to (103.center);
			\draw (104.center) to (105.center);
			\draw (106.center) to (107.center);
			\draw (108.center) to (109.center);
			\draw (111.center) to (110.center);
			\draw (113.center) to (112.center);
			\draw (115.center) to (114.center);
			\draw (117.center) to (116.center);
			\draw (119.center) to (118.center);
			\draw (121.center) to (120.center);
			\draw (84.center) to (89.center);
			\draw [style=Light grey column] (62.center)
			to (63.center)
			to (80.center)
			to (72.center)
			to cycle;
			\draw [style=Light grey column] (72.center) to (61.center);
			\draw (80.center) to (122.center);
			\draw (122.center) to (69.center);
			\draw (129.center) to (125.center);
			\draw (130.center) to (138.center);
			\draw (139.center) to (144.center);
			\draw (140.center) to (145.center);
			\draw (142.center) to (147.center);
			\draw (143.center) to (148.center);
			\draw (160.center) to (161.center);
			\draw (162.center) to (163.center);
			\draw (164.center) to (165.center);
			\draw (167.center) to (166.center);
			\draw (169.center) to (168.center);
			\draw (171.center) to (170.center);
			\draw (173.center) to (172.center);
			\draw (175.center) to (174.center);
			\draw (177.center) to (176.center);
			\draw (141.center) to (146.center);
			\draw [style=Light grey column] (138.center) to (131.center);
			\draw [style=Light grey column] (131.center) to (125.center);
			\draw (138.center) to (178.center);
			\draw (178.center) to (129.center);
			\draw (179.center) to (180.center);
			\draw (73.center) to (182.center);
			\draw (74.center) to (79.center);
			\draw (75.center) to (64.center);
			\draw (76.center) to (77.center);
		\end{pgfonlayer}
	\end{tikzpicture}
	\caption{The $5$-core of partition $(6,4,2)$ is $(1^2)$. On the left-hand side, this is found by removing two $5$-hooks, while on the right-hand side this is done by sliding the corresponding beads up by one place on their runners. The $e$-weight of $(6,4,2)$ is $2$.}
	\label{fig:core}
\end{figure}
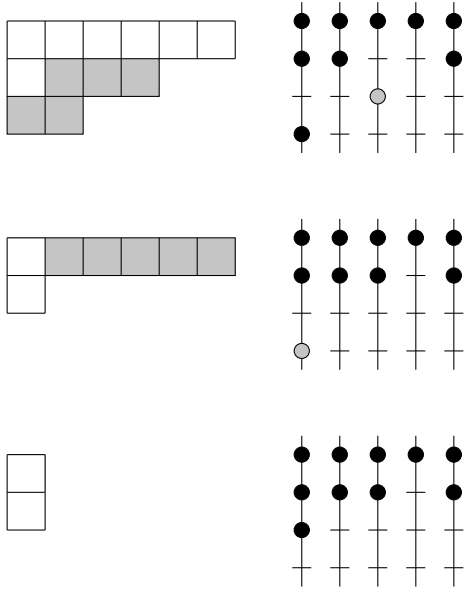

For the rest of the section we also fix an integer $d>1$. Recall the definition of the $d$-runner matrix of a partition from \Cref{de:runner matrix}. Our next (simple) step is to translate this definition to $\beta$-sets. Before doing so, we remark that the $d$-runner matrix of a partition $\lambda$ can be found using $e$-residues of the boxes of its Young diagram. The \textit{$e$-residue} of box $(i,j)\in Y(\lambda)$ is the remainder of $j-i$ modulo $e$ (between $0$ and $e-1$). From \Cref{de:runner matrix}, the entry $\mathcal{R}_d(\lambda)_{x,y}$ of the $d$-runner matrix of $\lambda$ is then the number of rows of $Y(\lambda)$ of length congruent to $x$ modulo $d$ such that their rightmost box has $e$-residue $y$. See \Cref{fig:runner} for an example.

As promised, we now move to $\beta$-sets. The \textit{$d$-emptiness} of an integer $v$ with respect to a $\beta$-set $B$ of a partition is the remainder of $\emp_B(v)$ modulo $d$ (between $0$ and $d-1$). Since there are only finitely many beads with a non-zero $d$-emptiness, the following definition of a $d$-runner matrix of $B$ makes sense. An example of the definition is in \Cref{fig:runner}.

\begin{definition}\label{de:runner matrix B}
	The \textit{$d$-runner matrix} $\mathcal{R}_d(B)$ of $B$ is a $(d-1)\times e$ matrix indexed by $1\leq x\leq d-1$ and $0\leq y\leq e-1$ where $\mathcal{R}_d(B)_{x,y}$ counts the number of beads with $d$-emptiness $x$ on runner $y$.
\end{definition}

\begin{remark}\label{re:runner matrix}
	The choice of the name `$d$-runner matrix' comes from the fact that the rows of a $d$-runner matrix correspond to $d$-emptinesses, which contribute `$d$' to the name, and the columns correspond to runners, which contribute `runner' to the name.
\end{remark}

\begin{figure}[h]
	\centering
	\begin{tikzpicture}[x=0.5cm, y=0.5cm]
		\begin{pgfonlayer}{nodelayer}
			\node [style=none] (55) at (15, 3) {};
			\node [style=none] (56) at (21, 3) {};
			\node [style=none] (57) at (21, 2) {};
			\node [style=none] (58) at (19, 2) {};
			\node [style=none] (59) at (19, 1) {};
			\node [style=none] (60) at (17, 1) {};
			\node [style=none] (61) at (17, 0) {};
			\node [style=none] (62) at (15, 0) {};
			\node [style=none] (63) at (15, 1) {};
			\node [style=none] (64) at (15, 2) {};
			\node [style=none] (65) at (16, 0) {};
			\node [style=none] (66) at (16, 3) {};
			\node [style=none] (67) at (17, 3) {};
			\node [style=none] (68) at (18, 3) {};
			\node [style=none] (69) at (19, 3) {};
			\node [style=none] (70) at (20, 3) {};
			\node [style=none] (71) at (20, 2) {};
			\node [style=none] (72) at (18, 1) {};
			\node [style=none] (73) at (18, 2) {};
			\node [style=none] (74) at (16, 2) {};
			\node [style=none] (75) at (16, 1) {};
			\node [style=none] (76) at (22.75, 3.5) {};
			\node [style=none] (77) at (23.75, 3.5) {};
			\node [style=none] (78) at (24.75, 3.5) {};
			\node [style=none] (79) at (25.75, 3.5) {};
			\node [style=none] (80) at (26.75, 3.5) {};
			\node [style=none] (81) at (22.75, -0.5) {};
			\node [style=none] (82) at (23.75, -0.5) {};
			\node [style=none] (83) at (24.75, -0.5) {};
			\node [style=none] (84) at (25.75, -0.5) {};
			\node [style=none] (85) at (26.75, -0.5) {};
			\node [style=Empty node] (86) at (22.75, 3) {};
			\node [style=Empty node] (87) at (23.75, 3) {};
			\node [style=Empty node] (88) at (24.75, 3) {};
			\node [style=Empty node] (89) at (25.75, 3) {};
			\node [style=Empty node] (90) at (26.75, 3) {};
			\node [style=Empty node] (91) at (22.75, 2) {};
			\node [style=Empty node] (92) at (23.75, 2) {};
			\node [style=Empty node] (93) at (26.75, 2) {};
			\node [style=Empty node] (94) at (24.75, 1) {};
			\node [style=none] (95) at (24.5, 2) {};
			\node [style=none] (96) at (25, 2) {};
			\node [style=none] (97) at (25.5, 2) {};
			\node [style=none] (98) at (26, 2) {};
			\node [style=none] (99) at (25.5, 1) {};
			\node [style=none] (100) at (26, 1) {};
			\node [style=none] (101) at (26.5, 1) {};
			\node [style=none] (102) at (27, 1) {};
			\node [style=none] (103) at (27, 0) {};
			\node [style=none] (104) at (26.5, 0) {};
			\node [style=none] (105) at (26, 0) {};
			\node [style=none] (106) at (25.5, 0) {};
			\node [style=none] (107) at (25, 0) {};
			\node [style=none] (108) at (24.5, 0) {};
			\node [style=none] (109) at (24, 0) {};
			\node [style=none] (110) at (23.5, 0) {};
			\node [style=none] (111) at (24, 1) {};
			\node [style=none] (112) at (23.5, 1) {};
			\node [style=none] (113) at (23, 1) {};
			\node [style=none] (114) at (22.5, 1) {};
			\node [style=Empty node] (115) at (22.75, 0) {};
			\node [style=none] (116) at (15.5, 2.5) {$0$};
			\node [style=none] (117) at (16.5, 2.5) {$1$};
			\node [style=none] (118) at (17.5, 2.5) {$2$};
			\node [style=none] (119) at (18.5, 2.5) {$3$};
			\node [style=none] (120) at (19.5, 2.5) {$4$};
			\node [style=none] (121) at (20.5, 2.5) {$0$};
			\node [style=none] (122) at (18.5, 1.5) {$2$};
			\node [style=none] (123) at (17.5, 1.5) {$1$};
			\node [style=none] (124) at (16.5, 1.5) {$0$};
			\node [style=none] (125) at (15.5, 1.5) {$4$};
			\node [style=none] (126) at (15.5, 0.5) {$3$};
			\node [style=none] (127) at (16.5, 0.5) {$4$};
			\node [style=none] (128) at (22.4, 3.25) {$\scriptstyle 0$};
			\node [style=none] (129) at (23.4, 3.25) {$\scriptstyle 0$};
			\node [style=none] (130) at (24.4, 3.25) {$\scriptstyle 0$};
			\node [style=none] (131) at (25.4, 3.25) {$\scriptstyle 0$};
			\node [style=none] (132) at (26.4, 3.25) {$\scriptstyle 0$};
			\node [style=none] (133) at (22.4, 2.25) {$\scriptstyle 0$};
			\node [style=none] (134) at (23.4, 2.25) {$\scriptstyle 0$};
			\node [style=none] (135) at (26.4, 2.25) {$\scriptstyle 2$};
			\node [style=none] (136) at (24.4, 1.25) {$\scriptstyle 1$};
			\node [style=none] (137) at (22.4, 0.25) {$\scriptstyle 0$};
			\node [style=none] (138) at (18, 2) {};
			\node [style=none] (139) at (19, 2) {};
			\node [style=none] (140) at (18, 1) {};
			\node [style=none] (141) at (19, 1) {};
			\node [style=none] (142) at (16, 1) {};
			\node [style=none] (143) at (17, 1) {};
			\node [style=none] (144) at (16, 0) {};
			\node [style=none] (145) at (17, 0) {};
		\end{pgfonlayer}
		\begin{pgfonlayer}{edgelayer}
			\draw [style=Light grey column] (61.center) to (62.center);
			\draw [style=Light grey column] (62.center) to (63.center);
			\draw [style=Light grey column] (63.center) to (75.center);
			\draw [style=Light grey column] (75.center) to (74.center);
			\draw [style=Light grey column] (74.center) to (58.center);
			\draw [style=Light grey column] (58.center) to (59.center);
			\draw [style=Light grey column] (59.center) to (60.center);
			\draw [style=Light grey column] (60.center) to (61.center);
			\draw (63.center) to (55.center);
			\draw (55.center) to (56.center);
			\draw (56.center) to (57.center);
			\draw (57.center) to (58.center);
			\draw (70.center) to (71.center);
			\draw (69.center) to (58.center);
			\draw (68.center) to (72.center);
			\draw (67.center) to (60.center);
			\draw (66.center) to (74.center);
			\draw (75.center) to (65.center);
			\draw (64.center) to (74.center);
			\draw (75.center) to (60.center);
			\draw (76.center) to (81.center);
			\draw (77.center) to (82.center);
			\draw (79.center) to (84.center);
			\draw (80.center) to (85.center);
			\draw (95.center) to (96.center);
			\draw (97.center) to (98.center);
			\draw (99.center) to (100.center);
			\draw (101.center) to (102.center);
			\draw (104.center) to (103.center);
			\draw (106.center) to (105.center);
			\draw (108.center) to (107.center);
			\draw (110.center) to (109.center);
			\draw (112.center) to (111.center);
			\draw (114.center) to (113.center);
			\draw (78.center) to (83.center);
			\draw [style=Light grey column] (141.center)
			to (140.center)
			to (138.center)
			to (139.center)
			to cycle;
			\draw [style=Light grey column] (145.center)
			to (144.center)
			to (142.center)
			to (143.center)
			to cycle;
		\end{pgfonlayer}
	\end{tikzpicture}
	\caption{Let $d=3$ and $e=5$. The first diagram captures the $e$-residues of the boxes of the Young diagram of $(6,4,2)$. Since only the second and the third part are not divisible by $d$, their reminders modulo $d$ are $1$ and $2$, respectively, and the (highlighted) rightmost boxes in the respective rows have $e$-residues $2$ and $4$, the $d$-runner matrix of $\lambda$ is $\big(\begin{smallmatrix}
		0 & 0 & 1 & 0 & 0\\ 
		0 & 0 & 0 & 0 & 1
		\end{smallmatrix}\big)$. We obtain the same $d$-runner matrix for its canonical $\beta$-set, which is drawn along the $d$-emptinesses of its beads in the second diagram.}
	\label{fig:runner}
\end{figure}
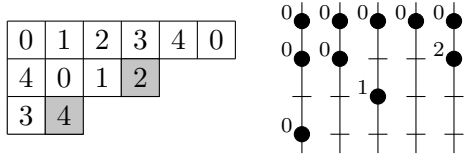

\begin{lemma}\label{le:runner matrix eq}
	Let $\lambda$ be a partition and $B=B_s(\lambda)$ be its $\beta$-set. Then for $1\leq x\leq d-1$ and $0\leq y\leq e-1$ we have $\mathcal{R}_d(\lambda)_{x,y}=\mathcal{R}_d(B)_{x,y+s}$, where $y+s$ is taken modulo $e$. In particular, the $d$-runner matrices of a partition and its canonical $\beta$-set coincide.
\end{lemma}

\begin{proof}
	Let $1\leq x\leq d-1$ and $0\leq y\leq e-1$ be integers. Firstly, the $i$'th largest bead $b^{(i)}$ of $B$ has $d$-emptiness $x$ if and only if $i\leq \ell(\lambda)$ and $\lambda_i=\emp_B(b^{(i)})$ has remainder $x$ modulo $d$, and secondly, it lies on the runner $y+s$ modulo $e$ if and only if $\lambda_i-i+s$ is congruent to $y+s$ modulo $e$, that is $\lambda_i-i$ is congruent to $y$ modulo $e$. Thus $\mathcal{R}_d(\lambda)_{x,y}$ and  $\mathcal{R}_d(B)_{x,y+s}$ count the same thing.
\end{proof}

We can use the definition of the $d$-runner matrix of $\beta$-sets to characterise the $\gamma$-realisable matrices --- the $(d-1)\times e$-matrices of non-negative integers $\mathcal{R}$ such that there is a partition with $e$-core $\gamma$ and $d$-runner matrix $\mathcal{R}$ --- and in particular show that any $(d-1)\times e$-matrix of non-negative integers is $\gamma$-realisable when $d$ and $e$ are coprime (and $\gamma$ is an arbitrary $e$-core partition). In the characterisation and the rest of the paper, we index any $(d-1)\times e$ matrix by $1\leq x\leq d-1$ and $0\leq y\leq e-1$.

For our integers $d,e>1$, let $g$ be their greatest common divisor. For an $e$-core partition $\gamma$ and $(d-1)\times e$-matrix of non-negative integers $\mathcal{R}$, let $\mathcal{D}=\mathcal{R} - \mathcal{R}_d(\gamma)$. We show that $\mathcal{R}$ is $\gamma$-realisable if and only if for all $0\leq j\leq g-1$

\begin{equation}\label{eq:realisable}
	\sum_{x=1}^{d-1}\sum_{\substack{k=0 \\ k\equiv j (\textnormal{mod } g)}}^{e-1} \mathcal{D}_{x,k} = \sum_{x=1}^{d-1}\sum_{\substack{k=0 \\ k\equiv j (\textnormal{mod } g)}}^{e-1} \mathcal{D}_{x,x+k},
\end{equation} 
where the second index is taken modulo $e$; this is a convention we use throughout the rest of this section. 

We make a couple of observations. Firstly, the sum of the left-hand side of \eqref{eq:realisable} over all $0\leq j\leq g-1$ is the sum of all entries of $D$, and the same holds for the right-hand side; therefore one can verify \eqref{eq:realisable} only for $g-1$ choices of $j$. In particular, if $g=1$ (that is, if $d$ are $e$ are coprime), then $\mathcal{R}$ is always $\gamma$-realisable. Secondly, if $\mathcal{D}$ is extended by row indexed by $0$, then \eqref{eq:realisable} can be replaced by an equivalent condition when the sums run from $x=0$ rather than from $x=1$; indeed, the new equality condition differs only by the sum of $\mathcal{D}_{0,k}$ for $0\leq k\leq e-1$ with $k\equiv j\, (\textnormal{mod } g)$ added to both sides. Finally, letting $\widehat{\mathcal{D}}_{x',y'} = \sum \mathcal{D}_{x,y}$ where the sum runs over all $0\leq x\leq d-1$ and $0\leq y\leq e-1$ such that $x\equiv x'\, (\textnormal{mod } g)$ and $y\equiv y'\, (\textnormal{mod } g)$, allows one to replace \eqref{eq:realisable} by $\sum_{x'=1}^{g-1} \widehat{\mathcal{D}}_{x',j}=\sum_{x'=1}^{g-1} \widehat{\mathcal{D}}_{x',x'+j}$. Both \eqref{eq:realisable} and this new equation are displayed graphically in Figure~\ref{fig:realisable}.

\begin{figure}[h]
	\centering
	\begin{tikzpicture}[x=0.5cm, y=0.5cm]
		\begin{pgfonlayer}{nodelayer}
			\node [style=Lead] (14) at (0, 3) {};
			\node [style=midsize] (21) at (-4, 6) {};
			\node [style=midsize] (22) at (-4, 5) {};
			\node [style=midsize] (23) at (-4, 4) {};
			\node [style=Lead] (34) at (-4, 3) {};
			\node [style=Lead] (35) at (4, 3) {};
			\node [style=Empty node] (36) at (-3, 6) {};
			\node [style=Empty node] (37) at (-2, 5) {};
			\node [style=Empty node] (38) at (-1, 4) {};
			\node [style=none] (54) at (-4.5, 6.5) {};
			\node [style=none] (55) at (-4.5, -0.5) {};
			\node [style=none] (56) at (7.5, 6.5) {};
			\node [style=none] (57) at (7.5, -0.5) {};
			\node [style=none] (58) at (-4.25, 6.25) {};
			\node [style=none] (59) at (-0.75, 6.25) {};
			\node [style=none] (60) at (-0.75, 3.75) {};
			\node [style=none] (61) at (-4.25, 3.75) {};
			\node [style=midsize] (62) at (0, 6) {};
			\node [style=midsize] (63) at (0, 5) {};
			\node [style=midsize] (64) at (0, 4) {};
			\node [style=Empty node] (65) at (1, 6) {};
			\node [style=Empty node] (66) at (2, 5) {};
			\node [style=Empty node] (67) at (3, 4) {};
			\node [style=none] (68) at (-0.25, 6.25) {};
			\node [style=none] (69) at (3.25, 6.25) {};
			\node [style=none] (70) at (3.25, 3.75) {};
			\node [style=none] (71) at (-0.25, 3.75) {};
			\node [style=midsize] (72) at (4, 6) {};
			\node [style=midsize] (73) at (4, 5) {};
			\node [style=midsize] (74) at (4, 4) {};
			\node [style=Empty node] (75) at (5, 6) {};
			\node [style=Empty node] (76) at (6, 5) {};
			\node [style=Empty node] (77) at (7, 4) {};
			\node [style=none] (78) at (3.75, 6.25) {};
			\node [style=none] (79) at (7.25, 6.25) {};
			\node [style=none] (80) at (7.25, 3.75) {};
			\node [style=none] (81) at (3.75, 3.75) {};
			\node [style=midsize] (82) at (-4, 2) {};
			\node [style=midsize] (83) at (-4, 1) {};
			\node [style=midsize] (84) at (-4, 0) {};
			\node [style=Empty node] (85) at (-3, 2) {};
			\node [style=Empty node] (86) at (-2, 1) {};
			\node [style=Empty node] (87) at (-1, 0) {};
			\node [style=none] (88) at (-4.25, 2.25) {};
			\node [style=none] (89) at (-0.75, 2.25) {};
			\node [style=none] (90) at (-0.75, -0.25) {};
			\node [style=none] (91) at (-4.25, -0.25) {};
			\node [style=midsize] (92) at (0, 2) {};
			\node [style=midsize] (93) at (0, 1) {};
			\node [style=midsize] (94) at (0, 0) {};
			\node [style=Empty node] (95) at (1, 2) {};
			\node [style=Empty node] (96) at (2, 1) {};
			\node [style=Empty node] (97) at (3, 0) {};
			\node [style=none] (98) at (-0.25, 2.25) {};
			\node [style=none] (99) at (3.25, 2.25) {};
			\node [style=none] (100) at (3.25, -0.25) {};
			\node [style=none] (101) at (-0.25, -0.25) {};
			\node [style=midsize] (102) at (4, 2) {};
			\node [style=midsize] (103) at (4, 1) {};
			\node [style=midsize] (104) at (4, 0) {};
			\node [style=Empty node] (105) at (5, 2) {};
			\node [style=Empty node] (106) at (6, 1) {};
			\node [style=Empty node] (107) at (7, 0) {};
			\node [style=none] (108) at (3.75, 2.25) {};
			\node [style=none] (109) at (7.25, 2.25) {};
			\node [style=none] (110) at (7.25, -0.25) {};
			\node [style=none] (111) at (3.75, -0.25) {};
			\node [style=none] (112) at (9, 3) {};
			\node [style=none] (113) at (10.5, 3) {};
			\node [style=midsize] (114) at (12, 4) {};
			\node [style=midsize] (115) at (12, 3) {};
			\node [style=midsize] (116) at (12, 2) {};
			\node [style=Empty node] (117) at (13, 4) {};
			\node [style=Empty node] (118) at (14, 3) {};
			\node [style=Empty node] (119) at (15, 2) {};
			\node [style=none] (120) at (11.5, 4.5) {};
			\node [style=none] (121) at (11.5, 1.5) {};
			\node [style=none] (122) at (15.5, 1.5) {};
			\node [style=none] (123) at (15.5, 4.5) {};
		\end{pgfonlayer}
		\begin{pgfonlayer}{edgelayer}
			\draw [bend right=15] (54.center) to (55.center);
			\draw [bend left=15] (56.center) to (57.center);
			\draw [style=Background] (60.center)
			to (61.center)
			to (58.center)
			to (59.center)
			to cycle;
			\draw [style=Background] (70.center)
			to (71.center)
			to (68.center)
			to (69.center)
			to cycle;
			\draw [style=Background] (80.center)
			to (81.center)
			to (78.center)
			to (79.center)
			to cycle;
			\draw [style=Background] (90.center)
			to (91.center)
			to (88.center)
			to (89.center)
			to cycle;
			\draw [style=Background] (100.center)
			to (101.center)
			to (98.center)
			to (99.center)
			to cycle;
			\draw [style=Background] (110.center)
			to (111.center)
			to (108.center)
			to (109.center)
			to cycle;
			\draw [style=Move it] (112.center) to (113.center);
			\draw [bend right=15] (120.center) to (121.center);
			\draw [bend left=15] (123.center) to (122.center);
		\end{pgfonlayer}
	\end{tikzpicture}
	\caption{Let $e=12$ and $d=8$, so then $g=4$. One obtains the left-hand side of \eqref{eq:realisable} for $j=0$ and matrix $\mathcal{D}$ on the left by adding its $18$ white and $3$ grey entries, while the right-hand side is obtained by adding its $18$ black and $3$ grey entries. The equations for other values of $j$ are obtained by moving the highlighting circles right by $j$ positions (working modulo $e$). Alternatively, one can add the $6$ light grey blocks to obtain matrix $\widehat{\mathcal{D}}$, the matrix on the right. For $j=0$, \eqref{eq:realisable} then states that the sum of white entries equals the sum of black entries.}
	\label{fig:realisable}
\end{figure}

The fact that \eqref{eq:realisable} is a necessary condition for $\mathcal{R}$ to be $\gamma$-realisable follows from the following lemma.

\begin{lemma}\label{le:realisable nec}
	Suppose that $\lambda$ and $\mu$ are partitions with the same $e$-core and let $\mathcal{D} = \mathcal{R}_d(\lambda) - \mathcal{R}_d(\mu)$. Then \eqref{eq:realisable} holds for all $0\leq j\leq g-1$, where $g$ is the greatest common divisor of $d$ and $e$.
\end{lemma}

\begin{proof}
	Since $\mu$ can be obtained from $\lambda$ by removing and adding a finite number of $e$-hooks, it suffices to consider the case when $\mu$ is obtained from $\lambda$ by adding a single $e$-hook. Pick integer $s$ divisible by $e$ such that $B=B_s(\lambda)$ and $B'=B_s(\mu)$ contain all non-positive integers and let $b\in B$ be such that $B'$ is obtained from $B$ by replacing $b$ by $b+e$ (so $b > 0$). By \Cref{le:runner matrix eq}, $\mathcal{R}_d(B) = \mathcal{R}_d(\lambda)$ and $\mathcal{R}_d(B') = \mathcal{R}_d(\mu)$. We extend these two matrices and, in turn, matrix $\mathcal{D}$ by an additional row $0$ by letting $\mathcal{R}_d(B)_{0, y}$ be the number of non-negative beads of $B$ with $d$-emptiness $0$ on runner $y$, and analogously for $\mathcal{R}_d(B')_{0, y}$. We further adjust the sums in \eqref{eq:realisable} to run from $x=0$ (see the second observation after \eqref{eq:realisable}). With this change the left-hand side of \eqref{eq:realisable} equals $0$ for all $0\leq j\leq g-1$, since $\lambda$ and $\mu$ have the same $e$-core. We now fix $0\leq j\leq g-1$ and show that the same holds for the right-hand side.
	
	For $b\in B$ we write $r_B(b)$ for the remainder of the difference of $\emp_B(b)$ and the runner of $b$ modulo $g$ (with $0\leq r_B(b)\leq g-1$) and we analogously define $r_{B'}(b')$ for $b'\in B'$. Then the right-hand side of (modified) \eqref{eq:realisable} equals the difference of the number of beads $b\in B$ with $r_B(b)=j$ and the number of beads $b'\in B'$ with $r_{B'}(b')=j$. If $b^{[1]}>b^{[2]}>\dots>b^{[t]}$ are the beads of $B$ (and $B'$) strictly between $b+e$ and $b$, $b^{[0]}=b+e$ and $b^{[t+1]} = b$, in this difference, we can count only beads $b\in B$ of the form $b^{[i+1]}$ and beads $b'\in B'$ of the form $b^{[i]}$ with $0\leq i\leq t$ (for other beads $b\in B\cap B'$, clearly, $r_B(b)=r_{B'}(b)$).  
	
	Using that $B'$ is obtained from $B$ by replacing $b$ by $b+e$ once more, we have $\emp_{B'}(b^{[i]}) - \emp_B(b^{[i+1]}) = b^{[i]}-b^{[i+1]}$ for all $0\leq i\leq t$. Working modulo $g$, we conclude that $r_B(b^{[i+1]}) = r_{B'}(b^{[i]})$ for all $0\leq i\leq t$, which proves that the right-hand side of \eqref{eq:realisable} is $0$. 
\end{proof}

Before the proof of the other direction, we sketch the construction of partition $\lambda$ with $e$-core $\gamma$ and $d$-runner matrix $\mathcal{R}$ satisfying \eqref{eq:realisable} for all $0\leq j\leq g-1$ in the following example.

\begin{example}\label{ex:existence}
	Let $e=5$, $d=3$, $\gamma$ be the $e$-core $(1^2)$ and $\mathcal{R}$ be the $2\times 5$ matrix $\big(\begin{smallmatrix}
		0 & 0 & 1 & 0 & 0\\ 
		0 & 0 & 0 & 0 & 1
	\end{smallmatrix}\big)$. The canonical $\beta$-set of $\gamma$ is $\left\lbrace 0, -1,-3, -4, \dots \right\rbrace $. We obtain the canonical $\beta$-set of a partition with $e$-core $\gamma$ and $d$-runner matrix $\mathcal{R}$ by adding a bead of $d$-emptiness $0$ on runner $0$, a bead of $d$-emptiness $1$ on runner $2$ and a bead of $d$-emptiness $2$ on runner $4$ to $S=\left\lbrace -4,-5,\dots \right\rbrace $; see the first diagram in \Cref{fig:construct}. We will try to add beads $-3,-2,\dots$ one by one; see the accompanying second diagram of \Cref{fig:construct}. We cannot add $-3$ as it would have $d$-emptiness $0$. We cannot add $-2$ as it does not lie on an available runner but we can add $-1$; its $d$-emptiness becomes $2$. Then $0$ and $1$ cannot be added for a wrong $d$-emptiness and a wrong runner reason, respectively. However, we can add $2$ and, eventually, $5$. We obtain $\left\lbrace 5,2,-1,-4,-5, \dots \right\rbrace $ the canonical $\beta$-set of $(6,4,2)$, which has $e$-core $\gamma$ and $d$-runner matrix $\mathcal{R}$.
\end{example}

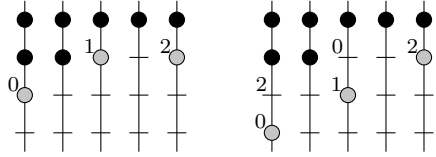
\begin{figure}[h]
	\centering
	\begin{tikzpicture}[x=0.5cm, y=0.5cm]
		\begin{pgfonlayer}{nodelayer}
			\node [style=none] (0) at (1.25, 3.5) {};
			\node [style=none] (1) at (2.25, 3.5) {};
			\node [style=none] (2) at (3.25, 3.5) {};
			\node [style=none] (3) at (4.25, 3.5) {};
			\node [style=none] (4) at (5.25, 3.5) {};
			\node [style=none] (5) at (1.25, -0.5) {};
			\node [style=none] (6) at (2.25, -0.5) {};
			\node [style=none] (7) at (3.25, -0.5) {};
			\node [style=none] (8) at (4.25, -0.5) {};
			\node [style=none] (9) at (5.25, -0.5) {};
			\node [style=Empty node] (10) at (1.25, 3) {};
			\node [style=Empty node] (11) at (2.25, 3) {};
			\node [style=Empty node] (12) at (3.25, 3) {};
			\node [style=Empty node] (13) at (4.25, 3) {};
			\node [style=Empty node] (15) at (1.25, 2) {};
			\node [style=Empty node] (16) at (2.25, 2) {};
			\node [style=none] (19) at (1, 0) {};
			\node [style=none] (20) at (1.5, 0) {};
			\node [style=none] (21) at (4, 2) {};
			\node [style=none] (22) at (4.5, 2) {};
			\node [style=none] (23) at (4, 1) {};
			\node [style=none] (24) at (4.5, 1) {};
			\node [style=none] (25) at (5, 1) {};
			\node [style=none] (26) at (5.5, 1) {};
			\node [style=none] (27) at (5.5, 0) {};
			\node [style=none] (28) at (5, 0) {};
			\node [style=none] (29) at (4.5, 0) {};
			\node [style=none] (30) at (4, 0) {};
			\node [style=none] (31) at (3.5, 0) {};
			\node [style=none] (32) at (3, 0) {};
			\node [style=none] (33) at (2.5, 0) {};
			\node [style=none] (34) at (2, 0) {};
			\node [style=none] (35) at (2.5, 1) {};
			\node [style=none] (36) at (2, 1) {};
			\node [style=none] (37) at (3.5, 1) {};
			\node [style=none] (38) at (3, 1) {};
			\node [style=Lead] (44) at (3.25, 2) {};
			\node [style=Lead] (45) at (5.25, 2) {};
			\node [style=Lead] (46) at (1.25, 1) {};
			\node [style=none] (47) at (7.75, 3.5) {};
			\node [style=none] (48) at (8.75, 3.5) {};
			\node [style=none] (49) at (9.75, 3.5) {};
			\node [style=none] (50) at (10.75, 3.5) {};
			\node [style=none] (51) at (11.75, 3.5) {};
			\node [style=none] (52) at (7.75, -0.5) {};
			\node [style=none] (53) at (8.75, -0.5) {};
			\node [style=none] (54) at (9.75, -0.5) {};
			\node [style=none] (55) at (10.75, -0.5) {};
			\node [style=none] (56) at (11.75, -0.5) {};
			\node [style=Empty node] (57) at (7.75, 3) {};
			\node [style=Empty node] (58) at (8.75, 3) {};
			\node [style=Empty node] (59) at (9.75, 3) {};
			\node [style=Empty node] (60) at (10.75, 3) {};
			\node [style=Empty node] (61) at (7.75, 2) {};
			\node [style=Empty node] (62) at (8.75, 2) {};
			\node [style=none] (63) at (7.5, 1) {};
			\node [style=none] (64) at (8, 1) {};
			\node [style=none] (65) at (10.5, 2) {};
			\node [style=none] (66) at (11, 2) {};
			\node [style=none] (67) at (10.5, 1) {};
			\node [style=none] (68) at (11, 1) {};
			\node [style=none] (69) at (11.5, 1) {};
			\node [style=none] (70) at (12, 1) {};
			\node [style=none] (71) at (12, 0) {};
			\node [style=none] (72) at (11.5, 0) {};
			\node [style=none] (73) at (11, 0) {};
			\node [style=none] (74) at (10.5, 0) {};
			\node [style=none] (75) at (10, 0) {};
			\node [style=none] (76) at (9.5, 0) {};
			\node [style=none] (77) at (9, 0) {};
			\node [style=none] (78) at (8.5, 0) {};
			\node [style=none] (79) at (9, 1) {};
			\node [style=none] (80) at (8.5, 1) {};
			\node [style=none] (81) at (10, 2) {};
			\node [style=none] (82) at (9.5, 2) {};
			\node [style=Lead] (84) at (11.75, 2) {};
			\node [style=Lead] (85) at (9.75, 1) {};
			\node [style=Lead] (86) at (7.75, 0) {};
			\node [style=none] (87) at (9.475, 2.3) {$\scriptstyle 0$};
			\node [style=none] (88) at (11.45, 2.3) {$\scriptstyle 2$};
			\node [style=none] (89) at (7.475, 1.3) {$\scriptstyle 2$};
			\node [style=none] (90) at (9.45, 1.3) {$\scriptstyle 1$};
			\node [style=none] (91) at (7.425, 0.3) {$\scriptstyle 0$};
			\node [style=none] (92) at (2.95, 2.3) {$\scriptstyle 1$};
			\node [style=none] (93) at (0.95, 1.3) {$\scriptstyle 0$};
			\node [style=none] (94) at (4.95, 2.3) {$\scriptstyle 2$};
			\node [style=Empty node] (95) at (5.25, 3) {};
			\node [style=Empty node] (96) at (11.75, 3) {};
		\end{pgfonlayer}
		\begin{pgfonlayer}{edgelayer}
			\draw (0.center) to (5.center);
			\draw (1.center) to (6.center);
			\draw (3.center) to (8.center);
			\draw (4.center) to (9.center);
			\draw (19.center) to (20.center);
			\draw (21.center) to (22.center);
			\draw (23.center) to (24.center);
			\draw (25.center) to (26.center);
			\draw (28.center) to (27.center);
			\draw (30.center) to (29.center);
			\draw (32.center) to (31.center);
			\draw (34.center) to (33.center);
			\draw (36.center) to (35.center);
			\draw (38.center) to (37.center);
			\draw (2.center) to (7.center);
			\draw (47.center) to (52.center);
			\draw (48.center) to (53.center);
			\draw (50.center) to (55.center);
			\draw (51.center) to (56.center);
			\draw (63.center) to (64.center);
			\draw (65.center) to (66.center);
			\draw (67.center) to (68.center);
			\draw (69.center) to (70.center);
			\draw (72.center) to (71.center);
			\draw (74.center) to (73.center);
			\draw (76.center) to (75.center);
			\draw (78.center) to (77.center);
			\draw (80.center) to (79.center);
			\draw (82.center) to (81.center);
			\draw (49.center) to (54.center);
		\end{pgfonlayer}
	\end{tikzpicture}
	\caption{The construction of the partition with prescribed $e$-core and $d$-runner matrix from \Cref{ex:existence}. In the first diagram, the highlighted beads are those which need to move on their runners to reach the denoted desired $d$-emptinesses. In the second diagram, we traverse the empty spaces in increasing order and move there a highlighted bead if it assigns the bead the correct $d$-emptiness. These to-be-$d$-emptinesses are computed in the diagram. They are $0$ on runner $2$ --- does not match, $2$ or runner $4$ --- matches, $2$ on runner $0$ --- does not match, $1$ on runner $2$ --- matches, and $0$ on runner $0$ --- matches.}
	\label{fig:construct}
\end{figure}

The general case follows the same idea. The key step is to use \eqref{eq:realisable} to show that we are able to place all the beads as in \Cref{ex:existence}.

\begin{lemma}\label{le:realisable suf}
	Let $\gamma$ be an $e$-core partition. Suppose that $\mathcal{R}$ is a $(d-1)\times e$ matrix of non-negative integers and that $\mathcal{D} = \mathcal{R} - \mathcal{R}_d(\gamma)$ satisfies \eqref{eq:realisable} for all $0\leq j\leq g-1$, where $g$ is the greatest common divisor of $d$ and $e$. Then there exists a partition with $e$-core $\gamma$ and $d$-runner matrix $\mathcal{R}$.
\end{lemma}

\begin{proof}
	Pick an integer $s$ divisible by $e$ such that $B=B_s(\gamma)$ contains all non-positive integers and for all $0\leq y\leq e-1$, the number of non-negative beads of $B$ on runner $y$ is at least the sum of the entries of $\mathcal{R}$ in column $y$. In this prove we extend all $d$-runner matrices of $\beta$-sets $C$ by an additional row $0$ by letting $\mathcal{R}_d(C)_{0,y}$ be the number of non-negative beads of $C$ with $d$-emptiness $0$ on runner $y$. By the choice of $s$, we can extend $\mathcal{R}$ by an additional row $0$ of non-negative integers such that the column sums of $\mathcal{R}$ and $\mathcal{R}_d(B)$ agree. In turn, we extend $\mathcal{D}$ by row $0$ and replace \eqref{eq:realisable} by the equivalent version
	  
	\begin{equation}\label{eq:modrealisable}
		\sum_{x=0}^{d-1}\sum_{\substack{k=0 \\ k\equiv j (\textnormal{mod } g)}}^{e-1} \mathcal{D}_{x,x+k}=0. 
	\end{equation}
	
	We claim that we can fix the negative beads of $B$ and move the remaining beads to obtain set $B'$ such that for any $0\leq x\leq d-1$ and $0\leq y\leq e-1$ there are $\mathcal{R}_{x,y}$ non-negative beads $b$ of $B'$ on runner $y$ with $d$-emptiness $x$. As such $B'$ is obtained by finitely many bead moves from $B$, it is the $\beta$-set $B_s(\lambda)$ of some partition $\lambda$ with $e$-core $\gamma$ and as the emptinesses of negative beads of $B'$ are $0$, the $d$-runner matrix of $B'$ (which is also the $d$-runner matrix of $\lambda$ by \Cref{le:runner matrix eq}) is $\mathcal{R}$; thus the proof is finished once such set $B'$ is constructed.
	
	We construct $B'$ as follows. We let $S = \left\lbrace -1, -2, \dots \right\rbrace $ and $\mathcal{M} = \mathcal{R}$. For $i=0,1,\dots$ we repeat the following: if $\mathcal{M}$ is a zero matrix, we return $B'=S$; otherwise, we set $y_i$ to be the runner of $i$ and $x_i$ to be the $d$-emptiness of $i-1$ in $S$ and if $\mathcal{M}_{x_i,y_i}>0$, we add $i$ to $S$ and decrease $\mathcal{M}_{x_i,y_i}$ by $1$. A simple induction on $i$ shows that after step $i$, we have $\mathcal{R} = \mathcal{M} + \mathcal{R}_d(S)$. Henceforth, if the algorithm terminates, then the set $B'$ has the desired property.
	
	Suppose that the algorithm does not terminate. Then there is $i_0\geq 0$ such that no $i\geq i_0$ is added to $S$. Let $\mathcal{M}'$ be the (non-zero) matrix $\mathcal{M}$ during step $i_0$ (and afterwards) and $S'$ be the set $S$ during step $i_0$ (and afterwards). We choose set $B'' \supset S'$ such that $B''$ and $B$ have the same number of non-negative beads on each runner and all the entries of $B''\setminus S'$ are at least $i_0$. Then $B''$ is the $\beta$-set $B_s(\mu)$ of some partition $\mu$ with $e$-core $\gamma$, and $\mathcal{R}_d(B'') = \mathcal{R}_d(\mu)$ and $\mathcal{R}_d(B)=\mathcal{R}_d(\gamma)$ have equal column sums. We conclude from \Cref{le:realisable nec} that 
	
	\begin{equation*}
		\sum_{x=0}^{d-1}\sum_{\substack{k=0 \\ k\equiv j (\textnormal{mod } g)}}^{e-1} \mathcal{R}_d(B)_{x,x+k}=\sum_{x=0}^{d-1}\sum_{\substack{k=0 \\ k\equiv j (\textnormal{mod } g)}}^{e-1} \mathcal{R}_d(B'')_{x,x+k} 
	\end{equation*}
	for all $0\leq j\leq g-1$. Using \eqref{eq:modrealisable} we have
	\begin{equation*}
		\sum_{x=0}^{d-1}\sum_{\substack{k=0 \\ k\equiv j (\textnormal{mod } g)}}^{e-1} \mathcal{R}_{x,x+k}=\sum_{x=0}^{d-1}\sum_{\substack{k=0 \\ k\equiv j (\textnormal{mod } g)}}^{e-1} \mathcal{R}_d(B'')_{x,x+k}, 
	\end{equation*}
	and, in turn,
	\begin{equation}\label{eq:enhanced real}
		\sum_{x=0}^{d-1}\sum_{\substack{k=0 \\ k\equiv j (\textnormal{mod } g)}}^{e-1} \mathcal{M}'_{x,x+k}=\sum_{x=0}^{d-1}\sum_{\substack{k=0 \\ k\equiv j (\textnormal{mod } g)}}^{e-1} \left( \mathcal{R}_d(B'')_{x,x+k} - \mathcal{R}_d(S')_{x,x+k}\right) 
	\end{equation}
	for all $0\leq j\leq g-1$.
	
	Let $b=\min \left( B''\setminus S'\right) \geq i_0$. We conclude from \eqref{eq:enhanced real} that there is $0\leq x\leq d-1$ and $0\leq y\leq e-1$ such that $\mathcal{M}'_{x,y}>0$ and $y-x$ is congruent to $b-\emp_{B''}(b)$ modulo $g$. By the Chinese remainder theorem, there is a non-negative integer $t$ such that $y\equiv b+t\, (\textnormal{mod } e)$ and $x\equiv \emp_{B''}(b)+t\, (\textnormal{mod } d)$. But then our algorithm would add $i=b+t\geq i_0$ to $S$ as it lies on runner $y$, the $d$-emptiness in $S'$ of $b+t-1$ is $x$ since $\emp_{S'}(b+t-1) = \emp_{S'}(b-1) + t = \emp_{B''}(b) + t$, and $\mathcal{M}'_{x,y}>0$, a contradiction. Thus our algorithm terminates, finishing the proof.
\end{proof}

\begin{remark}\label{re:greedy}
	The greedy algorithm in the proof of \Cref{le:realisable suf} guarantees that the constructed partition is a \dsunb partition. By \Cref{pr:equivalence max}, it is also the unique maximal element of $\mathcal{E}_{\mathcal{R}}(\gamma)$ in the dominance order. 
\end{remark}

\begin{corollary}\label{cor:existence}
	Let $\gamma$ be an $e$-core partition and $\mathcal{R}$ be a $(d-1)\times e$ matrix of non-negative integers. Then $\mathcal{R}$ is $\gamma$-realisable if and only if \eqref{eq:realisable} holds for $\mathcal{D} = \mathcal{R} - \mathcal{R}_d(\gamma)$ and all $0\leq j\leq g-1$, where $g$ is the greatest common divisor of $d$ and $e$. In particular, $\mathcal{R}$ is always $\gamma$-realisable if $d$ and $e$ are coprime.
\end{corollary}

\begin{proof}
	This follows from \Cref{le:realisable nec} and \Cref{le:realisable suf}.
\end{proof}

\section{Abacus Mullineux Algorithm}\label{se:mullineux}

We need the following definition to describe the new Abacus Mullineux Algorithm presented here.

\begin{definition}\label{de:leading}
	Let $B$ be a $\beta$-set of a partition. The \textit{leading beads} of $B$ are beads $b_0, b_1, \dots$ such that:
	\begin{enumerate}[label=\textnormal{(\roman*)}]
		\item $b_0$ is the greatest bead of $B$,
		\item for $i\geq 0$ the bead $b_{i+1}$ is the greatest bead which is less or equal to $b_i-e$.
	\end{enumerate}
\end{definition}  

One can summarise the definition of leading beads as beads which satisfy

\begin{equation}\label{eq:leading}
b_0 > b_0 -e \geq b_1 > b_1-e \geq \dots
\end{equation}
and there is no bead $b$ such that $b_i<b \leq b_{i-1} - e$ (for $i=0$ the right inequality is discarded).  Since $B$ contains all integers which are less than some integer $L$, for some index $i\geq 0$ we have $b_i = b_{i+1}+e = b_{i+2}+2e=\dots$. We refer to the least such index as the \textit{stable index} of $B$. Equivalently, the stable index of $B$ is the least index from which the sequence $(b_i+ie)_{i\geq 0}$ becomes constant.

\begin{definition}\label{de:shift}
	Let $B$ be a $\beta$-set of a partition. Let $b_0, b_1,\dots$ be its leading beads and $z$ be its stable index. We define the set $\Ms_e(B)$ to be $\left( B \setminus \left\lbrace b_i : i\geq 0 \right\rbrace\right)  \cup \left\lbrace b_i - e : i\geq 0 \right\rbrace $ and the integer $\ms_e(B)$ to be $b_z+ze$.
\end{definition}

From the definition of the leading beads, $\left\lbrace b_i : i\geq 0 \right\rbrace$ is a subset of $B$ and the union in the definition of $\Ms_e(B)$ is disjoint (if $b_i-e\in B$ for some $i\geq 0$, then $b_i-e=b_{i+1}$). An example of the leading beads and the maps $\Ms_e$ and $\ms_e$ is in \Cref{fig:lead}. 

\begin{figure}[h]
	\centering
	\begin{tikzpicture}[x=0.5cm, y=0.5cm]
		\begin{pgfonlayer}{nodelayer}
			\node [style=none] (0) at (1.25, 3.5) {};
			\node [style=none] (1) at (2.25, 3.5) {};
			\node [style=none] (2) at (3.25, 3.5) {};
			\node [style=none] (3) at (4.25, 3.5) {};
			\node [style=none] (4) at (5.25, 3.5) {};
			\node [style=none] (5) at (1.25, -0.5) {};
			\node [style=none] (6) at (2.25, -0.5) {};
			\node [style=none] (7) at (3.25, -0.5) {};
			\node [style=none] (8) at (4.25, -0.5) {};
			\node [style=none] (9) at (5.25, -0.5) {};
			\node [style=Empty node] (10) at (1.25, 3) {};
			\node [style=Empty node] (11) at (2.25, 3) {};
			\node [style=Empty node] (12) at (3.25, 3) {};
			\node [style=Empty node] (13) at (4.25, 3) {};
			\node [style=Empty node] (15) at (1.25, 2) {};
			\node [style=Empty node] (16) at (2.25, 2) {};
			\node [style=Empty node] (18) at (3.25, 1) {};
			\node [style=none] (19) at (3, 2) {};
			\node [style=none] (20) at (3.5, 2) {};
			\node [style=none] (21) at (4, 2) {};
			\node [style=none] (22) at (4.5, 2) {};
			\node [style=none] (23) at (4, 1) {};
			\node [style=none] (24) at (4.5, 1) {};
			\node [style=none] (25) at (5, 1) {};
			\node [style=none] (26) at (5.5, 1) {};
			\node [style=none] (27) at (5.5, 0) {};
			\node [style=none] (28) at (5, 0) {};
			\node [style=none] (29) at (4.5, 0) {};
			\node [style=none] (30) at (4, 0) {};
			\node [style=none] (31) at (3.5, 0) {};
			\node [style=none] (32) at (3, 0) {};
			\node [style=none] (33) at (2.5, 0) {};
			\node [style=none] (34) at (2, 0) {};
			\node [style=none] (35) at (2.5, 1) {};
			\node [style=none] (36) at (2, 1) {};
			\node [style=none] (37) at (1.5, 1) {};
			\node [style=none] (38) at (1, 1) {};
			\node [style=Lead] (39) at (1.25, 0) {};
			\node [style=Lead] (40) at (5.25, 2) {};
			\node [style=Lead] (41) at (5.25, 3) {};
			\node [style=none] (42) at (8.5, 3.5) {};
			\node [style=none] (43) at (9.5, 3.5) {};
			\node [style=none] (44) at (10.5, 3.5) {};
			\node [style=none] (45) at (11.5, 3.5) {};
			\node [style=none] (46) at (12.5, 3.5) {};
			\node [style=none] (47) at (8.5, -0.5) {};
			\node [style=none] (48) at (9.5, -0.5) {};
			\node [style=none] (49) at (10.5, -0.5) {};
			\node [style=none] (50) at (11.5, -0.5) {};
			\node [style=none] (51) at (12.5, -0.5) {};
			\node [style=Empty node] (52) at (8.5, 3) {};
			\node [style=Empty node] (53) at (9.5, 3) {};
			\node [style=Empty node] (54) at (10.5, 3) {};
			\node [style=Empty node] (55) at (11.5, 3) {};
			\node [style=Empty node] (56) at (8.5, 2) {};
			\node [style=Empty node] (57) at (9.5, 2) {};
			\node [style=Empty node] (58) at (10.5, 1) {};
			\node [style=none] (59) at (10.25, 2) {};
			\node [style=none] (60) at (10.75, 2) {};
			\node [style=none] (61) at (11.25, 2) {};
			\node [style=none] (62) at (11.75, 2) {};
			\node [style=none] (63) at (11.25, 1) {};
			\node [style=none] (64) at (11.75, 1) {};
			\node [style=none] (67) at (12.75, 0) {};
			\node [style=none] (68) at (12.25, 0) {};
			\node [style=none] (69) at (11.75, 0) {};
			\node [style=none] (70) at (11.25, 0) {};
			\node [style=none] (71) at (10.75, 0) {};
			\node [style=none] (72) at (10.25, 0) {};
			\node [style=none] (73) at (9.75, 0) {};
			\node [style=none] (74) at (9.25, 0) {};
			\node [style=none] (75) at (9.75, 1) {};
			\node [style=none] (76) at (9.25, 1) {};
			\node [style=Empty node] (77) at (12.5, 3) {};
			\node [style=Empty node] (78) at (8.5, 1) {};
			\node [style=none] (79) at (12.25, 2) {};
			\node [style=none] (80) at (12.75, 2) {};
			\node [style=none] (81) at (8.25, 0) {};
			\node [style=none] (82) at (8.75, 0) {};
			\node [style=square] (83) at (12.5, 1) {};
		\end{pgfonlayer}
		\begin{pgfonlayer}{edgelayer}
			\draw (0.center) to (5.center);
			\draw (1.center) to (6.center);
			\draw (3.center) to (8.center);
			\draw (4.center) to (9.center);
			\draw (19.center) to (20.center);
			\draw (21.center) to (22.center);
			\draw (23.center) to (24.center);
			\draw (25.center) to (26.center);
			\draw (28.center) to (27.center);
			\draw (30.center) to (29.center);
			\draw (32.center) to (31.center);
			\draw (34.center) to (33.center);
			\draw (36.center) to (35.center);
			\draw (38.center) to (37.center);
			\draw (2.center) to (7.center);
			\draw (42.center) to (47.center);
			\draw (43.center) to (48.center);
			\draw (45.center) to (50.center);
			\draw (46.center) to (51.center);
			\draw (59.center) to (60.center);
			\draw (61.center) to (62.center);
			\draw (63.center) to (64.center);
			\draw (68.center) to (67.center);
			\draw (70.center) to (69.center);
			\draw (72.center) to (71.center);
			\draw (74.center) to (73.center);
			\draw (76.center) to (75.center);
			\draw (44.center) to (49.center);
			\draw (79.center) to (80.center);
			\draw (81.center) to (82.center);
		\end{pgfonlayer}
	\end{tikzpicture}
	\caption{The leading beads of the set $B$ on the left are drawn grey. Its stable index is $1$. The square bead on the right is the bead $\ms_e(B)$ and the remaining beads are the beads of $\Ms_e(B)$.}
	\label{fig:lead}
\end{figure}
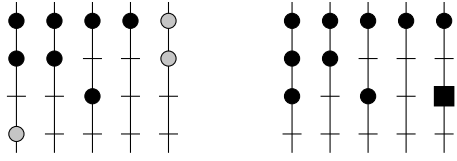

Throughout the paper we simplify the notation and write $\hat{b}$ for $\ms_e(B)$ whenever the set $B$ is clear from the context (the integer $e$ is always our fixed integer $e>1$). Similarly, the leading beads of a set called $B$ will always be denoted by $b_0,b_1,\dots$. The following two simple lemmas show some initial properties of the leading beads, the set $\Ms_e(B)$ and $\hat{b}$ which will be used throughout the paper.

In the next lemma, (i) can be replaced by the statement that the sequence $(b_i+ie)_{i\geq 0}$ is non-increasing.

\begin{lemma}\label{le:new bead}
	Let $B$ be a $\beta$-set of a partition with stable index $z$.
	\begin{enumerate}[label=\textnormal{(\roman*)}]
		\item For any $0\leq i\leq j$, we have $b_j +je \leq b_i +ie$.
		\item For any $i\geq 0$, $\hat{b}\leq b_i+ie$ with equality if and only if $i\geq z$. In particular, $\hat{b}\leq b_0$. 
	\end{enumerate}
\end{lemma}

\begin{proof}
	It is sufficient to consider $j=i+1$ in (i) which follows directly from the definition of leading beads. Moving to (ii), since $b_z\neq b_{z-1}-e$, we have $b_{z}< b_{z-1}-e$ (provided that $z>0$). Hence, for $0\leq i\leq z-1$, we have $\hat{b} = b_z +ze<b_{z-1}+(z-1)e\leq b_i+ie$, where the last inequality follows from (i). On the other hand if $i\geq z$, then $\hat{b} = b_z +ze=b_i+ie$, finishing the proof.
\end{proof}

The next lemma shows different descriptions of the set $\Ms_e(B)$.

\begin{lemma}\label{le:Me observation}
	Let $B$ be a $\beta$-set of a partition and $z$ be its stable index.
	\begin{enumerate}[label=\textnormal{(\roman*)}]
		\item For $j\geq z$, $\Ms_e(B)=\left( B \setminus \left\lbrace b_i :  0\leq i \leq j \right\rbrace\right)  \cup \left\lbrace b_i - e : 0\leq i< j \right\rbrace $.
		\item For any integer $i\geq 0$, the leading bead $b_i$ lies in $\Ms_e(B)$ if and only if $i>0$ and $b_i+e=b_{i-1}$.
		\item $b_z\notin\Ms_e(B)$.
		\item Let $0=i_0<i_1<\dots<i_t=z$ be the indices $i$ such that $b_i\notin \Ms_e(B)$. Then
		\[ B \setminus \left\lbrace b_{i_j} : 0\leq j \leq t \right\rbrace = B\cap \Ms_e(B) = \Ms_e(B)  \setminus \left\lbrace b_{i_j-1} - e : 1\leq j \leq t \right\rbrace.
		\]
		\item In the setting of (iv) we have $b_i = b_{i_j} - (i-i_j)e$ for any $0\leq j\leq t$ and $i_j\leq i< i_{j+1}$ (where the right inequality is omitted if $j=t$).  
	\end{enumerate}
\end{lemma}

\begin{proof}
	In (i), for any $i> j\geq z$, we have $b_i = b_{i-1}-e$. Thus $\left\lbrace b_i :  i>j \right\rbrace = \left\lbrace b_i - e : i\geq j \right\rbrace$. We immediately deduce (i) from the definition of $\Ms_e(B)$. For (ii), the leading bead $b_i$ lies in $\Ms_e(B)$ if and only if $b_i = b_{i'}-e$ for some integer $i'\geq 0$. From \eqref{eq:leading} if such $i'$ exists it is $i-1$, which recovers (ii). From the definition of the stable index $z$, either $z=0$ or $b_z < b_{z-1}-e$. So (iii) follows from (ii).
	
	The setting of (iv) is displayed in \Cref{fig:invariance}. Note that the equality $i_0=0$ is clear and the equality $i_t=z$ follows from (ii) --- it shows that $i_t\leq z$ --- and (iii). Since the non-leading beads of $B$ lie in $\Ms_e(B)$, we obtain the left equality. The only elements of $\Ms_e(B)$ which may not lie in $B$ are of the form $b_i-e$. From the definition of leading beads, $b_i-e$ does not lie in $B$ if and only if $b_i-e\neq b_{i+1}$, which by (ii) happens if and only if $b_{i+1}\notin \Ms_e(B)$. This proves the right equality. Finally, fix $j$ as in (v). We prove the statement by induction on $i$. If $i=i_j$, the statement is clear. If $i_j< i< i_{j+1}$ (where the right inequality is omitted if $j=t$), then by (ii) and the induction hypothesis, $b_i = b_{i-1}-e=b_{i_j} - (i -1-i_j)e -e=b_{i_j} - (i-i_j)e$, as required.
\end{proof}

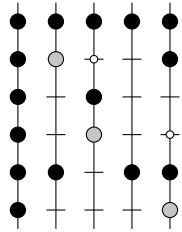
\begin{figure}[h]
	\centering
	\begin{tikzpicture}[x=0.5cm, y=0.5cm]
		\begin{pgfonlayer}{nodelayer}
			\node [style=none] (0) at (1.25, 3.5) {};
			\node [style=none] (1) at (2.25, 3.5) {};
			\node [style=none] (2) at (3.25, 3.5) {};
			\node [style=none] (3) at (4.25, 3.5) {};
			\node [style=none] (4) at (5.25, 3.5) {};
			\node [style=none] (5) at (1.25, -2.5) {};
			\node [style=none] (6) at (2.25, -2.5) {};
			\node [style=none] (7) at (3.25, -2.5) {};
			\node [style=none] (8) at (4.25, -2.5) {};
			\node [style=none] (9) at (5.25, -2.5) {};
			\node [style=Empty node] (10) at (1.25, 3) {};
			\node [style=Empty node] (11) at (2.25, 3) {};
			\node [style=Empty node] (12) at (3.25, 3) {};
			\node [style=Empty node] (13) at (4.25, 3) {};
			\node [style=Empty node] (15) at (1.25, 2) {};
			\node [style=Empty node] (18) at (3.25, 1) {};
			\node [style=none] (19) at (3, 2) {};
			\node [style=none] (20) at (3.5, 2) {};
			\node [style=none] (21) at (4, 2) {};
			\node [style=none] (22) at (4.5, 2) {};
			\node [style=none] (23) at (4, 1) {};
			\node [style=none] (24) at (4.5, 1) {};
			\node [style=none] (25) at (4, -2) {};
			\node [style=none] (26) at (4.5, -2) {};
			\node [style=none] (27) at (5.5, 0) {};
			\node [style=none] (28) at (5, 0) {};
			\node [style=none] (29) at (4.5, 0) {};
			\node [style=none] (30) at (4, 0) {};
			\node [style=none] (31) at (3.5, -2) {};
			\node [style=none] (32) at (3, -2) {};
			\node [style=none] (33) at (2.5, 0) {};
			\node [style=none] (34) at (2, 0) {};
			\node [style=none] (35) at (2.5, 1) {};
			\node [style=none] (36) at (2, 1) {};
			\node [style=none] (37) at (3.5, -1) {};
			\node [style=none] (38) at (3, -1) {};
			\node [style=Empty node] (40) at (5.25, -1) {};
			\node [style=Empty node] (41) at (1.25, 1) {};
			\node [style=Empty node] (42) at (1.25, 0) {};
			\node [style=Empty node] (43) at (1.25, -1) {};
			\node [style=Empty node] (44) at (1.25, -2) {};
			\node [style=Empty node] (45) at (2.25, -1) {};
			\node [style=Empty node] (46) at (5.25, 2) {};
			\node [style=Empty node] (47) at (5.25, 3) {};
			\node [style=Empty node] (49) at (4.25, -1) {};
			\node [style=Lead] (50) at (5.25, -2) {};
			\node [style=Lead] (51) at (3.25, 0) {};
			\node [style=Lead] (52) at (2.25, 2) {};
			\node [style=Small Empty] (53) at (5.25, 0) {};
			\node [style=Small Empty] (54) at (3.25, 2) {};
			\node [style=none] (55) at (2, -2) {};
			\node [style=none] (56) at (2.5, -2) {};
			\node [style=none] (57) at (5, 1) {};
			\node [style=none] (58) at (5.5, 1) {};
		\end{pgfonlayer}
		\begin{pgfonlayer}{edgelayer}
			\draw (0.center) to (5.center);
			\draw (1.center) to (6.center);
			\draw (3.center) to (8.center);
			\draw (4.center) to (9.center);
			\draw (19.center) to (20.center);
			\draw (21.center) to (22.center);
			\draw (23.center) to (24.center);
			\draw (25.center) to (26.center);
			\draw (28.center) to (27.center);
			\draw (30.center) to (29.center);
			\draw (32.center) to (31.center);
			\draw (34.center) to (33.center);
			\draw (36.center) to (35.center);
			\draw (38.center) to (37.center);
			\draw (2.center) to (7.center);
			\draw (55.center) to (56.center);
			\draw (57.center) to (58.center);
		\end{pgfonlayer}
	\end{tikzpicture}
	\caption{In the above diagram of a $\beta$-set $B$ the indices $i_j$ from \Cref{le:Me observation}(iv) are $0,2$ and $4$. The corresponding beads are highlighted. We have also highlighted (with a small white dot), the empty spaces $b_{i_j-1} - e$. As asserted by (iv), one obtains $\Ms_e(B)$ by removing the highlighted beads and inserting beads in the highlighted empty spaces.}
	\label{fig:invariance}
\end{figure}

The following are more technical observations needed for the proof of the main result of this section.

\begin{lemma}\label{le:shift}
	Let $B$ be a $\beta$-set of a partition such that $\Ms_e(B)$ contains all non-positive integers and let $z$ be the stable index of $B$.
	\begin{enumerate}[label=\textnormal{(\roman*)}]
		\item If $j$ is the greatest index such that $b_j\geq 0$, then $j\geq z$ and $\hat{b}=b_j+je$.
		\item If $r$ is the runner of $\hat{b}$, then the number of non-negative beads of $B$ on runner $r$ is one more that the number of non-negative beads of $\Ms_e(B)$ on runner $r$.
		\item Let $0\leq r\leq e-1$ be such that $r$ is \emph{not} the runner of $\hat{b}$. The numbers of non-negative beads of $B$ and of $\Ms_e(B)$ on runner $r$ agree. 
	\end{enumerate}
\end{lemma}

\begin{proof}
	We have $b_z\notin \Ms_e(B)$ by \Cref{le:Me observation}(iii). Thus $b_z>0$ and consequently $j$ in (i) is well-defined and $j\geq z$, which, by \Cref{le:new bead}(ii), implies $\hat{b}=b_j+je$, as required for (i). As $b_j-e=b_{j+1}<0\leq b_j\leq b_{j-1}-e$ (for $j=0$ the rightmost inequality is omitted), $b_j$ is the single non-negative leading bead $b$ such that $b-e$ is negative. Since $b_j$ lies on the same runner as $\hat{b}$, we deduce (ii) and (iii) from the definition of $\Ms_e(B)$ at once.
\end{proof}

We now connect the leading beads of a $\beta$-set of a partition $\lambda$ to operations on the Young diagram of $\lambda$. This connection will allow us to translate the Mullineux map on an abacus (but will not be used beyond this section). The \textit{$e$-rim} of partition $\lambda$ is a subset of the rim of $\lambda$ obtained by traversing the rim of $\lambda$ from $(1,\lambda_1)$ and repeating the following: we include the next $e$ boxes in the $e$-rim (or all the boxes, if there are less than $e$ boxes remaining) and then immediately move to the rightmost box of the next row (or end the procedure if there is no next row).

Alternatively, the repetitive procedure can be replaced by: include the first several boxes which form an $e$-divisible hook (or all boxes if no such $e$-divisible hook exists) and then immediately move to the rightmost box of the next row (or end the procedure if there is no next row). We can use this definition to introduce the \textit{proper $e$-rim} which is obtained in the same way but we do not include the final rim hook if it is not an $e$-divisible hook. See \Cref{fig:e-rim} for an example.

\begin{figure}[h]
	\centering
	\begin{tikzpicture}[x=0.5cm, y=0.5cm]
		\begin{pgfonlayer}{nodelayer}
			\node [style=none] (0) at (3.5, 3) {};
			\node [style=none] (1) at (9.5, 3) {};
			\node [style=none] (2) at (9.5, 2) {};
			\node [style=none] (3) at (7.5, 2) {};
			\node [style=none] (4) at (7.5, 1) {};
			\node [style=none] (5) at (5.5, 1) {};
			\node [style=none] (6) at (5.5, 0) {};
			\node [style=none] (7) at (3.5, 0) {};
			\node [style=none] (8) at (3.5, 1) {};
			\node [style=none] (9) at (3.5, 2) {};
			\node [style=none] (10) at (4.5, 0) {};
			\node [style=none] (11) at (4.5, 3) {};
			\node [style=none] (12) at (5.5, 3) {};
			\node [style=none] (13) at (6.5, 3) {};
			\node [style=none] (14) at (7.5, 3) {};
			\node [style=none] (15) at (8.5, 3) {};
			\node [style=none] (16) at (8.5, 2) {};
			\node [style=none] (17) at (6.5, 1) {};
			\node [style=none] (18) at (6.5, 2) {};
			\node [style=none] (19) at (4.5, 2) {};
			\node [style=none] (20) at (4.5, 1) {};
			\node [style=none] (21) at (13, 3) {};
			\node [style=none] (22) at (19, 3) {};
			\node [style=none] (23) at (19, 2) {};
			\node [style=none] (24) at (17, 2) {};
			\node [style=none] (25) at (17, 1) {};
			\node [style=none] (26) at (15, 1) {};
			\node [style=none] (27) at (15, 0) {};
			\node [style=none] (28) at (13, 0) {};
			\node [style=none] (29) at (13, 1) {};
			\node [style=none] (30) at (13, 2) {};
			\node [style=none] (31) at (14, 0) {};
			\node [style=none] (32) at (14, 3) {};
			\node [style=none] (33) at (15, 3) {};
			\node [style=none] (34) at (16, 3) {};
			\node [style=none] (35) at (17, 3) {};
			\node [style=none] (36) at (18, 3) {};
			\node [style=none] (37) at (18, 2) {};
			\node [style=none] (38) at (16, 1) {};
			\node [style=none] (39) at (16, 2) {};
			\node [style=none] (40) at (14, 2) {};
			\node [style=none] (41) at (14, 1) {};
			\node [style=none] (42) at (5.5, 2) {};
			\node [style=none] (43) at (15, 2) {};
		\end{pgfonlayer}
		\begin{pgfonlayer}{edgelayer}
			\draw [style=Light grey column] (13.center)
			to (1.center)
			to (2.center)
			to (3.center)
			to (4.center)
			to (5.center)
			to (42.center)
			to (18.center)
			to cycle;
			\draw [style=Light grey column] (8.center)
			to (5.center)
			to (6.center)
			to (7.center)
			to cycle;
			\draw [style=Light grey column] (24.center)
			to (25.center)
			to (26.center)
			to (43.center)
			to (39.center)
			to (34.center)
			to (22.center)
			to (23.center)
			to cycle;
			\draw (8.center) to (0.center);
			\draw (0.center) to (13.center);
			\draw (11.center) to (10.center);
			\draw (12.center) to (42.center);
			\draw (14.center) to (3.center);
			\draw (15.center) to (16.center);
			\draw (18.center) to (3.center);
			\draw (9.center) to (42.center);
			\draw (18.center) to (17.center);
			\draw (36.center) to (37.center);
			\draw (35.center) to (24.center);
			\draw (39.center) to (38.center);
			\draw (33.center) to (43.center);
			\draw (26.center) to (27.center);
			\draw (32.center) to (31.center);
			\draw (21.center) to (34.center);
			\draw (30.center) to (43.center);
			\draw (29.center) to (26.center);
			\draw (28.center) to (27.center);
			\draw (21.center) to (28.center);
			\draw (39.center) to (24.center);
		\end{pgfonlayer}
	\end{tikzpicture}
	\caption{Let $e=5$. The highlighted boxes in the left diagram form the $e$-rim of $(6,4,2)$ and in the right diagram they form the proper $e$-rim of the same partition. }
	\label{fig:e-rim}
\end{figure}
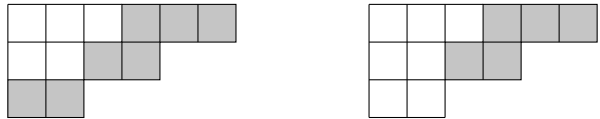

For $a$ a positive integer, we recall the correspondence (previously mentioned with $a=1$) between $ae$-hooks $R$ of a partition $\lambda$ and beads $b$ of its $\beta$-set $B$ such that $b-ae$ is an empty space of $B$. It assigns to $R$ the $\tp(R)$'th largest bead $b$ of $B$ and the removal of $R$ corresponds to replacing bead $b$ with $b-ae$. It is well-known and easy to check that the number of beads of $B$ greater than $b-ae$ equals $\bt(R)$. The next lemma describes the removal of the proper $e$-rim of $\lambda$ on a $\beta$-set of $\lambda$. The idea of the proof is sketched in \Cref{fig:proper e-rim}. In the statement we take $t'=-1$ if the proper $e$-rim is empty.

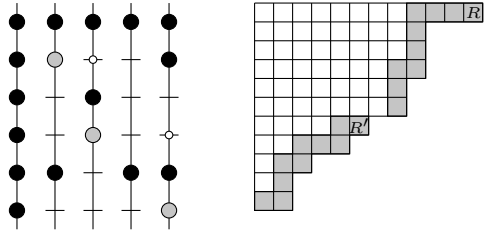
\begin{figure}[h]
 	\centering
 	\begin{tikzpicture}[x=0.5cm, y=0.5cm]
 		\begin{pgfonlayer}{nodelayer}
 			\node [style=none] (0) at (1.25, 3.5) {};
 			\node [style=none] (1) at (2.25, 3.5) {};
 			\node [style=none] (2) at (3.25, 3.5) {};
 			\node [style=none] (3) at (4.25, 3.5) {};
 			\node [style=none] (4) at (5.25, 3.5) {};
 			\node [style=none] (5) at (1.25, -2.5) {};
 			\node [style=none] (6) at (2.25, -2.5) {};
 			\node [style=none] (7) at (3.25, -2.5) {};
 			\node [style=none] (8) at (4.25, -2.5) {};
 			\node [style=none] (9) at (5.25, -2.5) {};
 			\node [style=Empty node] (10) at (1.25, 3) {};
 			\node [style=Empty node] (11) at (2.25, 3) {};
 			\node [style=Empty node] (12) at (3.25, 3) {};
 			\node [style=Empty node] (13) at (4.25, 3) {};
 			\node [style=Empty node] (15) at (1.25, 2) {};
 			\node [style=Empty node] (18) at (3.25, 1) {};
 			\node [style=none] (19) at (3, 2) {};
 			\node [style=none] (20) at (3.5, 2) {};
 			\node [style=none] (21) at (4, 2) {};
 			\node [style=none] (22) at (4.5, 2) {};
 			\node [style=none] (23) at (4, 1) {};
 			\node [style=none] (24) at (4.5, 1) {};
 			\node [style=none] (25) at (4, -2) {};
 			\node [style=none] (26) at (4.5, -2) {};
 			\node [style=none] (27) at (5.5, 0) {};
 			\node [style=none] (28) at (5, 0) {};
 			\node [style=none] (29) at (4.5, 0) {};
 			\node [style=none] (30) at (4, 0) {};
 			\node [style=none] (31) at (3.5, -2) {};
 			\node [style=none] (32) at (3, -2) {};
 			\node [style=none] (33) at (2.5, 0) {};
 			\node [style=none] (34) at (2, 0) {};
 			\node [style=none] (35) at (2.5, 1) {};
 			\node [style=none] (36) at (2, 1) {};
 			\node [style=none] (37) at (3.5, -1) {};
 			\node [style=none] (38) at (3, -1) {};
 			\node [style=Empty node] (40) at (5.25, -1) {};
 			\node [style=Empty node] (41) at (1.25, 1) {};
 			\node [style=Empty node] (42) at (1.25, 0) {};
 			\node [style=Empty node] (43) at (1.25, -1) {};
 			\node [style=Empty node] (44) at (1.25, -2) {};
 			\node [style=Empty node] (45) at (2.25, -1) {};
 			\node [style=Empty node] (46) at (5.25, 2) {};
 			\node [style=Empty node] (47) at (5.25, 3) {};
 			\node [style=Empty node] (49) at (4.25, -1) {};
 			\node [style=Lead] (50) at (5.25, -2) {};
 			\node [style=Lead] (51) at (3.25, 0) {};
 			\node [style=Small Empty] (53) at (5.25, 0) {};
 			\node [style=Small Empty] (54) at (3.25, 2) {};
 			\node [style=none] (55) at (7.5, 3.5) {};
 			\node [style=none] (56) at (13.5, 3.5) {};
 			\node [style=none] (57) at (13.5, 3) {};
 			\node [style=none] (58) at (12, 3) {};
 			\node [style=none] (59) at (12, 1.5) {};
 			\node [style=none] (60) at (11.5, 1.5) {};
 			\node [style=none] (61) at (11.5, 0.5) {};
 			\node [style=none] (71) at (11.5, 3.5) {};
 			\node [style=none] (72) at (11.5, 2) {};
 			\node [style=none] (73) at (11, 2) {};
 			\node [style=none] (74) at (11, 0.5) {};
 			\node [style=none] (82) at (13, 3.5) {};
 			\node [style=none] (83) at (12.5, 3.5) {};
 			\node [style=none] (84) at (12, 3.5) {};
 			\node [style=none] (85) at (11, 3.5) {};
 			\node [style=none] (86) at (10.5, 3.5) {};
 			\node [style=none] (87) at (10, 3.5) {};
 			\node [style=none] (88) at (9.5, 3.5) {};
 			\node [style=none] (89) at (9, 3.5) {};
 			\node [style=none] (90) at (8.5, 3.5) {};
 			\node [style=none] (91) at (7.5, 3) {};
 			\node [style=none] (92) at (7.5, 2.5) {};
 			\node [style=none] (93) at (7.5, 2) {};
 			\node [style=none] (94) at (7.5, 1.5) {};
 			\node [style=none] (95) at (7.5, 1) {};
 			\node [style=none] (96) at (7.5, 0.5) {};
 			\node [style=none] (97) at (7.5, 0) {};
 			\node [style=none] (98) at (7.5, -0.5) {};
 			\node [style=none] (99) at (7.5, -1) {};
 			\node [style=none] (100) at (7.5, -1.5) {};
 			\node [style=none] (102) at (8.5, -2) {};
 			\node [style=none] (103) at (10, -0.5) {};
 			\node [style=none] (104) at (11.5, 1) {};
 			\node [style=none] (105) at (12, 2) {};
 			\node [style=none] (106) at (12, 2.5) {};
 			\node [style=none] (107) at (12.5, 3) {};
 			\node [style=none] (108) at (13, 3) {};
 			\node [style=none] (109) at (2, -2) {};
 			\node [style=none] (110) at (2.5, -2) {};
 			\node [style=none] (112) at (13.25, 3.25) {$\scriptscriptstyle R$};
 			\node [style=none] (113) at (10.25, 0.25) {$\scriptscriptstyle R'$};
 			\node [style=none] (114) at (5, 1) {};
 			\node [style=none] (115) at (5.5, 1) {};
 			\node [style=none] (116) at (10.5, 0) {};
 			\node [style=none] (117) at (10, 0) {};
 			\node [style=none] (118) at (10.5, 0.5) {};
 			\node [style=none] (119) at (9, -1) {};
 			\node [style=none] (120) at (9, -0.5) {};
 			\node [style=none] (121) at (7.5, -2) {};
 			\node [style=none] (122) at (8.5, -1) {};
 			\node [style=none] (123) at (9.5, 0.5) {};
 			\node [style=none] (124) at (9.5, 0) {};
 			\node [style=none] (125) at (8.5, 0) {};
 			\node [style=none] (126) at (8, -1.5) {};
 			\node [style=none] (127) at (8, -0.5) {};
 			\node [style=none] (128) at (8.5, -0.5) {};
 			\node [style=none] (129) at (9.5, -0.5) {};
 			\node [style=none] (130) at (8.5, -1.5) {};
 			\node [style=none] (131) at (8, -2) {};
 			\node [style=none] (132) at (8, 3.5) {};
 			\node [style=Lead] (133) at (2.25, 2) {};
 		\end{pgfonlayer}
 		\begin{pgfonlayer}{edgelayer}
 			\draw (0.center) to (5.center);
 			\draw (1.center) to (6.center);
 			\draw (3.center) to (8.center);
 			\draw (4.center) to (9.center);
 			\draw (19.center) to (20.center);
 			\draw (21.center) to (22.center);
 			\draw (23.center) to (24.center);
 			\draw (25.center) to (26.center);
 			\draw (28.center) to (27.center);
 			\draw (30.center) to (29.center);
 			\draw (32.center) to (31.center);
 			\draw (34.center) to (33.center);
 			\draw (36.center) to (35.center);
 			\draw (38.center) to (37.center);
 			\draw (2.center) to (7.center);
 			\draw [style=Light grey column] (74.center)
 			to (73.center)
 			to (72.center)
 			to (71.center)
 			to (56.center)
 			to (57.center)
 			to (58.center)
 			to (59.center)
 			to (60.center)
 			to (61.center)
 			to cycle;
 			\draw (55.center) to (56.center);
 			\draw (91.center) to (57.center);
 			\draw (92.center) to (106.center);
 			\draw (93.center) to (105.center);
 			\draw (94.center) to (59.center);
 			\draw (95.center) to (104.center);
 			\draw (85.center) to (74.center);
 			\draw (71.center) to (61.center);
 			\draw (84.center) to (59.center);
 			\draw (83.center) to (107.center);
 			\draw (82.center) to (108.center);
 			\draw (56.center) to (57.center);
 			\draw (109.center) to (110.center);
 			\draw (114.center) to (115.center);
 			\draw [style=Light grey column] (125.center)
 			to (124.center)
 			to (123.center)
 			to (118.center)
 			to (116.center)
 			to (117.center)
 			to (103.center)
 			to (120.center)
 			to (119.center)
 			to (122.center)
 			to (102.center)
 			to (121.center)
 			to (100.center)
 			to (126.center)
 			to (127.center)
 			to (128.center)
 			to cycle;
 			\draw (96.center) to (61.center);
 			\draw (97.center) to (116.center);
 			\draw (98.center) to (103.center);
 			\draw (99.center) to (119.center);
 			\draw (100.center) to (130.center);
 			\draw (121.center) to (102.center);
 			\draw (55.center) to (121.center);
 			\draw (132.center) to (131.center);
 			\draw (90.center) to (102.center);
 			\draw (89.center) to (119.center);
 			\draw (88.center) to (129.center);
 			\draw (87.center) to (103.center);
 			\draw (86.center) to (116.center);
 		\end{pgfonlayer}
 	\end{tikzpicture}
 	\caption{Let $e=5$. The first $e$-divisible hook of the proper $e$-rim, called $R$, corresponds to the greatest bead --- the highlighted bead on runner $4$. From the construction of the proper $e$-rim, its removal corresponds to removing this bead and placing it in the first empty space above it --- the small white dot on runner $4$. The next $e$-divisible hook, called $R'$, starts one row below, so it corresponds to the greatest bead less than this empty space --- the highlighted bead on runner $2$. Similarly to $R$, the removal of $R'$ corresponds to moving this bead to the first empty space above it --- the second highlighted empty space. As the greatest bead less than this empty space --- the final highlighted bead --- has no empty spaces above it, the construction of the proper $e$-rim terminates here. Comparing this to \Cref{fig:invariance} (and \Cref{le:Me observation}(iv)) the highlighted beads are precisely the leading beads that disappear after applying $\Ms_e$.}
 	\label{fig:proper e-rim}
\end{figure}

\begin{lemma}\label{le:proper e-rim}
	Let $B$ be a $\beta$-set of a partition $\lambda$ with stable index $z$. Let $0=i_0<i_1<\dots<i_t=z$ be the indices $i$ such that $b_i\notin \Ms_e(B)$. If $R_0, R_1, \dots, R_{t'}$ are the $e$-divisible hooks added to the proper $e$-rim during its construction (in this order), then $t=t'+1$, $R_j$ corresponds to the leading bead $b_{i_j}$ and its removal corresponds to replacing $b_{i_j}$ with $b_{i_{j+1}-1}-e$.
\end{lemma}

\begin{proof}
	Firstly suppose that the proper $e$-rim is non-empty, so $t'\geq 0$. From the construction of the proper $e$-rim of $\lambda$, we have $\tp(R_0)=1$, $\bt(R_{j-1})+1=\tp(R_j)$ for all $1\leq j\leq t'$ and there is no $e$-divisible hook $R$ of $\lambda$ with $\bt(R_{t'})+1=\tp(R)$. We prove by induction on $0\leq j\leq t'$ that $t\geq j+1$, $R_j$ corresponds to the leading bead $b_{i_j}$ and its removal corresponds to moving $b_{i_j}$ to $b_{i_{j+1}-1}-e$.
	
	If $j=0$, then $\tp(R_0)=1$ implies that $R_0$ corresponds to the leading bead $b_{0}=b_{i_0}$. We will artificially use $i_0$ instead of $0$ so the below argument generalises for $j>0$. As $R_0$ is the least $e$-divisible hook of $\lambda$ with $\tp(R_0)=1$, removing $R_0$ corresponds to replacing $b_{i_0}$ with the least empty space of $B$ of the form $b_{i_0}-ae$ with $a>0$. In turn $t\geq 1$ as otherwise, by \Cref{le:Me observation}(v), for all $i\geq i_0$ we have $b_{i_0}-(i-i_0)e=b_i$, which is a bead of $B$. From \Cref{le:Me observation}(ii) and (v), we have $b_i=b_{i_0}-(i-i_0)e$ for all $i_0< i < i_1$ and $b_{i_1}<b_{i_1-1}-e = b_{i_0}-(i_1-i_0)e$. From the definition of leading beads, $b_{i_1-1}-e$ is an empty space of $B$ and thus the removal of $R_0$ corresponds to replacing $b_0$ with $b_{i_1-1}-e$.
	
	If $1\leq j\leq t'$ and the statement holds for $j-1$, then there are $\bt(R_{j-1})$ beads of $B$ greater than $b_{i_j-1}-e$, and so $\bt(R_{j-1})+1=\tp(R_j)$ implies that $R_j$ corresponds to the greatest bead less than the empty space $b_{i_j-1}-e$, that is $b_{i_j}$. We deduce the statement for $j$ using the same argument as in the previous paragraph (with $i_0$ and $i_1$ replaced with $i_j$ and $i_{j+1}$, respectively).
	
	It remains to show that $t=t'+1$. As there is no $e$-divisible hook $R$ of $\lambda$ with $\bt(R_{t'})+1=\tp(R)$ and there are $\bt(R_{t'})$ beads greater than $b_{i_{t'+1}-1}-e$, there is no $a>0$ such that $b_{i_{t'+1}}-ae$ is an empty space of $B$. Hence $i_{t'+1}\geq z=i_t$ which forces $t'+1=t$ as required.
	
	If the proper $e$-rim is empty, then there is no $e$-divisible hook $R$ with $\tp(R)=1$. Hence, $b_0-ae$ is a bead of $B$ for any $a>0$. Hence $z=0$ and $t=0=t'+1$, as required.   
\end{proof}

\begin{corollary}\label{co:proper e-rim}
	Let $B$ be a $\beta$-set of a partition $\lambda$ with stable index $z$. Then $\Ms_e(B)\cup \left\lbrace b_z\right\rbrace  $ is a $\beta$-set of the partition obtained by removing the proper $e$-rim of $\lambda$.
\end{corollary}

\begin{proof}
	By \Cref{le:proper e-rim}, if $0=i_0<i_1<\dots<i_t=z$ are the indices $i$ such that $b_i\notin \Ms_e(B)$, then the desired $\beta$-set is $\left( B \setminus \left\lbrace b_{i_j} : 0\leq j < t \right\rbrace\right)  \cup \left\lbrace b_{i_{j+1}-1} - e : 0\leq j<t \right\rbrace $. This is $\Ms_e(B)\cup \left\lbrace b_z\right\rbrace  $ by \Cref{le:Me observation}(iv).
\end{proof}

We now recall the map $J$ introduced in \cite{XuMullineux97}. For an $e$-regular partition $\lambda = (\List{\lambda}{t})$ we define the partition $J(\lambda)$ as the partition obtained by removing the $e$-rim of $\lambda$ and adding a box to each of its $t$ rows, except, possibly, the final one if the size of the $e$-rim is not divisible by $e$ (or equivalently, if the $e$-rim and the proper $e$-rim of $\lambda$ do not coincide). Continuing the example from \Cref{fig:e-rim}, we get $J((6,4,2)) = (4,3)$. A simpler, though, less precise description of $J$ is in \Cref{fig:inf} followed by a precise statement in \Cref{le:alternative J}.

\begin{figure}[h]
	\centering
	\begin{tikzpicture}[x=0.5cm, y=0.5cm]
		\begin{pgfonlayer}{nodelayer}
			\node [style=none] (0) at (3.5, 3) {};
			\node [style=none] (1) at (9.5, 3) {};
			\node [style=none] (2) at (9.5, 2) {};
			\node [style=none] (3) at (7.5, 2) {};
			\node [style=none] (4) at (7.5, 1) {};
			\node [style=none] (5) at (5.5, 1) {};
			\node [style=none] (8) at (3.5, 1) {};
			\node [style=none] (9) at (3.5, 2) {};
			\node [style=none] (11) at (4.5, 3) {};
			\node [style=none] (12) at (5.5, 3) {};
			\node [style=none] (13) at (6.5, 3) {};
			\node [style=none] (14) at (7.5, 3) {};
			\node [style=none] (15) at (8.5, 3) {};
			\node [style=none] (16) at (8.5, 2) {};
			\node [style=none] (17) at (6.5, 1) {};
			\node [style=none] (18) at (6.5, 2) {};
			\node [style=none] (19) at (4.5, 2) {};
			\node [style=none] (20) at (4.5, 1) {};
			\node [style=none] (42) at (5.5, 2) {};
			\node [style=none] (43) at (2.5, 3) {};
			\node [style=none] (44) at (2.5, 2) {};
			\node [style=none] (45) at (2.5, 1) {};
			\node [style=none] (46) at (5.5, 0) {};
			\node [style=none] (47) at (4.5, 0) {};
			\node [style=none] (48) at (3.5, 0) {};
			\node [style=none] (49) at (2.5, 0) {};
			\node [style=none] (50) at (2.5, -1) {};
			\node [style=none] (51) at (3.5, -1) {};
			\node [style=none] (52) at (3.5, -2) {};
			\node [style=none] (53) at (2.5, -2) {};
			\node [style=none] (54) at (3, -2.5) {$\vdots$};
		\end{pgfonlayer}
		\begin{pgfonlayer}{edgelayer}
			\draw [style=Light grey column] (13.center)
			to (1.center)
			to (2.center)
			to (3.center)
			to (4.center)
			to (5.center)
			to (42.center)
			to (18.center)
			to cycle;
			\draw (8.center) to (0.center);
			\draw (0.center) to (13.center);
			\draw (12.center) to (42.center);
			\draw (14.center) to (3.center);
			\draw (15.center) to (16.center);
			\draw (18.center) to (3.center);
			\draw (9.center) to (42.center);
			\draw (18.center) to (17.center);
			\draw [style=Light grey column] (53.center)
			to (45.center)
			to (5.center)
			to (46.center)
			to (48.center)
			to (52.center)
			to cycle;
			\draw (43.center) to (0.center);
			\draw (43.center) to (45.center);
			\draw (44.center) to (9.center);
			\draw (11.center) to (47.center);
			\draw (8.center) to (48.center);
			\draw (49.center) to (48.center);
			\draw (50.center) to (51.center);
		\end{pgfonlayer}
	\end{tikzpicture}
	\caption{The digram shows how to obtain $J(\lambda)$ using \Cref{le:alternative J} for $\lambda=(6,4,2)$. We add an extra infinite column to the Young diagram of $\lambda$ and then remove the `infinite $e$-rim' of this `Young diagram' to get $J(\lambda) = (4,3)$.}
	\label{fig:inf}
\end{figure}
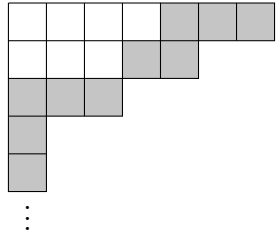

\begin{lemma}\label{le:alternative J}
	Let $\lambda=(\List{\lambda}{t})$ be a partition and $N$ be a non-negative integer such that the proper $e$-rim and the $e$-rim of $\lambda^+=(\List{1+\lambda}{t},1^N)$ coincide. Then $J(\lambda)$ is obtained by removing the (proper) $e$-rim of $\lambda^+$.
\end{lemma}

\begin{proof}
	We need to distinguish two cases. If the $e$-rim and the proper $e$-rim of $\lambda$ coincide, then the $e$-rim of $\lambda^+$ is the $e$-rim of $\lambda$ shifted by one to the right together with boxes $(i,1)$ with $t<i\leq t+N$. Thus its removal from $\lambda^+$ coincides with removing the $e$-rim of $\lambda$ and then adding a box to its first $t$ rows.
	
	If the $e$-rim and the proper $e$-rim of $\lambda$ do not coincide, then the $e$-rim of $\lambda^+$ is the $e$-rim of $\lambda$ shifted by one together with boxes $(i,1)$ with $t\leq i\leq t+N$. Thus its removal from $\lambda^+$ coincides with removing the $e$-rim of $\lambda$ and then adding a box to its first $t-1$ rows. This coincides with $J$.
\end{proof}

The key ingredient for our algorithm now follows from the last two lemmas.

\begin{lemma}\label{le:Ms is J}
	Let $\lambda$ be a partition with $\beta$-set $B$. Then $\Ms_e(B)$ is a $\beta$-set of $J(\lambda)$.
\end{lemma} 

\begin{proof}
	Without loss of generality suppose that $B$ contains all non-positive integers and let $b_0, b_1, \dots$ be its leading beads and $z$ be its stable index. We pick $i\geq z$ such that $b_i < 0$ and let $B^+ = B\setminus\left\lbrace b_i \right\rbrace $. Using the formula \eqref{eq:recover partition} just before \Cref{ex:empty}, $B^+$ is a $\beta$-set of a partition $\lambda^+=(\List{1+\lambda}{t},1^N)$ with $N > 0$. The leading beads of $B^+$ are $b_0,b_1,\dots, b_{i-1}, b_i -1, b_i-e-1,\dots$ and its stable index is $i$; see \Cref{fig:Bead Out}. Hence, by \Cref{co:proper e-rim} and \Cref{le:Me observation}(i) applied twice (with $j=i$), removal of the proper $e$-rim of $\lambda^+$ yields a partition $\mu$ with $\beta$-set
	\begin{align*}
	&\Ms_e(B+)\cup \left\lbrace b_i -1 \right\rbrace  \\ =&\left( B^+ \setminus \left\lbrace b_j : 0\leq j < i \right\rbrace \right) \cup \left\lbrace b_j - e : 0\leq j<i \right\rbrace  \\= &\left( B \setminus \left\lbrace b_j : 0\leq j\leq i \right\rbrace \right)  \cup \left\lbrace b_j - e : 0\leq j < i \right\rbrace \\ = &\Ms_e(B).
	\end{align*}
	In particular, $\mu$ does not depend on $i$, and so by increasing $i$ (and $N$) if necessary, we can assume that $\mu$ has fewer parts than $\lambda^+$. Hence, as $\mu$ is obtained by removing the proper $e$-rim of $\lambda^+$, we deduce that the proper $e$-rim of $\lambda^+$ coincides with its $e$-rim. The result follows from \Cref{le:alternative J}.
\end{proof}

\begin{figure}[h]
	\centering
	\begin{tikzpicture}[x=0.5cm, y=0.5cm]
		\begin{pgfonlayer}{nodelayer}
			\node [style=none] (0) at (1.25, 3.5) {};
			\node [style=none] (1) at (2.25, 3.5) {};
			\node [style=none] (2) at (3.25, 3.5) {};
			\node [style=none] (3) at (4.25, 3.5) {};
			\node [style=none] (4) at (5.25, 3.5) {};
			\node [style=none] (5) at (1.25, -2.5) {};
			\node [style=none] (6) at (2.25, -2.5) {};
			\node [style=none] (7) at (3.25, -2.5) {};
			\node [style=none] (8) at (4.25, -2.5) {};
			\node [style=none] (9) at (5.25, -2.5) {};
			\node [style=Empty node] (10) at (1.25, 3) {};
			\node [style=Empty node] (12) at (3.25, 3) {};
			\node [style=Empty node] (13) at (4.25, 3) {};
			\node [style=Empty node] (15) at (1.25, 2) {};
			\node [style=none] (23) at (4, 1) {};
			\node [style=none] (24) at (4.5, 1) {};
			\node [style=none] (25) at (4, -2) {};
			\node [style=none] (26) at (4.5, -2) {};
			\node [style=none] (27) at (5.5, 0) {};
			\node [style=none] (28) at (5, 0) {};
			\node [style=none] (29) at (4.5, 0) {};
			\node [style=none] (30) at (4, 0) {};
			\node [style=none] (31) at (3.5, -2) {};
			\node [style=none] (32) at (3, -2) {};
			\node [style=none] (33) at (2.5, 0) {};
			\node [style=none] (34) at (2, 0) {};
			\node [style=none] (35) at (3.5, 1) {};
			\node [style=none] (36) at (3, 1) {};
			\node [style=none] (37) at (3.5, -1) {};
			\node [style=none] (38) at (3, -1) {};
			\node [style=Empty node] (41) at (1.25, 1) {};
			\node [style=Empty node] (42) at (1.25, 0) {};
			\node [style=Empty node] (43) at (5.25, 1) {};
			\node [style=Empty node] (44) at (1.25, -2) {};
			\node [style=Empty node] (45) at (2.25, -1) {};
			\node [style=Empty node] (46) at (5.25, 2) {};
			\node [style=Empty node] (47) at (5.25, 3) {};
			\node [style=Empty node] (49) at (4.25, -1) {};
			\node [style=Lead] (50) at (5.25, -2) {};
			\node [style=Lead] (51) at (3.25, 0) {};
			\node [style=Lead] (52) at (2.25, 1) {};
			\node [style=none] (55) at (2, -2) {};
			\node [style=none] (56) at (2.5, -2) {};
			\node [style=none] (57) at (1, -1.25) {};
			\node [style=none] (58) at (1.5, -1.25) {};
			\node [style=Empty node] (59) at (4.25, 2) {};
			\node [style=Empty node] (60) at (3.25, 2) {};
			\node [style=Lead] (61) at (2.25, 2) {};
			\node [style=Lead] (62) at (2.25, 3) {};
			\node [style=none] (63) at (7.5, 3.5) {};
			\node [style=none] (64) at (8.5, 3.5) {};
			\node [style=none] (65) at (9.5, 3.5) {};
			\node [style=none] (66) at (10.5, 3.5) {};
			\node [style=none] (67) at (11.5, 3.5) {};
			\node [style=none] (68) at (7.5, -2.5) {};
			\node [style=none] (69) at (8.5, -2.5) {};
			\node [style=none] (70) at (9.5, -2.5) {};
			\node [style=none] (71) at (10.5, -2.5) {};
			\node [style=none] (72) at (11.5, -2.5) {};
			\node [style=Empty node] (73) at (8.5, 3) {};
			\node [style=Empty node] (74) at (9.5, 3) {};
			\node [style=Empty node] (75) at (10.5, 3) {};
			\node [style=none] (77) at (10.25, 1) {};
			\node [style=none] (78) at (10.75, 1) {};
			\node [style=none] (79) at (10.25, -2) {};
			\node [style=none] (80) at (10.75, -2) {};
			\node [style=none] (81) at (11.75, 0) {};
			\node [style=none] (82) at (11.25, 0) {};
			\node [style=none] (83) at (10.75, 0) {};
			\node [style=none] (84) at (10.25, 0) {};
			\node [style=none] (85) at (9.75, -2) {};
			\node [style=none] (86) at (9.25, -2) {};
			\node [style=none] (87) at (8.75, 0) {};
			\node [style=none] (88) at (8.25, 0) {};
			\node [style=none] (89) at (9.75, 1) {};
			\node [style=none] (90) at (9.25, 1) {};
			\node [style=none] (91) at (9.75, -1) {};
			\node [style=none] (92) at (9.25, -1) {};
			\node [style=Empty node] (94) at (7.5, 1) {};
			\node [style=Empty node] (95) at (7.5, 0) {};
			\node [style=Empty node] (96) at (11.5, 1) {};
			\node [style=Empty node] (97) at (7.5, -2) {};
			\node [style=Empty node] (98) at (8.5, -1) {};
			\node [style=Empty node] (99) at (11.5, 2) {};
			\node [style=Empty node] (100) at (11.5, 3) {};
			\node [style=Empty node] (101) at (10.5, -1) {};
			\node [style=Lead] (102) at (11.5, -2) {};
			\node [style=Lead] (103) at (9.5, 0) {};
			\node [style=Lead] (104) at (8.5, 1) {};
			\node [style=none] (105) at (8.25, -2) {};
			\node [style=none] (106) at (8.75, -2) {};
			\node [style=none] (107) at (7.25, -1.25) {};
			\node [style=none] (108) at (7.75, -1.25) {};
			\node [style=Empty node] (109) at (10.5, 2) {};
			\node [style=Empty node] (110) at (9.5, 2) {};
			\node [style=Lead] (111) at (7.5, 2) {};
			\node [style=Lead] (112) at (7.5, 3) {};
			\node [style=none] (113) at (8.25, 2) {};
			\node [style=none] (114) at (8.75, 2) {};
			\node [style=none] (115) at (2.6, 2.325) {$\scriptstyle b_i$};
			\node [style=none] (116) at (3.25, 4.5) {$B$};
			\node [style=none] (117) at (9.5, 4.5) {$B^+$};
			\node [style=Lead] (118) at (5.25, -1) {};
			\node [style=Lead] (119) at (11.5, -1) {};
		\end{pgfonlayer}
		\begin{pgfonlayer}{edgelayer}
			\draw (0.center) to (5.center);
			\draw (1.center) to (6.center);
			\draw (3.center) to (8.center);
			\draw (4.center) to (9.center);
			\draw (23.center) to (24.center);
			\draw (25.center) to (26.center);
			\draw (28.center) to (27.center);
			\draw (30.center) to (29.center);
			\draw (32.center) to (31.center);
			\draw (34.center) to (33.center);
			\draw (36.center) to (35.center);
			\draw (38.center) to (37.center);
			\draw (2.center) to (7.center);
			\draw (55.center) to (56.center);
			\draw (57.center) to (58.center);
			\draw (63.center) to (68.center);
			\draw (64.center) to (69.center);
			\draw (66.center) to (71.center);
			\draw (67.center) to (72.center);
			\draw (77.center) to (78.center);
			\draw (79.center) to (80.center);
			\draw (82.center) to (81.center);
			\draw (84.center) to (83.center);
			\draw (86.center) to (85.center);
			\draw (88.center) to (87.center);
			\draw (90.center) to (89.center);
			\draw (92.center) to (91.center);
			\draw (65.center) to (70.center);
			\draw (105.center) to (106.center);
			\draw (107.center) to (108.center);
			\draw (113.center) to (114.center);
		\end{pgfonlayer}
	\end{tikzpicture}
	\caption{The change of the leading beads when the bead $b_i$ of $B$ is removed to obtain $B^+$ as done in the proof of \Cref{le:Ms is J}.}
	\label{fig:Bead Out}
\end{figure}
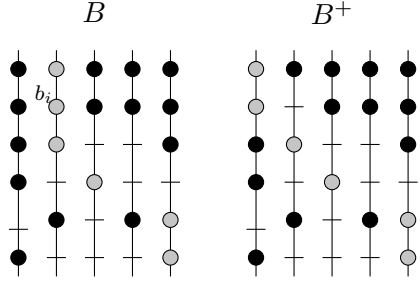

We define the \textit{Abacus Mullineux Algorithm} as follows. Given an $e$-regular partition $\lambda$, let $B$ be its $\beta$-set and let $S=B$ and $T=\varnothing$. We now repeat the following: if $S$ is a $\beta$-set of an empty partition, we replace $T$ by $T\cup S$, return it and terminate the procedure; otherwise, we add $\ms_e(S)$ to $T$ and replace $S$ by $\Ms_e(S)$. Examples of the algorithm are in \Cref{fig:MA} and \Cref{fig:SmallAbacusAlgorithm}.

\begin{figure}[h]
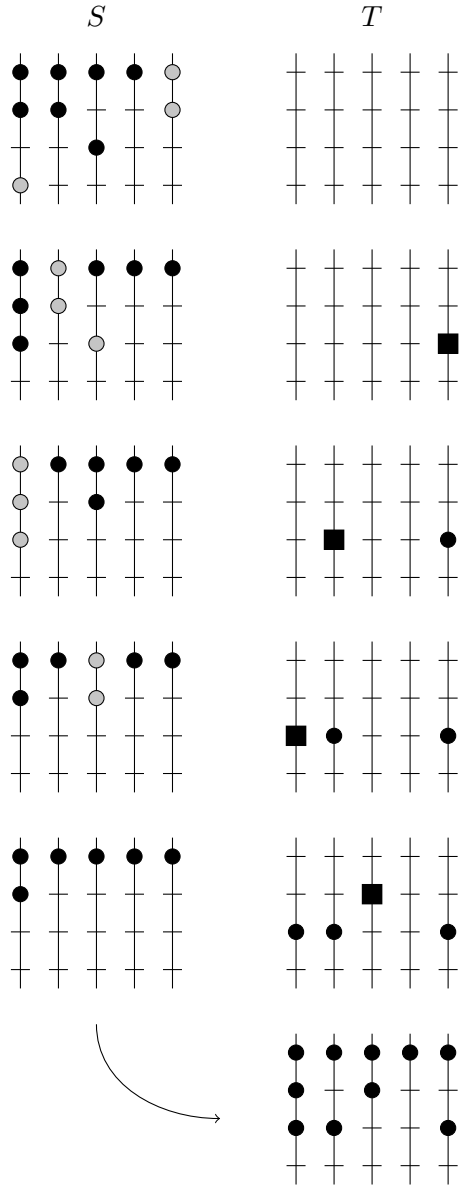

	\centering

	\caption{The Abacus Mullineux Algorithm applied with the canonical $\beta$-set of $\lambda=(6,4,2)$. The grey beads of the first four $\beta$-sets $S$ are their respective leading beads. The beads $\ms_e(S)$ added to $T$ are denoted by squares. On the second to last line $S$ is a $\beta$-set of the empty partition and so the algorithm terminates and returns set $T$ on the final line, which is the canonical $\beta$-set of $m_e(\lambda)' = (5,3^2,1)$.}
	\label{fig:MA}
\end{figure}

If $\lambda$ is $e$-regular and non-empty, then $J(\lambda)$ has smaller size than $\lambda$ and by \cite[Lemma~2]{XuMullineux97}, it is also $e$-regular. Hence there is a least $n\geq 0$ such that $J^n(\lambda)$ is the empty partition. One can then compute $m_e(\lambda)$ recursively using the following result due to Xu \cite[Theorem~1]{XuMullineux97} (reproved in \cite[Theorem~2]{XupSeries99}).

\begin{theorem}\label{th:xu}
	If $\lambda$ is an $e$-regular partition and $\mu=J(\lambda)$, then the first part of $m_e(\lambda)'$ is $|\lambda|-|\mu|$ and its remaining parts are the parts of $m_e(\mu)'$.
\end{theorem}

We can use \Cref{th:xu} to obtain an inductive proof of our main result of this section.

\begin{proposition}\label{pr:Mullineux algorithm}
	The Abacus Mullineux Algorithm applied to an $e$-regular partition $\lambda$ terminates and returns a $\beta$-set of $m_e(\lambda)'$. Moreover, this $\beta$-set is canonical if we start with the canonical $\beta$-set of $\lambda$.
\end{proposition}

\begin{proof}
	As $\lambda$ is $e$-regular, there is a least $n\geq 0$ such that $J^n(\lambda)$ is the empty partition. Thus, by \Cref{le:Ms is J}, after $n$ steps $S$ becomes a $\beta$-set of the empty partition and thus the algorithm terminates in the next step. Since $S$ is a $\beta$-set of some partition throughout the whole algorithm, without loss of generality we can start with a $\beta$-set $S_0 = B_s(\lambda)$ such that all sets $S$ throughout the process contain all non-positive integers.
	
	We now prove by induction on $|\lambda|$, that the algorithm terminates with a set $T_0$ (which clearly contains all non-positive integers) such that
	\begin{enumerate}[label=\textnormal{(\roman*)}]
		\item for any $0\leq r\leq e-1$ there are the same number of non-negative beads of $S_0$ and $T_0$ on runner $r$, and
		\item $T_0$ is a $\beta$-set of $m_e(\lambda)'$.
	\end{enumerate}
	 If $\lambda$ is the empty partition, this is clear. Otherwise $S_1=\Ms_e(S_0)$ is a $\beta$-set of a smaller partition $\mu=J(\lambda)$. Let $T_1$ be the outcome of the Abacus Mullineux Algorithm applied to $S_1$, so $T_0 = T_1 \cup \left\lbrace \hat{s}\right\rbrace$, where $\hat{s} = \ms_e(S_0)$. As the sets $S$ which arise during this algorithm are, apart from $S_0$, the same as when the algorithm is applied to $S_0$, they all contain all non-positive integers (and so does $T_1$), and thus the induction hypothesis applies to $S_1$, and so $T_1$ is a $\beta$-set of $m_e(\mu)'$. By \Cref{th:xu}, $m_e(\lambda)'$ equals $m_e(\mu)'$ with an extra first part (equal to $|\lambda|-|\mu|$), and thus there is $b> \max T_1\geq 0$ which can be added to $T_1$ to get a $\beta$-set of $m_e(\lambda)'$. Thus to prove (ii) we need to show that $\hat{s}=b$.
	 
	 Let $N$ be such that the $N$'th greatest element of $S_0$ is $0$. By \Cref{le:shift}(ii) and (iii) applied with $B=S_0$, $0$ is the $(N-1)$'st greatest bead of $S_1$ and thus, by the induction hypothesis, of $T_1$. Let $s_0,s_1,\dots$ be the leading beads of $S_0$ and write $j$ for the greatest index such that $s_j\geq 0$. Then, by \Cref{le:shift}(i), $j$ is at least the stable index of $S_0$ and $\hat{s} = s_j+ej$. Using the formula \eqref{eq:size} applied to $T_1, S_1, S_0$ and $T_1\cup \left\lbrace b\right\rbrace $ (all of which contain all non-positive integers) and the fact that $S_1=\left( S_0 \setminus \left\lbrace s_i : 0\leq i\leq j \right\rbrace\right)  \cup \left\lbrace s_i-e : 0\leq i< j \right\rbrace$ (see \Cref{le:Me observation}(i)), we compute
	 \begin{align*}
	 \hat{s}+ \sum_{0\leq t\in T_1}t - \binom{N-1}{2} &= \hat{s} + |m_e(\mu)'| = \hat{s} + |\mu|\\ &= \hat{s}+ \sum_{0\leq s\in S_1}s - \binom{N-1}{2}\\
	 &= \hat{s} + \sum_{0\leq s\in S_0}s - ej - s_j - \binom{N-1}{2} \\
	 &= \sum_{0\leq s\in S_0}s - \binom{N-1}{2} = |\lambda| + N -1 \\
	 &= |m_e(\lambda)'| + N -1= \sum_{0\leq t\in T_1\cup \left\lbrace b\right\rbrace }t - \binom{N-1}{2}\\
	 &= b + \sum_{0\leq t\in T_1}t - \binom{N-1}{2},
	 \end{align*}
	 and thus $\hat{s} = b$, which establishes (ii).
	 
	 By \Cref{le:shift}(ii) and (iii), $S_0$ and $S_1$ have the same number of non-negative beads on each runner apart from the runner of $\hat{s}$, where $S_0$ has a single extra non-negative bead. As $T_0=T_1 \cup \left\lbrace \hat{s} \right\rbrace $ and $\hat{s}=b >\max T_1\geq 0$, we deduce (i) immediately from the induction hypothesis. The first part of the proposition follows. 
	 
	 The combination of (i) and (ii) implies that $T_0 = B_s(m_e(\lambda)')$. Consequently, for any integer $s'$ if we start with $S=B_{s'}(\lambda)$, the algorithm returns $B_{s'}(m_e(\lambda)')$. Taking $s'=0$ shows the `moreover' part. 
\end{proof}

\begin{remark}\label{re:algorithm}
	There are three things to note.
	\begin{enumerate}[label=\textnormal{(\roman*)}]
		\item A similar (but much simpler) inductive argument on the size of $\lambda$, which uses the fact that $\hat{s} > \max T_1$ from the proof, shows that at every step, before we add $\ms_e(S)$ to $T$, we have $\ms_e(S)<\min T$ (or $T = \varnothing$). In particular, $\ms_e(S)\notin T$.
		\item A similar (but much simpler) inductive argument on the size of $\lambda$ also shows that in the final step the union $S\cup T$ is disjoint.
		\item One can replace $T$ by $T\cup S$ and terminate the process earlier, namely, once $S$ is a $\beta$-set of an $e$-core partition. This is because if $S$ is a $\beta$-set of an $e$-core partition, then its stable index is $0$ and the following step of the algorithm removes the greatest bead of $S$ and places it in $T$, transforming $S$ into a $\beta$-set of some other $e$-core partition, to which the same argument applies. From this observation we immediately see that $m_e$ maps $e$-core partitions to their conjugates, a fact which is not apparent from Mullineux's original algorithm.
	\end{enumerate}
	\end{remark}
	
\section{Balanced partitions}\label{se:balanced}

In the rest of the paper $d,e>1$ are two integers. Recall from \Cref{de:partitions} that a \dbal partition is a partition $\lambda$ such that the arm length of any $e$-divisible hook of $\lambda$ is congruent to $0$ modulo $d$. It is \dsbal if, instead, all such arm lengths are congruent to $-1$ modulo $d$. We can translate these notions to $\beta$-sets.

\begin{definition}\label{de:balanced}
	Let $B$ be a subset of $\Z$. The set $B$ is \textit{\dbal}if $d$ divides $\emp_B(b-ae+1, b)$ for any bead $b$ and any empty space $b-ae$ of $B$ (with $a>0$). It is \textit{\dsbal}if $d$ divides $\emp_B(b-ae, b)$ for any bead $b$ and any empty space $b-ae$ of $B$ (with $a>0$). 
\end{definition}

Using the earlier mentioned correspondence between $ae$-hooks of $\lambda$ and certain beads of its $\beta$-set, it is easy to see that $\lambda$ is \dbal (or $d$-shift balanced) if and only if its $\beta$-set is. An example of these definitions is presented in \Cref{fig:balanced}.

\begin{figure}[h]
	\centering
	\begin{tikzpicture}[x=0.5cm, y=0.5cm]
		\begin{pgfonlayer}{nodelayer}
			\node [style=none] (0) at (0, 3.75) {};
			\node [style=none] (1) at (1, 3.75) {};
			\node [style=none] (2) at (2, 3.75) {};
			\node [style=none] (3) at (3, 3.75) {};
			\node [style=none] (4) at (4, 3.75) {};
			\node [style=none] (5) at (0, -0.25) {};
			\node [style=none] (6) at (1, -0.25) {};
			\node [style=none] (7) at (2, -0.25) {};
			\node [style=none] (8) at (3, -0.25) {};
			\node [style=none] (9) at (4, -0.25) {};
			\node [style=Empty node] (10) at (0, 3.25) {};
			\node [style=Empty node] (11) at (1, 3.25) {};
			\node [style=Empty node] (12) at (2, 3.25) {};
			\node [style=Empty node] (13) at (3, 3.25) {};
			\node [style=Empty node] (15) at (0, 2.25) {};
			\node [style=Empty node] (16) at (1, 2.25) {};
			\node [style=Empty node] (18) at (2, 1.25) {};
			\node [style=none] (19) at (1.75, 2.25) {};
			\node [style=none] (20) at (2.25, 2.25) {};
			\node [style=none] (21) at (2.75, 2.25) {};
			\node [style=none] (22) at (3.25, 2.25) {};
			\node [style=none] (23) at (2.75, 1.25) {};
			\node [style=none] (24) at (3.25, 1.25) {};
			\node [style=none] (25) at (3.75, 1.25) {};
			\node [style=none] (26) at (4.25, 1.25) {};
			\node [style=none] (27) at (4.25, 0.25) {};
			\node [style=none] (28) at (3.75, 0.25) {};
			\node [style=none] (29) at (3.25, 0.25) {};
			\node [style=none] (30) at (2.75, 0.25) {};
			\node [style=none] (31) at (2.25, 0.25) {};
			\node [style=none] (32) at (1.75, 0.25) {};
			\node [style=none] (33) at (1.25, 0.25) {};
			\node [style=none] (34) at (0.75, 0.25) {};
			\node [style=none] (35) at (1.25, 1.25) {};
			\node [style=none] (36) at (0.75, 1.25) {};
			\node [style=none] (37) at (0.25, 1.25) {};
			\node [style=none] (38) at (-0.25, 1.25) {};
			\node [style=Empty node] (83) at (0, 0.25) {};
			\node [style=Empty node] (84) at (4, 3.25) {};
			\node [style=Empty node] (85) at (4, 2.25) {};
			\node [style=Empty node] (298) at (7, 3.25) {};
			\node [style=Empty node] (299) at (8, 1.25) {};
			\node [style=Empty node] (300) at (9, 3.25) {};
			\node [style=Empty node] (301) at (10, 3.25) {};
			\node [style=Empty node] (302) at (7, 2.25) {};
			\node [style=Empty node] (303) at (8, 3.25) {};
			\node [style=Empty node] (304) at (9, 2.25) {};
			\node [style=none] (325) at (7, 3.75) {};
			\node [style=none] (326) at (8, 3.75) {};
			\node [style=none] (327) at (9, 3.75) {};
			\node [style=none] (328) at (10, 3.75) {};
			\node [style=none] (329) at (11, 3.75) {};
			\node [style=none] (330) at (7, -0.25) {};
			\node [style=none] (331) at (8, -0.25) {};
			\node [style=none] (332) at (9, -0.25) {};
			\node [style=none] (333) at (10, -0.25) {};
			\node [style=none] (334) at (11, -0.25) {};
			\node [style=none] (337) at (9.75, 2.25) {};
			\node [style=none] (338) at (10.25, 2.25) {};
			\node [style=none] (339) at (9.75, 1.25) {};
			\node [style=none] (340) at (10.25, 1.25) {};
			\node [style=none] (341) at (11.25, 0.25) {};
			\node [style=none] (342) at (10.75, 0.25) {};
			\node [style=none] (343) at (10.25, 0.25) {};
			\node [style=none] (344) at (9.75, 0.25) {};
			\node [style=none] (345) at (9.25, 0.25) {};
			\node [style=none] (346) at (8.75, 0.25) {};
			\node [style=none] (347) at (8.25, 0.25) {};
			\node [style=none] (348) at (7.75, 0.25) {};
			\node [style=none] (351) at (10.75, 2.25) {};
			\node [style=none] (352) at (11.25, 2.25) {};
			\node [style=none] (353) at (6.75, 0.25) {};
			\node [style=none] (354) at (7.25, 0.25) {};
			\node [style=Empty node] (355) at (11, 1.25) {};
			\node [style=Empty node] (356) at (7, 1.25) {};
			\node [style=Empty node] (357) at (11, 3.25) {};
			\node [style=none] (358) at (7.75, 2.25) {};
			\node [style=none] (359) at (8.25, 2.25) {};
			\node [style=none] (376) at (9.25, 1.25) {};
			\node [style=none] (377) at (8.75, 1.25) {};
			\node [style=Empty node] (378) at (14, 3.25) {};
			\node [style=Empty node] (379) at (18, 2.25) {};
			\node [style=Empty node] (380) at (16, 3.25) {};
			\node [style=Empty node] (381) at (17, 3.25) {};
			\node [style=Empty node] (383) at (15, 3.25) {};
			\node [style=none] (385) at (14, 3.75) {};
			\node [style=none] (386) at (15, 3.75) {};
			\node [style=none] (387) at (16, 3.75) {};
			\node [style=none] (388) at (17, 3.75) {};
			\node [style=none] (389) at (18, 3.75) {};
			\node [style=none] (390) at (14, -0.25) {};
			\node [style=none] (391) at (15, -0.25) {};
			\node [style=none] (392) at (16, -0.25) {};
			\node [style=none] (393) at (17, -0.25) {};
			\node [style=none] (394) at (18, -0.25) {};
			\node [style=none] (395) at (16.75, 2.25) {};
			\node [style=none] (396) at (17.25, 2.25) {};
			\node [style=none] (397) at (16.75, 1.25) {};
			\node [style=none] (398) at (17.25, 1.25) {};
			\node [style=none] (399) at (18.25, 0.25) {};
			\node [style=none] (400) at (17.75, 0.25) {};
			\node [style=none] (401) at (17.25, 0.25) {};
			\node [style=none] (402) at (16.75, 0.25) {};
			\node [style=none] (403) at (16.25, 0.25) {};
			\node [style=none] (404) at (15.75, 0.25) {};
			\node [style=none] (405) at (15.25, 0.25) {};
			\node [style=none] (406) at (14.75, 0.25) {};
			\node [style=none] (407) at (13.75, 1.25) {};
			\node [style=none] (408) at (14.25, 1.25) {};
			\node [style=none] (409) at (13.75, 2.25) {};
			\node [style=none] (410) at (14.25, 2.25) {};
			\node [style=Empty node] (411) at (18, 1.25) {};
			\node [style=Empty node] (412) at (14, 0.25) {};
			\node [style=Empty node] (413) at (18, 3.25) {};
			\node [style=none] (414) at (14.75, 1.25) {};
			\node [style=none] (415) at (15.25, 1.25) {};
			\node [style=none] (416) at (16.25, 1.25) {};
			\node [style=none] (417) at (15.75, 1.25) {};
			\node [style=none] (418) at (-1, 10) {};
			\node [style=none] (419) at (5, 10) {};
			\node [style=none] (420) at (5, 9) {};
			\node [style=none] (421) at (3, 9) {};
			\node [style=none] (422) at (3, 8) {};
			\node [style=none] (423) at (1, 8) {};
			\node [style=none] (424) at (1, 7) {};
			\node [style=none] (425) at (-1, 7) {};
			\node [style=none] (426) at (0, 7) {};
			\node [style=none] (427) at (-1, 8) {};
			\node [style=none] (428) at (-1, 9) {};
			\node [style=none] (429) at (0, 10) {};
			\node [style=none] (430) at (1, 10) {};
			\node [style=none] (431) at (2, 10) {};
			\node [style=none] (432) at (2, 8) {};
			\node [style=none] (433) at (4, 10) {};
			\node [style=none] (434) at (4, 9) {};
			\node [style=none] (435) at (3, 10) {};
			\node [style=none] (436) at (6.5, 10) {};
			\node [style=none] (439) at (10.5, 9) {};
			\node [style=none] (440) at (6.5, 6) {};
			\node [style=none] (441) at (9.5, 7) {};
			\node [style=none] (442) at (8.5, 7) {};
			\node [style=none] (443) at (6.5, 7) {};
			\node [style=none] (444) at (7.5, 6) {};
			\node [style=none] (445) at (6.5, 8) {};
			\node [style=none] (446) at (6.5, 9) {};
			\node [style=none] (447) at (7.5, 10) {};
			\node [style=none] (448) at (8.5, 10) {};
			\node [style=none] (449) at (9.5, 10) {};
			\node [style=none] (450) at (9.5, 9) {};
			\node [style=none] (451) at (11.5, 10) {};
			\node [style=none] (452) at (11.5, 9) {};
			\node [style=none] (453) at (10.5, 10) {};
			\node [style=none] (454) at (13, 10) {};
			\node [style=none] (455) at (19, 10) {};
			\node [style=none] (456) at (19, 9) {};
			\node [style=none] (457) at (19, 8) {};
			\node [style=none] (458) at (17, 8) {};
			\node [style=none] (463) at (13, 8) {};
			\node [style=none] (464) at (13, 9) {};
			\node [style=none] (465) at (14, 10) {};
			\node [style=none] (466) at (15, 10) {};
			\node [style=none] (467) at (16, 10) {};
			\node [style=none] (469) at (18, 10) {};
			\node [style=none] (470) at (18, 8) {};
			\node [style=none] (471) at (17, 10) {};
			\node [style=none] (472) at (9.5, 8) {};
			\node [style=none] (474) at (7.5, 7) {};
			\node [style=Empty node] (475) at (15, 2.25) {};
			\node [style=Empty node] (476) at (16, 2.25) {};
			\node [style=none] (477) at (15, 7) {};
			\node [style=none] (478) at (14, 6) {};
			\node [style=none] (479) at (14, 5) {};
			\node [style=none] (480) at (13, 5) {};
			\node [style=none] (481) at (13, 6) {};
			\node [style=none] (482) at (13, 7) {};
			\node [style=none] (483) at (16, 8) {};
			\node [style=none] (484) at (-0.5, 7.5) {};
			\node [style=none] (485) at (0.5, 7.5) {};
			\node [style=none] (486) at (0.5, 8.25) {};
			\node [style=none] (487) at (2.5, 8.25) {};
			\node [style=none] (488) at (1.5, 8.75) {};
			\node [style=none] (489) at (2.5, 8.75) {};
			\node [style=none] (490) at (2.5, 9.5) {};
			\node [style=none] (491) at (4.5, 9.5) {};
			\node [style=none] (492) at (7, 6.5) {};
			\node [style=none] (493) at (7, 7.25) {};
			\node [style=none] (494) at (9.25, 7.25) {};
			\node [style=none] (495) at (9.25, 8.5) {};
			\node [style=none] (496) at (8.75, 7.75) {};
			\node [style=none] (497) at (8.75, 9.5) {};
			\node [style=none] (498) at (11, 9.5) {};
			\node [style=none] (499) at (18.25, 9.5) {};
			\node [style=none] (500) at (18.25, 8.75) {};
			\node [style=none] (501) at (15.5, 8.75) {};
			\node [style=none] (502) at (18.75, 9.5) {};
			\node [style=none] (503) at (18.75, 8.25) {};
			\node [style=none] (504) at (14.5, 8.25) {};
			\node [style=none] (505) at (14.5, 7.5) {};
			\node [style=none] (506) at (13.5, 7.5) {};
			\node [style=none] (507) at (13.5, 5.5) {};
			\node [style=none] (508) at (-0.35, 0.5) {$\scriptstyle a$};
			\node [style=none] (509) at (1.65, 1.5) {$\scriptstyle \tilde{a}$};
			\node [style=none] (510) at (10.65, 1.5) {$\scriptstyle b$};
			\node [style=none] (511) at (7.625, 1.5) {$\scriptstyle \tilde{b}$};
			\node [style=none] (512) at (13.65, 0.5) {$\scriptstyle c$};
			\node [style=none] (513) at (2, 4.5) {$A$};
			\node [style=none] (514) at (9, 4.5) {$B$};
			\node [style=none] (515) at (16, 4.5) {$C$};
		\end{pgfonlayer}
		\begin{pgfonlayer}{edgelayer}
			\draw (0.center) to (5.center);
			\draw (1.center) to (6.center);
			\draw (3.center) to (8.center);
			\draw (4.center) to (9.center);
			\draw (19.center) to (20.center);
			\draw (21.center) to (22.center);
			\draw (23.center) to (24.center);
			\draw (25.center) to (26.center);
			\draw (28.center) to (27.center);
			\draw (30.center) to (29.center);
			\draw (32.center) to (31.center);
			\draw (34.center) to (33.center);
			\draw (36.center) to (35.center);
			\draw (38.center) to (37.center);
			\draw (2.center) to (7.center);
			\draw (325.center) to (330.center);
			\draw (326.center) to (331.center);
			\draw (328.center) to (333.center);
			\draw (329.center) to (334.center);
			\draw (337.center) to (338.center);
			\draw (339.center) to (340.center);
			\draw (342.center) to (341.center);
			\draw (344.center) to (343.center);
			\draw (346.center) to (345.center);
			\draw (348.center) to (347.center);
			\draw (327.center) to (332.center);
			\draw (351.center) to (352.center);
			\draw (353.center) to (354.center);
			\draw (358.center) to (359.center);
			\draw (377.center) to (376.center);
			\draw (385.center) to (390.center);
			\draw (386.center) to (391.center);
			\draw (388.center) to (393.center);
			\draw (389.center) to (394.center);
			\draw (395.center) to (396.center);
			\draw (397.center) to (398.center);
			\draw (400.center) to (399.center);
			\draw (402.center) to (401.center);
			\draw (404.center) to (403.center);
			\draw (406.center) to (405.center);
			\draw (387.center) to (392.center);
			\draw (407.center) to (408.center);
			\draw (409.center) to (410.center);
			\draw (414.center) to (415.center);
			\draw (417.center) to (416.center);
			\draw (418.center) to (419.center);
			\draw (428.center) to (420.center);
			\draw (427.center) to (422.center);
			\draw (425.center) to (424.center);
			\draw (418.center) to (425.center);
			\draw (429.center) to (426.center);
			\draw (430.center) to (424.center);
			\draw (431.center) to (432.center);
			\draw (435.center) to (422.center);
			\draw (433.center) to (434.center);
			\draw (419.center) to (420.center);
			\draw (436.center) to (451.center);
			\draw (446.center) to (452.center);
			\draw (445.center) to (472.center);
			\draw (443.center) to (441.center);
			\draw (440.center) to (444.center);
			\draw (436.center) to (440.center);
			\draw (447.center) to (444.center);
			\draw (448.center) to (442.center);
			\draw (449.center) to (441.center);
			\draw (453.center) to (439.center);
			\draw (451.center) to (452.center);
			\draw (454.center) to (455.center);
			\draw (464.center) to (456.center);
			\draw (463.center) to (457.center);
			\draw (469.center) to (470.center);
			\draw (455.center) to (457.center);
			\draw (482.center) to (477.center);
			\draw (481.center) to (478.center);
			\draw (480.center) to (479.center);
			\draw (480.center) to (454.center);
			\draw (465.center) to (479.center);
			\draw (466.center) to (477.center);
			\draw (467.center) to (483.center);
			\draw (471.center) to (458.center);
			\draw [style=Extra box] (484.center) to (485.center);
			\draw [style=Extra box] (485.center) to (486.center);
			\draw [style=Extra box] (486.center) to (487.center);
			\draw [style=Extra box] (488.center) to (489.center);
			\draw [style=Extra box] (489.center) to (490.center);
			\draw [style=Extra box] (490.center) to (491.center);
			\draw [style=Extra box] (492.center) to (493.center);
			\draw [style=Extra box] (493.center) to (494.center);
			\draw [style=Extra box] (494.center) to (495.center);
			\draw [style=Extra box] (496.center) to (497.center);
			\draw [style=Extra box] (497.center) to (498.center);
			\draw [style=Extra box] (499.center) to (500.center);
			\draw [style=Extra box] (500.center) to (501.center);
			\draw [style=Extra box] (502.center) to (503.center);
			\draw [style=Extra box] (503.center) to (504.center);
			\draw [style=Extra box] (504.center) to (505.center);
			\draw [style=Extra box] (505.center) to (506.center);
			\draw [style=Extra box] (506.center) to (507.center);
		\end{pgfonlayer}
	\end{tikzpicture}
	\caption{Let $e=5$ and $d=3$. For the above partitions, their $e$-divisible hooks are marked by dashed lines and their canonical $\beta$-sets are displayed below each of them. The first partition and its canonical $\beta$-set are \dbal --- the underlying arm lengths are $3$, or alternatively, we have $\emp_A(a-4,a) = \emp_A(\tilde{a}-4,\tilde{a}) = 3$. The second partition is \dsbal --- the two relevant arm lengths are $2$, or alternatively, we have $\emp_B(b-5,b) = \emp_B(\tilde{b}-5,\tilde{b}) = 3$. The final partition is neither \dbal nor \dsbal --- see the $10$- and the $5$-hook, respectively. Alternatively, one computes $\emp_C(c-9,c) = 5$ and $\emp_C(c-5,c) = 4$.}
	\label{fig:balanced}
\end{figure}
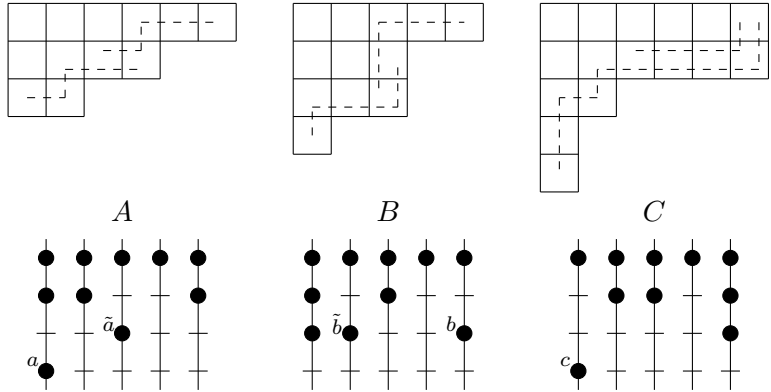

We now collect several lemmas to prove the main result of this section --- the map $\Ms_e$ preserves the property of being $d$-balanced. We start with some extra basic properties of leading beads.  

\begin{lemma}\label{le:auxiliary}
	Let $B$ be a $\beta$-set of a partition and let $0 \leq i < j$ be integers.
	\begin{enumerate}[label=\textnormal{(\roman*)}]
		\item If $b_i - b_j = (j-i)e$, then for any integer $k$ such that $i\leq k \leq j$ we have $b_i - b_k = (k-i)e$ and $b_k - b_j = (j-k)e$.
		\item If $b_j$ does not lie in $\Ms_e(B)$, then $b_i - b_j > (j-i)e$.
	\end{enumerate}
\end{lemma}

\begin{proof}
	By \Cref{le:new bead}(i), $b_j+je \leq b_k +ke \leq b_i + ie$ in (i). By our assumption the left-hand side and the right-hand side equal and thus $b_j+je =b_k +ke = b_i + ie$, proving (i). For (ii), by \Cref{le:Me observation}(ii) and the assumption that $j>0$, $b_j\neq b_{j-1}-e$ and so $b_j < b_{j-1}-e$. In turn, $b_j < b_i-(j-i)e$ by \Cref{le:new bead}(i), which rearranges to the desired inequality.
\end{proof}

The next lemma uses two observations. Firstly, if $b'\leq b$ are beads of a \dbal $\beta$-set $B$ lying on the same runner and there is an empty space $f<b'$ also on the same runner, then $\bd_B(b'+1,b)$ is congruent to $b-b'$ modulo $d$ as it is the difference of $\bd_B(f+1,b)$ and $\bd_B(f+1,b')$ which are congruent to $b-f$ and $b'-f$ modulo $d$, respectively; See \Cref{fig:beadcount} with $b=b_i$ and $b'=b_{i+1}$. If $b=b'$, this agrees with our convention that $\bd_B(v+1,v)=0$ for any integer $v$.

Secondly, for integers $u\leq v$ if \textit{there is no bead of $B$ between $u$ and $v$}, which means that no $b\in B$ satisfies $u\leq b\leq v$ then for any $w\geq v$ we have $\bd_B(u,w) = \bd_B(v+1,w)$, and for any $w\leq u$ we have $\bd_B(w,u-1) = \bd_B(w,v)$. Note that the definition of `there is no bead of $B$ between $u$ and $v$' extends to $u>v$ and trivially holds for such $u$ and $v$. However, with this extended definition we can deduce $\bd_B(u,w) = \bd_B(v+1,w)$ and $\bd_B(w,u-1) = \bd_B(w,v)$ (for appropriate $w$) only if $u\leq v+1$. An example of $u\leq v+1$ such that there is no bead of $B$ between $u$ and $v$ are $u=b_{i+1}+1$ and $v=b_i-e$ for any integer $i\geq 0$.

\begin{lemma}\label{le:bead count}
	Let $B$ be a \dbal $\beta$-set with stable index $z$. If $0\leq i\leq j\leq z$ are integers, then $\bd_B(b_j+1,b_i)$ is congruent to $(j-i)e$ modulo $d$.
\end{lemma}

\begin{proof}
	The beads and empty spaces used in the argument are lustrated in \Cref{fig:beadcount}. The result is clear if $i=j$. It suffices to assume that $j=i+1$. As $b_z$ does not lie in $\Ms_e(B)$ by \Cref{le:Me observation}(iii), by \Cref{le:auxiliary}(ii), $b_z\neq b_i - (z-i)e$. Hence we can pick the least $k\geq i+1$ such that $b_k\neq b_i - (k-i)e$, which implies that $b_k+1\leq b_i - (k-i)e$ from \Cref{le:new bead}(i). By the minimal choice of $k$ we have $b_{k-1}-e = b_i-(k-i)e$ (even if $k=i+1$), and thus there is no bead between $b_k+1$ and $b_i - (k-i)e$. In particular, $b_i - (k-i)e$ is an empty space of $B$.
	If $k>i+1$, then $b_{i+1}=b_i-e$ and we can use that $B$ is $d$-balanced to compute
	\begin{align*}
	\bd_B(b_{i+1}+1,b_i) &= \bd_B(b_i-(k-i)e+1, b_i) - \bd_B(b_i-(k-i)e+1, b_i - e)\\
	&\equiv (k-i)e - (k-i-1)e\\
	&\equiv e\, (\textnormal{mod } d), 
	\end{align*}
	as required. If $k=i+1$, since there are no beads between $b_{i+1}+1$ and $b_i - e$, we compute $\bd_B(b_{i+1}+1, b_i) = \bd_B(b_i-e+1, b_i)$ which is indeed congruent to $e$ modulo $d$ as $B$ is a \dbal set. 
\end{proof}

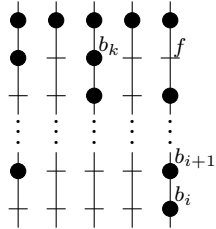
\begin{figure}[h]
	\centering
	\begin{tikzpicture}[x=0.5cm, y=0.5cm]
		\begin{pgfonlayer}{nodelayer}
			\node [style=none] (0) at (1.25, 3.5) {};
			\node [style=none] (1) at (2.25, 3.5) {};
			\node [style=none] (2) at (3.25, 3.5) {};
			\node [style=none] (3) at (4.25, 3.5) {};
			\node [style=none] (4) at (5.25, 3.5) {};
			\node [style=none] (5) at (1.25, 0.5) {};
			\node [style=none] (6) at (2.25, 0.5) {};
			\node [style=none] (7) at (3.25, 0.5) {};
			\node [style=none] (8) at (4.25, 0.5) {};
			\node [style=none] (9) at (5.25, 0.5) {};
			\node [style=Empty node] (10) at (1.25, 3) {};
			\node [style=Empty node] (11) at (2.25, 3) {};
			\node [style=Empty node] (12) at (3.25, 3) {};
			\node [style=Empty node] (13) at (4.25, 3) {};
			\node [style=Empty node] (15) at (1.25, 2) {};
			\node [style=Empty node] (16) at (3.25, 2) {};
			\node [style=Empty node] (18) at (5.25, 3) {};
			\node [style=none] (19) at (2, 2) {};
			\node [style=none] (20) at (2.5, 2) {};
			\node [style=none] (21) at (4, 2) {};
			\node [style=none] (22) at (4.5, 2) {};
			\node [style=none] (23) at (4, 1) {};
			\node [style=none] (24) at (4.5, 1) {};
			\node [style=none] (25) at (5, 2) {};
			\node [style=none] (26) at (5.5, 2) {};
			\node [style=none] (27) at (4.5, -1) {};
			\node [style=none] (28) at (4, -1) {};
			\node [style=none] (29) at (4.5, -2) {};
			\node [style=none] (30) at (4, -2) {};
			\node [style=none] (31) at (3.5, -2) {};
			\node [style=none] (32) at (3, -2) {};
			\node [style=none] (33) at (3.5, -1) {};
			\node [style=none] (34) at (3, -1) {};
			\node [style=none] (35) at (2.5, 1) {};
			\node [style=none] (36) at (2, 1) {};
			\node [style=none] (37) at (1.5, 1) {};
			\node [style=none] (38) at (1, 1) {};
			\node [style=Empty node] (39) at (5.25, 1) {};
			\node [style=Empty node] (40) at (3.25, 1) {};
			\node [style=none] (41) at (1.25, 0.25) {$\vdots$};
			\node [style=none] (42) at (2.25, 0.25) {$\vdots$};
			\node [style=none] (43) at (3.25, 0.25) {$\vdots$};
			\node [style=none] (44) at (4.25, 0.25) {$\vdots$};
			\node [style=none] (45) at (5.25, 0.25) {$\vdots$};
			\node [style=none] (46) at (1.25, -0.5) {};
			\node [style=none] (47) at (1.25, -2.5) {};
			\node [style=none] (48) at (2.25, -2.5) {};
			\node [style=none] (49) at (2.25, -0.5) {};
			\node [style=none] (50) at (3.25, -0.5) {};
			\node [style=none] (51) at (3.25, -2.5) {};
			\node [style=none] (52) at (4.25, -0.5) {};
			\node [style=none] (53) at (4.25, -2.5) {};
			\node [style=none] (54) at (5.25, -0.5) {};
			\node [style=none] (55) at (5.25, -2.5) {};
			\node [style=Empty node] (56) at (5.25, -1) {};
			\node [style=Empty node] (57) at (5.25, -2) {};
			\node [style=Empty node] (58) at (1.25, -1) {};
			\node [style=none] (59) at (2, -1) {};
			\node [style=none] (60) at (2.5, -1) {};
			\node [style=none] (61) at (2, -2) {};
			\node [style=none] (62) at (2.5, -2) {};
			\node [style=none] (63) at (1, -2) {};
			\node [style=none] (64) at (1.5, -2) {};
			\node [style=none] (65) at (5.6, -1.65) {$\scriptstyle b_i$};
			\node [style=none] (66) at (5.925, -0.65) {$\scriptstyle b_{i+1}$};
			\node [style=none] (68) at (3.65, 2.325) {$\scriptstyle b_k$};
			\node [style=none] (69) at (5.5, 2.275) {$\scriptstyle f$};
		\end{pgfonlayer}
		\begin{pgfonlayer}{edgelayer}
			\draw (0.center) to (5.center);
			\draw (1.center) to (6.center);
			\draw (3.center) to (8.center);
			\draw (4.center) to (9.center);
			\draw (19.center) to (20.center);
			\draw (21.center) to (22.center);
			\draw (23.center) to (24.center);
			\draw (25.center) to (26.center);
			\draw (28.center) to (27.center);
			\draw (30.center) to (29.center);
			\draw (32.center) to (31.center);
			\draw (34.center) to (33.center);
			\draw (36.center) to (35.center);
			\draw (38.center) to (37.center);
			\draw (2.center) to (7.center);
			\draw (46.center) to (47.center);
			\draw (49.center) to (48.center);
			\draw (50.center) to (51.center);
			\draw (52.center) to (53.center);
			\draw (54.center) to (55.center);
			\draw (59.center) to (60.center);
			\draw (61.center) to (62.center);
			\draw (63.center) to (64.center);
		\end{pgfonlayer}
	\end{tikzpicture}
	\caption{The proof of \Cref{le:bead count} with $f=b_i - (k-i)e$.}
	\label{fig:beadcount}
\end{figure}

A different way of using that there is no bead of $B$ between $u$ and $v$ with $u\leq v$ is the \textit{sandwich argument} --- we show that for some $w$ we have $u\leq w\leq v$ and thus there is no bead between $u$ and $w$ (and between $w$ and $v$), and particular $w$ is an empty space. One can then take advantage of this information.
The sandwich argument is used in the following lemma. We have already seen that two of the bounds in this lemma hold for any $\beta$-set $B$; the assumption that $B$ is \dbal lets us establish the remaining claims. The statements of parts (ii) and (iii) and the beads and the empty spaces from the proof of (iii) are depicted in \Cref{fig:wallI}.

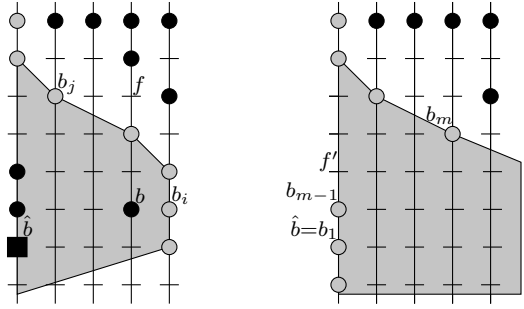
\begin{figure}[h]
	\centering
	\begin{tikzpicture}[x=0.5cm, y=0.5cm]
		\begin{pgfonlayer}{nodelayer}
			\node [style=none] (0) at (1.25, 3.5) {};
			\node [style=none] (4) at (5.25, 3.5) {};
			\node [style=none] (5) at (1.25, -4.5) {};
			\node [style=none] (6) at (2.25, -4.5) {};
			\node [style=none] (7) at (3.25, -4.5) {};
			\node [style=none] (8) at (4.25, -4.5) {};
			\node [style=none] (9) at (5.25, -4.5) {};
			\node [style=Empty node] (11) at (2.25, 3) {};
			\node [style=Empty node] (12) at (3.25, 3) {};
			\node [style=Empty node] (13) at (4.25, 3) {};
			\node [style=Lead] (15) at (1.25, 2) {};
			\node [style=Empty node] (16) at (4.25, 2) {};
			\node [style=Empty node] (18) at (5.25, 3) {};
			\node [style=none] (19) at (2, 2) {};
			\node [style=none] (20) at (2.5, 2) {};
			\node [style=none] (21) at (3, 2) {};
			\node [style=none] (22) at (3.5, 2) {};
			\node [style=none] (23) at (4, 1) {};
			\node [style=none] (24) at (4.5, 1) {};
			\node [style=none] (25) at (5, 2) {};
			\node [style=none] (26) at (5.5, 2) {};
			\node [style=none] (35) at (3.5, 1) {};
			\node [style=none] (38) at (1, 1) {};
			\node [style=Empty node] (39) at (5.25, 1) {};
			\node [style=Lead] (40) at (2.25, 1) {};
			\node [style=Lead] (46) at (4.25, 0) {};
			\node [style=Lead] (47) at (5.25, -1) {};
			\node [style=Lead] (49) at (5.25, -3) {};
			\node [style=Empty node] (51) at (1.25, -1) {};
			\node [style=Empty node] (52) at (1.25, -2) {};
			\node [style=none] (53) at (1.25, -4.25) {};
			\node [style=Empty node] (54) at (4.25, -2) {};
			\node [style=none] (55) at (1, 0) {};
			\node [style=none] (56) at (1.5, 0) {};
			\node [style=none] (57) at (2, 0) {};
			\node [style=none] (58) at (2.5, 0) {};
			\node [style=none] (59) at (3, 0) {};
			\node [style=none] (60) at (3.5, 0) {};
			\node [style=none] (61) at (5, 0) {};
			\node [style=none] (62) at (5.5, 0) {};
			\node [style=none] (63) at (4, -1) {};
			\node [style=none] (64) at (4.5, -1) {};
			\node [style=none] (65) at (3.5, -1) {};
			\node [style=none] (66) at (3, -1) {};
			\node [style=none] (67) at (2.5, -1) {};
			\node [style=none] (68) at (2, -1) {};
			\node [style=none] (69) at (2.5, -2) {};
			\node [style=none] (70) at (2, -2) {};
			\node [style=none] (71) at (3.5, -2) {};
			\node [style=none] (72) at (3, -2) {};
			\node [style=none] (73) at (1, -4) {};
			\node [style=none] (74) at (1.5, -4) {};
			\node [style=none] (75) at (2, -3) {};
			\node [style=none] (76) at (2.5, -3) {};
			\node [style=none] (77) at (3, -3) {};
			\node [style=none] (78) at (3.5, -3) {};
			\node [style=none] (79) at (4, -3) {};
			\node [style=none] (80) at (4.5, -3) {};
			\node [style=none] (81) at (5, -4) {};
			\node [style=none] (82) at (5.5, -4) {};
			\node [style=none] (83) at (4, -4) {};
			\node [style=none] (84) at (4.5, -4) {};
			\node [style=none] (85) at (3.5, -4) {};
			\node [style=none] (86) at (3, -4) {};
			\node [style=none] (87) at (2.5, -4) {};
			\node [style=none] (88) at (2, -4) {};
			\node [style=none] (89) at (2.25, 3.5) {};
			\node [style=none] (90) at (3.25, 3.5) {};
			\node [style=none] (91) at (4.25, 3.5) {};
			\node [style=none] (92) at (1.5, 1) {};
			\node [style=none] (93) at (3, 1) {};
			\node [style=none] (94) at (5.55, -1.625) {$\scriptstyle b_i$};
			\node [style=none] (95) at (4.45, -1.575) {$\scriptstyle b$};
			\node [style=none] (96) at (1.525, -2.425) {$\scriptstyle \hat{b}$};
			\node [style=none] (97) at (4.45, 1.3) {$\scriptstyle f$};
			\node [style=none] (98) at (2.575, 1.325) {$\scriptstyle b_j$};
			\node [style=none] (99) at (9.7, 3.5) {};
			\node [style=none] (100) at (13.7, 3.5) {};
			\node [style=none] (101) at (9.7, -4.5) {};
			\node [style=none] (102) at (10.7, -4.5) {};
			\node [style=none] (103) at (11.7, -4.5) {};
			\node [style=none] (104) at (12.7, -4.5) {};
			\node [style=none] (105) at (13.7, -4.5) {};
			\node [style=Empty node] (107) at (10.7, 3) {};
			\node [style=Empty node] (108) at (11.7, 3) {};
			\node [style=Empty node] (109) at (12.7, 3) {};
			\node [style=Lead] (111) at (9.7, 2) {};
			\node [style=Empty node] (112) at (13.7, 3) {};
			\node [style=none] (113) at (10.45, 2) {};
			\node [style=none] (114) at (10.95, 2) {};
			\node [style=none] (115) at (11.45, 2) {};
			\node [style=none] (116) at (11.95, 2) {};
			\node [style=none] (117) at (12.45, 1) {};
			\node [style=none] (118) at (12.95, 1) {};
			\node [style=none] (119) at (13.45, 2) {};
			\node [style=none] (120) at (13.95, 2) {};
			\node [style=none] (121) at (11.95, 1) {};
			\node [style=none] (122) at (9.45, 1) {};
			\node [style=Empty node] (123) at (13.7, 1) {};
			\node [style=Lead] (127) at (10.7, 1) {};
			\node [style=none] (130) at (9.7, -4.25) {};
			\node [style=Lead] (132) at (12.7, 0) {};
			\node [style=none] (133) at (9.45, 0) {};
			\node [style=none] (134) at (9.95, 0) {};
			\node [style=none] (135) at (10.45, 0) {};
			\node [style=none] (136) at (10.95, 0) {};
			\node [style=none] (137) at (11.45, 0) {};
			\node [style=none] (138) at (11.95, 0) {};
			\node [style=none] (139) at (13.45, 0) {};
			\node [style=none] (140) at (13.95, 0) {};
			\node [style=none] (141) at (12.45, -1) {};
			\node [style=none] (142) at (12.95, -1) {};
			\node [style=none] (143) at (11.95, -1) {};
			\node [style=none] (144) at (11.45, -1) {};
			\node [style=none] (145) at (10.95, -1) {};
			\node [style=none] (146) at (10.45, -1) {};
			\node [style=none] (147) at (10.95, -2) {};
			\node [style=none] (148) at (10.45, -2) {};
			\node [style=none] (149) at (11.95, -2) {};
			\node [style=none] (150) at (11.45, -2) {};
			\node [style=none] (151) at (9.45, -1) {};
			\node [style=none] (152) at (9.95, -1) {};
			\node [style=none] (153) at (10.45, -3) {};
			\node [style=none] (154) at (10.95, -3) {};
			\node [style=none] (155) at (11.45, -3) {};
			\node [style=none] (156) at (11.95, -3) {};
			\node [style=none] (157) at (12.45, -3) {};
			\node [style=none] (158) at (12.95, -3) {};
			\node [style=none] (159) at (13.45, -4) {};
			\node [style=none] (160) at (13.95, -4) {};
			\node [style=none] (161) at (12.45, -4) {};
			\node [style=none] (162) at (12.95, -4) {};
			\node [style=none] (163) at (11.95, -4) {};
			\node [style=none] (164) at (11.45, -4) {};
			\node [style=none] (165) at (10.95, -4) {};
			\node [style=none] (166) at (10.45, -4) {};
			\node [style=none] (167) at (10.7, 3.5) {};
			\node [style=none] (168) at (11.7, 3.5) {};
			\node [style=none] (169) at (12.7, 3.5) {};
			\node [style=none] (170) at (9.95, 1) {};
			\node [style=none] (171) at (11.45, 1) {};
			\node [style=none] (174) at (9.05, -2.475) {$\scriptstyle \hat{b}=b_1$};
			\node [style=none] (178) at (14.5, -4.25) {};
			\node [style=none] (179) at (13.5, -1) {};
			\node [style=none] (180) at (14, -1) {};
			\node [style=none] (181) at (13.5, -2) {};
			\node [style=none] (182) at (14, -2) {};
			\node [style=none] (183) at (13.5, -3) {};
			\node [style=none] (184) at (14, -3) {};
			\node [style=none] (185) at (13, -2) {};
			\node [style=none] (186) at (12.5, -2) {};
			\node [style=none] (187) at (12.5, 2) {};
			\node [style=none] (188) at (13, 2) {};
			\node [style=none] (189) at (14.5, -0.75) {};
			\node [style=none] (191) at (12.325, 0.525) {$\scriptstyle b_m$};
			\node [style=none] (192) at (8.975, -1.55) {$\scriptstyle b_{m-1}$};
			\node [style=square] (193) at (1.25, -3) {};
			\node [style=Lead] (195) at (1.25, 3) {};
			\node [style=Lead] (196) at (9.7, 3) {};
			\node [style=Lead] (197) at (9.7, -3) {};
			\node [style=Lead] (198) at (9.7, -2) {};
			\node [style=Lead] (199) at (9.7, -4) {};
			\node [style=Lead] (200) at (5.25, -2) {};
			\node [style=none] (201) at (9.4, -0.725) {$\scriptstyle f'$};
		\end{pgfonlayer}
		\begin{pgfonlayer}{edgelayer}
			\draw (0.center) to (5.center);
			\draw (4.center) to (9.center);
			\draw (19.center) to (20.center);
			\draw (21.center) to (22.center);
			\draw (23.center) to (24.center);
			\draw (25.center) to (26.center);
			\draw [style=Light grey column] (53.center)
			to (15.center)
			to (40.center)
			to (46.center)
			to (47.center)
			to (49.center)
			to cycle;
			\draw (55.center) to (56.center);
			\draw (57.center) to (58.center);
			\draw (59.center) to (60.center);
			\draw (61.center) to (62.center);
			\draw (63.center) to (64.center);
			\draw (66.center) to (65.center);
			\draw (68.center) to (67.center);
			\draw (70.center) to (69.center);
			\draw (72.center) to (71.center);
			\draw (73.center) to (74.center);
			\draw (75.center) to (76.center);
			\draw (77.center) to (78.center);
			\draw (79.center) to (80.center);
			\draw (88.center) to (87.center);
			\draw (86.center) to (85.center);
			\draw (83.center) to (84.center);
			\draw (81.center) to (82.center);
			\draw (89.center) to (6.center);
			\draw (90.center) to (7.center);
			\draw (91.center) to (8.center);
			\draw (93.center) to (35.center);
			\draw (38.center) to (92.center);
			\draw (99.center) to (101.center);
			\draw (100.center) to (105.center);
			\draw (113.center) to (114.center);
			\draw (115.center) to (116.center);
			\draw (117.center) to (118.center);
			\draw (119.center) to (120.center);
			\draw (133.center) to (134.center);
			\draw (135.center) to (136.center);
			\draw (137.center) to (138.center);
			\draw (139.center) to (140.center);
			\draw (141.center) to (142.center);
			\draw (144.center) to (143.center);
			\draw (146.center) to (145.center);
			\draw (148.center) to (147.center);
			\draw (150.center) to (149.center);
			\draw (151.center) to (152.center);
			\draw (153.center) to (154.center);
			\draw (155.center) to (156.center);
			\draw (157.center) to (158.center);
			\draw (166.center) to (165.center);
			\draw (164.center) to (163.center);
			\draw (161.center) to (162.center);
			\draw (159.center) to (160.center);
			\draw [in=90, out=-90] (167.center) to (102.center);
			\draw (168.center) to (103.center);
			\draw (169.center) to (104.center);
			\draw (171.center) to (121.center);
			\draw (122.center) to (170.center);
			\draw [style=Light grey column] (111.center)
			to (127.center)
			to (132.center)
			to (189.center)
			to (178.center)
			to (130.center)
			to cycle;
			\draw (135.center) to (136.center);
			\draw (122.center) to (170.center);
			\draw (133.center) to (134.center);
			\draw (151.center) to (152.center);
			\draw (146.center) to (145.center);
			\draw (137.center) to (138.center);
			\draw (144.center) to (143.center);
			\draw (141.center) to (142.center);
			\draw (179.center) to (180.center);
			\draw (181.center) to (182.center);
			\draw (183.center) to (184.center);
			\draw (159.center) to (160.center);
			\draw (161.center) to (162.center);
			\draw (164.center) to (163.center);
			\draw (166.center) to (165.center);
			\draw (153.center) to (154.center);
			\draw (148.center) to (147.center);
			\draw (150.center) to (149.center);
			\draw (186.center) to (185.center);
			\draw (157.center) to (158.center);
			\draw (187.center) to (188.center);
			\draw (167.center) to (102.center);
			\draw (103.center) to (168.center);
			\draw (169.center) to (104.center);
			\draw (100.center) to (105.center);
			\draw (155.center) to (156.center);
		\end{pgfonlayer}
	\end{tikzpicture}
	\caption{In the picture the leading beads are drawn grey and $\hat{b}$ is drawn as a box in the left diagram. The labels in the left diagram come from the proof of \Cref{le:wall I}(iii) (with $f = b - (j-i)e$) and the labels in the right diagram come from the proof of \Cref{le:overflow} (with $f' = b_{m-1} - e$). If the bead $\hat{b}$ is placed on the runner $0$, then \Cref{le:wall I}(ii) shows that there is a leading bead in each row (less than or equal to the row of $b_0$) with a possible single exception (as in the second diagram). In the latter case, we are in the setting of \Cref{le:overflow}. \Cref{le:wall I}(iii) then states there cannot be beads inside the highlighted area - the area between runner $0$ and the leading beads (so the bead $b$ in the left diagram is not allowed).}
	\label{fig:wallI}
\end{figure}

\begin{lemma}\label{le:wall I}
	Let $B$ be a \dbal $\beta$-set of a partition. For any non-negative integer $i$ the following holds.
	\begin{enumerate}[label=\textnormal{(\roman*)}]
		\item $(j-i)e\leq b_i - b_j \leq (j-i+1)e$ for any $j\geq i$.
		\item $b_i + (i-1)e \leq \hat{b} \leq b_i + ie$.
		\item There is no bead between $\hat{b} -ie +1$ and $b_i-1$.
	\end{enumerate} 
\end{lemma}

\begin{proof}
	We start with (iii); the setup of its proof is in \Cref{fig:wallI}. Let us suppose that (iii) fails to hold and let $b$ be the greatest bead such that $\hat{b} -ie +1\leq b\leq b_i-1$. Now let $j$ be the least non-negative integer such that $b_j < b - (j-i)e$ --- such $j$ exists (and is at most $z$, the stable index of $B$) as $b_z = \hat{b} -ze < b - (z-i)e$. Moreover, $j>i$ (and in particular $j>0$) since for any $k\leq i$ we have $b\leq b_i-1\leq b_k - (i-k)e - 1$ by \Cref{le:new bead}(i).
	
	By the minimality of $j$ and the fact that it is positive, $b_j < b-(j-i)e \leq b_{j-1} + (j-i-1)e - (j-i)e = b_{j-1}-e$ and, using the sandwich argument, there is no bead between $b_j+1$ and $b-(j-i)e$. We compute using \Cref{le:bead count} and the choice of $b$, which forces that there is no bead of $B$ between $b+1$ and $b_i-1$ (see also \Cref{fig:wallI}) that
	\begin{align*}
		\emp_B(b-(j-i)e + 1,b) &= (j-i)e - \bd_B(b-(j-i)e + 1,b)\\
		&=  (j-i)e - \bd_B(b_j + 1 ,b)\\
		&=  (j-i)e - \bd_B(b_j + 1 ,b_i) + 1\\
		&\equiv 1\, (\textnormal{mod } d),
	\end{align*}
	a contradiction to $B$ being a \dbal set. This proves (iii).
	
	The left inequality in (i) is just \Cref{le:new bead}(i). Suppose that the other inequality fails to hold for $B$ and pick minimal $j \geq i$ such that $b_i - b_j > (j-i+1)e$. Clearly $j>i$. By this choice
	\[
	b_j < b_i - (j-i+1)e \leq b_{j-1} + (j-i)e - (j-i+1)e = b_{j-1} - e.
	\]
	Consequently, $j\leq z$, the stable index of $B$, and by the sandwich argument there is no bead between $b_j+1$ and $b_i - (j-i+1)e$. As $B$ is $d$-balanced, $\bd_B(b_j+1, b_i) = \bd_B(b_i - (j-i+1)e+1, b_i)$ is congruent to $(j-i+1)e$ modulo $d$. We conclude that $d\mid e$ from \Cref{le:bead count}.
	
	On the other hand, by (iii), there is no bead between $\hat{b} -ie +1$ and $b_i-1$, and by \Cref{le:new bead}(ii) and our assumption, $\hat{b} -ie +1\leq b_j + (j-i)e+1\leq b_i-e$; thus $b_i-e$ is an empty space of $B$ and $\emp_B(b_i-e+1, b_i)=e-1$. Since $B$ is $d$-balanced, $d$ also divides $e-1$, a contradiction. 
	
	Finally, the right inequality of (ii) is \Cref{le:new bead}(ii). To show the other inequality, pick an integer $j \geq z$, the stable index of $B$ which is greater than or equal to $i$. By this choice of $j$ and (i), $\hat{b} = b_j + je \geq b_i + (i-1)e$, as required.
\end{proof}

We immediately conclude the following.

\begin{corollary}\label{cor:wall}
	Let $B$ be a \dbal $\beta$-set of a partition.
	\begin{enumerate}[label=\textnormal{(\roman*)}]
		\item For any positive integer $i$ we have $\hat{b} \geq b_i$.
		\item $\hat{b}\geq \max\Ms_e(B)$.
	\end{enumerate}	
\end{corollary}

\begin{proof}
	By \Cref{le:wall I}(ii) applied to a non-negative integer $1$, we get $\hat{b} \geq b_1 \geq b_i$, proving (i). By \Cref{le:wall I}(iii) applied with $i=0$, the only bead of $B$ greater than $\hat{b}$ can be $b_0$. Thus the only bead of $\Ms_e(B)$ greater than $\hat{b}$ can be $b_0-e$. But $b_0-e\leq \hat{b}$ by \Cref{le:wall I}(ii) applied with $i=0$, establishing (ii).
\end{proof}

While \Cref{le:wall I}(i) allows the option $b_i - b_j = (j-i+1)e$, the following observation shows that if such equality occurs, we are in a special setting, as in the second diagram of \Cref{fig:wallI} (in particular, there are $e$ consecutive empty spaces of $B$ of the form $b_{m-1}-e, b_{m-1}-e+1,\dots ,b_{m-1}-1$ as shown in the proof of \Cref{le:overflow} which forces $d\mid e-1$).

\begin{lemma}\label{le:overflow}
	Let $B$ be a \dbal $\beta$-set of a partition, $z$ be its stable index and $i<j$ be non-negative integers such that $b_i - b_j = (j-i+1)e$. Then
	\begin{enumerate}[label=\textnormal{(\roman*)}]
		\item $j\geq z$;
		\item for all $k\leq i$, $\hat{b} = b_k + (k-1)e$;
		\item $d\mid e-1$.
	\end{enumerate}
\end{lemma}

\begin{proof}
	By \Cref{le:wall I}(ii), $\hat{b} \leq b_j + je = b_i +(i-1)e \leq \hat{b}$ which forces equalities throughout. By \Cref{le:new bead}(ii), we get $j\geq z$, proving (i). Parts (i) and (ii) of \Cref{le:wall I} show that for any $k\leq i$ we have $\hat{b}\geq b_k +(k-1)e \geq b_i + (i-1)e = \hat{b}$. Hence we again need equalities throughout, showing (ii).
	
	For (iii) the important beads and empty spaces are drawn in the second diagram of \Cref{fig:wallI}. Let $m$ be the least integer such that $\hat{b} \neq b_m + (m-1)e$ (from earlier $j\geq m > i$). By \Cref{le:wall I}(iii), there is no bead between $\hat{b} - (m-1)e+1=b_{m-1}-e+1$ and $b_{m-1} - 1$. Thus $\emp_B(b_{m-1}-e+1, b_{m-1}) = e-1$. Moreover, $b_{m-1}-e$ is an empty space as $b_m\neq \hat{b}-(m-1)e = b_{m-1}-e$, and hence $d$ divides $\emp_B(b_{m-1}-e+1, b_{m-1}) = e-1$ as $B$ is a \dbal set. 
\end{proof}

The main result of the section follows. In the proof we use our convention that $b_0, b_1, \dots$ denote the leading beads of a $\beta$-set called $B$. We also frequently compare $\emp_B(u,v)$ and $\emp_{\Ms_e(B)}(u,v)$ for some $u+e\leq v$. To do so, we say that a leading bead $b_i$ \textit{moves out of the interval} if $u\leq b_i$ and $b_i-e<u$, that is $u\leq b_i<u+e$. On the other hand, we say that it \textit{moves into the interval} if $v< b_i$ and $b_i-e\leq v$, that is $v<b_i\leq v+e$.

It is clear that $\emp_{\Ms_e(B)}(u,v)$ is obtained from $\emp_B(u,v)$ by adding one for each leading bead which moves out of the interval and subtracting one for each one which moves into the interval. Moreover, from \eqref{eq:leading}, there is at most one leading bead which moves out of the interval and at most one which moves into the interval. 

In the proof we also use the \textit{bounding argument}: if there is no bead of $B$ between $u$ and $v$ and $b\leq v$ is a bead of $B$, then $b\leq u-1$. If moreover $u-1\leq b\leq v$, then $b=u-1$.

\begin{proposition}\label{pr:balanced}
	Let $B$ be a \dbal $\beta$-set of a partition. Then $\Ms_e(B)$ is also a \dbal set.
\end{proposition}

\begin{figure}[h]
	\centering
	\begin{tikzpicture}[x=0.5cm, y=0.5cm]
		\begin{pgfonlayer}{nodelayer}
			\node [style=none] (0) at (-7, 3) {};
			\node [style=none] (1) at (6, 3) {};
			\node [style=Empty node] (2) at (5.5, 3) {};
			\node [style=Empty node] (3) at (2.5, 3) {};
			\node [style=Empty node] (4) at (-3.5, 3) {};
			\node [style=none] (5) at (5.5, 2.45) {$\scriptstyle b_i$};
			\node [style=none] (6) at (5.25, 2.75) {};
			\node [style=none] (7) at (3.75, 2.75) {};
			\node [style=none] (8) at (-3.5, 2.45) {$\scriptstyle b_j$};
			\node [style=none] (9) at (3.5, 3.25) {};
			\node [style=none] (10) at (3.5, 2.75) {};
			\node [style=none] (11) at (3.5, 3.55) {$\scriptstyle b$};
			\node [style=none] (12) at (2.725, 2.45) {$\scriptstyle b_{i+1}$};
			\node [style=none] (13) at (-3.75, 2.75) {};
			\node [style=none] (14) at (-5.25, 2.75) {};
			\node [style=none] (15) at (-4.5, 3.25) {};
			\node [style=none] (16) at (-4.5, 2.75) {};
			\node [style=none] (17) at (-4.375, 3.55) {$\scriptstyle b-ae$};
			\node [style=none] (18) at (0.75, 2.75) {};
			\node [style=none] (19) at (2.25, 2.75) {};
			\node [style=none] (20) at (-7, 0.25) {};
			\node [style=none] (21) at (6, 0.25) {};
			\node [style=Empty node] (22) at (5.5, 0.25) {};
			\node [style=Empty node] (23) at (2.5, 0.25) {};
			\node [style=Empty node] (24) at (-4.5, 0.25) {};
			\node [style=none] (25) at (5.5, -0.3) {$\scriptstyle b_i$};
			\node [style=none] (26) at (5.25, 0) {};
			\node [style=none] (27) at (3.75, 0) {};
			\node [style=none] (28) at (-4.5, -0.3) {$\scriptstyle b_j$};
			\node [style=none] (29) at (3.5, 0.5) {};
			\node [style=none] (30) at (3.5, 0) {};
			\node [style=none] (31) at (3.5, 0.8) {$\scriptstyle b$};
			\node [style=none] (32) at (2.725, -0.3) {$\scriptstyle b_{i+1}$};
			\node [style=none] (33) at (-4.75, 0) {};
			\node [style=none] (34) at (-6.25, 0) {};
			\node [style=none] (37) at (-4.375, 0.8) {$\scriptstyle b-ae$};
			\node [style=none] (38) at (0.75, 0) {};
			\node [style=none] (39) at (2.25, 0) {};
			\node [style=none] (40) at (-7, -2.5) {};
			\node [style=none] (41) at (6, -2.5) {};
			\node [style=Empty node] (42) at (4.5, -2.5) {};
			\node [style=Empty node] (44) at (-3.5, -2.5) {};
			\node [style=none] (45) at (4.5, -3.05) {$\scriptstyle b_i$};
			\node [style=none] (46) at (4.25, -2.75) {};
			\node [style=none] (47) at (2.75, -2.75) {};
			\node [style=none] (48) at (-3.5, -3.05) {$\scriptstyle b_j$};
			\node [style=none] (51) at (3.5, -1.95) {$\scriptstyle b$};
			\node [style=none] (53) at (-3.75, -2.75) {};
			\node [style=none] (54) at (-5.25, -2.75) {};
			\node [style=none] (55) at (-4.5, -2.25) {};
			\node [style=none] (56) at (-4.5, -2.75) {};
			\node [style=none] (57) at (-4.5, -1.95) {$\scriptstyle b-ae$};
			\node [style=none] (60) at (-7, -5.25) {};
			\node [style=none] (61) at (6, -5.25) {};
			\node [style=Empty node] (62) at (4.5, -5.25) {};
			\node [style=Empty node] (64) at (-4.5, -5.25) {};
			\node [style=none] (65) at (4.5, -5.8) {$\scriptstyle b_i$};
			\node [style=none] (66) at (4.25, -5.5) {};
			\node [style=none] (67) at (2.75, -5.5) {};
			\node [style=none] (68) at (-4.5, -5.8) {$\scriptstyle b_j$};
			\node [style=none] (71) at (3.5, -4.7) {$\scriptstyle b$};
			\node [style=none] (73) at (-4.75, -5.5) {};
			\node [style=none] (74) at (-6.25, -5.5) {};
			\node [style=none] (77) at (-4.5, -4.7) {$\scriptstyle b-ae$};
			\node [style=Empty node] (78) at (3.5, -2.5) {};
			\node [style=Empty node] (79) at (3.5, -5.25) {};
		\end{pgfonlayer}
		\begin{pgfonlayer}{edgelayer}
			\draw (0.center) to (1.center);
			\draw [style=Move it, bend left, looseness=1.25] (6.center) to (7.center);
			\draw (9.center) to (10.center);
			\draw [style=Move it, bend left, looseness=1.25] (13.center) to (14.center);
			\draw (15.center) to (16.center);
			\draw [style=Move it, bend left, looseness=1.25] (19.center) to (18.center);
			\draw (20.center) to (21.center);
			\draw [style=Move it, bend left, looseness=1.25] (26.center) to (27.center);
			\draw (29.center) to (30.center);
			\draw [style=Move it, bend left, looseness=1.25] (33.center) to (34.center);
			\draw [style=Move it, bend left, looseness=1.25] (39.center) to (38.center);
			\draw (40.center) to (41.center);
			\draw [style=Move it, bend left, looseness=1.25] (46.center) to (47.center);
			\draw [style=Move it, bend left, looseness=1.25] (53.center) to (54.center);
			\draw (55.center) to (56.center);
			\draw (60.center) to (61.center);
			\draw [style=Move it, bend left, looseness=1.25] (66.center) to (67.center);
			\draw [style=Move it, bend left, looseness=1.25] (73.center) to (74.center);
		\end{pgfonlayer}
	\end{tikzpicture}
	\caption{The four cases in the proof of \Cref{pr:balanced}. In all the cases $j=a+i$. In the first two cases $b = b_i-e$ (and it may be the case that $b=b_{i+1}$), while in the last two cases $b_i-e< b < b_i$. In the first and the third case $b_j -e < b-ae < b_j$, while in the remaining two cases $b-ae = b_j$.}
	\label{fig:preservebalance}
\end{figure}
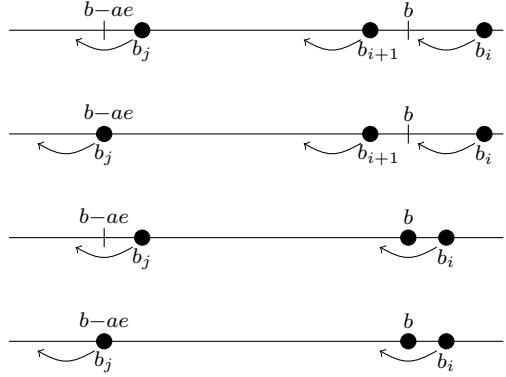

\begin{proof}
	Let $b$ be a bead of $\Ms_e(B)$ and $a$ be a positive integer such that $b-ae$ is an empty space in $\Ms_e(B)$. We show that $d\mid \emp_{\Ms_e(B)}(b-ae+1,b)$. This is done by distinguishing four cases shown in \Cref{fig:preservebalance}. We firstly assume that there is a non-negative integer $i$ such that $b= b_i - e$ and let $j=a+i$, so $a=j-i$. By \Cref{le:wall I}(i), $b- ae = b_i - (j-i+1)e \leq b_j\leq b_i -(j-i)e = b - (a-1)e$.
	
	If $b-ae\notin B$, then this can be strengthen to $b-ae < b_j \leq b - (a-1)e$. In such case, therefore, $\emp_{\Ms_e(B)}(b-ae+1, b) = \emp_B(b-ae+1, b)$ as $b_i$ moves into the interval and $b_j$ moves out of the interval; see the first diagram of \Cref{fig:preservebalance}. Hence, using that $B$ is \dbal
	\begin{align*}
	\emp_{\Ms_e(B)}(b-ae+1, b) &= \emp_B(b-ae+1, b_i) - \emp_B(b_i-e+1, b_i)\\
	&\equiv \bd_B(b_i-e+1, b_i) -e\\
	&\equiv \bd_B(b_{i+1}+1, b_i) - e\, (\textnormal{mod } d),
	\end{align*}
	which is $0$ modulo $d$ by \Cref{le:bead count}, which applies as $i<z$, the stable index of $B$, since $b_i-(a+1)e=b-ae\notin \Ms_e(B)$. 
	
	If $b-ae\in B$, then it must be a leading bead of $B$ and moreover $b-ae +e$ cannot be a leading bead of $B$ as $b-ae\notin \Ms_e(B)$. Hence $b_j = b-ae$ and $j\leq z$. Thus $\emp_{\Ms_e(B)}(b-ae+1, b) = \emp_B(b-ae+1, b) -1 $ (bead $b_i$ moves into the interval; see the second diagram of \Cref{fig:preservebalance}). This equals $ae-\bd_B(b-ae+1, b_i-e) -1 = ae - \bd_B(b_j + 1, b_{i+1}) - 1$, which, by \Cref{le:bead count} and the identity $a=j-i$, is congruent to $e-1$ modulo $d$. Since $b_i-b_j = (b+e) - (b-ae) = (j-i+1)e$, we conclude that $d$ divides $e-1$ by \Cref{le:overflow}(iii), as required.
	
	Now we move to the case that $b$ is not of the form $b_i-e$. Then $b$ is a non-leading bead of $B$ and thus there is the largest non-negative integer $i$ such that $b < b_i$ and in turn $b_i - e < b$. By \Cref{le:wall I}(iii), we get $b\leq \hat{b} -ie$ using the bounding argument. Again, let $j=a+i$. By parts (i) and (ii) of \Cref{le:wall I} we can write $b-ae\leq \hat{b} - je \leq b_j \leq b_i - ae< b- (a-1)e$.
	
	If $b-ae\notin B$, we have $b_j \neq b-ae$, and as in the second paragraph $\emp_{\Ms_e(B)}(b-ae+1, b) = \emp_B(b-ae+1, b) $; see the third diagram of \Cref{fig:preservebalance}. We are done as the letter is divisible by $d$ as $B$ is a \dbal set.
	
	Finally, if $b-ae\in B$, then as in the third paragraph $b_j = b-ae$, $j\leq z$ and $\emp_{\Ms_e(B)}(b-ae+1, b) = \emp_B(b-ae+1, b) -1 $; see the fourth diagram of \Cref{fig:preservebalance}. By \Cref{le:wall I}(ii), $\hat{b} -ie \leq b_j +ae = b\, (\leq b_i-1)$ and thus part (iii) of the same result together with the bounding argument implies that $b=\hat{b} -ie$ and there are no beads between $b+1$ and $b_i-1$. Therefore $\emp_B(b-ae+1, b) -1 = ae - \bd_B(b_j+1, b_i)$ which is divisible by $d$ by \Cref{le:bead count}.
\end{proof}

\section{$d$-combined pairs}\label{se:pairs}

To prove \Cref{th:Mullineuxbalanced} we need to adapt some of the definitions for sets to pairs of sets. Suppose that $(B,C)$ is a pair of subsets of integers. For integers $u\leq v$, we define the \textit{combined emptiness} by $\emp_{(B,C)}(u,v) = \emp_B(u,v) - \bd_C(u,v)$, or equivalently, $\emp_{(B,C)}(u,v) = \emp_C(u,v) - \bd_B(u,v)$. Note that $\emp_{(B,C)}(u,v)$ does \emph{not}, in general, count the empty spaces $f$ of $B\cup C$ such that $u\leq f\leq v$, but it does so if $B$ and $C$ are disjoint. If $B$ is a $\beta$-set of a partition and $C$ is a finite set, then, similarly to \Cref{de:emptyness}(iii), we define $\emp_{(B,C)}(v)$ as the limit of $\emp_{(B,C)}(u,v)$ as $u$ goes to $-\infty$. This is well-defined as the sequences $\emp_B(u,v)$ and $\bd_C(u,v)$ and, consequently, $\emp_{(B,C)}(u,v)$ become constant for fixed $v$ and sufficiently small $u$.

We also need the following technical definition. If $(B,C)$ is a pair of a $\beta$-set of a partition and a subset of integers, an empty space of $B\cup C$ is \textit{admissible} unless it is the \emph{only} empty space of $B$ of the form $b_0 - ke$ with $k\geq 0$ and moreover $b_0\in C$ (where $b_0=\max B$, as usually). From the definition it is clear that there is at most one non-admissible empty space for a fixed pair $(B,C)$. An example of a non-admissible empty space is the empty space $f$ in \Cref{fig:d-pair}.

In practice, instead of the technical definition of the admissible empty spaces we will use the following basic observations.

\begin{lemma}\label{le:admissible}
	Let $(B,C)$ be a pair of a $\beta$-set of a partition and a finite subset of integers. Let $B' = \Ms_e(B)$ and $C'=C \cup \{ \hat{b} \} $.
	\begin{enumerate}[label=\textnormal{(\roman*)}]
		\item If $a$ is a positive integer and $f\leq b_0$ is an empty space of $B\cup C$ such that $f-ae$ is also an empty space in $B\cup C$, then $f$ and $f-ae$ are admissible.
		\item Let $f$ be an empty space of $B\cup C$. If $f+e$ is not a leading bead of $B$, then $f$ is admissible.
		\item Suppose that $B$ is $d$-balanced, the stable index $z$ of $B$ is positive and $\hat{b} -ae$ lies in $B$ for all $0\leq a<z$. If $b_z\notin C$, then $b_z$ is a non-admissible empty space of $B'\cup C'$.
		\item Suppose that $B$ is \dbal and $i$ is a non-negative integer such that $b_i$ is an admissible empty space of $B'\cup C'$. If for some positive integer $a$, $b_i+ae$ is an empty space in $B\cup C$, then it is admissible.
	\end{enumerate} 
\end{lemma}

\begin{proof}
	If (i) fails to hold, then $f$ and $f-ae$ must be of the form $b_0 - je$ with $j\geq 0$. But then there are more than one empty spaces of this form, so they must be admissible. We show the contrapositive of (ii). Suppose that $f$ in (ii) is non-admissible. Then $f=b_0-je$ for some $j>0$ and all $b_0-ke$ with $0\leq k<j$ lie in $B$ and hence must be leading beads of $B$. Taking $k=j-1$, we get that $f+e$ is a leading bead, as required.
	
	Moving to (iii), the runners of $b_z$ in $B$ and $B'$ are displayed in \Cref{fig:admissible} on the left. As $z>0$, $b_z\neq \hat{b}$ and so by \Cref{le:Me observation}(iii) $b_z\notin B'\cup C'$. By the definition of the stable index all the beads $b_z+(z-a)e = \hat{b}-ae$ of $B$ with $0\leq a<z$ are non-leading. Moreover, $b_z+(z-a)e = \hat{b}-ae$ with $a>z$ are leading beads of $B$, and thus $B'$ contains $\hat{b}-ae$ for every non-negative $a$ except from $a=z$, when $\hat{b}-ae=b_z$. By \Cref{cor:wall}(ii), $\hat{b}\geq \max B'$ and so $\hat{b}= \max B'$. Since $\hat{b}\in C'$, we conclude that $b_z$ is a non-admissible empty space of $B'\cup C'$.
	
	The runners of $b_i$ in $B$ and $B'$ in the proof of (iv) are displayed in \Cref{fig:admissible} on the right. Suppose that (iv) fails to hold. Then $i\geq 1$ and $b_i+ae = b_0 -je$ for some $j> 0$ and $b_0-ke$ is a bead of $B$ for all $k\geq 0$ with $k\neq j$. Hence $b_i-ke$ lies in $B$ for all $k\geq 0$ and as $b_i\notin B'$, we conclude that $i=z$, the stable index of $B$. Observe that $b_0-e-ke$ is a bead of $B'$ for all $k\geq 0$ except for when $b_0-e-ke = b_z$ (in which case $k\geq 1$). Indeed if $b_z<b_0-e-ke<b_z+ae$, then $b_0-e-ke$ is a non-leading bead of $B$ and otherwise $b_0-ke$ is a leading bead; see \Cref{fig:admissible}. We get a contradiction by showing that $b_z$ is a non-admissible empty space of $B'\cup C'$. To do so it is sufficient to check that $\hat{b}=b_0-e$ (which lies in $B'$) which will in turn, by \Cref{cor:wall}(ii), be the largest element of $B'$ (also lying in $C'$). Since $e$ divides $b_i-b_0=b_z-b_0$, by the observation that $z=i\geq 1$, \Cref{le:wall I}(i) and \Cref{le:auxiliary}(ii), we have $b_z-b_0=(z+1)e$ which shows that $b_0-e=\hat{b}$, as required.
\end{proof}

\begin{figure}[h]
	\centering
	\begin{tikzpicture}[x=0.5cm, y=0.5cm]
		\begin{pgfonlayer}{nodelayer}
			\node [style=none] (0) at (-1, 5) {};
			\node [style=none] (1) at (-1, -2) {};
			\node [style=none] (7) at (-1.25, 1.5) {};
			\node [style=none] (8) at (-0.75, 1.5) {};
			\node [style=none] (9) at (-0.65, -1.125) {$\scriptstyle b_0$};
			\node [style=none] (10) at (-0.675, 3.9) {$\scriptstyle b_i$};
			\node [style=none] (11) at (1, 1) {};
			\node [style=none] (12) at (2.5, 1) {};
			\node [style=none] (13) at (4.5, 5) {};
			\node [style=none] (14) at (4.5, -2) {};
			\node [style=none] (20) at (4.25, 3.5) {};
			\node [style=none] (21) at (4.75, 3.5) {};
			\node [style=none] (22) at (4.25, -1.5) {};
			\node [style=none] (23) at (4.75, -1.5) {};
			\node [style=none] (24) at (4.725, -0.025) {$\scriptstyle \hat{b}$};
			\node [style=Empty node] (25) at (-1, 4.5) {};
			\node [style=Empty node] (26) at (4.5, 4.5) {};
			\node [style=Empty node] (27) at (-1, 3.5) {};
			\node [style=Empty node] (28) at (-1, 2.5) {};
			\node [style=Empty node] (29) at (-1, 0.5) {};
			\node [style=Empty node] (30) at (-1, -0.5) {};
			\node [style=Empty node] (31) at (-1, -1.5) {};
			\node [style=Empty node] (32) at (4.5, -0.5) {};
			\node [style=Empty node] (33) at (4.5, 0.5) {};
			\node [style=Empty node] (34) at (4.5, 1.5) {};
			\node [style=Empty node] (35) at (4.5, 2.5) {};
			\node [style=none] (36) at (-0.175, 1.75) {$\scriptstyle b_i+ae$};
			\node [style=none] (37) at (4.825, 3.75) {$\scriptstyle b_i$};
			\node [style=none] (38) at (-1, 5.5) {$B$};
			\node [style=none] (39) at (4.5, 5.5) {$B'$};
			\node [style=none] (40) at (-10.925, 5) {};
			\node [style=none] (41) at (-10.925, -2) {};
			\node [style=none] (45) at (-10.575, 3.875) {$\scriptstyle b_z$};
			\node [style=none] (46) at (-8.925, 1) {};
			\node [style=none] (47) at (-7.425, 1) {};
			\node [style=none] (48) at (-5.425, 5) {};
			\node [style=none] (49) at (-5.425, -2) {};
			\node [style=none] (50) at (-5.675, 3.5) {};
			\node [style=none] (51) at (-5.175, 3.5) {};
			\node [style=Empty node] (55) at (-10.925, 4.5) {};
			\node [style=Empty node] (56) at (-5.425, 4.5) {};
			\node [style=Empty node] (57) at (-10.925, 3.5) {};
			\node [style=Empty node] (58) at (-10.925, 2.5) {};
			\node [style=Empty node] (59) at (-10.925, 0.5) {};
			\node [style=Empty node] (60) at (-10.925, -0.5) {};
			\node [style=Empty node] (62) at (-5.425, -0.5) {};
			\node [style=Empty node] (63) at (-5.425, 0.5) {};
			\node [style=Empty node] (64) at (-5.425, 1.5) {};
			\node [style=Empty node] (65) at (-5.425, 2.5) {};
			\node [style=none] (67) at (-5.1, 3.775) {$\scriptstyle b_z$};
			\node [style=none] (68) at (-10.925, 5.5) {$B$};
			\node [style=none] (69) at (-5.425, 5.5) {$B'$};
			\node [style=Empty node] (70) at (-10.925, 1.5) {};
			\node [style=none] (72) at (-5.2, -0.025) {$\scriptstyle \hat{b}$};
			\node [style=none] (73) at (-11.175, -1.5) {};
			\node [style=none] (74) at (-10.675, -1.5) {};
			\node [style=none] (77) at (-5.675, -1.5) {};
			\node [style=none] (78) at (-5.175, -1.5) {};
		\end{pgfonlayer}
		\begin{pgfonlayer}{edgelayer}
			\draw (0.center) to (1.center);
			\draw (7.center) to (8.center);
			\draw [style=Move it] (11.center) to (12.center);
			\draw (13.center) to (14.center);
			\draw (20.center) to (21.center);
			\draw (22.center) to (23.center);
			\draw (40.center) to (41.center);
			\draw [style=Move it] (46.center) to (47.center);
			\draw (48.center) to (49.center);
			\draw (50.center) to (51.center);
			\draw (73.center) to (74.center);
			\draw (77.center) to (78.center);
		\end{pgfonlayer}
	\end{tikzpicture}
	\caption{The left two runners are the runners of $b_z$ in $B$ and $B'$ in \Cref{le:admissible}(iii). Since $z$ is the stable index, the beads below $b_z$ in $B$ are non-leading. The right two runners are the runners of $b_i$ (which is also $b_z$) in $B$ and $B'$ in the proof of \Cref{le:admissible}(iv). Observe that $b_i=b_z$ implies that all the displayed beads of $B$ apart from the one just below $b_i$ are leading. Thus $B'$ looks as on the right.}
	\label{fig:admissible}
\end{figure}
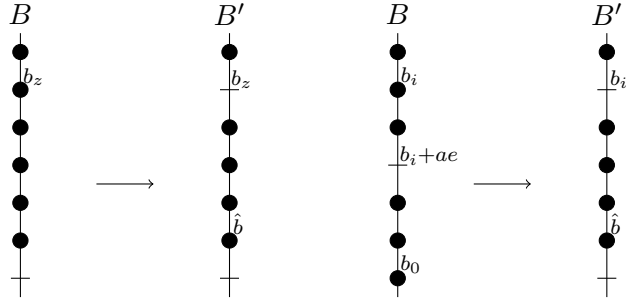

In \Cref{pr:balanced} we showed that the set $S$ in the Abacus Mullineux Algorithm stays \dbal as long as the initial $\beta$-set $S$ is. To prove \Cref{th:Mullineuxbalanced}, we also need to show that the set $T$ keeps some suitable properties throughout the algorithm, which forces it to be a \dsbal set at the end, provided that the initial $\beta$-set $S$ is a \dbal set. The right choices of these properties are parts (i) and (iii) of the following definition. 

\begin{definition}\label{de:balanced pair}
	A pair $(B,C)$ of a $\beta$-set and a subset of integers is called a \textit{$d$-combined pair} if the following hold.
	\begin{enumerate}[label=\textnormal{(\roman*)}]
		\item $\max B \leq \min C$ (or $C = \varnothing$).
		\item $B$ is $d$-balanced.
		\item If $c\in C$ and $c-ae$ (with $a>0$), is an admissible empty space of $B\cup C$, then $d\mid \emp_{(B,C)}(c-ae,c)$.
	\end{enumerate}
\end{definition}

Informally, the conditions (ii) and (iii) (assuming that (i) holds) say that $B\cup C$ behaves as a \dbal set if in \Cref{de:balanced} we use a bead from $B$ and as a \dsbal set if in \Cref{de:balanced} we use a bead from $C$. A helpful consequence of (i) and \Cref{le:new bead}(ii) are the inequalities $\hat{b}\leq b_0\leq \min C$ (if $C\neq \varnothing$). All the pairs $(S,T)$ from \Cref{fig:MA} are $d$-combined pairs. A more involved and detailed example is in \Cref{fig:d-pair}.

\begin{figure}[h]
	\centering
	\begin{tikzpicture}[x=0.5cm, y=0.5cm]
		\begin{pgfonlayer}{nodelayer}
			\node [style=none] (0) at (-3, 4.5) {};
			\node [style=none] (1) at (-3, 0.5) {};
			\node [style=none] (2) at (-2, 0.5) {};
			\node [style=none] (3) at (-1, 0.5) {};
			\node [style=none] (4) at (0, 0.5) {};
			\node [style=none] (5) at (1, 0.5) {};
			\node [style=none] (6) at (-2, 4.5) {};
			\node [style=none] (7) at (-1, 4.5) {};
			\node [style=none] (8) at (0, 4.5) {};
			\node [style=none] (9) at (1, 4.5) {};
			\node [style=Empty node] (10) at (-3, 4) {};
			\node [style=Empty node] (11) at (-2, 4) {};
			\node [style=Empty node] (12) at (-1, 4) {};
			\node [style=Empty node] (13) at (0, 4) {};
			\node [style=Empty node] (14) at (1, 4) {};
			\node [style=Empty node] (15) at (-2, 3) {};
			\node [style=Empty node] (16) at (1, 3) {};
			\node [style=empty square] (18) at (-1, 1) {};
			\node [style=empty square] (19) at (-2, 1) {};
			\node [style=empty square] (20) at (-1, 2) {};
			\node [style=Empty node] (21) at (-1, 2) {};
			\node [style=none] (22) at (-3.25, 3) {};
			\node [style=none] (23) at (-2.75, 3) {};
			\node [style=none] (24) at (-3.25, 2) {};
			\node [style=none] (25) at (-2.75, 2) {};
			\node [style=none] (26) at (-3.25, 1) {};
			\node [style=none] (27) at (-2.75, 1) {};
			\node [style=none] (28) at (-2.25, 2) {};
			\node [style=none] (29) at (-1.75, 2) {};
			\node [style=none] (30) at (-1.25, 3) {};
			\node [style=none] (31) at (-0.75, 3) {};
			\node [style=none] (32) at (-0.25, 3) {};
			\node [style=none] (33) at (0.25, 3) {};
			\node [style=none] (34) at (-0.25, 2) {};
			\node [style=none] (35) at (0.25, 2) {};
			\node [style=none] (36) at (-0.25, 1) {};
			\node [style=none] (37) at (0.25, 1) {};
			\node [style=none] (38) at (0.75, 2) {};
			\node [style=none] (39) at (1.25, 2) {};
			\node [style=none] (40) at (0.75, 1) {};
			\node [style=none] (41) at (1.25, 1) {};
			\node [style=none] (42) at (-1.8, 1.45) {$\scriptstyle c$};
			\node [style=none] (44) at (-1.5, 2) {};
			\node [style=none] (45) at (-0.8, 3.3) {$\scriptstyle f$};
			\node [style=none] (46) at (-0.825, 2.55) {$\scriptstyle b$};
		\end{pgfonlayer}
		\begin{pgfonlayer}{edgelayer}
			\draw (0.center) to (1.center);
			\draw (8.center) to (4.center);
			\draw (9.center) to (5.center);
			\draw (6.center) to (2.center);
			\draw (7.center) to (3.center);
			\draw (22.center) to (23.center);
			\draw (24.center) to (25.center);
			\draw (26.center) to (27.center);
			\draw (28.center) to (29.center);
			\draw (30.center) to (31.center);
			\draw (32.center) to (33.center);
			\draw (34.center) to (35.center);
			\draw (38.center) to (39.center);
			\draw (40.center) to (41.center);
			\draw (36.center) to (37.center);
		\end{pgfonlayer}
	\end{tikzpicture}
	\caption{The diagram displays beads of $B$ as black dots and all three beads of $C$ as squares. All the runners contain beads of $B$ above their displayed parts and no beads below their displayed parts. The pair $(B,C)$ is a $d$-combined pair for $e=5$ and $d=3$, since the condition (i) from \Cref{de:balanced pair} is clearly true, for (ii) we only need to compute that $\emp_B(b-e+1, b) = 3$ and for (iii) we only need to see that $\emp_{(B,C)}(c-e,c) = \emp_B(c-e,c) - \bd_C(c-e,c) = 5-2 = 3$, since the empty space $f$ is non-admissible --- $b$ is the maximal bead of $B$, it lies in $C$ and $f$ is the only empty space of $B$ of the form $b-ke$ with $k\geq 0$.}
	\label{fig:d-pair}
\end{figure}
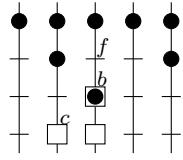

We now prove an analogous result to \Cref{le:wall I}(ii) but we replace the bead $\hat{b}$ by a bead $c$ of $C$. Again it implies that if $c$ is placed on runner $0$, then except for at most one row, each row (less or equal to the row of $b_0$) contains a leading bead.

\begin{lemma}\label{le:wall II}
	Let $(B,C)$ be a $d$-combined pair. Suppose that $c\in C$ and $r$ is the unique positive integer such that $c-re +1 \leq b_0\leq c-(r-1)e$. Then for any non-negative integer $i$ we have $c-(i+r)e\leq b_i \leq c- (i+r-1)e$. 
\end{lemma}

\begin{proof}
	 The beads and empty spaces used are sketched in \Cref{fig:wallII}. Suppose that the statement fails to hold for a $d$-combined pair $(B,C)$ and pick the least $i$ such that $c-(i+r)e\leq b_i \leq c- (i+r-1)e$ is false (clearly $i\geq 1$). By \Cref{le:new bead}(i), $b_i \leq b_0 - ie\leq c - (i+r-1)e$ and hence we conclude that $b_i + 1\leq c - (i+r)e$. By the minimality of $i$, $c-(i+r)e\leq b_{i-1}-e$ and thus $i\leq z$, the stable index of $B$ and by the sandwich argument there is no bead of $B$ between $b_i+1$ and $c-(i+r)e$. There is also no bead of $C$ as $c-(i+r)e< b_0\leq \min C$.
	 
	 By \Cref{le:wall I}(ii) we have $\hat{b} \leq b_i+ie < c-re\leq b_0-1$ and thus \Cref{le:wall I}(iii) and the sandwich argument implies that there is no bead of $B$ between $c-re$ and $b_0-1$. Again, there is no bead of $C$ either. Moreover, by \Cref{le:admissible}(i), the empty spaces $c-(i+r)e$ and $c-re$ are admissible. Hence $d$ divides the difference
	 \begin{align*}
	 	D &= \emp_{(B,C)}(c-(i+r)e, c) - \emp_{(B,C)}(c-re, c)\\
	 	&= \emp_{(B,C)}(c-(i+r)e, c- re-1)\\
	 	&= ie - \bd_B(c-(i+r)e, c- re-1)\\
	 	&= ie - \bd_B(b_i + 1, b_0)+1\\
	 	&\equiv 1\, (\textnormal{mod } d),
	 \end{align*}
	 where the final congruence follows from \Cref{le:bead count} with $0\leq i\leq z$. This is the desired contradiction.
\end{proof}

\begin{figure}[h]
	\centering
	\begin{tikzpicture}[x=0.5cm, y=0.5cm]
		\begin{pgfonlayer}{nodelayer}
			\node [style=none] (0) at (-3, 4.5) {};
			\node [style=none] (1) at (-3, -1.5) {};
			\node [style=none] (2) at (-2, -1.5) {};
			\node [style=none] (3) at (-1, -1.5) {};
			\node [style=none] (4) at (0, -1.5) {};
			\node [style=none] (5) at (1, -1.5) {};
			\node [style=none] (6) at (-2, 4.5) {};
			\node [style=none] (7) at (-1, 4.5) {};
			\node [style=none] (8) at (0, 4.5) {};
			\node [style=none] (9) at (1, 4.5) {};
			\node [style=Empty node] (10) at (-3, 4) {};
			\node [style=Empty node] (11) at (-2, 4) {};
			\node [style=Empty node] (12) at (-1, 4) {};
			\node [style=Empty node] (13) at (0, 4) {};
			\node [style=Empty node] (14) at (1, 4) {};
			\node [style=Empty node] (15) at (-2, 2) {};
			\node [style=Empty node] (16) at (-3, 3) {};
			\node [style=empty square] (18) at (-1, -1) {};
			\node [style=empty square] (19) at (-3, -1) {};
			\node [style=empty square] (20) at (0, -1) {};
			\node [style=Empty node] (21) at (1, 0) {};
			\node [style=none] (22) at (-1.25, 1) {};
			\node [style=none] (23) at (-0.75, 1) {};
			\node [style=none] (24) at (-2.75, 2) {};
			\node [style=none] (25) at (-3.25, 2) {};
			\node [style=none] (26) at (-3.25, 0) {};
			\node [style=none] (27) at (-2.75, 0) {};
			\node [style=none] (28) at (-2.25, 3) {};
			\node [style=none] (29) at (-1.75, 3) {};
			\node [style=none] (30) at (-1.25, 3) {};
			\node [style=none] (31) at (-0.75, 3) {};
			\node [style=none] (32) at (-0.25, 3) {};
			\node [style=none] (33) at (0.25, 3) {};
			\node [style=none] (34) at (-0.25, 2) {};
			\node [style=none] (35) at (0.25, 2) {};
			\node [style=none] (36) at (0.75, -1) {};
			\node [style=none] (37) at (1.25, -1) {};
			\node [style=none] (38) at (0.75, 2) {};
			\node [style=none] (39) at (1.25, 2) {};
			\node [style=none] (40) at (0.75, 1) {};
			\node [style=none] (41) at (1.25, 1) {};
			\node [style=none] (42) at (0.175, -0.525) {$\scriptstyle c$};
			\node [style=Empty node] (43) at (0, 1) {};
			\node [style=none] (44) at (-1.25, 2) {};
			\node [style=none] (45) at (-0.75, 2) {};
			\node [style=none] (46) at (-2.25, 1) {};
			\node [style=none] (47) at (-1.75, 1) {};
			\node [style=none] (50) at (-2.25, 0) {};
			\node [style=none] (51) at (-1.75, 0) {};
			\node [style=none] (52) at (-1.25, 0) {};
			\node [style=none] (53) at (-0.75, 0) {};
			\node [style=none] (54) at (-0.25, 0) {};
			\node [style=none] (55) at (0.25, 0) {};
			\node [style=none] (56) at (-2.25, -1) {};
			\node [style=none] (57) at (-1.75, -1) {};
			\node [style=Empty node] (58) at (-3, 1) {};
			\node [style=Empty node] (59) at (1, 3) {};
			\node [style=none] (60) at (-1.7, 2.35) {$\scriptstyle b_i$};
			\node [style=none] (61) at (-2.8, 0.35) {$\scriptstyle \hat{b}$};
			\node [style=none] (62) at (0.225, 0.3) {$\scriptstyle f$};
			\node [style=none] (63) at (0.225, 2.225) {$\scriptstyle g$};
		\end{pgfonlayer}
		\begin{pgfonlayer}{edgelayer}
			\draw (0.center) to (1.center);
			\draw (8.center) to (4.center);
			\draw (9.center) to (5.center);
			\draw (6.center) to (2.center);
			\draw (7.center) to (3.center);
			\draw (22.center) to (23.center);
			\draw (24.center) to (25.center);
			\draw (26.center) to (27.center);
			\draw (28.center) to (29.center);
			\draw (30.center) to (31.center);
			\draw (32.center) to (33.center);
			\draw (34.center) to (35.center);
			\draw (38.center) to (39.center);
			\draw (40.center) to (41.center);
			\draw (36.center) to (37.center);
			\draw (56.center) to (57.center);
			\draw (50.center) to (51.center);
			\draw (46.center) to (47.center);
			\draw (44.center) to (45.center);
			\draw (52.center) to (53.center);
			\draw (54.center) to (55.center);
		\end{pgfonlayer}
	\end{tikzpicture}
	\caption{The beads and empty spaces from the proof of \Cref{le:wall II}. The beads of $B$ are black dots, the beads of $C$ are squares, $f=c-re$ and $g=c - (r+i)e$. Note that in this example $r=1$.}
	\label{fig:wallII}
\end{figure}

Given the nature of the Abacus Mullineux Algorithm, the proof of \Cref{th:Mullineuxbalanced} will be done inductively. The first ingredient of the inductive step is that the transformation of $(S,T)$ in the Abacus Mullineux Algorithm preserves $d$-combined pairs. The proof is similar to the proof of \Cref{pr:balanced}. We again write $b_0,b_1,\dots$ for the leading beads of $B$. This time we frequently compare $\emp_{(B,C)}(u,v)$ and $\emp_{(B',C')}(u,v)$ where $B' = \Ms_e(B)$ and $C'=C\cup \{ \hat{b} \}$ and $u+e\leq v$.

Recall that there is at most one leading bead $b_i$ of $B$ which moves out of the interval, that is, which satisfy $u\leq b_i < u+e$. Similarly, there is at most one leading bead $b_i$ of $B$ which moves into the interval, that is, which satisfy $v<b_i\leq v+e$. From the definition of $C'$, it is clear that only integer $c$ such that $u\leq c\leq v$ and $c$ lies precisely in one of $C$ and $C'$ is $c=\hat{b}$. If $u\leq \hat{b}\leq v$, we say that \textit{$\hat{b}$ moves into the interval}. With this extended definition, it still remains true that  $\emp_{(B',C')}(u,v)$ is obtained from $\emp_{(B,C)}(u,v)$ by adding one for each leading bead of $B$ which moves out of the interval and subtracting one for each leading bead of $B$ or $\hat{b}\in C'$ which moves into the interval.

Throughout the proof, it is helpful to keep in mind that for a $d$-combined pair $(B,C)$ any integer less than $b_0$ is an empty space of $C$. 

\begin{proposition}\label{pr:shift}
	Let $(B,C)$ be a $d$-combined pair such that $\hat{b}\notin C$. If $B' = \Ms_e(B)$ and $C'=C\cup \{ \hat{b} \}$, then $(B', C') $ is a $d$-combined pair.
\end{proposition}

\begin{proof}
	We know that $\max B'\leq \hat{b}\leq b_0\leq \min C$ (where the last inequality is omitted if $C=\varnothing$) by \Cref{cor:wall}(ii) and \Cref{le:new bead}(ii). Hence $\max B' \leq \min C'$. So \Cref{de:balanced pair}(i) holds for $(B',C')$. \Cref{pr:balanced} shows \Cref{de:balanced pair}(ii) and hence it remains to show (iii). Let $c$ be a bead of $C'$ and $c-ae$ (with $a>0$) be an admissible empty space of $B'\cup C'$. We need to show that $d\mid \emp_{(B',C')}(c-ae,c)$. The chain of inequalities in the first sentence shows that the beads of $C$ and $C'$ which are greater than $b_0$ are the same. Thus, if $c-ae > b_0$ (and in turn $c-ae >\max B'$), then
	\[
	\emp_{(B',C')}(c-ae,c) = \emp_{C'}(c-ae,c) = \emp_{C}(c-ae,c)= \emp_{(B,C)}(c-ae,c),
	\]
	which is divisible by $d$ as $(B,C)$ is a $d$-combined pair (and $c>b_0\geq \hat{b}$ lies in $C$ and $c-ae$ is an empty space of $B\cup C$ which is admissible --- it cannot be of the form $b_0-ke$ with $k\geq 0$). So we assume that $c-ae\leq b_0$. We distinguish four cases which are displayed in \Cref{fig:shift}.
	
	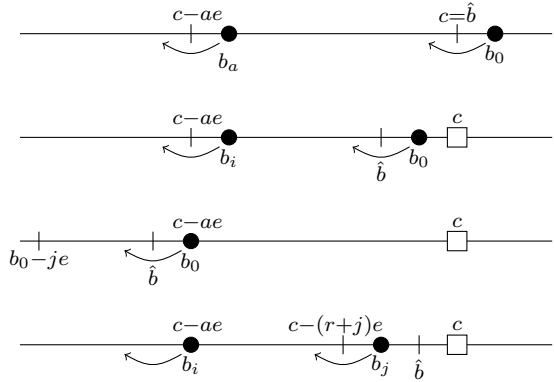
\begin{figure}[h]
		\centering
		\begin{tikzpicture}[x=0.5cm, y=0.5cm]
			\begin{pgfonlayer}{nodelayer}
				\node [style=none] (0) at (-8, 3) {};
				\node [style=none] (1) at (6, 3) {};
				\node [style=Empty node] (2) at (4.5, 3) {};
				\node [style=Empty node] (4) at (-2.5, 3) {};
				\node [style=none] (5) at (4.5, 2.45) {$\scriptstyle b_0$};
				\node [style=none] (6) at (4.25, 2.75) {};
				\node [style=none] (7) at (2.75, 2.75) {};
				\node [style=none] (8) at (-2.5, 2.45) {$\scriptstyle b_a$};
				\node [style=none] (9) at (3.5, 3.25) {};
				\node [style=none] (10) at (3.5, 2.75) {};
				\node [style=none] (11) at (3.5, 3.575) {$\scriptstyle c=\hat{b}$};
				\node [style=none] (13) at (-2.75, 2.75) {};
				\node [style=none] (14) at (-4.25, 2.75) {};
				\node [style=none] (15) at (-3.5, 3.25) {};
				\node [style=none] (16) at (-3.5, 2.75) {};
				\node [style=none] (17) at (-3.325, 3.5) {$\scriptstyle c-ae$};
				\node [style=none] (20) at (-8, 0.25) {};
				\node [style=none] (21) at (6, 0.25) {};
				\node [style=Empty node] (23) at (2.5, 0.25) {};
				\node [style=Empty node] (24) at (-2.5, 0.25) {};
				\node [style=none] (28) at (-2.5, -0.3) {$\scriptstyle b_i$};
				\node [style=none] (31) at (1.5, -0.65) {$\scriptstyle \hat{b}$};
				\node [style=none] (32) at (2.5, -0.3) {$\scriptstyle b_0$};
				\node [style=none] (33) at (-2.75, 0) {};
				\node [style=none] (34) at (-4.25, 0) {};
				\node [style=none] (37) at (-3.325, 0.75) {$\scriptstyle c-ae$};
				\node [style=none] (38) at (0.75, 0) {};
				\node [style=none] (39) at (2.25, 0) {};
				\node [style=none] (40) at (-8, -2.5) {};
				\node [style=none] (41) at (6, -2.5) {};
				\node [style=Empty node] (44) at (-3.5, -2.5) {};
				\node [style=none] (48) at (-3.5, -3.05) {$\scriptstyle b_0$};
				\node [style=none] (51) at (3.5, -2) {$\scriptstyle \scriptstyle c$};
				\node [style=none] (53) at (-3.75, -2.75) {};
				\node [style=none] (54) at (-5.25, -2.75) {};
				\node [style=none] (57) at (-3.325, -2) {$\scriptstyle c-ae$};
				\node [style=none] (60) at (-8, -5.25) {};
				\node [style=none] (61) at (6, -5.25) {};
				\node [style=Empty node] (62) at (1.5, -5.25) {};
				\node [style=Empty node] (64) at (-3.5, -5.25) {};
				\node [style=none] (65) at (1.5, -5.8) {$\scriptstyle b_j$};
				\node [style=none] (66) at (1.25, -5.5) {};
				\node [style=none] (67) at (-0.25, -5.5) {};
				\node [style=none] (68) at (-3.5, -5.8) {$\scriptstyle b_i$};
				\node [style=none] (71) at (3.5, -4.75) {$\scriptstyle c$};
				\node [style=none] (73) at (-3.75, -5.5) {};
				\node [style=none] (74) at (-5.25, -5.5) {};
				\node [style=none] (77) at (-3.325, -4.75) {$\scriptstyle c-ae$};
				\node [style=empty square] (80) at (3.5, 0.25) {};
				\node [style=none] (81) at (3.5, 0.75) {$\scriptstyle c$};
				\node [style=none] (82) at (-3.5, 0.5) {};
				\node [style=none] (83) at (-3.5, 0) {};
				\node [style=none] (84) at (1.5, 0.5) {};
				\node [style=none] (85) at (1.5, 0) {};
				\node [style=empty square] (86) at (3.5, -2.5) {};
				\node [style=none] (87) at (-4.5, -2.25) {};
				\node [style=none] (88) at (-4.5, -2.75) {};
				\node [style=none] (89) at (-4.5, -3.4) {$\scriptstyle \hat{b}$};
				\node [style=none] (92) at (-7.5, -3) {$\scriptstyle b_0 - je$};
				\node [style=none] (93) at (-7.5, -2.25) {};
				\node [style=none] (94) at (-7.5, -2.75) {};
				\node [style=none] (95) at (0.5, -5) {};
				\node [style=none] (96) at (0.5, -5.5) {};
				\node [style=none] (97) at (0.3, -4.75) {$\scriptstyle c-(r+j)e$};
				\node [style=empty square] (98) at (3.5, -5.25) {};
				\node [style=none] (99) at (2.5, -5) {};
				\node [style=none] (100) at (2.5, -5.5) {};
				\node [style=none] (101) at (2.5, -5.9) {$\scriptstyle \hat{b}$};
			\end{pgfonlayer}
			\begin{pgfonlayer}{edgelayer}
				\draw (0.center) to (1.center);
				\draw [style=Move it, bend left, looseness=1.25] (6.center) to (7.center);
				\draw (9.center) to (10.center);
				\draw [style=Move it, bend left, looseness=1.25] (13.center) to (14.center);
				\draw (15.center) to (16.center);
				\draw (20.center) to (21.center);
				\draw [style=Move it, bend left, looseness=1.25] (33.center) to (34.center);
				\draw [style=Move it, bend left, looseness=1.25] (39.center) to (38.center);
				\draw (40.center) to (41.center);
				\draw [style=Move it, bend left, looseness=1.25] (53.center) to (54.center);
				\draw (60.center) to (61.center);
				\draw [style=Move it, bend left, looseness=1.25] (66.center) to (67.center);
				\draw [style=Move it, bend left, looseness=1.25] (73.center) to (74.center);
				\draw (82.center) to (83.center);
				\draw (84.center) to (85.center);
				\draw (87.center) to (88.center);
				\draw (93.center) to (94.center);
				\draw (95.center) to (96.center);
				\draw (99.center) to (100.center);
			\end{pgfonlayer}
		\end{tikzpicture}
		\caption{The four cases from the proof of \Cref{pr:shift} --- the first case is $c=\hat{b}$, while in the remaining three cases $c> \hat{b}$ and $c-ae\notin B$ (the second case) or $c-ae = b_0$ (the third case) or $c-ae = b_i$ with $i>0$ (the fourth case). Note that in the first case $c-ae$ and $b_a$ may coincide.}
		\label{fig:shift}
	\end{figure}
	
	We firstly suppose that $c=\hat{b}$. By \Cref{le:wall I}(ii), $\hat{b}-ae \leq b_a < \hat{b} - (a-1)e$, where we used a sharp inequality on the right-hand side as $\hat{b}-ae$ is an empty space in $B'$. For the same reason, $b_{a-1}\neq \hat{b}-(a-1)e$; thus $a\leq z$, the stable index of $B$ (\Cref{le:new bead}(ii)). Using \Cref{le:wall I}(ii) once more, we conclude that $\hat{b} < b_0\leq \hat{b} + e$, where $\hat{b} = b_0$ is excluded using \Cref{le:new bead}(ii) and $z\geq a >0$. Thus $\emp_{(B',C')}(\hat{b}-ae, \hat{b}) = \emp_{(B,C)}(\hat{b}-ae, \hat{b}) - 1$ as $b_a$ moves out of the interval and $b_0$ and $\hat{b}$ moves into the interval; see the first diagram of \Cref{fig:shift}. Since $\hat{b}< b_0\leq \min C$ (or $C=\varnothing$) we can further write
	\begin{align*}
	\emp_{(B',C')}(\hat{b}-ae, \hat{b}) &= \emp_{(B,C)}(\hat{b}-ae, \hat{b}) - 1\\
	&= \emp_B(\hat{b}-ae, \hat{b}) - 1\\ &= ae - \bd_B(\hat{b}-ae, \hat{b}).
	\end{align*}
	As $\hat{b}-ae\notin B'$, there is not a non-leading bead $\hat{b}-ae$ of $B$. If $\hat{b}-ae$ is a leading bead of $B$, then, from the inequalities $\hat{b}-ae \leq b_a < \hat{b} - (a-1)e$, it is $b_a$, and so we can replace $\hat{b}-ae$ by $b_a$ in the above $\bd_B$. We can do the same even if $\hat{b}-ae$ is an empty space of $B$, using \Cref{le:wall I}(iii). The same lemma allows us to replace $\hat{b}$ by $b_0 - 1$. Thus the desired emptiness is $ae - \bd_B(b_a, b_0-1) = ae - \bd_B(b_a+1, b_0)$, which is divisible by $d$ by \Cref{le:bead count} (as $B$ is a \dbal set and $a\leq z$).
	
	We now suppose that $c\neq \hat{b}$ (and so $c\in C$ and $c\geq b_0\geq \hat{b}$). We let $r$ be as in \Cref{le:wall II}, that is, $r$ is a positive integer such that $c-re +1 \leq b_0\leq c-(r-1)e$, and let $i=a-r$ which is non-negative by the assumption that $c-ae < b_0$. So $a=i+r$. We distinguish two cases. Firstly assume that $c-ae\notin B$. Note that $c-ae$ is an admissible empty space of $B\cup C$ by \Cref{le:admissible}(ii) as $c-ae+e$ is not a leading bead of $B$ since $c-ae\notin B'$. Using \Cref{le:wall II}, we find that $c-ae< b_i < c-(a-1)e$. By \Cref{le:wall II} applied to $z$, the stable index of $B$, it also holds that $c-ae \leq c-re\leq b_z + ze= \hat{b}$. Hence $\emp_{(B',C')}(c-ae,c) = \emp_{(B,C)}(c-ae,c)$ as $\hat{b}$ moves into the interval and $b_i$ moves out of the interval; see the second diagram of \Cref{fig:shift}. The latter combined emptiness is divisible by $d$ as $(B,C)$ is a $d$-combined pair.
	
	We can further assume that $c> \hat{b}$ and $c-ae\in B$. As $c-ae$ is not in $B'$, it must be a leading bead of $B$. By \Cref{le:wall II}, it is either $b_i$ or $b_{i+1}$. We eliminate the latter option. If $ b_{i+1} = c- ae$, then as $b_i\geq b_{i+1} + e = c- (a-1)e$, but also, by \Cref{le:wall II}, $b_i\leq c - (a-1)e$, we conclude that $b_i = c-(a-1)e$. But then $c-ae\in B'$, a contradiction. So $c-ae = b_i$. We need to distinguish the final two cases: $i=0$ and $i>0$.
	
	If $c-ae = b_0$, then $\hat{b} < b_0$, and thus $0$ is not the stable index of $B$ and there is the least $j\geq 1$ such that $b_j\neq b_0-je$. We have
	\begin{align*}
	\emp_{(B',C')}(c-ae,c) &= \emp_{(B,C)}(b_0, c) + 1\\
	&= \emp_{(B,C)}(b_0-je,c) - \emp_{(B,C)}(b_0-je, b_0-1) + 1,
	\end{align*}
	as $b_0$ moves out of the interval; see the third diagram of \Cref{fig:shift}. By the minimality of $j$, we know that $b_0-je$ is an empty space in $B$ (and in $C$). In fact, it is an admissible empty space of $B\cup C$ (since $b_0=c-ae$ is not in $C'$ and thus it is not in $C$). Hence $\emp_{(B,C)}(b_0-je,c)$ is divisible by $d$. Moreover, $\emp_{(B,C)}(b_0-je, b_0-1)-1 = \emp_B(b_0-je, b_0-1)-1 = \emp_B(b_0-je+1, b_0)$ is divisible by $d$ as $B$ is a \dbal set, and thus $d$ divides $\emp_{(B',C')}(c-ae,c)$.
	 
	If $c-ae = b_i$ with $i>0$, we claim that there is $0\leq j < i$ such that $c-(r+j)e < b_j$ and there is no bead of $B$ between $c-(r+j)e$ and $b_j-1$. Note that the first condition holds true for all $0\leq j<i$ since, by \Cref{le:new bead}(i),  $b_j \geq  c - (r+j)e$, and the equality cannot be attained, as otherwise \Cref{le:auxiliary}(i) shows that $b_{i-1} = c-(a-1)e$ and contradicts that $c-ae\notin B'$. To prove the claim we suppose that the second condition fails to hold for all $0\leq j < i$ and we reach a contradiction by showing that $c-ae$ is a non-admissible empty space in $B'\cup C'$.
	 
	Using parts (ii) and (iii) of \Cref{le:wall I}, we have $\hat{b}-je \leq b_i +ie -je = c - (r+j)e\leq b_j-1$ and the only possible bead of $B$ between $\hat{b}-je$ and $b_j-1$ is $\hat{b}-je$. As, by our assumption, there is a bead of $B$ between $c-(r+j)e$ and $b_j-1$, we conclude that $\hat{b}-je = c - (r+j)e$ and this is a bead in $B$ for all $0\leq j < i$. The equality $\hat{b}-je = c - (r+j)e$ says that $\hat{b} = b_i +ie$ and by \Cref{le:new bead}(ii) we have $i\geq z$, and in fact $i=z$ ($b_i\notin B'$). We conclude that $c-ae=b_z$ is a non-admissible empty space in $B'\cup C'$ by \Cref{le:admissible}(iii), a contradiction.
	 
	Pick $j$ such that $0\leq j < i$ and $c-(r+j)e < b_j$ and there is no bead of $B$ between $c-(r+j)e$ and $b_j-1$. Since $b_i\leq \hat{b}$ by \Cref{cor:wall}(i), we have $\emp_{(B',C')}(c-ae,c) = \emp_{(B,C)}(c-ae,c)$ as $\hat{b}$ moves into the interval and $b_i$ moves out of the interval; see the fourth diagram of \Cref{fig:shift}. This can be expanded as
	\begin{align*}
	&\emp_{(B,C)}(b_i,c-(r+j)e-1) + \emp_{(B,C)}(c-(r+j)e, c)\\=\; &(i-j)e - \bd_B(b_i, c-(r+j)e-1) + \emp_{(B,C)}(c-(r+j)e, c)\\=\; &(i-j)e - \bd_B(b_i, b_j-1) + \emp_{(B,C)}(c-(r+j)e, c)\\=\; &(i-j)e - \bd_B(b_i+1, b_j) + \emp_{(B,C)}(c-(r+j)e, c)\\ \equiv\; &\emp_{(B,C)}(c-(r+j)e, c)\, (\textnormal{mod } d),
	\end{align*}
	where the final congruence follows from \Cref{le:bead count} as $i\leq z$ ($b_i$ is not in $B'$). We will be done if we show that $c-(r+j)e$ is an admissible empty space of $B\cup C$. By the choice of $j$, it is an empty space of $B\cup C$ and thus \Cref{le:admissible}(iv) applies to $i$ and `$b_i+ae$' equal to $c-(r+j)e$.
\end{proof}

Let $(B,C)$ be a pair of a $\beta$-set of a partition and a finite set of integers. We define the \textit{combined $d$-runner matrix} $\mathcal{R}_d(B,C)$ as a $(d-1)\times e$ matrix indexed by $1\leq x\leq d-1$ and $0\leq y\leq e-1$ where $\mathcal{R}_d(B,C)_{x,y}$ counts the number of beads of $B$ with $d$-emptiness $x$ on runner $y$ and beads $c$ of $C$ on runner $y$ such that $\emp_{(B,C)}(c)$ is congruent to $x$ modulo $d$.

\begin{example}\label{ex:combined runner}
	Let $e=5$, $d=3$, $B=\left\lbrace 7,4,1,-1,-2,\dots \right\rbrace $ and $C=\left\lbrace 12,11,7 \right\rbrace$ as on \Cref{fig:d-pair}. Then the beads of $B$ with non-zero emptinesses are $7,4$ and $1$ and $\emp_B(7)=5$, $\emp_B(4)=3$ and $\emp_B(1)=1$. As $\emp_{(B,C)}(12)=\emp_{(B,C)}(11)=7$ and $\emp_{(B,C)}(7)=4$, the combined $d$-runner matrix $\mathcal{R}_d(B,C)$ equals $\big(\begin{smallmatrix}
		0 & 2 & 2 & 0 & 0\\ 
		0 & 0 & 1 & 0 & 0
	\end{smallmatrix}\big)$. 
\end{example}

Our final ingredient to prove our first main theorem follows.

\begin{proposition}\label{pr:invariance}
	Let $(B,C)$ be a $d$-combined pair such that $\hat{b}\notin C$ and let $B'= \Ms_e(B)$ and $C' = C\cup \{ \hat{b} \} $. Then $\mathcal{R}_d(B,C) = \mathcal{R}_d(B',C')$.
\end{proposition}

\begin{proof}
	Let $z$ be the stable index of $B$ and $0=i_0<i_1<\dots <i_t=z$ be indices of the leading beads $b_i$ of $B$ such that $b_i\notin B'$, as in \Cref{le:Me observation}(iv) (see \Cref{fig:invariance}). Then \Cref{le:Me observation}(iv) tells us that
	\begin{equation}\label{eq:B'}
		B' = \left( B \setminus \left\lbrace b_{i_j} : 0\leq j \leq t \right\rbrace\right)  \cup \left\lbrace b_{i_j-1} - e : 1\leq j \leq t \right\rbrace ,
	\end{equation}
	where all the elements in the \textit{added set} do not lie in $B$ while the elements in the \textit{removed set} all lie in $B$. Let $\phi$ be a bijection between the disjoint unions $B\coprod C$ and $B'\coprod C'$ defined by $\phi(b_{i_j}) = b_{i_{j+1}-1} - e$ for $0\leq j<t$, $\phi(b_z) = \hat{b}$ and $\phi$ acts as the identity on the beads in $B\cap B'$ and $C$. By \Cref{le:Me observation}(v), $b_{i_j}$ and $b_{i_{j+1}-1}$ lie on the same runner for any $0\leq j<t$. It follows that $\phi$ preserves the runners of beads. To complete the proof we need to show that $\phi$ also preserves appropriate (combined) emptinesses of beads modulo $d$; that is, $\emp(u) \equiv \emp(\phi(u))\, (\textnormal{mod } d)$, where, on the left-hand side $\emp = \emp_B$ if $u\in B$ and $\emp = \emp_{(B,C)}$ if $u\in C$, and analogously for the right-hand side.
	
	We start with $b\in B\cap B'$. From the definition of leading beads and \Cref{le:Me observation}(ii) we have a chain of inequalities
	\begin{equation}\label{eq:twochain}
	b_{i_t} < \dots < b_{i_2 - 1}-e < b_{i_1} < b_{i_1 - 1} -e < b_{i_0}
	\end{equation}
	such that there is no bead between $b_{i_j}+1$ and $b_{i_j-1}-e$ (and no bead greater than $b_{i_0}$); thus $b_{i_{j+1}-1}-e<b<b_{i_j}$ for some $0\leq j\leq t$ (where the left inequality is discarded if $j=t$). Using \eqref{eq:B'} and \eqref{eq:twochain}, there are $t-j$ elements less or equal to $b$ both in the added set and the removed set; so we have $\emp_B(b) = \emp_{B'}(b)$. If $c\in C$, then $\hat{b}\leq b_0\leq c$ and so the elements of the added and removed sets are all less or equal to $c$. Taking the addition of $\hat{b}$ to $C$ to obtain $C'$ into consideration, we have $\emp_{(B,C)}(c) = \emp_{(B',C')}(c)$.
	
	Next, we show that for $0\leq j < t$, $\emp_{B'}(b_{i_{j+1}-1} - e)$ and $\emp_B(b_{i_j})$ are congruent modulo $d$. The first term equals $\emp_B(b_{i_{j+1}-1} - e)$ --- there are $t-j$ elements less or equal to $b_{i_{j+1}-1} - e$ in both the added and the removed set in \eqref{eq:B'}. By \Cref{le:Me observation}(v), we have $b_{i_{j+1}-1} - e = b_{i_j} - (i_{j+1} - i_j)e$ and hence, as $B$ is $d$-balanced, $d$ divides $\emp_B(b_{i_j} - (i_{j+1} - i_j)e + 1, b_{i_j}) = \emp_B(b_{i_j}) - \emp_B(b_{i_{j+1}-1} - e)$, as claimed.
	
	It remains to show that $\emp_{(B',C')}(\hat{b})$ is congruent to $\emp_B(b_z)$ modulo $d$. If $z=0$, then $b_z=\hat{b}$ and $\emp_{(B',C')}(\hat{b}) = \emp_B(\hat{b})$ is clear from \eqref{eq:B'}. Otherwise, $b_z < \hat{b} = \min C'$, and thus $b_z$ is an empty space in $B'\cup C'$. If it is admissible, then by \Cref{pr:shift}, $d$ divides $\emp_{(B',C')}(b_z,\hat{b})$. This is true even if it is not admissible as then $\max B'\in C'$, and hence $\max B' = \min C' = \hat{b}$, which shows that $\hat{b}$ lies in $B'$ and, consequently, $d$ divides $\emp_{B'}(b_z+1, \hat{b})$ ($B'$ is a \dbal set), which equals $\emp_{(B',C')}(b_z,\hat{b})$. Therefore
	\[
	\emp_{(B',C')}(\hat{b}) \equiv \emp_{(B',C')}(b_z-1) = \emp_B(b_z-1) = \emp_B(b_z)\, (\textnormal{mod } d),
	\]
	as required. 
\end{proof}

\begin{proof}[Proof of the main part of \Cref{th:Mullineuxbalanced}]
	Let $B$ be the canonical $\beta$-set of $\lambda$. Then $(B,\varnothing)$ is a $d$-combined pair. If we apply the Abacus Mullineux Algorithm with $B$ as the starting $\beta$-set, then, by inductive argument using \Cref{pr:shift} (and \Cref{re:algorithm}(i)), the pair $(S,T)$ remains a $d$-combined pair until the final step when $S$ is a $\beta$-set of the empty partition. Then there are no empty spaces $f\leq \max S$ of $S$ and, in particular, no non-admissible empty spaces of $S\cup T$. By \Cref{de:balanced pair}(iii), $S\cup T$ is a \dsbal set.
	
	A similar inductive argument, using \Cref{pr:invariance}, shows that the combined $d$-runner matrix $\mathcal{R}_d(S,T)$, which initially equals $\mathcal{R}_d(B)$, remains unchanged until the final step of the algorithm. But then $S$ is a $\beta$-set of the empty partition and thus $\mathcal{R}_d(S,T)$ is just the $d$-runner matrix of $S\cup T$ (see \Cref{re:algorithm}(ii)). Hence, by \Cref{pr:Mullineux algorithm}, the canonical $\beta$-set of $m_e(\lambda)'$ is a \dsbal set with the same $d$-runner matrix as $B$, as required.
\end{proof}
	
\section{Skewed partitions}\label{se:unbalances}

We now complete the picture in \Cref{fig:BigPicture} by examining \dunb and \dsunb partitions. Recall from \Cref{de:partitions} that a partition $\lambda$ is \dunb if the arm length of any $e$-divisible hook of $\lambda$ is \emph{not} congruent to $0$ modulo $d$. And it is \dsunb if instead all such arm lengths are \emph{not} congruent to $-1$ modulo $d$. As for the \dbal and \dsbal partitions we can straightforwardly translate these definitions to $\beta$-sets.

\begin{definition}\label{de:unbalanced}
	Let $B$ be a subset of $\Z$. The set $B$ is \textit{\dunb}if $d$ does \emph{not} divide $\emp_B(b-ae+1, b)$ for any bead $b$ and any empty space $b-ae$ of $B$ (with $a>0$). It is \textit{\dsunb}if $d$ does \emph{not} divide $\emp_B(b-ae, b)$ for any bead $b$ and any empty space $b-ae$ of $B$ (with $a>0$). 
\end{definition}

We start by showing the `moreover' part of \Cref{pr:equivalence max}.

\begin{lemma}\label{le:regularity}
	Let $\lambda$ be a \dsunb partition. Then $\lambda$ is $e$-regular if and only if there is $0$ in each row of $\mathcal{R}_d(\lambda)$.
\end{lemma}

\begin{proof}
	Let $B$ be the canonical $\beta$-set of a \dsunb partition $\lambda$. We prove that $\lambda$ is not $e$-regular if and only if there is a row of $\mathcal{R}_d(\lambda)$ with no $0$. Suppose that $\lambda$ is not $e$-regular. Then there is $b\in B$ such that $b,b+1,\dots,b+e-1$ all lie in $B$ and $b>f$, where $f$ is the least empty space of $B$ (see \eqref{eq:recover partition} on page \pageref{eq:recover partition}). Let $0\leq i\leq e-1$ be the remainder of $f-b$ modulo $e$. Since $B$ is $d$-shift skewed, $\emp_B(b+i)=\emp_B(f,b+i)$ is not divisible by $d$. Since $b,b+1,\dots,b+e-1$ have the same emptiness, which, as we just saw, is not divisible by $d$, and lie on all $e$ runners, we obtain the desired row of $\mathcal{R}_d(\lambda)$ with no $0$.
	
	Conversely, suppose that row $x$ of $\mathcal{R}_d(\lambda)$ has no $0$ and let $b$ be the least bead of $B$ with $d$-emptiness $x$. We claim that $b,b+1,\dots, b+e-1$ are all beads of $B$, witnessing that $\lambda$ is not $e$-regular (as $x>0$). If not, there is least $1\leq i\leq e-1$ such that $b+i\notin B$. Let $y$ be the runner of $b+i$. By the minimal choice of $b$ and the assumption that $\mathcal{R}_d(\lambda)_{x,y}>0$, there is bead $b'\geq b$ lying on runner $y$ with $d$-emptiness $x$. Then $b'>b+i$ and $\emp_B(b+i, b') = \emp_B(b') - \emp_B(b+i-1) = \emp_B(b') - \emp_B(b)$, which is divisible by $d$, a contradiction with $B$ being $d$-shift skewed.
\end{proof}	

We now proceed to prove the first direction of the main part of \Cref{pr:equivalence max}. To do so, we introduce the following terminology. Given $B\subseteq \Z$, a \textit{swap of $i$ and $j$ of $B$} changes $B$ to a new set by replacing $i$ and $j$ by $j$ and $i$, respectively. If both or none of $i$ and $j$ lie in $B$, then $B$ does not change. If only one, say $i$, lies in $B$, then $B$ is replaced by $\left( B \setminus\left\lbrace i\right\rbrace\right)  \cup \left\lbrace j\right\rbrace  $. In such case, if $i>j$, we call this an \textit{up-move}, otherwise, we say it is a \textit{down-move}. The words `up' and `down' are chosen to correspond to the moves on the abacus with $|i-j|$ runners.

Observe that if $|i-j|=ae$ for some positive integer $a$ and $B=B_s(\lambda)$ for a partition $\lambda$, then the swap of $i$ and $j$ of $B$ results in a set of the form $B_s(\mu)$ for some partition $\mu$ with the same $e$-core as $\lambda$ (see the discussion after \Cref{ex:empty}). Moreover, in such case the $e$-weight of $\mu$ is the $e$-weight of $\lambda$ \emph{decreased by $a$} if the swap was an up-move, the $e$-weight of $\lambda$ \emph{increased by $a$} if the swap was a down-move, and the $e$-weight of $\lambda$ otherwise.

We will need the following observation about dominance order discussed before (and in) \Cref{ex:dominance}.

\begin{lemma}\label{le:UD}
	Let $B_{\alpha}$ be a $\beta$-set of a partition $\alpha$, $b$ be a bead of $B_{\alpha}$, $a$ be a positive integer and $i$ be a non-zero integer. Let $\gamma$ be a partition whose $\beta$-set $B_{\gamma}$ is obtained from $B_{\alpha}$ by the swap of $b-ae$ and $b$ followed by the swap of $b+i-ae$ and $b+i$. Suppose that the first swap was an up-move and the second was a down-move. If $i>0$, then $\gamma\rhd \alpha$, otherwise, $\gamma\lhd \alpha$.
\end{lemma}

\begin{proof}
	Let $\delta$ be a partition whose $\beta$-set $B_{\delta}$ is obtained from $B_{\alpha}$ by the swap of $b-ae$ and $b$, which is an up-move. By our assumption, it is also obtained from $B_{\gamma}$ by the swap of $b+i-ae$ and $b+i$, which is also an up-move.
	
	Using the correspondence between $ae$-hooks of a partition and beads $b'$ of its $\beta$-set such that $b'-ae$ is an empty space (see the discussion before \Cref{le:proper e-rim}), we see that there is a rim hook $R$ of $\alpha$ and a rim hook $R'$ of $\gamma$ such their removals from the respective partitions result in $\delta$. Since $\bt(R)$ is the number of beads of $B_{\alpha}$ greater than $b-ae$, that is the number of beads of $B_{\delta}$ greater or equal to $b-ae$ and, similarly, $\bt(R')$ is the number of beads of $B_{\delta}$ greater or equal to $b+i-ae$, we have $\bt(R)>\bt(R')$ if $i>0$ and $\bt(R)<\bt(R')$ if $i<0$. Since $\alpha$ and $\gamma$ have equal sizes, we deduce the statement using the simple fact about the dominance order mentioned before \Cref{ex:dominance}. 
\end{proof}  

While, if $i$ and $j$ both lie in $B$, the swap of $i$ and $j$ does nothing to $B$, one should think of this operation as swapping two \textit{mobile beads} at positions $i$ and $j$ on an abacus. Similarly, if $i$ and $j$ do not lie in $B$, the swap of $i$ and $j$ is a swap of two \textit{mobile empty spaces}. In general, the swap of $i$ and $j$ should be thought of as a swap of two \textit{mobile objects}. For instance, if $i\notin B$ and we perform a swap of $i$ and $j$ and then a swap of $j$ and $k$, then the mobile empty space at position $i$ of $B$ ends up at position $k$.

We can now define the first algorithm \textit{A1}. Let $B$ be a $\beta$-set of a partition, $b$ be its bead and $b-ae$ be its empty space (with $a>0$). The algorithm A1 applied to $B$, $b$ and $b-ae$ proceeds as follows. If there is no bead between $b-ae$ and $b-1$, we simply swap $b-ae$ and $b$ and return the new set. Otherwise, in turns, for $i= 0,1,2,\dots$ we swap $b+i-ae$ and $b+i$ of $B$ and then if the total number of performed up-moves is congruent to the total number of performed down-moves modulo $d$, we terminate the process. By discreteness of the process and the fact that the first swap is an up-move, at the end the number of performed up-moves either equals the number of performed down-moves or is greater by $d$. An example of this algorithm is in \Cref{fig:A1}.

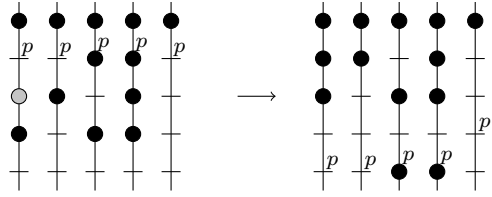
\begin{figure}[h]
	\centering
	\begin{tikzpicture}[x=0.5cm, y=0.5cm]
		\begin{pgfonlayer}{nodelayer}
			\node [style=none] (0) at (1.25, 3.5) {};
			\node [style=none] (1) at (2.25, 3.5) {};
			\node [style=none] (2) at (3.25, 3.5) {};
			\node [style=none] (3) at (4.25, 3.5) {};
			\node [style=none] (4) at (5.25, 3.5) {};
			\node [style=none] (5) at (1.25, -1.5) {};
			\node [style=none] (6) at (2.25, -1.5) {};
			\node [style=none] (7) at (3.25, -1.5) {};
			\node [style=none] (8) at (4.25, -1.5) {};
			\node [style=none] (9) at (5.25, -1.5) {};
			\node [style=Empty node] (11) at (2.25, 3) {};
			\node [style=Empty node] (12) at (3.25, 3) {};
			\node [style=Empty node] (13) at (4.25, 3) {};
			\node [style=Empty node] (15) at (1.25, 3) {};
			\node [style=Empty node] (18) at (3.25, 2) {};
			\node [style=none] (19) at (3, 1) {};
			\node [style=none] (20) at (3.5, 1) {};
			\node [style=none] (21) at (5, 2) {};
			\node [style=none] (22) at (5.5, 2) {};
			\node [style=none] (23) at (5, 1) {};
			\node [style=none] (24) at (5.5, 1) {};
			\node [style=none] (25) at (2, 2) {};
			\node [style=none] (26) at (2.5, 2) {};
			\node [style=none] (27) at (5.5, 0) {};
			\node [style=none] (28) at (5, 0) {};
			\node [style=none] (29) at (4.5, -1) {};
			\node [style=none] (30) at (4, -1) {};
			\node [style=none] (31) at (1.5, -1) {};
			\node [style=none] (32) at (1, -1) {};
			\node [style=none] (33) at (2.5, 0) {};
			\node [style=none] (34) at (2, 0) {};
			\node [style=none] (35) at (2.5, -1) {};
			\node [style=none] (36) at (2, -1) {};
			\node [style=none] (37) at (3.5, -1) {};
			\node [style=none] (38) at (3, -1) {};
			\node [style=Empty node] (42) at (1.25, 0) {};
			\node [style=Empty node] (44) at (2.25, 1) {};
			\node [style=Empty node] (45) at (3.25, 0) {};
			\node [style=Empty node] (46) at (4.25, 2) {};
			\node [style=Empty node] (47) at (5.25, 3) {};
			\node [style=Empty node] (48) at (4.25, 1) {};
			\node [style=Empty node] (49) at (4.25, 0) {};
			\node [style=none] (50) at (1, 2) {};
			\node [style=none] (51) at (1.5, 2) {};
			\node [style=none] (52) at (5, -1) {};
			\node [style=none] (53) at (5.5, -1) {};
			\node [style=Lead] (54) at (1.25, 1) {};
			\node [style=none] (55) at (7, 1) {};
			\node [style=none] (56) at (8, 1) {};
			\node [style=none] (57) at (9.25, 3.5) {};
			\node [style=none] (58) at (10.25, 3.5) {};
			\node [style=none] (59) at (11.25, 3.5) {};
			\node [style=none] (60) at (12.25, 3.5) {};
			\node [style=none] (61) at (13.25, 3.5) {};
			\node [style=none] (62) at (9.25, -1.5) {};
			\node [style=none] (63) at (10.25, -1.5) {};
			\node [style=none] (64) at (11.25, -1.5) {};
			\node [style=none] (65) at (12.25, -1.5) {};
			\node [style=none] (66) at (13.25, -1.5) {};
			\node [style=Empty node] (67) at (10.25, 3) {};
			\node [style=Empty node] (68) at (11.25, 3) {};
			\node [style=Empty node] (69) at (12.25, 3) {};
			\node [style=Empty node] (70) at (9.25, 3) {};
			\node [style=Empty node] (71) at (11.25, 1) {};
			\node [style=none] (72) at (11, 2) {};
			\node [style=none] (73) at (11.5, 2) {};
			\node [style=none] (74) at (13, 2) {};
			\node [style=none] (75) at (13.5, 2) {};
			\node [style=none] (76) at (13, 1) {};
			\node [style=none] (77) at (13.5, 1) {};
			\node [style=none] (78) at (10, 1) {};
			\node [style=none] (79) at (10.5, 1) {};
			\node [style=none] (80) at (13.5, 0) {};
			\node [style=none] (81) at (13, 0) {};
			\node [style=none] (82) at (12.5, 0) {};
			\node [style=none] (83) at (12, 0) {};
			\node [style=none] (84) at (9.5, -1) {};
			\node [style=none] (85) at (9, -1) {};
			\node [style=none] (86) at (10.5, 0) {};
			\node [style=none] (87) at (10, 0) {};
			\node [style=none] (88) at (10.5, -1) {};
			\node [style=none] (89) at (10, -1) {};
			\node [style=none] (90) at (11.5, 0) {};
			\node [style=none] (91) at (11, 0) {};
			\node [style=Empty node] (92) at (9.25, 1) {};
			\node [style=Empty node] (93) at (10.25, 2) {};
			\node [style=Empty node] (94) at (11.25, -1) {};
			\node [style=Empty node] (95) at (12.25, 2) {};
			\node [style=Empty node] (96) at (13.25, 3) {};
			\node [style=Empty node] (97) at (12.25, 1) {};
			\node [style=Empty node] (98) at (12.25, -1) {};
			\node [style=none] (99) at (9, 0) {};
			\node [style=none] (100) at (9.5, 0) {};
			\node [style=none] (101) at (13, -1) {};
			\node [style=none] (102) at (13.5, -1) {};
			\node [style=Empty node] (103) at (9.25, 2) {};
			\node [style=none] (104) at (1.475, 2.25) {$\scriptstyle p$};
			\node [style=none] (105) at (2.475, 2.25) {$\scriptstyle p$};
			\node [style=none] (106) at (3.475, 2.375) {$\scriptstyle p$};
			\node [style=none] (107) at (4.475, 2.375) {$\scriptstyle p$};
			\node [style=none] (108) at (5.475, 2.25) {$\scriptstyle p$};
			\node [style=none] (109) at (9.5, -0.75) {$\scriptstyle p$};
			\node [style=none] (110) at (10.5, -0.75) {$\scriptstyle p$};
			\node [style=none] (111) at (11.5, -0.625) {$\scriptstyle p$};
			\node [style=none] (112) at (12.5, -0.625) {$\scriptstyle p$};
			\node [style=none] (113) at (13.5, 0.25) {$\scriptstyle p$};
		\end{pgfonlayer}
		\begin{pgfonlayer}{edgelayer}
			\draw (0.center) to (5.center);
			\draw (1.center) to (6.center);
			\draw (3.center) to (8.center);
			\draw (4.center) to (9.center);
			\draw (19.center) to (20.center);
			\draw (21.center) to (22.center);
			\draw (23.center) to (24.center);
			\draw (25.center) to (26.center);
			\draw (28.center) to (27.center);
			\draw (30.center) to (29.center);
			\draw (32.center) to (31.center);
			\draw (34.center) to (33.center);
			\draw (36.center) to (35.center);
			\draw (38.center) to (37.center);
			\draw (2.center) to (7.center);
			\draw (50.center) to (51.center);
			\draw (52.center) to (53.center);
			\draw [style=Move it] (55.center) to (56.center);
			\draw (57.center) to (62.center);
			\draw (58.center) to (63.center);
			\draw (60.center) to (65.center);
			\draw (61.center) to (66.center);
			\draw (72.center) to (73.center);
			\draw (74.center) to (75.center);
			\draw (76.center) to (77.center);
			\draw (78.center) to (79.center);
			\draw (81.center) to (80.center);
			\draw (83.center) to (82.center);
			\draw (85.center) to (84.center);
			\draw (87.center) to (86.center);
			\draw (89.center) to (88.center);
			\draw (91.center) to (90.center);
			\draw (59.center) to (64.center);
			\draw (99.center) to (100.center);
			\draw (101.center) to (102.center);
		\end{pgfonlayer}
	\end{tikzpicture}
	\caption{Let $e=5$ and $d=3$. The picture shows the effect of the algorithm A1 applied to the first $\beta$-set with $b$ given by the highlighted bead and $a=1$. There are $14$ swaps. Three of them are up-moves, performed for $i=0,1$ and $5$ and three are down-moves, performed for $i=2,12$ and $13$. One can see that all the swaps can be performed runner-wise. Moreover, on each runner we have a mobile object which propagates downwards during the algorithm performed runner-wise. The propagating mobile objects are denoted by $p$.}
	\label{fig:A1}
\end{figure}

If the algorithm terminates, then, since only finitely many swaps were made, the final set is a $\beta$-set of a partition with the same $e$-core as the underlying partition of the initial set $B$. And it is the canonical $\beta$-set if $B$ is.

It is clear that the swaps performed during the algorithm can be done runner-wise on the James abacus with $ae$ runners (if the algorithm terminates). That is, we can firstly, in the original order, perform all swaps on the runner of $b$, that is, swaps of $b+iae-ae$ and $b+iae$ (with $i\geq 0$), then, in the original order, perform all swaps on the runner of $b+1$, that is, swaps of $b+1+iae-ae$ and $b+1+iae$ (with $i\geq 0$) and so on until we perform all swaps on the runner of $b+ae-1$. Thinking of the swaps performed during the algorithm as swaps of mobile objects, for any $0\leq r < ae$, the mobile object at position $b+r-ae$ of $B$ \textit{propagates} downwards on its runner. For example, the mobile empty space at position $b-ae$ of $B$ moves to position $b$, then to $b+ae$ and so on; see \Cref{fig:A1}. Hence, on each of the $ae$ runners we cannot perform both up- and down-moves. In \Cref{fig:A1}, only up-moves are performed on runners $0$ and $1$ and only down-moves are performed on runners $2$ and $3$.

Now, it is clear, that if the algorithm A1 terminates, then we can perform the up- and down-moves performed in this algorithm in any order with the same result as long as the relative order of up-moves and the relative order of down-moves are preserved; see \Cref{fig:UD}.

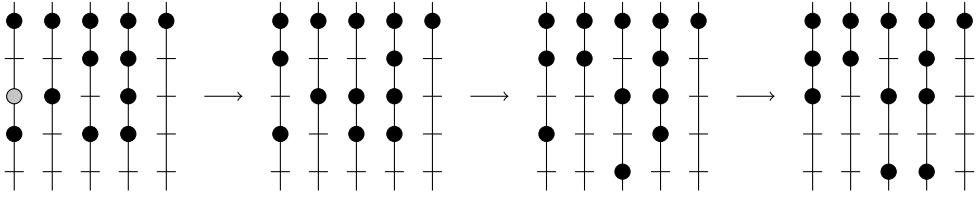
\begin{figure}[h]
	\centering
	\begin{tikzpicture}[x=0.5cm, y=0.5cm]
		\begin{pgfonlayer}{nodelayer}
			\node [style=none] (0) at (2.25, 3.5) {};
			\node [style=none] (1) at (3.25, 3.5) {};
			\node [style=none] (2) at (4.25, 3.5) {};
			\node [style=none] (3) at (5.25, 3.5) {};
			\node [style=none] (4) at (6.25, 3.5) {};
			\node [style=none] (5) at (2.25, -1.5) {};
			\node [style=none] (6) at (3.25, -1.5) {};
			\node [style=none] (7) at (4.25, -1.5) {};
			\node [style=none] (8) at (5.25, -1.5) {};
			\node [style=none] (9) at (6.25, -1.5) {};
			\node [style=Empty node] (11) at (3.25, 3) {};
			\node [style=Empty node] (12) at (4.25, 3) {};
			\node [style=Empty node] (13) at (5.25, 3) {};
			\node [style=Empty node] (15) at (2.25, 3) {};
			\node [style=Empty node] (18) at (4.25, 2) {};
			\node [style=none] (19) at (4, 1) {};
			\node [style=none] (20) at (4.5, 1) {};
			\node [style=none] (21) at (6, 2) {};
			\node [style=none] (22) at (6.5, 2) {};
			\node [style=none] (23) at (6, 1) {};
			\node [style=none] (24) at (6.5, 1) {};
			\node [style=none] (25) at (3, 2) {};
			\node [style=none] (26) at (3.5, 2) {};
			\node [style=none] (27) at (6.5, 0) {};
			\node [style=none] (28) at (6, 0) {};
			\node [style=none] (29) at (5.5, -1) {};
			\node [style=none] (30) at (5, -1) {};
			\node [style=none] (31) at (2.5, -1) {};
			\node [style=none] (32) at (2, -1) {};
			\node [style=none] (33) at (3.5, 0) {};
			\node [style=none] (34) at (3, 0) {};
			\node [style=none] (35) at (3.5, -1) {};
			\node [style=none] (36) at (3, -1) {};
			\node [style=none] (37) at (4.5, -1) {};
			\node [style=none] (38) at (4, -1) {};
			\node [style=Empty node] (42) at (2.25, 0) {};
			\node [style=Empty node] (44) at (3.25, 1) {};
			\node [style=Empty node] (45) at (4.25, 0) {};
			\node [style=Empty node] (46) at (5.25, 2) {};
			\node [style=Empty node] (47) at (6.25, 3) {};
			\node [style=Empty node] (48) at (5.25, 1) {};
			\node [style=Empty node] (49) at (5.25, 0) {};
			\node [style=none] (50) at (2, 2) {};
			\node [style=none] (51) at (2.5, 2) {};
			\node [style=none] (52) at (6, -1) {};
			\node [style=none] (53) at (6.5, -1) {};
			\node [style=Lead] (54) at (2.25, 1) {};
			\node [style=none] (55) at (7.25, 1) {};
			\node [style=none] (56) at (8.25, 1) {};
			\node [style=none] (57) at (9.25, 3.5) {};
			\node [style=none] (58) at (10.25, 3.5) {};
			\node [style=none] (59) at (11.25, 3.5) {};
			\node [style=none] (60) at (12.25, 3.5) {};
			\node [style=none] (61) at (13.25, 3.5) {};
			\node [style=none] (62) at (9.25, -1.5) {};
			\node [style=none] (63) at (10.25, -1.5) {};
			\node [style=none] (64) at (11.25, -1.5) {};
			\node [style=none] (65) at (12.25, -1.5) {};
			\node [style=none] (66) at (13.25, -1.5) {};
			\node [style=Empty node] (67) at (10.25, 3) {};
			\node [style=Empty node] (68) at (11.25, 3) {};
			\node [style=Empty node] (69) at (12.25, 3) {};
			\node [style=Empty node] (70) at (9.25, 3) {};
			\node [style=Empty node] (71) at (11.25, 1) {};
			\node [style=none] (72) at (11, 2) {};
			\node [style=none] (73) at (11.5, 2) {};
			\node [style=none] (74) at (13, 2) {};
			\node [style=none] (75) at (13.5, 2) {};
			\node [style=none] (76) at (13, 1) {};
			\node [style=none] (77) at (13.5, 1) {};
			\node [style=none] (78) at (10, 2) {};
			\node [style=none] (79) at (10.5, 2) {};
			\node [style=none] (80) at (13.5, 0) {};
			\node [style=none] (81) at (13, 0) {};
			\node [style=none] (82) at (12.5, -1) {};
			\node [style=none] (83) at (12, -1) {};
			\node [style=none] (84) at (9.5, -1) {};
			\node [style=none] (85) at (9, -1) {};
			\node [style=none] (86) at (10.5, 0) {};
			\node [style=none] (87) at (10, 0) {};
			\node [style=none] (88) at (10.5, -1) {};
			\node [style=none] (89) at (10, -1) {};
			\node [style=none] (90) at (11.5, -1) {};
			\node [style=none] (91) at (11, -1) {};
			\node [style=Empty node] (92) at (9.25, 0) {};
			\node [style=Empty node] (93) at (10.25, 1) {};
			\node [style=Empty node] (94) at (11.25, 0) {};
			\node [style=Empty node] (95) at (12.25, 2) {};
			\node [style=Empty node] (96) at (13.25, 3) {};
			\node [style=Empty node] (97) at (12.25, 1) {};
			\node [style=Empty node] (98) at (12.25, 0) {};
			\node [style=none] (99) at (9, 1) {};
			\node [style=none] (100) at (9.5, 1) {};
			\node [style=none] (101) at (13, -1) {};
			\node [style=none] (102) at (13.5, -1) {};
			\node [style=Empty node] (103) at (9.25, 2) {};
			\node [style=none] (104) at (14.25, 1) {};
			\node [style=none] (105) at (15.25, 1) {};
			\node [style=none] (106) at (16.25, 3.5) {};
			\node [style=none] (107) at (17.25, 3.5) {};
			\node [style=none] (108) at (18.25, 3.5) {};
			\node [style=none] (109) at (19.25, 3.5) {};
			\node [style=none] (110) at (20.25, 3.5) {};
			\node [style=none] (111) at (16.25, -1.5) {};
			\node [style=none] (112) at (17.25, -1.5) {};
			\node [style=none] (113) at (18.25, -1.5) {};
			\node [style=none] (114) at (19.25, -1.5) {};
			\node [style=none] (115) at (20.25, -1.5) {};
			\node [style=Empty node] (116) at (17.25, 3) {};
			\node [style=Empty node] (117) at (18.25, 3) {};
			\node [style=Empty node] (118) at (19.25, 3) {};
			\node [style=Empty node] (119) at (16.25, 3) {};
			\node [style=Empty node] (120) at (18.25, 1) {};
			\node [style=none] (121) at (18, 2) {};
			\node [style=none] (122) at (18.5, 2) {};
			\node [style=none] (123) at (20, 2) {};
			\node [style=none] (124) at (20.5, 2) {};
			\node [style=none] (125) at (20, 1) {};
			\node [style=none] (126) at (20.5, 1) {};
			\node [style=none] (127) at (17, 1) {};
			\node [style=none] (128) at (17.5, 1) {};
			\node [style=none] (129) at (20.5, 0) {};
			\node [style=none] (130) at (20, 0) {};
			\node [style=none] (131) at (19.5, -1) {};
			\node [style=none] (132) at (19, -1) {};
			\node [style=none] (133) at (16.5, -1) {};
			\node [style=none] (134) at (16, -1) {};
			\node [style=none] (135) at (17.5, 0) {};
			\node [style=none] (136) at (17, 0) {};
			\node [style=none] (137) at (17.5, -1) {};
			\node [style=none] (138) at (17, -1) {};
			\node [style=none] (139) at (18.5, 0) {};
			\node [style=none] (140) at (18, 0) {};
			\node [style=Empty node] (141) at (16.25, 0) {};
			\node [style=Empty node] (142) at (17.25, 2) {};
			\node [style=Empty node] (143) at (18.25, -1) {};
			\node [style=Empty node] (144) at (19.25, 2) {};
			\node [style=Empty node] (145) at (20.25, 3) {};
			\node [style=Empty node] (146) at (19.25, 1) {};
			\node [style=Empty node] (147) at (19.25, 0) {};
			\node [style=none] (148) at (16, 1) {};
			\node [style=none] (149) at (16.5, 1) {};
			\node [style=none] (150) at (20, -1) {};
			\node [style=none] (151) at (20.5, -1) {};
			\node [style=Empty node] (152) at (16.25, 2) {};
			\node [style=none] (153) at (21.25, 1) {};
			\node [style=none] (154) at (22.25, 1) {};
			\node [style=none] (155) at (23.25, 3.5) {};
			\node [style=none] (156) at (24.25, 3.5) {};
			\node [style=none] (157) at (25.25, 3.5) {};
			\node [style=none] (158) at (26.25, 3.5) {};
			\node [style=none] (159) at (27.25, 3.5) {};
			\node [style=none] (160) at (23.25, -1.5) {};
			\node [style=none] (161) at (24.25, -1.5) {};
			\node [style=none] (162) at (25.25, -1.5) {};
			\node [style=none] (163) at (26.25, -1.5) {};
			\node [style=none] (164) at (27.25, -1.5) {};
			\node [style=Empty node] (165) at (24.25, 3) {};
			\node [style=Empty node] (166) at (25.25, 3) {};
			\node [style=Empty node] (167) at (26.25, 3) {};
			\node [style=Empty node] (168) at (23.25, 3) {};
			\node [style=Empty node] (169) at (25.25, 1) {};
			\node [style=none] (170) at (25, 2) {};
			\node [style=none] (171) at (25.5, 2) {};
			\node [style=none] (172) at (27, 2) {};
			\node [style=none] (173) at (27.5, 2) {};
			\node [style=none] (174) at (27, 1) {};
			\node [style=none] (175) at (27.5, 1) {};
			\node [style=none] (176) at (24, 1) {};
			\node [style=none] (177) at (24.5, 1) {};
			\node [style=none] (178) at (27.5, 0) {};
			\node [style=none] (179) at (27, 0) {};
			\node [style=none] (180) at (26.5, 0) {};
			\node [style=none] (181) at (26, 0) {};
			\node [style=none] (182) at (23.5, -1) {};
			\node [style=none] (183) at (23, -1) {};
			\node [style=none] (184) at (24.5, 0) {};
			\node [style=none] (185) at (24, 0) {};
			\node [style=none] (186) at (24.5, -1) {};
			\node [style=none] (187) at (24, -1) {};
			\node [style=none] (188) at (25.5, 0) {};
			\node [style=none] (189) at (25, 0) {};
			\node [style=Empty node] (190) at (23.25, 1) {};
			\node [style=Empty node] (191) at (24.25, 2) {};
			\node [style=Empty node] (192) at (25.25, -1) {};
			\node [style=Empty node] (193) at (26.25, 2) {};
			\node [style=Empty node] (194) at (27.25, 3) {};
			\node [style=Empty node] (195) at (26.25, 1) {};
			\node [style=Empty node] (196) at (26.25, -1) {};
			\node [style=none] (197) at (23, 0) {};
			\node [style=none] (198) at (23.5, 0) {};
			\node [style=none] (199) at (27, -1) {};
			\node [style=none] (200) at (27.5, -1) {};
			\node [style=Empty node] (201) at (23.25, 2) {};
		\end{pgfonlayer}
		\begin{pgfonlayer}{edgelayer}
			\draw (0.center) to (5.center);
			\draw (1.center) to (6.center);
			\draw (3.center) to (8.center);
			\draw (4.center) to (9.center);
			\draw (19.center) to (20.center);
			\draw (21.center) to (22.center);
			\draw (23.center) to (24.center);
			\draw (25.center) to (26.center);
			\draw (28.center) to (27.center);
			\draw (30.center) to (29.center);
			\draw (32.center) to (31.center);
			\draw (34.center) to (33.center);
			\draw (36.center) to (35.center);
			\draw (38.center) to (37.center);
			\draw (2.center) to (7.center);
			\draw (50.center) to (51.center);
			\draw (52.center) to (53.center);
			\draw [style=Move it] (55.center) to (56.center);
			\draw (57.center) to (62.center);
			\draw (58.center) to (63.center);
			\draw (60.center) to (65.center);
			\draw (61.center) to (66.center);
			\draw (72.center) to (73.center);
			\draw (74.center) to (75.center);
			\draw (76.center) to (77.center);
			\draw (78.center) to (79.center);
			\draw (81.center) to (80.center);
			\draw (83.center) to (82.center);
			\draw (85.center) to (84.center);
			\draw (87.center) to (86.center);
			\draw (89.center) to (88.center);
			\draw (91.center) to (90.center);
			\draw (59.center) to (64.center);
			\draw (99.center) to (100.center);
			\draw (101.center) to (102.center);
			\draw [style=Move it] (104.center) to (105.center);
			\draw (106.center) to (111.center);
			\draw (107.center) to (112.center);
			\draw (109.center) to (114.center);
			\draw (110.center) to (115.center);
			\draw (121.center) to (122.center);
			\draw (123.center) to (124.center);
			\draw (125.center) to (126.center);
			\draw (127.center) to (128.center);
			\draw (130.center) to (129.center);
			\draw (132.center) to (131.center);
			\draw (134.center) to (133.center);
			\draw (136.center) to (135.center);
			\draw (138.center) to (137.center);
			\draw (140.center) to (139.center);
			\draw (108.center) to (113.center);
			\draw (148.center) to (149.center);
			\draw (150.center) to (151.center);
			\draw [style=Move it] (153.center) to (154.center);
			\draw (155.center) to (160.center);
			\draw (156.center) to (161.center);
			\draw (158.center) to (163.center);
			\draw (159.center) to (164.center);
			\draw (170.center) to (171.center);
			\draw (172.center) to (173.center);
			\draw (174.center) to (175.center);
			\draw (176.center) to (177.center);
			\draw (179.center) to (178.center);
			\draw (181.center) to (180.center);
			\draw (183.center) to (182.center);
			\draw (185.center) to (184.center);
			\draw (187.center) to (186.center);
			\draw (189.center) to (188.center);
			\draw (157.center) to (162.center);
			\draw (197.center) to (198.center);
			\draw (199.center) to (200.center);
		\end{pgfonlayer}
	\end{tikzpicture}
	\caption{Let $U_1, U_2, U_3$ be the up-moves from \Cref{fig:A1} (performed in this order) and $D_1, D_2, D_3$ be the down-moves. We can perform the up- and down-moves in order $U_1,D_1,U_2,D_2,U_3,D_3$ with the same effect. The intermediate $\beta$-sets after performing the pairs of moves $U_j$ and $D_j$ with $j=1,2,3$ are displayed above. Note that the underlying partitions: $(6^2,5,4,3^2,2^2), (6^2,5,4^3,2), (10, 7, 5,4^2, 1)$ and $(11^2,3^2,2,1)$ increase in the dominance order.}
	\label{fig:UD}
\end{figure}

We can now show the first property of the algorithm A1.

\begin{lemma}\label{le:a1-order}
	Let $B$ be a $\beta$-set of a partition $\lambda$, $b$ be its bead and $b-ae$ be its empty space (with $a>0$). Then the algorithm A1 applied to $B$, $b$ and $b-ae$ terminates and returns a $\beta$-set of a partition which is either less in size than $\lambda$ or which has the same size as $\lambda$ and is greater in the dominance order. 
\end{lemma}

\begin{proof}
	We can only focus on the case when there is a bead $b'$ between $b-ae$ and $b-1$, as the result is clear otherwise. Observe that for any $i\geq 0$ the mobile object at position $b+i$ of $B$ has not moved before the swap of $b+i-ae$ and $b+i$. Thus, there are only finitely many possible up-moves the algorithm can make as there are only finitely many beads of $B$ greater or equal to $b$. On the other hand, the (propagating) mobile bead $b'$ can be swapped with $b'+ ae$, then $b' +2ae$ and so on and from some point all these swaps will be down-moves. So there is no bound on the number of down-moves, and hence the algorithm terminates. Let $B'$ be the terminating set and let $\mu$ be the corresponding partition. If at the end the number of up-moves is the number of down-moves increased by $d$, then the $e$-weight of $\mu$ is the $e$-weight of $\lambda$ decreased by $da$, and hence $|\mu| < |\lambda|$.
	
	Now suppose that the number of the up-moves and the number of the down-moves agree and denote them by $n$. For $1\leq j\leq n$ let $U_j$ and $D_j$ be the $j$'th up-move and down-move, respectively. Since the algorithm did not terminate earlier, throughout the algorithm the number of up-moves was at least the number of down-moves; thus $U_j$ happened before $D_j$.
	
	As observed above, $B'$ can be obtained from $B$ by performing the up- and down-moves $U_1, D_1, \dots, U_n, D_n$ in this order. We will be done once we show that in this order of swaps the pair of swaps $U_j$ and $D_j$ increases the dominance order of the underlying partition; see \Cref{fig:UD}. If $\alpha$ and $\gamma$ are the underlying partitions before $U_j$ and after $D_j$, respectively, then \Cref{le:UD} applies to these partitions. Since $U_j$ happened before $D_j$ in the original order of the swaps, $i>0$ in \Cref{le:UD}, and thus $\gamma$ dominates $\alpha$, as required.
\end{proof} 

In the next result about the algorithm A1 we think of the algorithm as a permutation of three intervals of mobile objects on the James abacus of the starting $\beta$-set; see \Cref{fig:PermuteA1}. 

\begin{figure}[h]
	\centering
	\begin{tikzpicture}[x=0.5cm, y=0.5cm]
		\begin{pgfonlayer}{nodelayer}
			\node [style=none] (0) at (-5.5, 0) {};
			\node [style=none] (1) at (8.5, 0) {};
			\node [style=none] (2) at (-5, 0) {};
			\node [style=none] (3) at (-2, 0) {};
			\node [style=none] (4) at (2, 0) {};
			\node [style=none] (5) at (8, 0) {};
			\node [style=none] (6) at (-4.5, 0.25) {};
			\node [style=none] (7) at (-4.5, -0.25) {};
			\node [style=none] (8) at (-2.5, 0.25) {};
			\node [style=none] (9) at (-2.5, -0.25) {};
			\node [style=none] (10) at (-1.5, 0.25) {};
			\node [style=none] (11) at (-1.5, -0.25) {};
			\node [style=Empty node] (12) at (1.5, 0) {};
			\node [style=Empty node] (13) at (2.5, 0) {};
			\node [style=none] (14) at (7.5, 0.25) {};
			\node [style=none] (15) at (7.5, -0.25) {};
			\node [style=none] (16) at (-4.3, -0.5) {$\scriptstyle b-ae$};
			\node [style=none] (17) at (-3.5, 1.25) {$I_1$};
			\node [style=none] (18) at (0, 1.25) {$I_2$};
			\node [style=none] (19) at (-2.5, -0.5) {$\scriptstyle c$};
			\node [style=none] (20) at (-1.45, -0.5) {$\scriptstyle c+1$};
			\node [style=none] (21) at (1.5, -0.5) {$\scriptstyle b-1$};
			\node [style=none] (22) at (2.5, -0.5) {$\scriptstyle b$};
			\node [style=none] (23) at (7.5, -0.5) {$\scriptstyle b+t$};
			\node [style=none] (24) at (5, 1.25) {$I_3$};
			\node [style=none] (25) at (-5, -1.75) {};
			\node [style=none] (26) at (-5, 1.75) {};
			\node [style=none] (27) at (-4.5, 1.75) {};
			\node [style=none] (28) at (-4.5, -1.75) {};
			\node [style=none] (29) at (-2.5, 1.75) {};
			\node [style=none] (30) at (-1.5, 1.75) {};
			\node [style=none] (31) at (-2, 1.75) {};
			\node [style=none] (32) at (-2.5, -1.75) {};
			\node [style=none] (33) at (-2, -1.75) {};
			\node [style=none] (34) at (-1.5, -1.75) {};
			\node [style=none] (35) at (1.5, -1.75) {};
			\node [style=none] (36) at (2, -1.75) {};
			\node [style=none] (37) at (2.5, -1.75) {};
			\node [style=none] (38) at (1.5, 1.75) {};
			\node [style=none] (39) at (2, 1.75) {};
			\node [style=none] (40) at (2.5, 1.75) {};
			\node [style=none] (41) at (7.5, 1.75) {};
			\node [style=none] (42) at (8, 1.75) {};
			\node [style=none] (43) at (7.5, -1.75) {};
			\node [style=none] (44) at (8, -1.75) {};
			\node [style=none] (45) at (-1.5, 1.75) {};
			\node [style=none] (46) at (1.5, -2.5) {};
			\node [style=none] (47) at (1.5, -3.5) {};
			\node [style=none] (48) at (1.5, -3.5) {};
			\node [style=none] (49) at (-5.5, -6) {};
			\node [style=none] (50) at (8.5, -6) {};
			\node [style=none] (51) at (-5, -6) {};
			\node [style=none] (54) at (8, -6) {};
			\node [style=none] (55) at (5.5, -5.75) {};
			\node [style=none] (56) at (5.5, -6.25) {};
			\node [style=none] (57) at (0.5, -5.75) {};
			\node [style=none] (58) at (0.5, -6.25) {};
			\node [style=none] (59) at (1.5, -5.75) {};
			\node [style=none] (60) at (1.5, -6.25) {};
			\node [style=Empty node] (61) at (4.5, -6) {};
			\node [style=Empty node] (62) at (-4.5, -6) {};
			\node [style=none] (63) at (7.5, -5.75) {};
			\node [style=none] (64) at (7.5, -6.25) {};
			\node [style=none] (66) at (6.5, -4.75) {$I_1$};
			\node [style=none] (67) at (3, -4.75) {$I_2$};
			\node [style=none] (73) at (-2, -4.75) {$I_3$};
			\node [style=none] (74) at (-5, -7.75) {};
			\node [style=none] (75) at (-5, -4.25) {};
			\node [style=none] (76) at (-4.5, -4.25) {};
			\node [style=none] (77) at (-4.5, -7.75) {};
			\node [style=none] (78) at (0.5, -4.25) {};
			\node [style=none] (80) at (1, -4.25) {};
			\node [style=none] (81) at (0.5, -7.75) {};
			\node [style=none] (82) at (1, -7.75) {};
			\node [style=none] (83) at (1.5, -7.75) {};
			\node [style=none] (84) at (4.5, -7.75) {};
			\node [style=none] (85) at (5, -7.75) {};
			\node [style=none] (86) at (5.5, -7.75) {};
			\node [style=none] (87) at (4.5, -4.25) {};
			\node [style=none] (88) at (5, -4.25) {};
			\node [style=none] (89) at (5.5, -4.25) {};
			\node [style=none] (90) at (7.5, -4.25) {};
			\node [style=none] (91) at (8, -4.25) {};
			\node [style=none] (92) at (7.5, -7.75) {};
			\node [style=none] (93) at (8, -7.75) {};
			\node [style=none] (94) at (1.5, -4.25) {};
			\node [style=none] (95) at (-6.5, 0) {$B$};
			\node [style=none] (96) at (-6.5, -6) {$B'$};
			\node [style=none] (97) at (-5, 2.5) {};
			\node [style=none] (98) at (-2, 2.5) {};
			\node [style=none] (99) at (2, 2.5) {};
			\node [style=none] (100) at (8, 2.5) {};
			\node [style=none] (101) at (5, 3) {$t+1$};
			\node [style=none] (102) at (0, 3) {$ae-r_t-1$};
			\node [style=none] (103) at (-3.5, 3) {$r_t+1$};
			\node [style=none] (104) at (1.2, -3) {$\scriptstyle \psi$};
		\end{pgfonlayer}
		\begin{pgfonlayer}{edgelayer}
			\draw (5.center) to (1.center);
			\draw (2.center) to (0.center);
			\draw (14.center) to (15.center);
			\draw (10.center) to (11.center);
			\draw (8.center) to (9.center);
			\draw (6.center) to (7.center);
			\draw (2.center) to (5.center);
			\draw (27.center) to (26.center);
			\draw (26.center) to (25.center);
			\draw (25.center) to (28.center);
			\draw (29.center) to (45.center);
			\draw (32.center) to (34.center);
			\draw (31.center) to (33.center);
			\draw (38.center) to (40.center);
			\draw (35.center) to (37.center);
			\draw (39.center) to (36.center);
			\draw (41.center) to (42.center);
			\draw (44.center) to (43.center);
			\draw (42.center) to (44.center);
			\draw [style=Move it] (46.center) to (48.center);
			\draw (54.center) to (50.center);
			\draw (51.center) to (49.center);
			\draw (63.center) to (64.center);
			\draw (59.center) to (60.center);
			\draw (57.center) to (58.center);
			\draw (55.center) to (56.center);
			\draw (51.center) to (54.center);
			\draw (76.center) to (75.center);
			\draw (75.center) to (74.center);
			\draw (74.center) to (77.center);
			\draw (78.center) to (94.center);
			\draw (81.center) to (83.center);
			\draw (80.center) to (82.center);
			\draw (87.center) to (89.center);
			\draw (84.center) to (86.center);
			\draw (88.center) to (85.center);
			\draw (90.center) to (91.center);
			\draw (93.center) to (92.center);
			\draw (91.center) to (93.center);
			\draw [style=measuredots] (97.center) to (98.center);
			\draw [style=measuredots] (98.center) to (99.center);
			\draw [style=measuredots] (99.center) to (100.center);
		\end{pgfonlayer}
	\end{tikzpicture}
	\caption{In the picture, $b$ is a bead of $B$, $b-ae$ is an empty space of $B$ and the other integers can be either beads or empty spaces of $B$. Suppose that the last swap in the algorithm A1 applied to $B$, $b$ and $b-ae$ is the swap of $b+t-ae$ and $b+t$. If $0\leq r_t\leq ae-1$ is the remainder of $t$ modulo $ae$, then the propagating moving object at position $c=b+r_t-ae$ of $B$ moves to $b+t$ during A1. Hence the propagating moving objects lying in interval $I_1$ increase their positions by $ae+t-r_t$, while those in interval $I_2$ increase their positions only by $t-r_t$. For instance, in \Cref{fig:A1} where $e=5$, $a=1$ and $t=13$, the propagating moving objects on the first four runners increase their position by $15$ while the propagating moving object on the final runner increases its position only by $10$. Finally, the moving objects in interval $I_3$ decrease their positions by $ae$ and the remaining moving objects are not affected by A1. Thus the outcome of A1, set $B'$, can be obtained by applying permutation $\psi$ of the three intervals $I_1, I_2$ and $I_3$ as in the picture.}
	\label{fig:PermuteA1}
\end{figure}
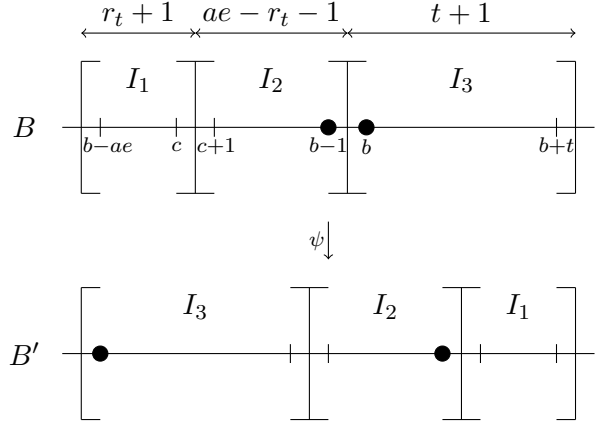

\begin{lemma}\label{le:a1-runner}
 	Let $B$ be a $\beta$-set of a partition, $b$ be its bead and $b-ae$ be its empty space (with $a>0$) such that $d\mid \emp_B(b-ae, b)$. The algorithm A1 applied to $B$, $b$ and $b-ae$ preserves the $d$-runner matrix.
\end{lemma}

\begin{proof}
	Denote by $B'$ the set returned by the algorithm A1 applied to $B$, $b$ and $b-ae$. We need to show that the $d$-runner matrices of $B$ and $B'$ agree.
	
	The statement is clear if there is no bead between $b-ae$ and $b-1$ since then all emptinesses of beads in $B'$ apart from the emptiness of $b-ae$ remain as in $B$, and the emptiness of $b-ae$ in $B'$ and the emptiness of $b$ in $B$ agree modulo $d$ as $d\mid \emp_B(b-ae,b)$. So suppose that there is a bead between $b-ae$ and $b-1$.
	
	Suppose that A1 terminates with the swap of $b+t-ae$ and $b+t$ (with $t>0$) and for $0\leq i\leq t$ let $0\leq r_i<ae$ be the remainder of $i$ modulo $ae$. We define three intervals of integers
	\begin{align*}
		I_1 &=\left\lbrace v\in \Z: b-ae\leq v\leq b+r_t-ae \right\rbrace,\\
		I_2 &=\left\lbrace v\in \Z: b+r_t-ae< v< b \right\rbrace,\\
		I_3 &=\left\lbrace v\in \Z: b\leq v\leq b+t \right\rbrace,
	\end{align*}
	as in the first diagram in \Cref{fig:PermuteA1}, and write $E_i$ for the number of empty spaces of $B$ in $I_i$ for $i=1,2,3$. For example, $E_2 = \emp_B(b+r_t-ae+1, b-1)$. Then the terminating set, call it $B'$, can be obtained from $B$ by simultaneously replacing each $b'\in B$ by $\psi(b')$, where $\psi:\Z \to \Z$ is a bijection given by
	\begin{align}\label{al:intervals}
		\psi(v) =
		\begin{cases*}
			v + ae + t -r_t, & \text{if $v\in I_1$;}\\
			v + t -r_t, & \text{if $v\in I_2$;}\\
			v - ae, & \text{if $v\in I_3$;}\\
			v, & \text{otherwise,}
		\end{cases*}
	\end{align}
	for any $v\in \Z$; see \Cref{fig:PermuteA1}. Clearly $\psi$ preserves runners (of the James abacus with $e$-runners). To finish the proof we show that it also preserves $d$-emptinesses. From \eqref{al:intervals} (and \Cref{fig:PermuteA1}) we obtain
	\begin{align*}
		\emp_{B'}(\phi(v)) - \emp_{B}(v)=
		\begin{cases*}
			E_2 + E_3, & \text{if $v\in I_1$;}\\
			E_3 - E_1, & \text{if $v\in I_2$;}\\
			- (E_1 + E_2), & \text{if $v\in I_3$;}\\
			0, & \text{otherwise,}
		\end{cases*}
	\end{align*}
	for any $v\in \Z$. Thus it suffices to show that $d\mid E_1+E_2$ and $d\mid E_2 + E_3$. We know that $E_1+E_2 = \emp_B(b-ae, b-1) = \emp_B(b-ae, b)$ which is indeed divisible by $d$. We can also write
	\begin{align*}
		E_2 + E_3 &= \emp_{B'}(b-ae, b+t-r_t -1)\\
		&= \emp_{B'}(b+t-r_t -1) - \emp_{B'}(b-ae-1)\\
		&= \emp_{B'}(b+t-r_t -1) - \emp_B(b-ae-1),
	\end{align*}
	which remains valid even if $I_2$ is an empty set, that is, when $r_t=ae-1$. 
	
	For $0\leq i\leq t$, let $B_i$ be the $\beta$-set before step $i$ of A1 and let $B_{t+1} = B'$. So for any $0\leq i\leq t$, $B_{i+1}$ is obtained from $B_i$ by the swap of $b+i-ae$ and $b+i$. We examine how the $d$-emptiness of the propagating mobile empty space at position $b-ae$ of $B$ changes throughout the process. Let $p_i = b+i-1-r_{i-1}$ which is, by a simple inductive argument, the position of this mobile empty space in $B_i$. If $ae\mid i$, then the mobile objects at positions $b+i-ae+1, b+i-ae+2,\dots, b+i-1$ of $B_i$ are the propagating mobile objects at positions $b-ae+1, b-ae+2,\dots, b-1$ of $B$, respectively. Hence
	\begin{align*}
		\emp_{B_{i+1}}(p_{i+1})-\emp_{B_i}(p_i) &= \emp_{B_{i+1}}(b+i)-\emp_{B_i}(b+i-ae)\\
		&= \emp_{B_i}(b+i)-\emp_{B_i}(b+i-ae)\\
		&= \emp_{B_i}(b+i-ae+1,b+i-1) + \emp_{B_i}(b+i,b+i)\\
		&= \emp_{B}(b-ae+1,b-1) + \emp_{B_i}(b+i,b+i)\\
		&= \emp_{B}(b-ae,b) -1 + \emp_{B_i}(b+i,b+i)\\
		&\equiv \emp_{B_i}(b+i,b+i) -1 (\textnormal{mod } d)
	\end{align*}
	Since the $i$'th swap swaps our mobile empty space at position $p_i$ in $B_i$ and the mobile object at position $b+i$ of $B_i$, we conclude that $\emp_{B_{i+1}}(p_{i+1})-\emp_{B_i}(p_i)$ is $-1$ modulo $d$ if the $i$'th swap is an up-move and $0$, otherwise (when it is neither an up-move nor a down-move).
	
	If $ae\nmid i$, then our mobile empty space at position $p_i$ of $B_i$ does not move during the $i$'th swap (that is $p_{i+1}=p_i$) and it is clear that $\emp_{B_{i+1}}(p_{i+1})-\emp_{B_i}(p_{i})$ is $1$ if the $i$'th step is a down-move, $-1$ if the $i$'th step is an up-move and $0$ otherwise. Since at the end the numbers of the up- and the down-moves agree modulo $d$, we conclude that $\emp_{B_{t+1}}(p_{t+1})=\emp_{B'}(b+t-r_t) $ is congruent to $\emp_{B_0}(p_0)=\emp_B(b-ae)$ modulo $d$. As $p_i$ is an empty space of $B_i$ for all $0\leq i \leq t+1$, also $\emp_{B'}(b+t-r_t-1)$ and $\emp_B(b-ae-1)$ agree modulo $d$, and hence $d\mid E_2 + E_3$, as desired.
\end{proof}

We can immediately deduce the first direction in \Cref{pr:equivalence max}. Recall that if $\gamma$ is an $e$-core partition and $\mathcal{R}$ is a $\gamma$-realisable matrix, then $\mathcal{E}_{\mathcal{R}}(\gamma)$ is the set of all partitions with $e$-core $\gamma$ and $d$-runner matrix $\mathcal{R}$ of minimal $e$-weight.

\begin{corollary}\label{cor:max implies dsunb}
	Let $\gamma$ be an $e$-core partition and $\mathcal{R}$ be a $\gamma$-realisable $(d-1)\times e$ matrix of non-negative integers. Then a maximal partition in the dominance order in $\mathcal{E}_{\mathcal{R}}(\gamma)$ is $d$-shift skewed.
\end{corollary}

\begin{proof}
	We show the contrapositive. Let $\lambda$ be a partition with $e$-core $\gamma$ and $d$-runner matrix $\mathcal{R}$ and let $B$ be its canonical $\beta$-set. If it is not $d$-shift skewed, then there is $b\in B$ such that $b-ae\notin B$ and $d\mid\emp_B(b-ae,b)$. The algorithm A1 applied to $B$, $b$ and $b-ae$ returns the canonical $\beta$-set of a partition with the same $e$-core, the same $d$-runner matrix (by \Cref{le:a1-runner}) and either less in size than $\lambda$ or greater in the dominance order than $\lambda$ (by \Cref{le:a1-order}). This shows that $\lambda$ is not a maximal partition in the dominance order in $\mathcal{E}_{\mathcal{R}}(\gamma)$, which proves the result.
\end{proof}

We now repeat the same steps to show an analogous statement for \dunb partitions.

We define the second algorithm \textit{A2}. Let $B$ be a $\beta$-set of a partition, $b$ be its bead and $b-ae$ be its empty space (with $a>0$). The algorithm A2 applied to $B$, $b$ and $b-ae$ proceeds as follows. In turns, for $i= 0,1,2,\dots$ we swap $b-i-ae$ and $b-i$ of $B$ and then if the total number of performed up-moves is congruent to the total number of performed down-moves modulo $d$, we terminate the process. As for A1, at the end the number of performed up-moves either equals the number of performed down-moves or is greater by $d$. An example of this algorithm is in \Cref{fig:A2}.

Again, if the algorithm terminates, the final set is a $\beta$-set of a partition with the same $e$-core as the underlying partition of the initial set $B$. And it is the canonical $\beta$-set if $B$ is. As for A1, we can also perform the swaps runner-wise --- on each runner we propagate a mobile object of the form $b+r-ae$ (for suitable $0<r\leq ae$) of $B$ \emph{upwards}; see \Cref{fig:A2}. As before, we cannot perform both up- and down-moves on any chosen runner, and hence, we can reorder the performed up- and down-moves in any way as long as the relative order of up-moves and the relative order of down-moves remain unchanged.

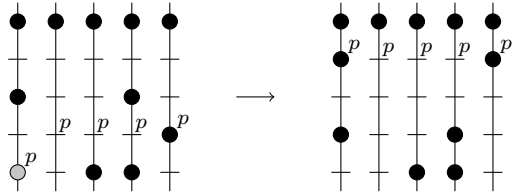
\begin{figure}[h]
	\centering
	\begin{tikzpicture}[x=0.5cm, y=0.5cm]
		\begin{pgfonlayer}{nodelayer}
			\node [style=none] (0) at (1.25, 4.5) {};
			\node [style=none] (1) at (2.25, 4.5) {};
			\node [style=none] (2) at (3.25, 4.5) {};
			\node [style=none] (3) at (4.25, 4.5) {};
			\node [style=none] (4) at (5.25, 4.5) {};
			\node [style=none] (5) at (1.25, -0.5) {};
			\node [style=none] (6) at (2.25, -0.5) {};
			\node [style=none] (7) at (3.25, -0.5) {};
			\node [style=none] (8) at (4.25, -0.5) {};
			\node [style=none] (9) at (5.25, -0.5) {};
			\node [style=Empty node] (11) at (2.25, 4) {};
			\node [style=Empty node] (12) at (3.25, 4) {};
			\node [style=Empty node] (13) at (4.25, 4) {};
			\node [style=Empty node] (15) at (1.25, 4) {};
			\node [style=Empty node] (18) at (3.25, 0) {};
			\node [style=none] (19) at (3, 1) {};
			\node [style=none] (20) at (3.5, 1) {};
			\node [style=none] (23) at (4, 1) {};
			\node [style=none] (24) at (4.5, 1) {};
			\node [style=none] (25) at (3, 2) {};
			\node [style=none] (26) at (3.5, 2) {};
			\node [style=none] (27) at (5.5, 0) {};
			\node [style=none] (28) at (5, 0) {};
			\node [style=none] (33) at (2.5, 0) {};
			\node [style=none] (34) at (2, 0) {};
			\node [style=Empty node] (42) at (5.25, 4) {};
			\node [style=Empty node] (45) at (1.25, 2) {};
			\node [style=Empty node] (46) at (4.25, 2) {};
			\node [style=Empty node] (48) at (5.25, 1) {};
			\node [style=Empty node] (49) at (4.25, 0) {};
			\node [style=none] (50) at (1, 1) {};
			\node [style=none] (51) at (1.5, 1) {};
			\node [style=none] (52) at (5, 2) {};
			\node [style=none] (53) at (5.5, 2) {};
			\node [style=Lead] (54) at (1.25, 0) {};
			\node [style=none] (55) at (7, 2) {};
			\node [style=none] (56) at (8, 2) {};
			\node [style=none] (104) at (2, 1) {};
			\node [style=none] (105) at (2.5, 1) {};
			\node [style=none] (106) at (5, 3) {};
			\node [style=none] (107) at (5.5, 3) {};
			\node [style=none] (108) at (2, 2) {};
			\node [style=none] (109) at (2.5, 2) {};
			\node [style=none] (110) at (1, 3) {};
			\node [style=none] (111) at (1.5, 3) {};
			\node [style=none] (112) at (2, 3) {};
			\node [style=none] (113) at (2.5, 3) {};
			\node [style=none] (114) at (3, 3) {};
			\node [style=none] (115) at (3.5, 3) {};
			\node [style=none] (116) at (4, 3) {};
			\node [style=none] (117) at (4.5, 3) {};
			\node [style=none] (118) at (9.75, 4.5) {};
			\node [style=none] (119) at (10.75, 4.5) {};
			\node [style=none] (120) at (11.75, 4.5) {};
			\node [style=none] (121) at (12.75, 4.5) {};
			\node [style=none] (122) at (13.75, 4.5) {};
			\node [style=none] (123) at (9.75, -0.5) {};
			\node [style=none] (124) at (10.75, -0.5) {};
			\node [style=none] (125) at (11.75, -0.5) {};
			\node [style=none] (126) at (12.75, -0.5) {};
			\node [style=none] (127) at (13.75, -0.5) {};
			\node [style=Empty node] (128) at (10.75, 4) {};
			\node [style=Empty node] (129) at (11.75, 4) {};
			\node [style=Empty node] (130) at (12.75, 4) {};
			\node [style=Empty node] (131) at (9.75, 4) {};
			\node [style=Empty node] (132) at (11.75, 0) {};
			\node [style=none] (133) at (11.5, 1) {};
			\node [style=none] (134) at (12, 1) {};
			\node [style=none] (135) at (12.5, 2) {};
			\node [style=none] (136) at (13, 2) {};
			\node [style=none] (137) at (11.5, 2) {};
			\node [style=none] (138) at (12, 2) {};
			\node [style=none] (139) at (14, 0) {};
			\node [style=none] (140) at (13.5, 0) {};
			\node [style=none] (141) at (11, 0) {};
			\node [style=none] (142) at (10.5, 0) {};
			\node [style=Empty node] (143) at (13.75, 4) {};
			\node [style=Empty node] (144) at (9.75, 3) {};
			\node [style=Empty node] (145) at (12.75, 1) {};
			\node [style=Empty node] (146) at (13.75, 3) {};
			\node [style=Empty node] (147) at (12.75, 0) {};
			\node [style=none] (148) at (9.5, 0) {};
			\node [style=none] (149) at (10, 0) {};
			\node [style=none] (150) at (13.5, 2) {};
			\node [style=none] (151) at (14, 2) {};
			\node [style=none] (153) at (10.5, 1) {};
			\node [style=none] (154) at (11, 1) {};
			\node [style=none] (155) at (13.5, 1) {};
			\node [style=none] (156) at (14, 1) {};
			\node [style=none] (157) at (10.5, 2) {};
			\node [style=none] (158) at (11, 2) {};
			\node [style=none] (159) at (9.5, 2) {};
			\node [style=none] (160) at (10, 2) {};
			\node [style=none] (161) at (10.5, 3) {};
			\node [style=none] (162) at (11, 3) {};
			\node [style=none] (163) at (11.5, 3) {};
			\node [style=none] (164) at (12, 3) {};
			\node [style=none] (165) at (12.5, 3) {};
			\node [style=none] (166) at (13, 3) {};
			\node [style=Empty node] (167) at (9.75, 1) {};
			\node [style=none] (168) at (1.6, 0.35) {$\scriptstyle p$};
			\node [style=none] (169) at (5.6, 1.35) {$\scriptstyle p$};
			\node [style=none] (170) at (4.5, 1.25) {$\scriptstyle p$};
			\node [style=none] (171) at (3.5, 1.25) {$\scriptstyle p$};
			\node [style=none] (172) at (2.5, 1.25) {$\scriptstyle p$};
			\node [style=none] (173) at (10.1, 3.35) {$\scriptstyle p$};
			\node [style=none] (174) at (11, 3.25) {$\scriptstyle p$};
			\node [style=none] (175) at (12, 3.25) {$\scriptstyle p$};
			\node [style=none] (176) at (13, 3.25) {$\scriptstyle p$};
			\node [style=none] (177) at (14.1, 3.35) {$\scriptstyle p$};
		\end{pgfonlayer}
		\begin{pgfonlayer}{edgelayer}
			\draw (0.center) to (5.center);
			\draw (1.center) to (6.center);
			\draw (3.center) to (8.center);
			\draw (4.center) to (9.center);
			\draw (19.center) to (20.center);
			\draw (23.center) to (24.center);
			\draw (25.center) to (26.center);
			\draw (28.center) to (27.center);
			\draw (34.center) to (33.center);
			\draw (2.center) to (7.center);
			\draw (50.center) to (51.center);
			\draw (52.center) to (53.center);
			\draw [style=Move it] (55.center) to (56.center);
			\draw (104.center) to (105.center);
			\draw (106.center) to (107.center);
			\draw (108.center) to (109.center);
			\draw (110.center) to (111.center);
			\draw (112.center) to (113.center);
			\draw (114.center) to (115.center);
			\draw (116.center) to (117.center);
			\draw (118.center) to (123.center);
			\draw (119.center) to (124.center);
			\draw (121.center) to (126.center);
			\draw (122.center) to (127.center);
			\draw (133.center) to (134.center);
			\draw (135.center) to (136.center);
			\draw (137.center) to (138.center);
			\draw (140.center) to (139.center);
			\draw (142.center) to (141.center);
			\draw (120.center) to (125.center);
			\draw (148.center) to (149.center);
			\draw (150.center) to (151.center);
			\draw (153.center) to (154.center);
			\draw (155.center) to (156.center);
			\draw (157.center) to (158.center);
			\draw (159.center) to (160.center);
			\draw (161.center) to (162.center);
			\draw (163.center) to (164.center);
			\draw (165.center) to (166.center);
		\end{pgfonlayer}
	\end{tikzpicture}
	\caption{Let $e=5$ and $d=3$. The picture shows the effect of the algorithm A2 applied to the first $\beta$-set with $b$ given by the highlighted bead and $a=1$. There are $11$ swaps. Four of them are up-moves, performed for $i=0,1,6$ and $10$ and one down-move, performed for $i=2$. One can see that all the swaps can be performed runner-wise. Moreover, on each runner we have a mobile object which propagates upwards during the algorithm performed runner-wise. The propagating mobile objects are denoted by $p$.}
	\label{fig:A2}
\end{figure}

\begin{lemma}\label{le:a2-order}
	Let $B$ be a $\beta$-set of a partition $\lambda$, $b$ be its bead and $b-ae$ be its empty space (with $a>0$). Suppose that each row of the $d$-runner matrix of $B$ contains $0$. Then the algorithm A2 applied to $B$, $b$ and $b-ae$ terminates and returns a $\beta$-set of a partition which is either less in size than $\lambda$ or which has the same size as $\lambda$ and is less in the dominance order. 
\end{lemma}

\begin{proof}
	Observe that for any $i\geq 0$ the mobile object at position $b-ae-i$ of $B$ has not moved before the swap of $b-i-ae$ and $b-i$. Thus, there are only finitely many possible up-moves the algorithm can make as there are only finitely many empty spaces of $B$ less or equal to $b-ae$.
	
	If there is an empty space $b-ae<f<b$ of $B$, it will be as a mobile empty space swapped with $f-ae$, then with $f-2ae$ and so on, and from some point all of these swaps will be down-moves. Hence, in this case, there are infinitely many possible down-moves and A2 terminates.
	
	If not, then $b, b-1, \dots, b-ae+1$ are beads in $B$. They all have the same $d$-emptiness and span all runners of the James abacus with $e$ runners, and hence, by the condition imposed on the $d$-runner matrix, all these beads' $d$-emptinesses are $0$. Thus there are at least $d$ empty spaces $f\leq b-ae$. Since the propagating moving objects at positions $b, b-1, \dots, b-ae+1$ of $B$ are all beads, there will be no down-moves and at least $d$ up-moves, so A2 terminates.
	
	The rest of the proof follows the steps of the proof of \Cref{le:a1-order}. Note that the final sentence has to be changed to: since $U_j$ happened before $D_j$ in the original order of the swaps, $i<0$ in \Cref{le:UD}, and thus $\alpha$ dominates $\gamma$, as required. 
\end{proof} 

Again, the algorithm A2 preserves the $d$-runner matrix, provided that it terminates. This is not difficult to show if we think of the algorithm as a permutation of three intervals of the James abacus of the starting $\beta$-set; see \Cref{fig:PermuteA2}.

\begin{figure}[h]
	\centering
	\begin{tikzpicture}[x=0.5cm, y=0.5cm]
		\begin{pgfonlayer}{nodelayer}
			\node [style=none] (48) at (1.5, -8.5) {};
			\node [style=none] (49) at (-5.5, -6) {};
			\node [style=none] (50) at (8.5, -6) {};
			\node [style=none] (51) at (-5, -6) {};
			\node [style=none] (54) at (8, -6) {};
			\node [style=none] (55) at (5.5, -5.75) {};
			\node [style=none] (56) at (5.5, -6.25) {};
			\node [style=none] (57) at (0.5, -5.75) {};
			\node [style=none] (58) at (0.5, -6.25) {};
			\node [style=none] (59) at (1.5, -5.75) {};
			\node [style=none] (60) at (1.5, -6.25) {};
			\node [style=Empty node] (61) at (7.5, -6) {};
			\node [style=Empty node] (62) at (-4.5, -6) {};
			\node [style=none] (63) at (4.5, -5.75) {};
			\node [style=none] (64) at (4.5, -6.25) {};
			\node [style=none] (65) at (5.55, -6.5) {$\scriptstyle c$};
			\node [style=none] (66) at (6.5, -4.75) {$I_1$};
			\node [style=none] (67) at (3, -4.75) {$I_2$};
			\node [style=none] (68) at (7.5, -6.5) {$\scriptstyle b$};
			\node [style=none] (69) at (2.075, -6.5) {$\scriptstyle b-ae+1$};
			\node [style=none] (70) at (4.5, -6.5) {$\scriptstyle c-1$};
			\node [style=none] (71) at (-3.975, -6.5) {$\scriptstyle b-t-ae$};
			\node [style=none] (72) at (0.25, -6.5) {$\scriptstyle b-ae$};
			\node [style=none] (73) at (-2, -4.75) {$I_3$};
			\node [style=none] (74) at (-5, -7.75) {};
			\node [style=none] (75) at (-5, -4.25) {};
			\node [style=none] (76) at (-4.5, -4.25) {};
			\node [style=none] (77) at (-4.5, -7.75) {};
			\node [style=none] (78) at (0.5, -4.25) {};
			\node [style=none] (80) at (1, -4.25) {};
			\node [style=none] (81) at (0.5, -7.75) {};
			\node [style=none] (82) at (1, -7.75) {};
			\node [style=none] (83) at (1.5, -7.75) {};
			\node [style=none] (84) at (4.5, -7.75) {};
			\node [style=none] (85) at (5, -7.75) {};
			\node [style=none] (86) at (5.5, -7.75) {};
			\node [style=none] (87) at (4.5, -4.25) {};
			\node [style=none] (88) at (5, -4.25) {};
			\node [style=none] (89) at (5.5, -4.25) {};
			\node [style=none] (90) at (7.5, -4.25) {};
			\node [style=none] (91) at (8, -4.25) {};
			\node [style=none] (92) at (7.5, -7.75) {};
			\node [style=none] (93) at (8, -7.75) {};
			\node [style=none] (94) at (1.5, -4.25) {};
			\node [style=none] (95) at (-5.5, -11.75) {};
			\node [style=none] (96) at (8.5, -11.75) {};
			\node [style=none] (97) at (-5, -11.75) {};
			\node [style=none] (98) at (-2, -11.75) {};
			\node [style=none] (99) at (2, -11.75) {};
			\node [style=none] (100) at (8, -11.75) {};
			\node [style=none] (101) at (-4.5, -11.5) {};
			\node [style=none] (102) at (-4.5, -12) {};
			\node [style=none] (103) at (1.5, -11.5) {};
			\node [style=none] (104) at (1.5, -12) {};
			\node [style=none] (105) at (-1.5, -11.5) {};
			\node [style=none] (106) at (-1.5, -12) {};
			\node [style=Empty node] (107) at (-2.5, -11.75) {};
			\node [style=Empty node] (108) at (2.5, -11.75) {};
			\node [style=none] (109) at (7.5, -11.5) {};
			\node [style=none] (110) at (7.5, -12) {};
			\node [style=none] (112) at (-3.5, -10.5) {$I_1$};
			\node [style=none] (113) at (0, -10.5) {$I_2$};
			\node [style=none] (119) at (5, -10.5) {$I_3$};
			\node [style=none] (120) at (-5, -13.5) {};
			\node [style=none] (121) at (-5, -10) {};
			\node [style=none] (122) at (-4.5, -10) {};
			\node [style=none] (123) at (-4.5, -13.5) {};
			\node [style=none] (124) at (-2.5, -10) {};
			\node [style=none] (125) at (-1.5, -10) {};
			\node [style=none] (126) at (-2, -10) {};
			\node [style=none] (127) at (-2.5, -13.5) {};
			\node [style=none] (128) at (-2, -13.5) {};
			\node [style=none] (129) at (-1.5, -13.5) {};
			\node [style=none] (130) at (1.5, -13.5) {};
			\node [style=none] (131) at (2, -13.5) {};
			\node [style=none] (132) at (2.5, -13.5) {};
			\node [style=none] (133) at (1.5, -10) {};
			\node [style=none] (134) at (2, -10) {};
			\node [style=none] (135) at (2.5, -10) {};
			\node [style=none] (136) at (7.5, -10) {};
			\node [style=none] (137) at (8, -10) {};
			\node [style=none] (138) at (7.5, -13.5) {};
			\node [style=none] (139) at (8, -13.5) {};
			\node [style=none] (140) at (-1.5, -10) {};
			\node [style=none] (141) at (1.5, -9.5) {};
			\node [style=none] (142) at (-6.5, -6) {$B$};
			\node [style=none] (143) at (-6.5, -11.75) {$B'$};
			\node [style=none] (144) at (1.2, -9) {$\scriptstyle \psi$};
			\node [style=none] (145) at (-5, -3.5) {};
			\node [style=none] (146) at (1, -3.5) {};
			\node [style=none] (147) at (5, -3.5) {};
			\node [style=none] (148) at (8, -3.5) {};
			\node [style=none] (149) at (6.5, -3) {$r_t+1$};
			\node [style=none] (150) at (3, -3) {$ae-r_t-1$};
			\node [style=none] (151) at (-2, -3) {$t+1$};
		\end{pgfonlayer}
		\begin{pgfonlayer}{edgelayer}
			\draw (54.center) to (50.center);
			\draw (51.center) to (49.center);
			\draw (63.center) to (64.center);
			\draw (59.center) to (60.center);
			\draw (57.center) to (58.center);
			\draw (55.center) to (56.center);
			\draw (51.center) to (54.center);
			\draw (76.center) to (75.center);
			\draw (75.center) to (74.center);
			\draw (74.center) to (77.center);
			\draw (78.center) to (94.center);
			\draw (81.center) to (83.center);
			\draw (80.center) to (82.center);
			\draw (87.center) to (89.center);
			\draw (84.center) to (86.center);
			\draw (88.center) to (85.center);
			\draw (90.center) to (91.center);
			\draw (93.center) to (92.center);
			\draw (91.center) to (93.center);
			\draw (100.center) to (96.center);
			\draw (97.center) to (95.center);
			\draw (109.center) to (110.center);
			\draw (105.center) to (106.center);
			\draw (103.center) to (104.center);
			\draw (101.center) to (102.center);
			\draw (97.center) to (100.center);
			\draw (122.center) to (121.center);
			\draw (121.center) to (120.center);
			\draw (120.center) to (123.center);
			\draw (124.center) to (140.center);
			\draw (127.center) to (129.center);
			\draw (126.center) to (128.center);
			\draw (133.center) to (135.center);
			\draw (130.center) to (132.center);
			\draw (134.center) to (131.center);
			\draw (136.center) to (137.center);
			\draw (139.center) to (138.center);
			\draw (137.center) to (139.center);
			\draw [style=Move it] (48.center) to (141.center);
			\draw [style=measuredots] (145.center) to (146.center);
			\draw [style=measuredots] (146.center) to (147.center);
			\draw [style=measuredots] (147.center) to (148.center);
		\end{pgfonlayer}
	\end{tikzpicture}
	\caption{In the picture, $b$ is a bead of $B$, $b-ae$ is an empty space of $B$ and the other integers can be either beads or empty spaces of $B$. Suppose that the last swap in the algorithm A2 applied to $B$, $b$ and $b-ae$ is the swap of $b-t-ae$ and $b-t$. If $0\leq r_t\leq ae-1$ is the remainder of $t$ modulo $ae$, then the propagating moving object at position $c=b-r_t$ of $B$ moves to $b-t-ae$ during A2. Hence the propagating moving objects lying in interval $I_1$ decrease their positions by $ae+t-r_t$, while those in interval $I_2$ decrease their positions only by $t-r_t$. For instance, in \Cref{fig:A2} where $e=5$, $a=1$ and $t=10$, the propagating moving object on the first runner decreases its position by $15$ while the propagating moving objects on the final four runners decrease their positions only by $10$. Finally, the moving objects in interval $I_3$ increase their positions by $ae$ and the remaining moving objects are not affected by A2. Thus the outcome of A2, set $B'$, can be obtained by applying permutation $\psi$ of the three intervals $I_1, I_2$ and $I_3$ as in the picture.}
	\label{fig:PermuteA2}
\end{figure}
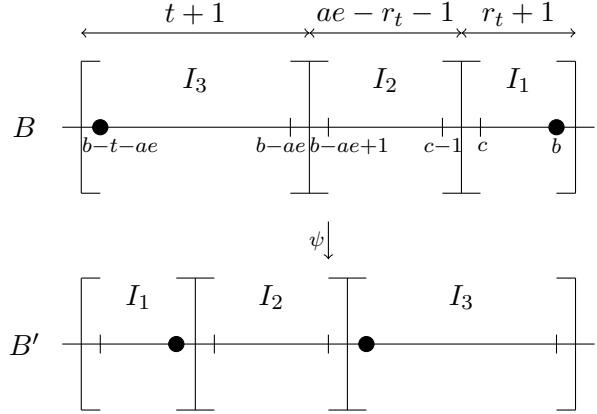

\begin{lemma}\label{le:a2-runner}
	Let $B$ be a $\beta$-set of a partition, $b$ be its bead and $b-ae$ be its empty space of (with $a>0$) such that $d\mid \emp_B(b-ae+1, b)$. Suppose that the algorithm A2 applied to $B$, $b$ and $b-ae$ terminates. Then it preserves the $d$-runner matrix.
\end{lemma}

\begin{proof}
	Denote by $B'$ the set returned by the algorithm A2 applied to $B$, $b$ and $b-ae$. We need to show that the $d$-runner matrices of $B$ and $B'$ agree.
	
	Suppose that A2 terminates with the swap of $b-t-ae$ and $b-t$ (with $t>0$) and, again, for $0\leq i\leq t$ let $0\leq r_i<ae$ be the remainder of $i$ modulo $ae$. We define three intervals of integers
	\begin{align*}
		I_1 &=\left\lbrace v\in \Z: b-r_t\leq v\leq b \right\rbrace,\\
		I_2 &=\left\lbrace v\in \Z: b-ae< v< b -r_t \right\rbrace,\\
		I_3 &=\left\lbrace v\in \Z: b-t-ae\leq v\leq b-ae \right\rbrace,
	\end{align*}
	as in the first diagram in \Cref{fig:PermuteA2}, and write $E_i$ for the number of empty spaces of $B$ in $I_i$ for $i=1,2,3$. The terminating set, call it $B'$, can be obtained from $B$ by simultaneously replacing each $b'\in B$ by $\psi(b')$, where $\psi:\Z \to \Z$ is a bijection given by
	\begin{align}\label{al:intervals2}
		\psi(v) =
		\begin{cases*}
			v - t -ae + r_t, & \text{if $v\in I_1$;}\\
			v - t + r_t, & \text{if $v\in I_2$;}\\
			v + ae, & \text{if $v\in I_3$;}\\
			v, & \text{otherwise,}
		\end{cases*}
	\end{align}
	for any $v\in \Z$; see \Cref{fig:PermuteA2}. Clearly $\psi$ preserves runners (of the James abacus with $e$ runners). To finish the proof we show that it also preserves $d$-emptinesses. From \eqref{al:intervals2} (and \Cref{fig:PermuteA2}) we obtain
	\begin{align*}
		\emp_{B'}(\phi(v)) - \emp_{B}(v)=
		\begin{cases*}
			- (E_2 + E_3), & \text{if $v\in I_1$;}\\
			E_1 - E_3, & \text{if $v\in I_2$;}\\
			E_1 + E_2, & \text{if $v\in I_3$;}\\
			0, & \text{otherwise,}
		\end{cases*}
	\end{align*}
	for any $v\in \Z$. As in the proof of \Cref{le:a1-runner} it suffices to show that $d\mid E_1+E_2$ and $d\mid E_2 + E_3$. The rest of the proof is analogous to the proof of \Cref{le:a1-runner}. In particular, $d\mid E_1 + E_2$ follows from $d\mid \emp_B(b-ae+1,b)$ and $d\mid E_2 + E_3$ follows after verifying that, modulo $d$, the $d$-emptiness of the mobile bead at position $b$ of $B$ increases by $1$ after each down-move, decreases by $1$ after each up-move and does not change otherwise.
\end{proof}

To conclude the first implication of \Cref{pr:equivalence min} we need to do slightly more work than for \Cref{pr:equivalence max} due to the `$l$-restricted' part of the statement.

\begin{corollary}\label{cor:min implies dunb}
	Let $\gamma$ be an $e$-core partition and $\mathcal{R}$ be a $\gamma$-realisable $(d-1)\times e$ matrix of non-negative integers with at least one $0$ in each row, and write $l$ for the least common multiple of $d$ and $e$. Then a minimal partition in the dominance order in $\mathcal{E}_{\mathcal{R}}(\gamma)$ is $l$-restricted and $d$-skewed.
\end{corollary}

\begin{proof}
		We show the contrapositive. Let $\lambda$ be a partition with $e$-core $\gamma$ and $d$-runner matrix $\mathcal{R}$ as in the statement and let $B$ be its canonical $\beta$-set. If it is not $d$-skewed, then there is $b\in B$ such that $b-ae\notin B$ and $d\mid\emp_B(b-ae+1,b)$. The algorithm A2 applied to $B$, $b$ and $b-ae$ returns the canonical $\beta$-set of a partition with the same $e$-core, the same $d$-runner matrix (by \Cref{le:a2-runner}) and either less in size than $\lambda$ or less in the dominance order than $\lambda$ (by \Cref{le:a2-order}). This shows that $\lambda$ is not a minimal partition in the dominance order in $\mathcal{E}_{\mathcal{R}}(\gamma)$, as required.
		
		On the other hand, if $\lambda$ is not $l$-restricted, then there is $b\in B$ such that there is no bead between $b-l$ and $b-1$. But then $d$ divides $\emp_B(b-l, b)=l$, and algorithm A1 applied to $B$, $b$ and $b-l$ returns the canonical $\beta$ set of a partition $\mu$ with the same $e$-core and the same $d$-runner matrix as $\lambda$ (by \Cref{le:a1-runner}). By definition of A1, the algorithm only swaps $b$ and $b-l$ and then terminates, so $\mu$ is less in size than $\lambda$. This shows that $\lambda \notin \mathcal{E}_{\mathcal{R}}(\gamma)$, which finishes the proof.
\end{proof}

\begin{remark}\label{re:preserves}
	The second paragraph of the proof shows that all partitions in $\mathcal{E}_{\mathcal{R}}(\gamma)$ are $l$-restricted. One can also show that if $\lambda$ and $\mu$ are as in the second paragraph of the proof and $\lambda$ is $d$-skewed, then so is $\mu$. 
\end{remark}

To prove the other directions in \Cref{pr:equivalence max} and \Cref{pr:equivalence min} we show that for each $e$-core $\gamma$ and each $d$-runner matrix $\mathcal{R}$ there is at most one partition $\lambda$ which is \dsunb (\emph{or} $l$-restricted and $d$-skewed), has $e$-core $\gamma$ and has $d$-runner matrix $\mathcal{R}$. We use the notation $b\tensor[_B]{\sim}{_{B'}}b'$ to mean that the runner and the $d$-emptiness of bead $b$ of $B$ agree with the runner and the $d$-emptiness of bead $b'$ of $B'$, respectively. If $B$ and $B'$ are clear from the context, we just write $b\sim b'$.

\begin{lemma}\label{le:dsunb unique}
	Let $\gamma$ be an $e$-core partition and $\mathcal{R}$ be a $(d-1)\times e$ matrix of non-negative integers. There is at most one partition $\lambda$ which is $d$-shift skewed, has $e$-core $\gamma$ and has $d$-runner matrix $\mathcal{R}$.
\end{lemma}

\begin{proof}
	Suppose that $\lambda\neq \mu$ are \dsunb partitions with $e$-core $\gamma$ and $d$-runner matrix $\mathcal{R}$. Pick $\beta$-sets $B=B_s(\lambda)$ and $B'=B_s(\mu)$ of $\lambda$ and $\mu$, respectively, such that both sets contain all non-positive integers. Note that both $\beta$-sets have the same number of positive beads on a given runner of a given $d$-emptiness. Pick the least $b$ (necessary positive) such that $b$ lies only in one set $B$ or $B'$, and without loss of generality, assume that $b\in B\setminus B'$.
	
	Then there is $b'\in B'$ such that $b'>b$ and $b\sim b'$. Since $B'$ is $d$-shift balanced, $d$ does not divide $\emp_{B'}(b, b')$. But this equals $\emp_{B'}(b') - \emp_{B'}(b-1) = \emp_{B'}(b') - \emp_B(b-1) = \emp_{B'}(b') - \emp_B(b)$, which is divisible by $d$ by the choice of $b'$, a contradiction.
\end{proof}

Unlike the proof of \Cref{cor:max implies dsunb} (and the proof of the preceding lemmas), the proof of \Cref{le:dsunb unique} does not seem to easily adapt for \dunb partitions. For a bead $b$ of a set $B$, provided that $b$ is not the greatest bead of $B$, we define its \textit{successor} $s(b)$ to be the least bead of $B$ which is greater than $b$. If $b$ is not the least bead of $B$, we define its \textit{predecessor} $p(b)$ to be the greatest bead of $B$ which is less than $b$.

For partitions $\lambda\neq\mu$ with equal $e$-cores and $d$-runner matrices, let $B=B_s(\lambda)$ and $B'=B_s(\mu)$ be their $\beta$-sets such that both of them contain all non-positive integers. We write $B_{\geq 0}$ and $B'_{\geq 0}$ for the sets of non-negative beads of $B$ and $B'$, respectively. We define sequences $(l_i)$ and $(r_i)$ of non-negative beads of $B$ and $(l'_i)$ and $(r'_i)$ of non-negative beads of $B'$ recurrently as follows:

\begin{enumerate}[label=\textnormal{(\roman*)}]
	\item $l_1 = b_0$,
	\item $l'_i=\max\left( \{ b'\in B'_{\geq 0} | l_i\sim b'\} \setminus \{ l'_j | j<i \}\right)   $,
	\item $r'_i = s(l'_i)$ if defined, otherwise, terminate the process here,
	\item $r_i = \max \left( \{ b\in B_{\geq 0} |b\sim r'_i \} \setminus \{ r_j | j<i \}\right)  $,
	\item $l_{i+1} = p(r_i)$.
\end{enumerate}

From (ii), for any $i$ such that $l_i$ and $l'_i$ are both defined, we have $l_i\sim l'_i$. As $\lambda$ and $\mu$ have the same $e$-core and $d$-runner matrix, $B$ and $B'$ have the same number of non-negative beads on a given runner of a given $d$-emptiness; hence $l'_i$ in (ii) is well-defined. Similarly for (iv). Finally, we cannot have $r_i=0$ for any $i$: otherwise, $r'_j=0$ for some $j\leq i$ from (iv) since $0$ is the least non-negative bead of $B$ on runner $0$ with $d$-emptiness $0$, but $r'_j=0$ is not a successor of a non-negative bead of $B'$. Hence (v) yields a non-negative bead. Thus all steps are well-defined. From (ii) and (iv) it is clear that there are no repetitions in $(l'_i)$ and $(r_i)$, respectively. Consequently, using (iii), there are no repetitions in $(r'_i)$ and since $b_0$ is not a predecessor of any bead of $B$, using (i) and (v), there are no repetitions in $(l_i)$ either. In particular, the procedure eventually terminates, say at the $(N+1)$'st iteration, in (iii) when $l'_{N+1}=b'_0$ is the greatest bead of $B'$.

We now specialise to the case when $\lambda$ and $\mu$ are \dunb partitions. An illustration of the above-defined sequences and some more notation introduced in the proof of \Cref{le:same runner} is in \Cref{fig:Unique}.

 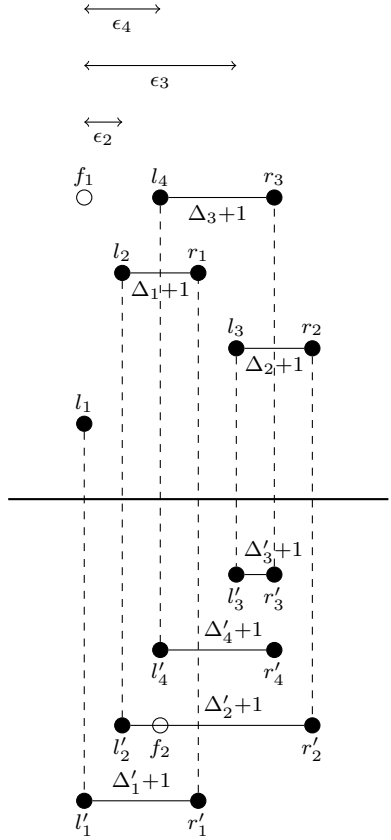
\begin{figure}[h]
 	\centering
 	\begin{tikzpicture}[x=0.5cm, y=0.5cm]
 		\begin{pgfonlayer}{nodelayer}
 			\node [style=Empty node] (0) at (-2, 2) {};
 			\node [style=Empty node] (1) at (-2, -8) {};
 			\node [style=Empty node] (2) at (1, -8) {};
 			\node [style=Empty node] (3) at (1, 6) {};
 			\node [style=Empty node] (4) at (-1, 6) {};
 			\node [style=Empty node] (5) at (-1, -6) {};
 			\node [style=Empty node] (6) at (4, -6) {};
 			\node [style=Empty node] (7) at (4, 4) {};
 			\node [style=Empty node] (8) at (2, 4) {};
 			\node [style=Empty node] (9) at (2, -2) {};
 			\node [style=Empty node] (10) at (3, -2) {};
 			\node [style=Empty node] (11) at (3, 8) {};
 			\node [style=Empty node] (12) at (0, 8) {};
 			\node [style=Empty node] (13) at (0, -4) {};
 			\node [style=Empty node] (14) at (3, -4) {};
 			\node [style=none] (15) at (-4, 0) {};
 			\node [style=none] (16) at (6, 0) {};
 			\node [style=none] (17) at (-0.5, -7.6) {$\scriptstyle \Delta'_1+1$};
 			\node [style=none] (18) at (1.925, -5.6) {$\scriptstyle \Delta'_2+1$};
 			\node [style=none] (19) at (1.925, -3.6) {$\scriptstyle \Delta'_4+1$};
 			\node [style=none] (20) at (2.975, -1.4) {$\scriptstyle \Delta'_3+1$};
 			\node [style=none] (21) at (0, 5.55) {$\scriptstyle \Delta_1+1$};
 			\node [style=none] (22) at (1.5, 7.55) {$\scriptstyle \Delta_3+1$};
 			\node [style=none] (23) at (3, 3.55) {$\scriptstyle \Delta_2 + 1$};
 			\node [style=none] (24) at (-2, 2.55) {$\scriptstyle l_1$};
 			\node [style=none] (25) at (-2, 10) {};
 			\node [style=none] (26) at (-1, 10) {};
 			\node [style=none] (27) at (-2, 11.5) {};
 			\node [style=none] (28) at (2, 11.5) {};
 			\node [style=none] (29) at (-2, 13) {};
 			\node [style=none] (30) at (0, 13) {};
 			\node [style=none] (31) at (-1.5, 9.55) {$\scriptstyle \epsilon_2$};
 			\node [style=none] (32) at (0, 11.05) {$\scriptstyle \epsilon_3$};
 			\node [style=none] (33) at (-1, 12.55) {$\scriptstyle \epsilon_4$};
 			\node [style=none] (35) at (-2, 8.55) {$\scriptstyle f_1$};
 			\node [style=none] (38) at (0, -6.625) {$\scriptstyle f_2$};
 			\node [style=midsize] (41) at (0, -6) {};
 			\node [style=midsize] (42) at (-2, 8) {};
 			\node [style=none] (43) at (-2, -8.625) {$\scriptstyle l'_1$};
 			\node [style=none] (44) at (-1, -6.625) {$\scriptstyle l'_2$};
 			\node [style=none] (45) at (2, -2.625) {$\scriptstyle l'_3$};
 			\node [style=none] (46) at (0, -4.625) {$\scriptstyle l'_4$};
 			\node [style=none] (47) at (1, -8.625) {$\scriptstyle r'_1$};
 			\node [style=none] (48) at (4, -6.625) {$\scriptstyle r'_2$};
 			\node [style=none] (49) at (3, -2.625) {$\scriptstyle r'_3$};
 			\node [style=none] (50) at (3, -4.625) {$\scriptstyle r'_4$};
 			\node [style=none] (51) at (-1, 6.55) {$\scriptstyle l_2$};
 			\node [style=none] (52) at (0, 8.55) {$\scriptstyle l_4$};
 			\node [style=none] (53) at (2, 4.55) {$\scriptstyle l_3$};
 			\node [style=none] (54) at (4, 4.5) {$\scriptstyle r_2$};
 			\node [style=none] (55) at (3, 8.5) {$\scriptstyle r_3$};
 			\node [style=none] (56) at (1, 6.475) {$\scriptstyle r_1$};
 		\end{pgfonlayer}
 		\begin{pgfonlayer}{edgelayer}
 			\draw (1) to (2);
 			\draw (4) to (3);
 			\draw (5) to (6);
 			\draw [style=Extra box] (0) to (1);
 			\draw [style=Extra box] (2) to (3);
 			\draw [style=Extra box] (4) to (5);
 			\draw [style=Extra box] (7) to (6);
 			\draw [style=Extra box] (8) to (9);
 			\draw [style=Extra box] (10) to (11);
 			\draw (9) to (10);
 			\draw (11) to (12);
 			\draw [style=Extra box] (12) to (13);
 			\draw (13) to (14);
 			\draw [style=Border edge] (15.center) to (16.center);
 			\draw (7) to (8);
 			\draw [style=measuredots] (25.center) to (26.center);
 			\draw [style=measuredots] (27.center) to (28.center);
 			\draw [style=measuredots] (29.center) to (30.center);
 		\end{pgfonlayer}
 	\end{tikzpicture}
 	\caption{The diagram illustrates the construction of the first few beads of the sequences of beads introduced before \Cref{le:same runner}. The beads above the bold line belong to a $\beta$-set $B$ and those below belong to a $\beta$-set $B'$. Starting from $l_1=b_0$ one can follow the line to find beads $l_1, l'_1, r'_1, r_1, l_2, \dots, r'_4$. The dashed lines connect beads $b\in B$ and $b'\in B'$ such that $b\sim b'$. The non-dashed lines do not contain any beads (except at the endpoints). The picture does not fit the James abacus with $e$ runners, but for both $B$ and $B'$ we follow the rules `the lower, the greater' and for beads and empty spaces in the same row `the further right, the greater'. The picture also contains members of the sequences $(\Delta_i), (\Delta'_i)$ and $(\epsilon_i)$ from the proof of \Cref{le:same runner} and their interpretation as lengths of lines or distances. The empty spaces $f_1$ and $f_2$ introduced when showing that $\epsilon_i\geq 0$ and $l'_i < l'_j$ are drawn for $i=4$ and $i=4$ and $j=2$, respectively. They would give the desired contradictions if the (open) interval of length $\Delta_3+1$ extended to the runner of $l_1=b_0$ and if $l'_4$ lied below the interval of length $\Delta'_2 + 1$, respectively.}
 	\label{fig:Unique}
 \end{figure}

\begin{lemma}\label{le:same runner}
	Suppose that $\lambda\neq \mu$ are \dunb partitions with the same $e$-core and $d$-runner matrix. If $B$ and $B'$ are as above and the greatest bead $b_0$ of $B$ does not lie in $B'$, then the greatest bead $b'_0$ of $B'$ satisfies $b_0\sim b'_0$.
\end{lemma}

\begin{proof}
	By changing $B$ and $B'$, without loss of generality we can assume that $b_0$ lies on runner $0$. We consider the above four sequences and let $N\geq 0$ be such that $l'_{N+1}=b'_0$ in the above procedure. We define sequences $(\Delta_i)_{1\leq i\leq N}$ by $\Delta_i = \emp_B(l_{i+1}+1,r_i) = r_i - l_{i+1}-1$ and $(\Delta'_i)_{1\leq i\leq N}$ by $\Delta'_i=\emp_{B'}(l'_i+1, r'_i) = r'_i - l'_i-1$. We also let $\epsilon_i = \sum_{j=1}^{i-1} (\Delta'_j - \Delta_j)$ for $1\leq i\leq N+1$.
	
	By simple induction on $i$ and the definitions (i)--(v), we see that for $1\leq i\leq N+1$, the runner of $l'_i$ (and $l_i$) is congruent to $\epsilon_i$ modulo $e$, and for $1\leq i\leq N$, the runner of $r'_i$ (and $r_i$) is congruent to $\epsilon_i + \Delta'_i + 1$ modulo $e$. Similarly, for $i$ in the corresponding range, one shows that
	\begin{align}
	\emp_{B}(l_i)\equiv \emp_{B'}(l'_i) &\equiv \emp_B(b_0) + \epsilon_i\, (\textnormal{mod } d), \textnormal{ and}\label{al:emp1}\\
	\emp_{B}(r_i)\equiv \emp_{B'}(r'_i) &\equiv \emp_B(b_0) + \epsilon_i +\Delta'_i\, (\textnormal{mod } d).\label{al:emp2}
	\end{align}
	We can rewrite two of these congruences as
	\begin{align}
		\emp_B(l_i+1, b_0) &\equiv -\epsilon_i\, (\textnormal{mod } d), \textnormal{ and}\label{al:emp3}\\
		\emp_B(r_i+1, b_0) &\equiv -(\epsilon_i + \Delta'_i)\, (\textnormal{mod } d)\label{al:emp4}.
	\end{align}
	
	We now claim that all $\epsilon_i$ are non-negative. If not, then pick the least $i$ such that $\epsilon_i<0$ (clearly $i>1$). From the definition $\epsilon_i = \epsilon_{i-1} + \Delta'_{i-1} - \Delta_{i-1}$. By the minimal choice of $i$ we get $0< \epsilon_{i-1} + \Delta'_{i-1} +1 \leq \Delta_{i-1}$. Thus $f_1=r_{i-1} - (\epsilon_{i-1} + \Delta'_{i-1} + 1)$ (sketched in \Cref{fig:Unique}) lies between $l_i+1$ and $r_{i-1}-1$ and thus it is an empty space of $B$. It also lies on runner $0$ by the earlier observation that $r_{i-1}$ lies on the runner congruent to $\epsilon_{i-1} + \Delta'_{i-1} + 1$ modulo $e$, and hence $d\nmid \emp_B(f_1+1,b_0)$. But this equals $\emp_B(f_1+1, r_{i-1}) + \emp_B(r_{i-1}+1, b_0) = \epsilon_{i-1} + \Delta'_{i-1} + \emp_B(r_{i-1}+1, b_0)$ which is divisible by $d$ by \eqref{al:emp4}, a contradiction which proves the claim.
	
	Clearly, $\epsilon_i\leq \epsilon_{i-1} + \Delta'_{i-1}$ for all $i\leq N+1$, and thus using the previous claim, a simple induction on $i\leq N$ shows that the union of the intervals $[\epsilon_j, \epsilon_j +\Delta'_j]$ with $j\leq i$ is an interval of the form $[0, m_i]$. Using this, we prove by induction on $i\leq N+1$ that $l'_i\leq l'_1$. If $i=1$, this is clear. For $2\leq i\leq N+1$, we have $0\leq \epsilon_i\leq \epsilon_{i-1} + \Delta'_{i-1}$ and thus the above statement about intervals (for $i-1$) guarantee an index $j<i$ such that $\epsilon_i\in [\epsilon_j, \epsilon_j + \Delta'_j]$. We claim that $l'_i<l'_j$, which in turn yields $l'_i<l'_1$.
	
	If not, then $l'_i>l'_j$ and in turn $l'_i\geq r'_j$. Notice that $\epsilon_i\neq \epsilon_j$ as otherwise, by the earlier observation about the runners of $l'_i$ and $l'_j$ and \eqref{al:emp1}, $l'_i$ and $l'_j$ would lie on the same runner, have the same $d$-emptiness and have $l'_i>l'_j$, which contradicts (ii) in the definition of our four sequences. Thus $f_2=r'_j - (\epsilon_j + \Delta'_j + 1)+ \epsilon_i$ (sketched in \Cref{fig:Unique}) satisfies $l'_j + 1= r'_j-\Delta'_j\leq f_2\leq r'_j-1$ and hence it is an empty space of $B'$. Since $r'_j$ lies on runner congruent to $\epsilon_j + \Delta'_j + 1$ modulo $e$, the empty space $f_2$ of $B'$ lies on the runner congruent to $\epsilon_i$ modulo $e$, which is the runner of $l'_i$. Therefore $d$ does not divide $\emp_{B'}(f_2+1, l'_i)$. This emptiness can be expanded as
	\begin{align*}
	&\emp_{B'}(f_2+1,r'_j) + \emp_{B'}(r'_j+1, l'_i)\\ =\; &\epsilon_j + \Delta'_j - \epsilon_i + \emp_{B'}(l'_i) - \emp_{B'}(r'_j)\\ =\; &(\emp_{B'}(l'_i) - \epsilon_i) - (\emp_{B'}(r'_j) - \epsilon_j - \Delta'_j ).
	\end{align*}
	Both of the brackets are congruent to $\emp_B(b_0)$ modulo $d$ by \eqref{al:emp1} and \eqref{al:emp2}, respectively, and thus we deduce that $d$ divides $\emp_{B'}(f_2+1, l'_i)$, a contradiction proving that $l'_i<l'_j$.
	
	In particular, $b'_0 =l'_{N+1}\leq l'_1$, which forces equality and $N=0$. From (ii) in the definition of our four sequences, $b_0\sim b'_0$, as required.
\end{proof}

We are now ready to deduce the uniqueness result for \dunb partitions, which, unlike its \dsunb analogue, requires an extra $l$-restricted assumption.

\begin{lemma}\label{le:dunb unique}
	Let $\gamma$ be an $e$-core partition and $\mathcal{R}$ be a $(d-1)\times e$ matrix of non-negative integers, and write $l$ for the least common multiple of $d$ and $e$. There is at most one $l$-restricted partition $\lambda$ which is $d$-skewed, has $e$-core $\gamma$ and has $d$-runner matrix $\mathcal{R}$.
\end{lemma}

\begin{proof}
	Suppose this fails and take the lexicographically least $l$-restricted \dunb partition $\lambda$, such that there is a lexicographically smaller $l$-restricted \dunb partition $\mu$ with the same $e$-core and $d$-runner matrix. Then take their $\beta$-sets $B=B_s(\lambda)$ and $B'=B_s(\mu)$, respectively, such that both contain all non-positive integers. As $\lambda$ is lexicographically greater than $\mu$, we have the following inequality between the largest beads of $B$ and $B'$: $b_0\geq b'_0$. If they equal, then $\emp_B(b_0) = \emp_{B'}(b'_0)$, and thus $b_0$ and $b'_0$ can be removed from $B$ and $B'$, respectively, to get $\beta$-sets of smaller $l$-restricted partitions $\widetilde{\lambda}$ and $\widetilde{\mu}$ which still have the same $e$-core, $d$-runner matrix, they are $d$-skewed, and $\widetilde{\lambda}$ is lexicographically larger than $\widetilde{\mu}$, which contradicts the initial choice of $\lambda$.
	
	Thus $b_0>b'_0$, which by \Cref{le:same runner} implies that $b_0\sim b'_0$ and, in particular, $b'_0 = b_0-ae$ for some $a>0$. We can again remove the greatest beads of $B$ and $B'$ to get partitions $\widetilde{\lambda}$ and $\widetilde{\mu}$, both less than $\lambda$ in the lexicographical order. As they are $l$-restricted, $d$-skewed, and their $e$-cores and $d$-runner matrices agree (since $b_0\sim b'_0$), to avoid contradiction with the minimality of $\lambda$, we need $\widetilde{\lambda}=\widetilde{\mu}$.
	
	Consequently, there is no bead of $B$ between $b'_0$ and $b_0-1$. Thus $\emp_B(b'_0+1, b_0) = ae-1$. Using $\widetilde{\lambda}=\widetilde{\mu}$, this can be rewritten as $\emp_B(b_0) - \emp_B(b'_0) = \emp_B(b_0) - \emp_{B'}(b'_0)-1$, which is congruent to $-1$ modulo $d$ as $b_0 \sim b'_0$. Hence $d\mid ae$, and consequently $ae\geq l$, the least common multiple of $d$ and $e$. Since there is no bead between $b'_0=b_0-ae$ and $b_0-1$, we conclude that $\lambda_1-\lambda_2\geq ae\geq l$ (where $\lambda_2=0$ if $\lambda=(\lambda_1)$), a contradiction with $\lambda$ being $l$-restricted.
\end{proof}

We can finally deduce that with respect to the dominance order the sets $\mathcal{E}_{\mathcal{R}}(\gamma)$ always have a unique maximal element and, under the usual mild assumption, also a unique minimal element. In turn, we obtain the propositions from \Cref{se:intro}.

\begin{proposition}\label{le:dominance}
	Let $\gamma$ be an $e$-core partition and $\mathcal{R}$ be a $\gamma$-realisable $(d-1)\times e$ matrix of non-negative integers. In $\mathcal{E}_{\mathcal{R}}(\gamma)$, there is a unique maximal element in the dominance order. If there is $0$ in each row of $\mathcal{R}$, then there is also a unique minimal element in the dominance order.
\end{proposition}

\begin{proof}
	The first part follows from \Cref{cor:max implies dsunb} and \Cref{le:dsunb unique}. The second part follows from \Cref{cor:min implies dunb} and \Cref{le:dunb unique}.
\end{proof}

\begin{proof}[Proof of \Cref{pr:equivalence max}]
	This follows from \Cref{cor:max implies dsunb}, \Cref{le:dsunb unique} and \Cref{le:regularity}.
\end{proof}

\begin{proof}[Proof of \Cref{pr:equivalence min}]
	This follows from \Cref{cor:min implies dunb} and \Cref{le:dunb unique}.
\end{proof}

Let us remark that our proofs provide a way to find the maximal (and the minimal, under the usual assumption) element of $\mathcal{E}_{\mathcal{R}}(\gamma)$ starting from any partition $\lambda$ with $e$-core $\gamma$ and $d$-runner matrix $\mathcal{R}$. Indeed, we can start with the canonical $\beta$-set $B$ of $\lambda$, and then, whenever there is a bead $b$ of $B$ and an empty space $b-ae$ of $B$ (with $a>0$) such that $d\mid \emp_B(b-ae,b)$, we apply A1 with $B$, $b$ and $b-ae$ and update $B$. By \Cref{le:a1-order}, A1 can be performed only finitely many times and by \Cref{pr:equivalence max}, at the end we get the canonical $\beta$-set of the maximal element of $\mathcal{E}_{\mathcal{R}}(\gamma)$. The minimal element is found analogously: one applies A2 instead of A1 with a small caveat; it may happen then during the procedure we reach the canonical $\beta$-set $B$ of a \dunb partition that is not $l$-restricted (with $l$ being the least common multiple of $d$ and $e$). Then one can apply A1 to $B$, $b$ and $b-l$, where $b\in B$ and there is no bead between $b-l$ and $b-1$ (as in the second paragraph of the proof of \Cref{cor:min implies dunb}). Note that for the maximal element we can also use the algorithm from \Cref{le:realisable suf}; see \Cref{re:greedy}. 

We end this section by addressing the assumptions in \Cref{pr:equivalence min}. While there is always a minimal element of $\mathcal{E}_{\mathcal{R}}(\gamma)$ with respect to the dominance order (for $\mathcal{R}$ $\gamma$-realisable), by the next lemma this is not the case for \dunb partitions with given $e$-core $\gamma$ and $\gamma$-realisable $d$-runner matrix. The lemma shows that \Cref{pr:equivalence min} fails if we omit the constraint on the $d$-runner matrix.

\begin{lemma}\label{le:nonexistence}
	Let $\mathcal{R}$ be a $(d-1)\times e$ matrix of non-negative integers with a row of positive integers. Then there is no \dunb partition with $d$-runner matrix $\mathcal{R}$.
\end{lemma}

\begin{proof}
	Suppose that row $x$ of $\mathcal{R}$ (with $1\leq x\leq d-1$) has only positive entries. Let $B$ be the canonical $\beta$-set of a partition with $d$-runner matrix $\mathcal{R}$. Take the $x$'th least empty space $f$ of $B$. Since each runner of $B$ has a bead with $d$-emptiness $x$, which implicitly must be greater than $f$, there is $a>0$ such that $f+ae$ is a bead of $B$ with $d$-emptiness $x$. Then $\emp_B(f+1, f+ae) = \emp_B(f+ae) - x$ and this is divisible by $d$. Therefore, $B$ is not a \dunb set.
\end{proof}

Similarly, provided that $d$ and $e$ are not coprime, \Cref{pr:equivalence min} fails if we omit `$l$-restricted' from its statement --- partition $(l)$ is $d$-skewed (the arm lengths of $e$-divisible hooks of $(l)$ are all congruent to $-1$ modulo $e$, hence are not divisible by $d$) with $e$-core $\gamma$ given by the empty partition and $d$-runner matrix $\mathcal{R}$ given by the zero matrix, but is not the minimal element of $\mathcal{E}_{\mathcal{R}}(\gamma) = \left\lbrace \gamma\right\rbrace$. However, if $d$ and $e$ are coprime then `$l$-restricted' is not needed in \Cref{pr:equivalence min} since then any $d$-skewed partition is automatically $l$-restricted, as shown in the next lemma. In turn, `$e$-restricted' can be omitted from \Cref{th:Mullineuxbalanced}.

\begin{lemma}\label{le:e-restricted}
	Suppose that $d,e>1$ are coprime integers and $\lambda$ is a $d$-skewed partition and let $l=de$. Then $\lambda$ is $l$-restricted.
\end{lemma}

\begin{proof}
	Suppose for contradiction that $\lambda$ is not $l$-restricted. Then its $\beta$-set $B$ has a bead $b$ such that there is no bead between $b-de$ and $b-1$. Let $0<a<d$ be an inverse of $e$ modulo $d$. Then $d$ divides $\emp_B(b-ae+1,b) = ae-1$, a contradiction with $\lambda$ being $d$-skewed.
\end{proof}

For general integers $d,e> 1$, by the final sentence of \Cref{re:preserves}, all $d$-skewed partitions with given $e$-core $\gamma$ and $\gamma$-realisable $d$-runner matrix $\mathcal{R}$ must have the form $(\mu_1 +a_1l, \mu_2+a_2l, \dots, \mu_k +a_kl)$, where $\mu=(\mu_1,\mu_2,\dots, \mu_t)$ is the minimal element of $\mathcal{E}_{\mathcal{R}}(\gamma)$, $a_1, a_2,\dots, a_k$ are non-negative integers (where $\mu_i = 0$ for $i>t$) and $l$ is the least common multiple of $d$ and $e$. We omit the proof here but, in fact, a partition of this form is $d$-skewed if and only if $a_i=0$ for $i\geq g$, the greatest common divisor of $d$ and $e$, which reproves that `$l$-restricted' is not needed in \Cref{pr:equivalence min} when $d$ and $e$ are coprime.

We thus obtain a full classification of $d$-skewed partitions: partition $\lambda$ with $e$-core $\gamma$ and $d$-runner matrix $\mathcal{R}$ is \dunb if and only if $\mathcal{R}$ contains $0$ in each row and $\lambda = (\mu_1 +a_1l, \mu_2+a_2l, \dots, \mu_k +a_kl)$, where $\mu$ is the minimal element of $\mathcal{E}_{\mathcal{R}}(\gamma)$ and $a_1, a_2,\dots, a_k$ are non-negative integers such that $a_i=0$ whenever $i\geq g$.

\subsection*{Acknowledgements}
The author would like to thank Mark Wildon for careful reading of two version of the manuscript, his insightful comments and for suggesting the initial project which eventually led to this paper, David Hemmer for ideas which helped significantly generalise the results, \'{A}lvaro Guti\'{e}rrez for helpful discussions about \Cref{le:same runner}, Steffen K\"{o}nig and Stacey Law for suggesting valuable ideas for improvements to the introduction, Bim Gustavsson for pointing out a couple of typos in the manuscript, and everyone else whose interesting comments and questions led to improvements and generalisations of the current version of the manuscript.  

%
%

\bibliographystyle{abbrv}
\bibliography{MSNrefsMul}
\end{document}